\documentclass[11pt]{amsart}

\usepackage[pagewise]{lineno}\nolinenumbers

\usepackage[english]{babel}

\usepackage[margin=1in]{geometry}

\usepackage{amsmath, amssymb, amsthm, esint,   mathtools, mathrsfs, bbm}
\usepackage[T1]{fontenc}
\usepackage[latin9]{inputenc}
\usepackage{geometry}
\usepackage{graphicx}
\usepackage{enumitem}
\usepackage{lmodern}
\usepackage[colorlinks=true, allcolors=blue]{hyperref}
\usepackage{cancel}
\usepackage{xcolor}
\usepackage{bbm}
\usepackage[normalem]{ulem}
\newcommand{\cred}{\color{red}}
\newcommand{\cpurp}{\color{purple}}

\numberwithin{equation}{section}
\numberwithin{figure}{section}

\theoremstyle{plain}
\newtheorem{thm}{\protect\theoremname}
\theoremstyle{definition}
\newtheorem{defn}{\protect\definitionname}
\theoremstyle{remark}
\newtheorem{rem}{\protect\remarkname}
\theoremstyle{plain}
\newtheorem{lem}{\protect\lemmaname}
\theoremstyle{plain}
\newtheorem{cor}{\protect\corollaryname}
\theoremstyle{plain}
\newtheorem{prop}{\protect\propositionname}
\usepackage{babel}
\providecommand{\corollaryname}{Corollary}
\providecommand{\definitionname}{Definition}
\providecommand{\lemmaname}{Lemma}
\providecommand{\remarkname}{Remark}
\providecommand{\theoremname}{Theorem}
\providecommand{\propositionname}{Proposition}
\numberwithin{thm}{section}
\numberwithin{defn}{section}
\numberwithin{rem}{section}
\numberwithin{lem}{section}
\numberwithin{cor}{section}
\numberwithin{prop}{section}
\title[Stochastic Damped Compressible Euler Equations: Existence and Long-Time Behavior]{Stochastic Compressible Euler Equations with Frictional Damping: Existence of $L^\infty$ Martingale Solutions and Asymptotic Porous Medium-Like Behavior}
\author[R. Dai, J. Kuan, K. Tawri, S. \v{C}ani\'{c}, K. Trivisa]{Rongyi Dai$^1$, Jeffrey Kuan$^2$, Krutika Tawri$^3$, Sun\v{c}ica \v{C}ani\'{c}$^1$, Konstantina Trivisa$^2$}

\address{	\newline
$^1$ Department of Mathematics, University of California Berkeley, CA, USA.\\
\newline
$^2$ Department of Mathematics, University of Maryland, MD, USA.\\
\newline
$^3$  Department of Applied Mathematics, University of Washington, WA, USA.}
	
\email{amy\_dai@berkeley.edu (Rongyi Dai), jeffreyk@umd.edu (Jeffrey Kuan), ktawri@uw.edu (Krutika Tawri), canics@berkeley.edu (Sun\v{c}ica \v{C}ani\'{c}), trivisa@umd.edu (Konstantina Trivisa)}

\begin{document}
\begin{abstract}
  
 \if 1 = 0
 We consider the isentropic compressible Euler equations with linear frictional damping driven by a multiplicative white-in-time stochastic forcing. We establish global-in-time existence of $L^\infty$ martingale solutions, satisfying an entropy inequality, for this problem on bounded one-dimensional domains, with $L^\infty$ initial data and with Dirichlet boundary conditions imposed on the momentum. We further prove that the weak solutions converge exponentially fast to a steady state of the system. In addition, time asymptotically, the density in this stochastic problem satisfies the deterministic porous media equations while the momentum obeys Darcy's law. 
      These results are delicate in the stochastic setting, since random temporal perturbations can significantly influence the long-time statistics of the system.
      Our approach relies on carefully deriving moment estimates for the entropy to demonstrate the decay of stochastic effects.
      
\fi
 {
      We study the one-dimensional isentropic compressible Euler equations with linear (frictional) damping, subject to multiplicative, white-in-time stochastic forcing. The system is posed on a bounded interval with $L^\infty$ initial data and Dirichlet boundary conditions imposed on the momentum.
We establish the global-in-time existence of $L^\infty$ martingale solutions that satisfy an appropriate entropy inequality. Then, we analyze the long-time behavior of these solutions and show that, under suitable assumptions on the noise, they converge almost surely and exponentially fast to a constant steady state of the system. The limiting density is well-approximated by the asymptotic solution of the deterministic porous medium equation, while the momentum exhibits the asymptotic behavior predicted by Darcy's law.
The analysis in the stochastic setting is delicate, as temporal white-noise perturbations can significantly influence the long-time statistics of the solution. Our approach hinges on deriving sharp moment estimates for the entropy, which enable us to quantify and ultimately prove the decay of stochastic effects.
To the best of our knowledge, this work provides the first rigorous pathwise convergence result
for the long-time behavior of solutions to the stochastic isentropic compressible Euler equations
with linear damping.
}

\end{abstract}
	\maketitle

	\section{Introduction}\label{sec_introduction}
    \subsection{Problem Definition}
		{
        We consider the one-dimensional isentropic compressible Euler equations with linear damping,
        perturbed by multiplicative white-in-time noise. The system describes the evolution of the
        density $\rho$ and the momentum $m = \rho u$, where $u$ denotes the fluid velocity, and is
        governed by the stochastic partial differential equations:
        }
	\begin{equation}\label{problem}
		\begin{cases}
			d \rho + \partial_x m \,dt = 0\\
			d m + \partial_x (\frac{m^2}{\rho} + p(\rho))\,dt = -\alpha m \,dt + \sigma(x, \rho, m)\,dW(t)\\
            \end{cases}
            \quad \text{ for } t\geq 0, \; 0\leq x\leq 1.
        \end{equation}
        The first equation (the continuity equation) expresses conservation of mass, while the second
        equation (the momentum equation) describes the balance of momentum. The parameter $\alpha > 0$ 
        in the momentum equation denotes the frictional damping coefficient. 
        The function $p = p(\rho)$ is pressure 
        $p(\rho) : [0,\infty) \to \mathbb{R}$, which is given by the $\gamma$--law:
        \begin{equation*}
        p(\rho) = \kappa \rho^{\gamma}, \qquad \text{ where } \kappa = \frac{\theta^{2}}{\gamma}, \quad \theta = \frac{\gamma - 1}{2}, \quad \text{ and } \gamma > 1,
        \end{equation*}
        describing a polytropic, perfect gas. Here,  $\gamma > 1 $ is known as the adiabatic gas exponent.
        We supply problem \eqref{problem} with initial conditions for the density and momentum
	\begin{equation}\label{InitialCond}
            \rho(0, x) = \rho_0(x), \qquad m(0, x) = m_0(x), \qquad \text{ for } x \in [0, 1],
        \end{equation}
        where $\displaystyle \int_0^1 \rho_0(x) dx := \rho_* > 0$,
        and we impose Dirichlet boundary conditions for the momentum at the boundary of the spatial domain:
        {\begin{equation}\label{BC}
			m(t, 0) = m(t, 1) = 0, \qquad \text{ for all } t \ge 0. 
	\end{equation}}
	
   The system \eqref{problem} can be written in terms of the vector state variable $\displaystyle U = \begin{pmatrix} \rho \\ m \\ \end{pmatrix}$ as follows:
	\begin{equation}\label{problem_matrix}
		\begin{cases}
			dU + \partial_x F(U) \,dt = G(U)\,dt +\Phi(x, U)\,dW(t),& \text{ for } t\geq 0,\; \; 0\leq x\leq 1 \\
			U(0, x) = U_0(x),& x \in [0, 1]\\
			m(t, 0) = m(t, 1) = 0, & t \in [0, \infty)
		\end{cases}
	\end{equation}
	where 
	\begin{equation}\label{defn_FGPhi}
	    F(U) := \begin{pmatrix}m\\ \frac{m^2}{\rho} + p(\rho)
	\end{pmatrix}, \;\; G(U):= \begin{pmatrix}
		0 \\ -\alpha m
	\end{pmatrix}
	, \;\; \Phi(x, U) :=  \begin{pmatrix}
		0\\ \sigma(x, U)
	\end{pmatrix}.
	\end{equation}
    In this work, we will be assuming that our initial data is in $L^\infty(0,1)$.
   
    \medskip

    \noindent \textbf{Description of the stochastic noise.} 
        We model the random external forcing in the momentum equation by a multiplicative
        white-in-time noise. To formalize this, let $(\Omega, \mathcal{F}, \mathbb{P})$ be a
        probability space equipped with a filtration $(\mathcal{F}_{t})_{t \in [0,T]}$ satisfying the usual
        conditions (i.e. right continuous and such that $\mathcal{F}_0$ is complete). Let $W$ denote a one-dimensional, real-valued $(\mathcal{F}_{t})_{t \in [0,T]}$-Brownian
        motion defined on this filtered probability space. The stochastic forcing enters the momentum
        equation \eqref{problem} through a noise coefficient $\sigma(x,\rho,m)$, which depends
        nonlinearly on the spatial variable and the state variables of the system.
        We assume that $\sigma$ is Lipschitz continuous in the state variables,  namely
        \begin{equation}\label{noise1}
            |\nabla_{\rho,m}\sigma(x,\rho,m)| \leq \sqrt{A_0}, \qquad \text{for some constant } A_0 > 0,
        \end{equation}
        and that it satisfies
     \begin{equation}\label{noise2}
	 \sigma(x, \rho, 0) = 0 \text{ for all } \rho > 0, \;\;\;\;\; \sigma(x, 0, m) = 0 \text{ for all } m \in \mathbb{R}.
     \end{equation}
We remark that the Lipschitz continuity assumption {with respect to} the state variables, along with $\sigma(\cdot, 0, 0)=0$, is common in literature -- see, e.g. \cite{MR3021756, bauzet_dirichlet_2014} in the context of stochastic scalar conservation laws, and \cite{breit_local_2018, MR4078230, MR3921105} in the context of stochastic compressible Euler equations and the Navier Stokes equation.

\medskip

\noindent \textbf{Entropy.} System \eqref{problem_matrix} admits a family of convex weak entropy--entropy flux pairs
$(\eta, H)$, which play a central role in the selection of physically relevant weak solutions, see \cite{perthame_kinetic_2003}. 
More specifically,
for any convex function $g \in C^{2}(\mathbb{R})$, we associate a convex entropy--flux pair
$(\eta, H)$ defined by
	\begin{equation}\label{entropy_pair_formula}
		\begin{split}
			\eta(U) &= \int_{\mathbb{R}} g(\xi)\chi(\rho, \xi-u)\,d\xi = \rho \int_{-1}^{1}g(u + z\rho^\theta)(1-z^2)\,dz,\\
			H(U) &=  \int_{\mathbb{R}} g(\xi)[\theta \xi + (1-\theta)u]\chi(\rho, \xi-u)\,d\xi = \rho \int_{-1}^{1}g(u + z\rho^\theta)(u + z\theta\rho^\theta)(1-z^2)\,dz,
		\end{split}
	\end{equation}
	where $\chi(U) = c_\lambda(\rho^{2\theta} - u^2)^\lambda_+$, $\lambda = \frac{3-\gamma}{2(\gamma-1)}$, and $c_\lambda = \left(\int_{-1}^{1}(1-z^2)^\lambda_+\right)^{-1}$. 
    
\medskip    
    A brief calculation shows that:
    \begin{itemize}
    \item for $g(\xi) = 1$, we get $\eta(U) = \rho$;
    \item for $g(\xi) = \xi$, we get $\eta(U) = m $;
    \item  for  $ g(\xi)= \frac{1}{2}\xi^2$, we get
        \begin{equation}\label{defn_etaE}
            \eta(U) = \frac{1}{2}\frac{m^2}{\rho} + \frac{\kappa}{\gamma - 1}\rho^\gamma =: \eta_E(U),
        \end{equation}
    which corresponds to the {\emph{mechanical energy}} of the system.
    \end{itemize}

To formally derive the entropy formulation for this stochastic system, we apply It\^{o}'s formula
to the composition $U \mapsto \eta(U)$, where $(\eta,H)$ is a convex entropy--entropy flux pair
associated with \eqref{problem_matrix}. Using the compatibility relation
\[
\nabla H(U) \;=\; \nabla \eta(U)^{\top} \nabla F(U),
\]
we obtain
\begin{equation}\label{ito_entropy}
\begin{aligned}
    d\eta(U)
    + \partial_x H(U)\,dt
    = - \alpha m \,\partial_m \eta(U)\,dt
    + \partial_m \eta(U)\,\sigma(x,U)\,dW(t)  
     + \frac{1}{2}\,\sigma^2(x,U)\,\partial_m^2\eta(U)\,dt .
\end{aligned}
\end{equation}

\noindent
To capture entropy dissipation across shocks, 
the entropy balance \eqref{ito_entropy} is replaced by the
following entropy \emph{inequality} \cite{diperna_convergence_1983} :
\[
\begin{aligned}
    d\eta(U)
    + \partial_x H(U)\,dt
    + \alpha m\,\partial_m\eta(U)\,dt
    - \partial_m \eta(U)\,\sigma(x,U)\,dW(t)
    - \frac{1}{2}\,\partial_m^2\eta(U)\,\sigma^2(x,U)\,dt
    \;\le\; 0,
\end{aligned}
\]
which holds in the sense of distributions.

\smallskip

\noindent
This formal computation guides the definition of a martingale solution: a probabilistically weak
solution in which both the solution and the underlying stochastic basis are part of the
unknowns. More precisely, we will be working with the following martingale $L^\infty$ weak entropy solutions.

{
    \begin{defn}[{Martingale $L^\infty$ Weak Entropy Solution}]\label{martingale_soln_defn}
    A martingale $L^\infty$ weak entropy solution to \eqref{problem} with deterministic initial data $U_0=(\rho_0,m_0)$ is a multiplet $({\Omega}, {\mathcal{F}}, {\mathbb{P}}, ({\mathcal{F}}_t)_{t\geq 0}, {W}, {U})$, where $({\Omega}, {\mathcal{F}}, {\mathbb{P}}, ({\mathcal{F}_t})_{t\geq 0})$ is a filtered probability space, ${W}$ is an $({\mathcal{F}_t})_{t \geq 0}$-adapted Wiener process, and ${U}$ is an $({\mathcal{F}_t})_{t \geq 0}$-adapted predictable process such that: 
        \begin{enumerate}
            \item {$U \in L^\infty((0,\infty)\times(0,1))\cap  C_{w,loc}(0, \infty; L^2(0,1))\cap C_{loc}(0, \infty; H^{-2}(0,1)) $}, ${\mathbb{P}}$-almost surely.
            \item {$U$ is $({\mathcal{F}}_t)_{t\geq 0}-$adapted.}
            \item For every weak entropy-entropy flux pair $(\eta, H)$ defined in \eqref{entropy_pair_formula} {corresponding to any convex function $g \in C^{2}(\mathbb{R})$}, and for all nonnegative test functions $\varphi \in C^1_c((0,1))$ and {$\psi \in C^1_c([0, \infty))$, we have for almost every $t \geq 0$:}
		\begin{equation}\label{entropy_ineq}
			\begin{split}
				&\langle \eta(U_0), \varphi\rangle \psi(0) - \langle \eta(U(t)), \varphi\rangle \psi(t)  + \int_{0}^{t} \langle \eta({U})(s), \varphi\rangle\psi'(s)\,dx + \int_0^t\langle H({U})(s), \partial_x \varphi\rangle \psi(s)\,ds\\
                &- \int_0^t \langle \alpha {m}\partial_m\eta({U}),\varphi\rangle \psi(s)\,ds + \int_{0}^{t} \frac{1}{2}\langle \sigma^2(x, {U})\partial^2_m \eta({U}), \varphi\rangle \psi(s) \,ds\\
                & + \int_{0}^{t} \langle \sigma(x, {U})\partial_m \eta({U}), \varphi\rangle \psi(s)\,d{W}(s) \geq 0,\;\; \text{ ${\mathbb{P}}$-almost surely.}
			\end{split}
		\end{equation}
        \item The momentum $m$ satisfies Dirichlet boundary condition {\cpurp \eqref{BC}} in a weak distributional sense; see Appendix \ref{appendix_bc} for details.
        \end{enumerate}
    \end{defn}
    \medskip
    \begin{rem}
        The entropy inequality \eqref{entropy_ineq} in Item 2 of Definition \ref{martingale_soln_defn} implies that the following weak formulation holds true for every deterministic test function $\varphi \in C_c^1((0,1))$ and $\psi \in C_c^1([0, \infty))$, $\mathbb{P}$-almost surely for every $t \geq 0$:
        \begin{equation}\label{weak_form}
            \begin{split}
				&\langle U_0, \varphi\rangle \psi(0) - \langle U(t), \varphi\rangle \psi(t)  + \int_{0}^{t} \langle {U}(s), \varphi\rangle\psi'(s)\,dx + \int_0^t\langle F({U}(s)), \partial_x \varphi\rangle \psi(s)\,ds\\
                &- \int_0^t \langle G({U}(s)),\varphi\rangle \psi(s)\,ds + \int_{0}^{t} \langle \Phi({U}(s)), \varphi\rangle \psi(s)\,d{W}(s) =0.
			\end{split}
        \end{equation}
      This weak formulation \eqref{weak_form} is obtained by taking $g(\xi) = \pm 1$ in \eqref{entropy_pair_formula}, which gives us $\eta(U) = \pm \rho$, and $g(\xi) = \pm \xi$, which gives us $\eta(U) = \pm m$. These choices of $g$ are admissible in the entropy inequality because their second derivatives are zero, so that both the positive and negative versions of $g$ are convex, which then implies that the associated entropy function $\eta$ is also convex.
    \end{rem}
    {
    \begin{rem}
    In Item 2, we require the entropy inequality to hold for any entropy-flux pair $(\eta, H)$ generated by the formulas \eqref{entropy_pair_formula} via a convex function $g \in C^{2}(\mathbb{R})$. We note that more specific conditions are required in standard works on compressible isentropic Euler equations, see \cite{berthelin_stochastic_2019, MR1383202}, where the entropy inequality is only required to hold for \textbf{subquadratic} convex functions $g \in C^{2}(\mathbb{R})$, satisfying:
    \begin{equation*}
    |g(\xi)| \le C(1 + |\xi|^{2}), \qquad |g'(\xi)| \le C(1 + |\xi|),
    \end{equation*}
    for some positive constant $C > 0$. However, we are able to omit the subquadratic growth condition on the function $g \in C^{2}(\mathbb{R})$, due to the compact support properties of the noise coefficient that we will use later in \eqref{noise_assumption}. In particular, these noise coefficient assumptions will allow us to obtain $L^{\infty}((0, \infty) \times (0, 1))$ solutions in space and time, so that $\rho$ and $m$ are hence bounded pointwise almost everywhere.
    \end{rem}
    }
    
    \begin{rem}\label{rem_boundary_condn}
         Item 3 in Definition \ref{martingale_soln_defn} is not standard as it imposes Dirichlet boundary condition on the momentum, which takes values merely in $L^\infty((0,\infty)\times (0,1))$.  This argument justifying the existence of aforementioned trace  is dependent on $U$ satisfying the weak formulation \eqref{weak_form}. The idea of the proof is borrowed from previous works such as 
        \cite{heidrich_global_1994} and \cite{pan_initial_2008}.
We refer the reader to Appendix \ref{appendix_bc}, where we provide the relevant details for completeness of the exposition.

    \end{rem}
    }

{\subsection{Summary of Main Results.}}
We show two main results related to martingale $L^\infty$ weak entropy solutions to \eqref{problem}: (1) global-in-time existence and (2) long-time behavior. More specifically, we show that the asymptotics of any martingale $L^\infty$ weak entropy solution to \eqref{problem} is almost surely well-approximated by an asymptotic solution of the  {\emph{deterministic}} porous medium problem \eqref{eqn_PME} below.
These two results are summarized in Result I and Result II below. 

\medskip
\noindent
{\bf{Result I: Existence.}} The following theorem states the first main result of this manuscript, which is the global-in-time existence of  a martingale $L^\infty$ weak entropy solution to \eqref{problem}.

\medskip
\begin{thm}[{\bf{Existence}}] \label{main_theorem} Assume that the initial data $U_0=(\rho_0, m_0)$ satisfies 
		\begin{equation}\label{init_cond}
			\begin{split}
				0\leq \rho_0(x) \leq M_1, \;\; |m_0(x)| \leq { M_2\rho_0(x)}, \text{ for some constant }M_1, M_2 > 0,
			\end{split}
		\end{equation}
		and that the noise coefficient $\sigma(x, U)$ satisfies {the following conditions:}
		\begin{equation}\label{noise_assumption}
			\begin{split}
				&\bullet \text{Compact support: } \text{supp } \sigma(x, \cdot, \cdot) \subset [0,M_1] \times [-M_1M_2, { M_1M_2}], \\
				&\bullet \text{Behavior at vanishing state variables: }  \sigma(x, \rho, 0) = 0, \;\;\forall \rho >0; \sigma(x, 0, m) = 0, \;\; \forall m\in \mathbb{R},\\
				&\bullet \text{Lipschitz continuity: }|\nabla_{\rho, m}\sigma(x, \rho, m) | \leq \sqrt{A_0}, \;\;\text{for some constant } A_0 > 0,\\
				&\bullet \text{Boundary behavior: } \sigma (0, \rho, m) = \sigma(1, \rho, m) = 0.			
			\end{split}
		\end{equation}
		Then there exists a global-in-time martingale $L^\infty$ weak entropy solution in the sense of Definition \ref{martingale_soln_defn} to \eqref{problem} with initial data $U_0$. 
	\end{thm}

\begin{rem} 
The assumption 
\[
|m_0(x)| \le M_2\, \rho_0(x)
\]
implies that at every point where $\rho_0(x)\neq 0$, the initial velocity is bounded:
$
|u_0(x)| := \frac{|m_0(x)|}{\rho_0(x)} \le M_2.
$
Moreover, the initial kinetic energy density remains bounded:
\[
\frac{m_0(x)^2}{\rho_0(x)} 
\le \frac{M_2^2\, \rho_0(x)^2}{\rho_0(x)} 
= M_2^2\, \rho_0(x) 
\le M_2^2 M_1.
\]
Consequently, the total initial energy is uniformly bounded on all non-vacuum regions.
\end{rem}

\begin{rem}
The assumption of compact support for the noise has appeared previously in the literature on stochastic conservation laws, particularly in the context of bounded solutions; see, for example, the work of Bauzet~\cite{bauzet_time-splitting_2014}. We also emphasize that this assumption is reasonable from a physical perspective.
Indeed, the system~\eqref{problem} is closely related to, e.g., one-dimensional hyperbolic conservation law models for blood flow in elastic arteries~\cite{CanicKim}. In this interpretation, the density corresponds to the arterial cross-sectional area, while the momentum represents the momentum of averaged flow quantities, such as the cross-sectionally averaged fluid velocity. Assuming that the magnitude of the stochastic forcing -- modeling e.g. random fluctuations in heart contractions -- depends on the flow momentum and vanishes outside a compact set, amounts to stating that beyond physiologically meaningful regimes of arterial wall deformation and blood velocity, stochastic effects become irrelevant. 
In our setting, this assumption -- together with the $L^\infty$ bound on the 
initial data -- is essential for deriving global-in-time $L^\infty(0,1)$ 
estimates on the solution.
\end{rem}

The existence of martingale weak entropy solutions to the compressible Euler 
equations of gas dynamics has been investigated by several authors; see Section~\ref{sec_lit_review} for a detailed overview of 
the literature. The result most closely related to the present work is due to 
Berthelin and Vovelle~\cite{berthelin_stochastic_2019}.
In contrast to our setting, the analysis in~\cite{berthelin_stochastic_2019} 
focused on the compressible Euler equations without 
any damping term, and incorporated multiplicative stochastic forcing 
satisfying
\[
\Bigg( \sum_{k \ge 1} |\sigma_k(x,\rho,u)|^2 \Bigg)^{1/2}
 \le A_0\, \rho \big( 1 + u^2 + \rho^{2\theta} \big)^{1/2}.
\]
Moreover, the problem in \cite{berthelin_stochastic_2019} was posed on the one-dimensional torus and considered 
Cauchy data in \(L^2\).
In the present work, we focus on martingale \(L^\infty\) weak entropy 
solutions on a bounded domain with \(L^\infty\) initial data.  
This choice is natural for the analysis of the long-time dynamics of 
\eqref{problem} studied in the second part of this manuscript: uniform bounds on the solutions play a crucial role in 
deriving the decay estimates established in Result~II below.

    
   The main ideas behind the proof of Theorem~\ref{main_theorem} can be summarized as follows. To show existence of martingale $L^\infty$ weak entropy solutions to \eqref{problem} on the whole time interval $[0, \infty)$, we first construct solutions to \eqref{problem} on a fixed but arbitrarily large time interval $[0, T]$. For this purpose, we approximate the hyperbolic problem by a two-level approximation scheme:
   \begin{itemize} 
     \item \textbf{$\epsilon$-level:} We first introduce a parabolic regularization of the hyperbolic system \eqref{problem_matrix} by augmenting it with an artificial viscosity term $\epsilon \partial_x^2 U$;
     \item \textbf{$\tau$-level:} For every $\epsilon > 0$ and $T>0$, we subdivide the time interval $[0,T]$ into $N$ subintervals of width $\tau = \Delta t$, and employ a time-splitting scheme that separates the deterministic and stochastic components of the evolution to obtain an approximate solution to the $\epsilon$-regularized problem on the fixed time interval $[0, T]$.
     \end{itemize} 
     The goal is to prove the existence of a global-in-time  martingale $L^\infty$ weak entropy solution by: (1)  first taking the limit as $\tau \to 0$ for each fixed $\epsilon > 0$ and $T>0$, (2) extend the just contructed $\epsilon$-solution to a global-in-time solution on $[0,\infty)$,  and (3) let $\epsilon \to 0$. 

     To show convergence of the scheme as  $\tau\to 0$, we employ compactness arguments (tightness of laws of the approximate solutions) through an Arzela-Ascoli-type argument, thereby obtaining a solution to the $\epsilon$-regularized problem \eqref{regularized_problem1}--\eqref{regularized_problem2} on the time interval $[0, T]$. 
 
 The stochastic solution constructed on the finite time interval $[0,T]$ is then extended to a global-in-time solution on $[0,\infty)$ by exploiting pathwise uniqueness on $[0,T]$, which follows from the Gy\"ongy--Krylov theorem~\cite{MR4491500}, together with a standard gluing argument.

 The passage $\epsilon\to 0$ is more delicate and utilizes the theory of stochastic compensated compactness given in \cite{feng_stochastic_2008}, Young measures, and other techniques from the classical deterministic theory of isentropic Euler equations developed in \cite{diperna_convergence_1983,   MR719807, MR1383202}.
     
     The two-level approximation approach is similar to the approach presented in Berthelin and Vovelle \cite{berthelin_stochastic_2019}; however, our time-splitting scheme at the $\tau$-level is different. Our scheme is of Lie-Trotter type \cite{Bensoussan_Glowinski_Rascanu_1992}, similar to the scheme used by Bauzet in \cite{bauzet_time-splitting_2014}.
In our splitting scheme, on each time subinterval of width $\tau = \Delta t$ we 
define a {\emph{continuous}} stochastic process, defined for all $t \in (t_n,t_{n+1})$.
This stochastic process interpolates between the solution at time $t\in (t_n,t_{n+1})$ obtained by solving the deterministic subproblem with initial data given by the solution at $t_n$, and a solution at time $t\in (t_n,t_{n+1})$ obtained by solving the stochastic subproblem 
 with the initial data given by the deterministic solution at time $t_{n+1}$; 
see \eqref{Scheme} in Section~\ref{sec_splitting_scheme}. 
When viewed as a semi-discrete scheme at discrete points $t_n, n = 1, ..., N$, this is the classical Lie-Trotter scheme introduced in \cite{Bensoussan_Glowinski_Rascanu_1992}, and used in \cite{bauzet_time-splitting_2014}. 
The final result is a continuous stochastic process approximating the solution 
 to the $\epsilon$-regularized equations.

	 The existence proof presented in this manuscript differs from the proof in \cite{berthelin_stochastic_2019} in the following two aspects:
\begin{enumerate}
\item We aim to obtain solutions that belong to the space $L^\infty((0,\infty)\times (0,1))$, whereas the entropy solutions constructed in \cite{berthelin_stochastic_2019} take values in the weaker space $C(0, T; H^{-2}(\mathbb{T}))$. We note that in \cite{berthelin_stochastic_2019} the initial condition is assumed to be only in $L^2(\mathbb{T})$, and the noise has a more general structure.
\item We impose Dirichlet boundary conditions on the momentum, as opposed to periodic boundary conditions considered in \cite{berthelin_stochastic_2019}.
\end{enumerate}

    Details of the proof of Theorem~\ref{main_theorem} are presented in 
     Sections \ref{sec_existence_pathwise_regularized_soln} -- \ref{sec_existence_martingale_soln}.

\if 1 = 0
\medskip
{\color{orange}This paragraph looks a bit redundant with Remark 1.3 }Before moving on to the second main result of the paper, we emphasize that 
the assumption that the noise coefficient $\sigma(x,\rho,m)$ has compact 
support---depending on the $L^\infty$ bounds of the initial data---plays a 
fundamental role in constructing solutions that remain in $L^\infty(Q_T)$ 
for every $T>0$. 
In particular, the combination of the boundedness of the initial data and 
the compact-support assumption on the noise, together with the classical 
theory of invariant regions for nonlinear diffusion equations (see the work 
of Chueh, Conley, and Smoller~\cite{MR430536}), guarantees uniform 
global-in-time $L^\infty(0,1)$ bounds for the solutions of the viscous 
regularized problem~\eqref{regularized_problem1}--\eqref{regularized_problem2}. These bounds are preserved 
in the vanishing-viscosity limit 
{\cpurp and are crucial for analyzing the long-time 
behavior of martingale weak entropy solutions to~\eqref{problem}, studied in the second part of the paper. }
\fi

 \medskip    
\noindent{\bf Result II: Long-time behavior.} The main goal of this paper is to investigate the \emph{long-time behavior} 
of martingale $L^\infty$ weak entropy solutions to~\eqref{problem}. In particular, we establish the following two main results in this context:
\begin{enumerate}
\item First, we show that any 
pair $(\rho,m)$ of density and momentum solving~\eqref{problem} in the sense 
of Definition~\ref{martingale_soln_defn}, under suitable assumptions on the 
noise, converges almost surely as $t \to \infty$, to a constant steady 
state of the system. The rate of convergence is exponential in $L^2(0,1)$.
\item Secondly, we demonstrate that, asymptotically, this 
steady state is well-approximated by the  classical solution to the 
\emph{deterministic (decoupled) porous medium system}: 
        \begin{equation}\label{eqn_PME}
			\left\{
            \begin{split}
                \partial_t \rho &= \partial_x^2 p(\rho),\ {\text{where}} \ p(\rho) = c\rho^\gamma,\\
                m &= -\frac{1}{\alpha}p(\rho)_x,
            \end{split}
            \right.
        \end{equation}
        
        with the initial and boundary conditions
        \begin{equation}
        \left\{
            \begin{split}
                \rho(0,x) = \tilde{\rho}_0(x), \quad  x\in (0,1),\\
                m(t,0) = m(t,1) = 0,   \quad t \ge 0,
            \end{split}
            \right.
        \end{equation}
    where the initial data $\rho_0$ for the stochastic compressible Euler equations \eqref{problem} and the initial data $\tilde{\rho}_0$ for the deterministic porous medium problem \eqref{eqn_PME} satisfy
    $$
     \int_0^1 \rho_0(x) dx = \int_0^1 \tilde{\rho}_0(x) dx.
    $$
    \end{enumerate}

In the deterministic setting, the long-time behavior of the damped Euler 
equations has been the subject of considerable study; see, for example 
\cite{pan_initial_2008, pan_3d_2009}. In particular, the results in our paper 
can be viewed as a stochastic generalization of the work of Pan and Zhao 
\cite{pan_initial_2008}, where they constructed global-in-time $L^\infty$ 
weak entropy solutions to the \emph{deterministic} initial--boundary value 
problem for the damped compressible Euler equations on bounded domains and 
proved that these solutions converge exponentially fast, as $t \to \infty$, 
to steady states. They also showed that the associated porous medium equation 
exhibits the same asymptotic behavior, thereby providing a rigorous 
justification of Darcy's law in the long-time limit.

The stochastic regime is considerably more delicate to analyze, as temporal 
random perturbations can substantially influence the long-time statistics of 
the system, and consequently it remains much less understood. See 
Section~\ref{sec_lit_review} for a review of the relevant literature. 
What is particularly striking in our \emph{stochastic} setting results is that the 
asymptotics of the \emph{stochastically perturbed} Euler system are governed 
by the asymptotics of the corresponding \emph{deterministic} porous medium 
equation. More precisely, the two main results described above are stated in 
the following theorems.

{\begin{thm}[{\bf{Exponential decay to constant state}}]\label{thm_limit_longt}
Assume that the deterministic initial data $U_0=(\rho_0,m_0)\in L^\infty(0,1)$ and the noise coefficient $\sigma$ satisfy the assumptions \eqref{init_cond} and \eqref{noise_assumption}.
		Let $U = (\rho, m)$ and $(\bar\Omega,\bar{\mathbb{P}},(\bar{\mathcal{F}}_t)_{t\geq 0},\bar{\mathcal{F}}, W)$ be the martingale $L^\infty$ weak entropy solution to \eqref{problem}, in the sense of Definition \ref{martingale_soln_defn}, {constructed in Theorem \ref{main_theorem}}. Let $$\rho_* := \int_{0}^{1}\rho_0(x)\,dx.$$ Then there exist deterministic positive constants $\tilde{C} = \tilde{C}(\alpha, \gamma, M_1)$, $r = r(\alpha, \gamma, A_0, M_1)$, and a set $\bar{\Omega}_0\subset \bar\Omega$ with $\bar{\mathbb{P}}(\bar{\Omega}_0)=1$ such that if the Lipschitz continuity constant in \eqref{noise_assumption}$_3$ of the noise coefficient satisfies $A_0 < \tilde{C} \min{(\alpha, 1)}$, then for every $\omega \in \bar{\Omega}_0$ there exists a constant $C_0 = C_0(\omega, \alpha, \gamma, A_0, \|U_0\|_{L^\infty})>0$, such that
		\begin{equation}\label{ineq_thm2.2}
			\int_{0}^{1}(\rho(t) - \rho_*)^2 + m^2(t) \,dx \leq C_0 e^{-rt},\qquad\text{for every }t\geq 0. 
		\end{equation}
	\end{thm}

	\if 1 = 0
    We will now give a brief context of deterministic counterpart of our result, that is the long time dynamics of the deterministic linearly damped isentropic Euler equation \eqref{problem}, before a more comprehensive literature review in the following subsection.
		\begin{equation}\label{problem}
		\begin{cases}
			\partial_t \rho + \partial_x m = 0\\
			\partial_t m + \partial_x (\frac{m^2}{\rho} + p(\rho)) = -\alpha m.
		\end{cases}
	\end{equation}
    Assume that these equations are supplemented with the following initial and boundary conditions:
    \begin{equation}
      \begin{split}\label{eulerdata}
            \rho(0, x) = \rho_0(x), \; m(0, x) = m_0(x), &\quad x \in [0, 1]\\
            m(t, 0) = m(t, 1) = 0, &\quad t \geq 0.
      \end{split}
    \end{equation}
    
    	 It was proven by Pan and Zhao in \cite{pan_initial_2008} that any $L^\infty$ weak entropy solution $(\rho,m)$ to \eqref{problem}-\eqref{eulerdata}, converges exponentially fast to the constant state $(\rho_*, 0)$ in $L^2(0,1)$. In the same paper, it is also proved that the solution $\rho_{PME}$ to the porous medium equation \eqref{eqn_PME} and {the pressure gradient }$-\partial_xp(\rho_{PME})$ converge exponentially fast to the same constant state {$(\rho_*, 0)$} in $L^2(0,1)$,
	 	\begin{equation}\label{eqn_PME}
	 	\begin{cases}
	 		\partial_t \rho = \partial_x^2 p(\rho), \qquad p(\rho) = c\rho^\gamma \\
            \rho(0, x) = \rho_0(x)\\
            \partial_x p(\rho(t, 0)) = \partial_x p(\rho(t, 1)) = 0.\\
	 	\end{cases}
	 \end{equation}

    In light of Pan and Zhao's work \cite{pan_initial_2008}, Theorem \ref{thm_limit_longt} then implies that we have the following result.
\fi
    
    \begin{thm}[{\bf{Porous Medium Approximation of the Long-Time Behavior}}] \label{ie-pme}
  Assume that the deterministic initial data $U_0=(\rho_0,m_0)\in L^\infty(0,1)$ and the noise coefficient $\sigma$ satisfy the assumptions \eqref{init_cond} and \eqref{noise_assumption}.
		Let $(\rho_{IE}, m_{IE})$ and $(\bar\Omega,\bar{\mathbb{P}},(\bar{\mathcal{F}}_t)_{t\geq 0},\bar{\mathcal{F}}, W)$ be a martingale $L^\infty$ weak solution to \eqref{problem} in the sense of Definition \ref{martingale_soln_defn},
    and let $(\rho_{PME},m_{PME})$ denote the global-in-time classical solution of the porous medium problem \eqref{eqn_PME} 
    with the same initial density $\rho_0$.
    
        Then there exist deterministic positive constants $\tilde{C} = \tilde{C}(\alpha, \gamma, M_1)$, $r = r(\alpha, \gamma, A_0, M_1)$, and a set $\bar{\Omega}_0\subset \bar\Omega$ with $\bar{\mathbb{P}}(\bar{\Omega}_0)=1$ such that if the Lipschitz continuity constant in \eqref{noise_assumption}$_3$ of the noise coefficient satisfies $A_0 < \tilde{C} \min{(\alpha, 1)}$, then for every $\omega \in \bar{\Omega}_0$ there exists a constant $C_0 = C_0(\omega, \alpha, \gamma, A_0, \|U_0\|_{L^\infty})>0$, such that
        \begin{equation}
            \int_0^1 \left[(\rho_{IE}(t) - \rho_{PME}(t))^2 + (m_{IE}(t) - m_{PME}(t))^2\right] \,dx \leq C_0 e^{-rt}.
        \end{equation}
    \end{thm}

     
     \if 1 = 0
     {This theorem states that the behavior of the $L^\infty$ martingale weak solutions to the {\it stochastic} isentropic Euler equations with linear damping \eqref{problem} is exponentially close to that of the $L^\infty$ solutions of the {\it deterministic} porous medium equations \eqref{eqn_PME} at large times.}

     To intuitively understand this result, we note here the following phenomenon observed in \cite{pan_initial_2008} in the deterministic setting \eqref{problem} (i.e. when $\sigma=0$ in \eqref{problem}).
     Because of the dissipative nature of the damping term $-\alpha m$, the inertial term $\partial_x\frac{m^2}{\rho}$ decays faster than the pressure and the damping terms, resulting in $\partial_xp(\rho)$ to be balanced by $-\alpha m$. This is precisely the Darcy law. Formally, we can then plug $m = -\partial_x p(\rho)$ into the continuity equation \eqref{problem} and recover the porous medium equation \eqref{eqn_PME}. Therefore, our work builds up on this philosophy that shows that if the noise coefficient $\sigma(x, \rho, m)$ in the stochastic case \eqref{problem} has an appropriate structure, then the stochastic term also decays {fast enough} and the solution to the stochastic problem \eqref{problem} still exhibits the same long time behavior as the deterministic problem \eqref{problem}.
\fi 
	 
	\begin{rem}
    We have already discussed the role played by the assumption of compactly supported noise coefficient (first point in \eqref{noise_assumption}) in obtaining bounded solutions. To further understand how the structure of the noise is chosen, we observe that the second and the third point in
     assumption \eqref{noise_assumption} together imply that magnitude of the noise coefficient is dominated by that of the momentum since we have
     \begin{equation}\label{sigma_bound}
         |\partial_m \sigma| \leq \sqrt{A_0} \implies |\sigma(x, \rho, m) - \underbrace{\sigma(x, \rho, 0)}_{=0} |\leq \sqrt{A_0}|m|.
     \end{equation}
   Hence the condition appearing in the statement of Theorem~\ref{thm_limit_longt}, which bounds the Lipschitz constant $A_0$ in terms of the damping coefficient $\alpha$, ensures 
    {that the contribution of the stochastic term to the energy production in expectation does not overtake the dissipative effects of the damping.}
  \end{rem}

The main ideas behind the proof of Theorem~\ref{thm_limit_longt} are based on deriving 
sharp estimates on higher moments of the entropy and developing new methods 
to capture the time-asymptotic decay of the stochastic contributions in 
\eqref{problem}. 
More specifically, we establish decay in time of the 
$L^2_\omega L^2_x$ norm of the $\epsilon$-regularized solution $U_\epsilon$ 
(see Definition~\ref{defn_U_epsilon}). This is achieved by considering the entropy-dissipation method developed for the hyperbolic system~\eqref{problem}  in 
\cite{huang_convergence_2005, huang_asymptotic_2006}, and extending it
to the parabolic $\epsilon$-approximation, together with additional technical modifications to handle the stochastic forcing. 

In particular, we show that
\[
\mathbb{E}\big\|\rho_\epsilon(t) - \textstyle\int_0^1 \rho_{\epsilon 0}\,dx \big\|^2_{L^2(0,1)} \to 0,
\qquad 
\mathbb{E}\|m_\epsilon(t)\|^2_{L^2(0,1)} \to 0,
\quad \text{as } t \to \infty.
\]
While this readily implies that 
\[
\|\rho_\epsilon(t_n) - \textstyle\int_0^1\rho_{\epsilon0}\,dx \|_{L^2(0,1)} \to 0,
\qquad 
\|m_\epsilon(t_n)\|_{L^2(0,1)} \to 0
\]
almost surely along some subsequence $\{t_n\}_n$, the desired conclusion does not follow immediately, 
since there may exist another subsequence of times along which these quantities fail to decay.

 \if 1 = 0  
 {\color{gray}{We emphasize that obtaining exponential decay estimates in expectation at the $\epsilon$-approximate level is essential for establishing the desired almost sure decay-in-time result for the limiting system. This is because the approximate solution satisfies an entropy {\it equality}, rather than a mere entropy {\it inequality} satisfied by the limiting solution, allowing tools such as the Ito formula to be applied only at this level.
 This enables us to derive higher-moment decay estimates from lower moment bounds. These higher-moment estimates, in turn, allow us to control the growth of stochastic integrals appearing in the energy equality of the approximate system. This is essential for invoking the results of Yuskovych \cite{MR4671722} which ensure that the stochastic terms vanish in the limit.}}
\fi

We emphasize that obtaining exponential decay estimates in expectation at the $\epsilon$-approximate level is essential for establishing the desired almost sure decay-in-time result for the $\epsilon \to 0$ limiting system. This is because the $\epsilon$-approximate solution satisfies an entropy \emph{equality}, in contrast to the entropy \emph{inequality} satisfied by the $\epsilon \to 0$ limiting solution. 
Working with $\epsilon$-approximations allows tools such as the It\^o formula to be applied, and it enables us to derive higher-moment decay estimates from lower-moment bounds. 
These higher-moment estimates, in turn, provide the necessary control over the stochastic integrals appearing in the energy equality for the $\epsilon$-approximate system, which is crucial for invoking the results of Yuskovych~\cite{MR4671722} to ensure that the stochastic contributions vanish in the limit $\epsilon \to 0$.

Details of the proofs of Theorems~\ref{thm_limit_longt} and \ref{ie-pme} are presented in Section~\ref{sec_long_time_behavior}.

	\medskip

We conclude this introduction by highlighting the three principal novel contributions of this work:
\begin{enumerate}
    \item To the best of our knowledge, this is the first result that extends the study of the long-time behavior of the deterministic isentropic Euler equations with linear damping~\eqref{problem} to their stochastic perturbation.

    \item We establish almost sure convergence in time of solutions to a stationary state. In contrast, most of the existing literature on the long-time behavior of stochastic compressible fluid flows -- primarily in the context of the Navier-Stokes equations -- addresses only the existence of stationary solutions or invariant measures, without proving any form of convergence of the time-averaged statistics of solutions to those of the stationary state.

    \item We develop new techniques to study long-time asymptotic behavior of stochastically perturbed compressible Euler equations that rely on: (1)  deriving higher-moment energy estimates at the level of the viscous approximation, and (2) applying recent results of Yuskovych~\cite{MR4671722} concerning almost sure convergence in time of stochastic integrals, established for stochastic differential equations.
\end{enumerate}

	\subsection{Literature review}\label{sec_lit_review}
	\subsubsection{Existence of global solution}
	\textbf{In the deterministic setting}, the study of existence of global-in-time weak entropy solutions to the isentropic Euler equation dates back to the early 1980s, with the seminal work of DiPerna \cite{MR719807} in which the existence of global-in-time bounded entropy solution in one space dimension on $\mathbb{R}$ is proved by assuming a uniformly bounded initial condition and considering a special value for $\gamma$. 
    In this work, the method of vanishing viscosity is used to construct approximate solutions and the theory of compensated compactness is developed to pass to the limit. Later in 1996, Lions, Perthame and Souganidis \cite{MR1383202} offered an alternative and simplified proof for all $\gamma > 1$ using a kinetic formulation and the theory of compensated compactness.
    The existence of global $L^\infty$ weak entropy solutions for a general inhomogeneous isentropic Euler system with $\gamma >1$ is proven by Ding, Chen, and Luo in \cite{xiaxi_convergence_1989} using a fractional step Lax-Friedrich scheme and compensated compactness. The $\gamma=1$ case was handled by Huang and Pan in  \cite{huang_asymptotic_2006} using the vanishing viscosity method.\\
	
\noindent \textbf{In the stochastic setting},  the global-in-time existence of solution for stochastic compressible Euler equations in 1D was proven relatively recently by Berthelin and Vovelle in \cite{berthelin_stochastic_2019}.
    The construction of approximate solutions and compactness framework in this work was based on several previous studies on general stochastic \textit{scalar} hyperbolic conservation laws 
    \begin{equation}\label{eqn_SSCL}
        du(t,x) + \nabla_x \cdot F(u(t, x))\,dt = \sigma(x, u(t, x))\,dW(t).
    \end{equation}
The earliest work in the context of well-posedness is due to Feng and Nualart \cite{feng_stochastic_2008}, where \eqref{eqn_SSCL} is studied in multiple spatial dimensions driven by multiplicative space-time white noise, with $L^p$ initial data. The approach involves constructing approximate solutions via artificial viscosity and obtaining compactness by extending Tartar's compensated compactness framework \cite{MR584398} to the stochastic setting. 
    In \cite{MR2652180}, Debussche and Vovelle employed the kinetic formulation to prove existence and uniqueness for both additive and multiplicative noise on a $d$-dimensional torus. Later, Bauzet, Vallet, and Wittbold \cite{bauzet_dirichlet_2014} established the existence and uniqueness of measure-valued entropy solutions with $L^2$ initial data under homogeneous Dirichlet boundary conditions, using the semi-Kruzhkov entropy framework. To support this analysis numerically, Bauzet developed a Lie-Trotter time-splitting scheme in \cite{bauzet_time-splitting_2014}, first studied in the context of stochastic PDEs in \cite{Bensoussan_Glowinski_Rascanu_1992}, and showed that this scheme converges to the unique measure-valued entropy solution. Notably, in this manuscript we develop a continuous variant of Lie-Trotter splitting scheme, which gives rise to a continuous stochastic process on each time subinterval, see \eqref{Scheme} below. This splitting scheme, adapted to our specific problem setting, comprises one layer of our two-layer approximation procedure. A similar two-layer approximation approach was employed in the work of Berthelin and Vovelle \cite{berthelin_stochastic_2019}, where a different time-splitting scheme was used. Finally, in \cite{MR3800013}, Karlsen and Storrosten adopted the same operator splitting approach to prove existence results for the case of spatially inhomogeneous noise on a $d$-dimensional torus.
    
    \if 1 = 0
    \textbf{For stochastic compressible isentropic Euler equation}, Berthelin and Vovelle  proved the existence of martingale weak entropy solution to a Cauchy problem for isentropic Euler equation with multiplicative stochastic noise satisfying $(\sum_{k \geq 1}|\sigma_k(x, \rho, u)|^2)^{1/2} \leq A_0\rho(1+u^2 + \rho^{2\theta})^{1/2}$, and with $L^2$ initial data on a one-dimensional torus \cite{berthelin_stochastic_2019}. They constructed the solution to the parabolic approximation of the hyperbolic equation with a time-splitting scheme, and used the theory of compensated compactness, in a similar spirit as in the deterministic case \cite{MR719807, MR1383202}, to show that the vanishing viscosity limit of the parabolic approximation is the weak entropy solution. Later, more results on the solvability of stochastic compressible Euler equation were obtained. In 2019, Breit and Mensah in \cite{MR3921105} proved the existence of a unique local strong pathwise solution to the stochastic barotropic Euler equation on $\mathbb{R}^d$, prescribed with suffiently smooth initial data and far-field boundary condition. Their work is based on Kim's result in  \cite{kim_stochastic_2011} on the global regular solutions to symmetric hyperbolic systems under suitable assumptions on the noise and suffiently smooth initial data. In 2020, Breit, Feireisel and Hofmanova proved in \cite{MR4078230} nonuniqueness of weak pathwise solution with sufficiently smooth initial data, adapting the method of convex integration. In \cite{MR4472926}, Hofmanova et al. introduced the notion of dissipative measure-valued martingale solution to the same equation and proved global existence of such a solution on a 3D torus for any finite energy initial data.  To the best of our knowledge, there are no results yet for stochastic compressible Euler equations with additional source terms.
    
    We also briefly mention that for existence results on \textbf{stochastic compressible Navier-Stokes equation}, one can refer to \cite{breit_local_2018, MR3549199, MR2997374}. For \textbf{stochastic incompressible Euler equation}, one can refer to \cite{bessaih_2-d_1999, MR3483892, brzezniak_stochastic_2001,  capinski_stochastic_1999, MR3161482, lang_well-posedness_2023}.
    \fi
    %
    %
    \medskip

    \noindent \textbf{For stochastic compressible isentropic Euler equations},
Berthelin and Vovelle~\cite{berthelin_stochastic_2019} proved the existence of martingale weak entropy solutions to the Cauchy problem for the isentropic Euler equations with multiplicative stochastic forcing satisfying
\[
\Big(\sum_{k \geq 1} |\sigma_k(x,\rho,u)|^2 \Big)^{1/2}
\leq A_0 \rho \big(1 + u^2 + \rho^{2\theta}\big)^{1/2},
\]
with $L^2$ initial data posed on the one-dimensional torus.
Their approach is based on constructing solutions to a parabolic approximation of the hyperbolic system via a time-splitting scheme and employing compensated compactness arguments, in the spirit of the deterministic theory
\cite{MR719807, MR1383202}, to pass to the vanishing viscosity limit and obtain a weak entropy solution.

Subsequently, further results on the solvability of the stochastic compressible Euler equations have been established.
In particular, Breit and Mensah~\cite{MR3921105} proved the existence and uniqueness of a local strong pathwise solution to the stochastic barotropic Euler equations on $\mathbb{R}^d$, assuming sufficiently smooth initial data and far-field boundary conditions. Their analysis builds on earlier work of Kim~\cite{kim_stochastic_2011}, who established global regular solutions to symmetric hyperbolic systems under suitable assumptions on the noise and smooth initial data.
Later, Breit, Feireisl, and Hofmanov\'{a}~\cite{MR4078230} demonstrated nonuniqueness of weak pathwise solutions, even for smooth initial data, by adapting convex integration techniques.
In~\cite{MR4472926}, Hofmanov\'{a} \emph{et al.} introduced the notion of a dissipative measure-valued martingale solution and proved the global existence of such solutions on the three-dimensional torus for arbitrary finite-energy initial data. 

Finally, we mention that in recent work by Chen, Huang, and Wang in \cite{ChenHuangWang}, the global existence of weak martingale solutions to the 1D stochastic isentropic Euler equations on $\mathbb{R}$ is shown, and this work provides an extension of earlier results in Berthelin and Vovelle \cite{berthelin_stochastic_2019}, where a boundedness condition on higher order moments (higher order integrability) of the initial data is required. The new result in \cite{ChenHuangWang} extends the study of stochastic compressible isentropic Euler equations by (1) considering initial data that has only \textit{finite relative energy} rather than higher moment bounds as was required in past literature \cite{berthelin_stochastic_2019}, (2) posing the equations on $\mathbb{R}$ rather than on a bounded domain such as $\mathbb{T}$, (3) considering initial data (having more generally finite \textit{relative} energy rather than just finite energy) which can admit positive far field density at infinity, and (4) allowing for more general pressure laws rather than just the power law (polytropic) pressure law. As a result, new methods for stochastic compensated compactness arguments are introduced in \cite{ChenHuangWang} for the limit passage in the artificial viscosity parameter; these compactness methods are interesting, as they no longer appeal to higher order moment/entropy estimates as in the work by Berthelin and Vovelle \cite{berthelin_stochastic_2019}, and they hence give a new limit procedure for approximate solutions to these equations.


\medskip
\noindent
\textbf{Related models.}
For existence results concerning the \emph{stochastic compressible Navier--Stokes equations}, see
\cite{breit_local_2018, MR3549199, MR2997374}.
For the \emph{stochastic incompressible Euler equations}, see
\cite{bessaih_2-d_1999, MR3483892, brzezniak_stochastic_2001, capinski_stochastic_1999, MR3161482, lang_well-posedness_2023}.

\subsubsection{Long time behavior }
{\bf{Deterministic setting.}} The study of long time behavior of solutions to deterministic 
hyperbolic conservation laws with linear damping and with small and smooth initial data dates back to the early 1990s \cite{hsiao_convergence_1992,MR1254543}.
In particular, in \cite{hsiao_convergence_1992}, Hsiao and Liu study the large-time asymptotic behavior of solutions to a Cauchy problem for one-dimensional system of hyperbolic conservation laws with linear damping
\[
\partial_t U + A(U)\,\partial_x U = - D\,U,
\]
where $D$ a positive definite matrix.
Under suitable structural assumptions on the system and smallness conditions on the initial data, the authors show that solutions exist globally in time and converge, as $t \to \infty$, to \emph{nonlinear diffusion waves}. These diffusion waves are governed by an associated nonlinear parabolic equation that captures the dominant dissipative behavior of the original hyperbolic system.

The authors of \cite{hsiao_convergence_1992,MR1254543} verified that time asymptotically, the solution of the deterministic isentropic compressible Euler equations is equivalent to the self-similar solution (also called the Barenblatt solution) of the porous medium problem \eqref{eqn_PME}. Later this result was extended to include non-smooth solutions with shock waves away from the vacuum \cite{MR1640089, hsiao_nonlinear_1998, luo_interaction_2000}.
	
	In the early 2000's, this problem was studied more extensively in \cite{huang_convergence_2005, huang_asymptotic_2006, huang_l1_2011} allowing vacuum states in the solution. In particular, the work in \cite{huang_convergence_2005} considers the one-dimensional compressible isentropic Euler equations with frictional (linear) damping, allowing for the presence of vacuum in the initial data. The authors address a long-standing conjecture asserting that, for solutions with finite total mass, the density profile of the damped Euler system should, as $t \to \infty$, approach a Barenblatt self-similar solution $\rho(x/\sqrt{t})$ of the porous medium equation, which is a canonical nonlinear diffusion profile.
They prove that any entropy weak solution that is bounded in $L^\infty$ and has finite initial mass converges strongly in $L^1$, with explicit decay rates, to the Barenblatt profile associated with the porous medium equation. Moreover, the momentum converges in the long-time limit in accordance with Darcy's law derived from the limiting diffusion equation.
The analysis is based on robust entropy dissipation methods that capture the dissipative effects of the damping and reconcile the interplay between nonlinear convection, damping, and vacuum. This approach yields a rigorous justification of convergence to the Barenblatt solution without imposing smallness or high regularity assumptions on the initial data. The work in 
 \cite{huang_asymptotic_2006} complements~\cite{huang_convergence_2005} by clarifying the mechanism and rate at which solutions approach the diffusion profile.


In \cite{pan_initial_2008}, Pan and Zhao adopted the entropy dissipation method to study the deterministic initial-boundary value problem for the 1D compressible isentropic Euler equations with linear damping posed on a {\emph{bounded spatial domain}} with Diriclet boundary data for the momentum. 
The authors establish global-in-time existence and uniqueness of entropy weak solutions for large initial data, allowing for the presence of vacuum. A central contribution of the paper is the rigorous treatment of boundary effects, which play a crucial role in the analysis of damped hyperbolic systems on bounded domains and introduce additional technical difficulties compared to the Cauchy problem.
Furthermore, Pan and Zhao investigate the long-time asymptotic behavior of solutions. They show that the damping mechanism dominates the dynamics as time progresses, leading to decay of the momentum and convergence of the density toward a spatially homogeneous equilibrium determined by the total mass. 
%
This work extends earlier results on the Cauchy problem for compressible Euler equations with damping to bounded domains and provides a rigorous foundation for understanding the combined effects of damping, vacuum, and boundaries on the global dynamics and asymptotic behavior of solutions.

In 2009, Pan and Zhao extended their result to the 3D case \cite{pan_3d_2009}. We mention that the asymptotic behavior of the deterministic porous medium problem \eqref{eqn_PME} with Neumann boundary condition on pressure is also studied in \cite{MR625601}, except that the authors do not characterize the behavior of $\partial_x p(\rho)$.

\if 1 = 0
The long time behavior of stochastic {\bf{incompressible}} fluid flows relatively well studied comparing to their compressible counterparts. We refer to \cite{breit_stationary_2019}, \cite{coti_zelati_invariant_2021}, and the references therein, for a comprehensive review of stochastic incompressible fluid flows (mostly Navier Stokes equation).
	The first result concerning the large time behavior of stochastic \textbf{compressible} fluid flows, which proves the existence of a statistically stationary solution to the compressible Navier-Stokes equation driven by a multiplicative noise in three-dimensional space, is given in \cite{breit_stationary_2019}. 
    Shortly after this work, Coti Zelati, Glatt-Holtz, and Trivisa proved in \cite{coti_zelati_invariant_2021} the existence of an invariant measure to the one-dimensional compressible Navier-Stokes equation for $\gamma=1$ driven by a multiplicative noise, supplemented by Dirichlet boundary condition for the velocity, by taking advantage of the existence of globally well-posed and probablisitically strong solution in 1D. 
	
	 We also briefly mention the work on invariant measures or singleton random attractor of solution for stochastic scalar hyperbolic balance law, such as Burgers equation with additive noise \cite{boritchev_sharp_2013, e_invariant_2000, MR1952472}, and general stochastic hyperbolic conservation law with additive noise \cite{debussche_invariant_2015}. See also the closely related work of Gess and Souganidis \cite{MR3666564} where the noise is applied to the flux function: $\partial_t u + \sum_{i} \partial_{x_i} A(u) \circ \,dW_i(t) = 0$ and it is proved that the spatial average of the random initial data is the unique random attractor of the solution and the convergence holds in expectation and almost surely. 
     \fi
     
     %
     %
 \if 1 = 0
     \textbf{Stochastic setting.}
The long-time behavior of stochastic {\cpurp fluids has been studied for over 30 years. However, the emphasis has been mostly on \textbf{incompressible} fluid flows. The long-time asymptotic behavior of stochastic \textbf{compressible} fluid equations is a very recent and still developing field.}

{\cpurp In terms of \textbf{incompressible} fluid equations most work focused on the stochastic Navier-Stokes equations. In particular, for the 2D stochastic Navier-Stokes equations, the long-time dynamics are now well understood; the questions that have been studied include existence and uniqueness of invariant measures under suitable non-degeneracy of noise, ergodicity and exponential mixing, and convergence of time averages to stationary statistics.
In 3D, despite unresolved issues of uniqueness, significant progress has still been made in the areas of existence of stationary martingale solutions,
invariant measures, and
ergodic properties under suitable assumptions on the noise.
These results are often framed in terms of statistical equilibria, rather than pointwise convergence of solutions.}
We refer the reader to \cite{breit_stationary_2019}, \cite{coti_zelati_invariant_2021}, and the references therein for comprehensive overviews of stochastic incompressible fluid models, primarily in the context of the Navier--Stokes equations.

{\cpurp 
In terms of \textbf{compressible} fluid flows, systematic study began only in the late 2000s--2010s, driven largely by advances in the deterministic compressible theory.} The first result addressing the long-time behavior of compressible Navier-Stokes equations was obtained in \cite{breit_stationary_2019}, where the authors establish  the existence of a statistically stationary solution to the three-dimensional compressible Navier--Stokes equations driven by multiplicative noise. 
Shortly thereafter, Coti Zelati, Glatt-Holtz, and Trivisa proved in \cite{coti_zelati_invariant_2021} the existence of an invariant measure for the one-dimensional compressible Navier--Stokes equations with adiabatic exponent $\gamma=1$, subject to multiplicative noise and supplemented with Dirichlet boundary conditions for the velocity. Their analysis exploits the global well-posedness and probabilistic strong solvability available in one spatial dimension.

\fi 

\medskip

\noindent \textbf{Stochastic setting.}
The long-time behavior of stochastic fluid equations has been studied for more than three decades; however, the majority of this work has focused on \textbf{incompressible} flows. In contrast, the analysis of the long-time asymptotic behavior of stochastic \textbf{compressible} fluid equations is a much more recent and still rapidly developing area.

For \textbf{incompressible} fluid equations, most of the existing literature concerns the stochastic Navier--Stokes equations. In particular, for the two-dimensional stochastic Navier--Stokes system, the long-time dynamics are by now relatively well understood. Results in this setting include the existence and uniqueness of invariant measures under suitable non-degeneracy assumptions on the noise, ergodicity and exponential mixing, as well as convergence of time-averaged statistics to stationary states.  
In three spatial dimensions, despite the lack of a complete deterministic well-posedness theory, substantial progress has nevertheless been made, including results on the existence of statistically stationary martingale solutions, invariant measures, and ergodic properties under appropriate assumptions on the stochastic forcing. In both two and three dimensions, the long-time behavior is typically described in terms of \emph{statistical equilibria}, rather than pointwise convergence of individual solution trajectories.
We refer the reader to \cite{breit_stationary_2019}, \cite{coti_zelati_invariant_2021}, and the references therein for comprehensive overviews of stochastic incompressible fluid models, primarily in the context of the Navier--Stokes equations.

In contrast, the systematic study of stochastic \textbf{compressible} fluid equations began only in the late 2000s and early 2010s, largely driven by advances in the deterministic theory of compressible flows. The first result addressing the long-time behavior of the stochastic compressible Navier--Stokes equations was obtained in \cite{breit_stationary_2019}, where the existence of a statistically stationary solution to the three-dimensional system driven by multiplicative noise was established. 
Shortly thereafter, Coti Zelati, Glatt-Holtz, and Trivisa proved in \cite{coti_zelati_invariant_2021} the existence of an invariant measure for the one-dimensional compressible Navier--Stokes equations with adiabatic exponent $\gamma=1$, subject to multiplicative noise and Dirichlet boundary conditions on the velocity. Their analysis crucially exploits the global well-posedness and probabilistically strong solvability available in one spatial dimension.

We also briefly mention related work on invariant measures and random attractors for stochastic scalar hyperbolic balance laws, such as the Burgers equation with additive noise \cite{boritchev_sharp_2013, e_invariant_2000, MR1952472}, as well as for general stochastic hyperbolic conservation laws with additive noise \cite{debussche_invariant_2015}. 

Closely related is the work of Gess and Souganidis~\cite{MR3666564}, in which stochastic perturbations are applied directly to the flux function
\[
\partial_t u + \sum_{i} \partial_{x_i} A(u) \circ \, dW_i(t) = 0.
\]
A central theme of the paper is that the {\emph convective}} stochastic forcing, despite being rough in time, induces a regularizing and stabilizing effect on the dynamics.
They show that the spatial average of the random initial data constitutes the unique random attractor of the solution, with convergence holding both in expectation and almost surely. In particular, they show that the law of the solution converges to a unique invariant measure, implying ergodicity of the system. 
The analysis combines tools from kinetic formulations, stochastic analysis, and contractivity properties in $L^1$, allowing the authors to overcome the lack of classical smoothing present in deterministic scalar conservation laws.
Overall, this work provides one of the most complete and rigorous descriptions of the long-time dynamics of stochastic hyperbolic equations, demonstrating that randomness can fundamentally alter the asymptotic behavior by inducing regularization, ergodicity, and convergence to a unique statistical equilibrium.

We conclude this literature review by discussing the work presented in a recent preprint \cite{KuanTawriTrivisa2025}, which addresses the long-time {\emph{statistical}} behavior of the one-dimensional stochastic isentropic compressible Euler equations with linear damping, posed on the torus and driven by multiplicative white-in-time noise. This noise is not necessarily compactly supported and has more general growth conditions. The main result is the existence of statistically stationary solutions in the $L^p$ phase space within the class of martingale weak entropy solutions for any adiabatic exponent $\gamma > 1$. 

The present work complements the results of \cite{KuanTawriTrivisa2025} by providing a more detailed description of the long-time asymptotic behavior of the stochastic isentropic compressible Euler equations with linear damping. While \cite{KuanTawriTrivisa2025} focuses on the existence of statistically stationary solutions and characterizes the long-time behavior at the level of statistics, the results in this paper yield pathwise (almost sure) information on the asymptotics of solutions corresponding to arbitrary deterministic bounded initial data in 
$L^\infty$ and physically relevant Dirichlet boundary condition, under the assumption that the noise is compactly supported and is dominated by the frictional damping effects.

Using techniques that are significantly different from those used in \cite{KuanTawriTrivisa2025}, we prove that the solutions for the density and momentum, constructed from our existence proof, converge to a constant equilibrium state almost surely and exponentially fast in time, where this constant equilibrium state has density given by the total initial mass and momentum zero. These constant equilibrium states are well approximated by the asymptotic profile of the corresponding deterministic porous medium equation.

To the best of our knowledge, the present work provides the first rigorous pathwise convergence result for the long-time behavior of solutions to the stochastic isentropic compressible Euler equations with linear damping.

\subsection{Organization of the paper}
	{{The paper is organized as follows.
    In Section~\ref{TwoApproximations}, we construct approximate solutions to problem \eqref{problem} by introducing a two-level approximation scheme: a parabolic approximation of the hyperbolic problem via the parameter $\epsilon$, and a semi-discretization in time via the parameter $\tau = \Delta t$.
In Section~\ref{sec_existence_pathwise_regularized_soln}, we present details of the parabolic approximation via the parameter $\epsilon$, 
and introduce an appropriate notion of the solution to the resulting $\epsilon$-regularized problem. 
In Section~\ref{sec_splitting_scheme}, 
we discretize the time interval into subintervals of width $\tau = \Delta t$, and construct a time-splitting scheme that approximates solutions of the $\epsilon$-regularized system by splitting the problem into the deterministic and stochastic subproblems.
In Section~\ref{sec_tau_existence}, we show the existence of solutions to the deterministic and stochastic subproblems,
and prove an entropy equality satisfied by the approximate solutions to the parabolic regularization. 
In Section~\ref{tauto0}, we pass to the limit $\tau \to 0$ in the time-splitting scheme to prove the existence of a pathwise solution to the regularized system for fixed $\epsilon>0$. 
In Section~\ref{sec_existence_martingale_soln}, we let $\epsilon \to 0$ and prove  the existence of a martingale $L^\infty$ weak entropy solution to~\eqref{problem}.

Section~\ref{sec_long_time_behavior} is devoted to the study of  the \emph{long-time behavior} 
of martingale $L^\infty$ weak entropy solutions to~\eqref{problem}.
In particular, in Section~\ref{sec_decay_in_expectation}, we revisit the $\epsilon$-regularized system and prove exponential decay in time of
\[
\mathbb{E}\Big\|\rho_\epsilon(t) - \textstyle\int_0^1 \rho_{\epsilon 0}\,dx\Big\|^2_{L^2(0,1)}
\quad \text{and} \quad
\mathbb{E}\|m_\epsilon(t)\|^2_{L^2(0,1)}.
\]
While this result guarantees the existence of a sequence of times $(t_n)_{n\in\mathbb{N}}$ along which the desired exponential decay to a constant state of martingale $L^\infty$ weak entropy solutions holds almost surely, it does not by itself yield our main result, which is pathwise convergence to a constant state as $t\to\infty$.  We prove the pathwise convergence to a constant state in Section~\ref{sec_a.s._convergence}.
}}

We conclude the manuscript by showing that the long-time behavior of the martingale $L^\infty$ weak entropy solution to the stochastic isentropic Euler equations with linear damping is well-approximated by the classical solution to a deterministic 
porous medium problem -- the main result of this manuscript.

In the Appendix~\ref{appendix_bc}, we provide a rigorous justification of the homogeneous Dirichlet boundary condition on the momentum, and in Appendix~\ref{Appendix2}, we state some useful algebraic identities used throughout the manuscript.


\section{Construction of approximate solutions in $\epsilon$ and $\tau = \Delta t$. }\label{TwoApproximations}
In the first half of this manuscript, we prove the existence of global-in-time martingale $L^{\infty}$ weak entropy solutions to \eqref{problem}, namely the result stated in Theorem \ref{main_theorem}. The proof of Theorem~1.1 proceeds by first establishing the existence of martingale $L^\infty$ weak entropy solutions on an arbitrarily large but finite time interval $[0,T]$, and then extending these solutions to the whole time interval $[0,\infty)$. To show the existence on $[0,T]$, we employ a two-level scheme:
\begin{enumerate}
\item At the first level, indexed by the parameter $\epsilon>0$, we introduce a parabolic regularization of the hyperbolic system by augmenting it with an artificial viscosity term of the form $\epsilon \partial_x^2 U$. 
\smallskip
\item At the second level, indexed by the parameter $\tau$, for each fixed $\epsilon>0$ and $T>0$, we construct approximate solutions to the regularized problem on the interval $[0,T]$ by means of a time-splitting scheme that separates the deterministic and stochastic components of the evolution. 
\end{enumerate}
This approach allows us to obtain well-defined approximate solutions and derive uniform estimates that are essential for passing to the limit and proving the desired existence result.

We begin by introducing the parabolic approximation to the original problem, and proving the existence of a pathwise bounded solution to it for each fixed parameter $\epsilon>0$.

\subsection{The $\epsilon$-approximation via parabolic regularization} %
\label{sec_existence_pathwise_regularized_soln}
 For each fixed $\epsilon > 0$, consider the following parabolic approximation of~\eqref{problem} posed on the space--time domain $[0,1] \times [0,\infty)$:
 \begin{align}
\left\{\begin{aligned}\label{regularized_problem1}
d \rho + \partial_x m \,dt
&= \epsilon \partial_x^2 \rho \,dt, \\
d m + \partial_x\!\left(\frac{m^2}{\rho} + p(\rho)\right) dt
&= \epsilon \partial_x^2 m \,dt
   - \alpha m \,dt
   + \sigma_\epsilon(x, U)\,dW(t),
\end{aligned}
\right.
\\[1ex]
\left\{
\begin{aligned}\label{regularized_problem2}
(\rho(0,x), m(0,x))
&= (\rho_{\epsilon0}(x), m_{\epsilon0}(x)),
\qquad x \in [0,1], \\
\partial_x \rho(t,0)
= \partial_x \rho(t,1)
&= 0,
\qquad t \ge 0, \\
m(t,0)
= m_(t,1)
&= 0,
\qquad t \ge 0,
\end{aligned}
\right.
\end{align}
where $\sigma_\epsilon$ and $(\rho_{\epsilon0},m_{\epsilon0})$ are the regularized noise and initial data, respectively, defined as follows.

   \if 1 = 0
\begin{align}
\left\{\begin{aligned}\label{regularized_problem1}
d \rho_\epsilon + \partial_x m_\epsilon \,dt
&= \epsilon \partial_x^2 \rho_\epsilon \,dt, \\
d m_\epsilon + \partial_x\!\left(\frac{m_\epsilon^2}{\rho_\epsilon} + p(\rho_\epsilon)\right) dt
&= \epsilon \partial_x^2 m_\epsilon \,dt
   - \alpha m_\epsilon \,dt
   + \sigma_\epsilon(x, U_\epsilon)\,dW(t),
\end{aligned}
\right.
\\[1ex]
\left\{
\begin{aligned}\label{regularized_problem2}
(\rho_\epsilon(0,x), m_\epsilon(0,x))
&= (\rho_{\epsilon0}(x), m_{\epsilon0}(x)),
\qquad x \in [0,1], \\
\partial_x \rho_\epsilon(t,0)
= \partial_x \rho_\epsilon(t,1)
&= 0,
\qquad t \ge 0, \\
m_\epsilon(t,0)
= m_\epsilon(t,1)
&= 0,
\qquad t \ge 0 .
\end{aligned}
\right.
\end{align}
\fi
\medskip

\noindent {\bf{Regularized initial data and noise coefficient.}}\label{rem_init_cond}
    We start by defining the regularized initial data $U_{\epsilon0} = (\rho_{\epsilon 0}, m_{\epsilon 0})$.
    For this purpose recall assumption \eqref{init_cond} on the boundedness of the initial data $(\rho_0, m_0)$ to the original problem \eqref{problem}, stated in the existence theorem, Theorem~\ref{main_theorem}:
        \begin{equation*}\label{init_cond1}
			\begin{split}
				0\leq \rho_0(x) \leq M_1, \;\; |m_0(x)| \leq { M_2\rho_0(x)}, \text{ for some constants }M_1, M_2 > 0.
			\end{split}
		\end{equation*}

 Define the regularized initial data by first truncating $\rho_0$ from below, and defining its even and odd extensions $\rho_0^{even}$ and $m_0^{odd}$ to be:
		\begin{equation}\label{rho_epsilon0}
			\rho_0^{even} = \begin{cases}
				\max(\rho_0, \epsilon)(-x) & x \in [-1, 0]\\
				\max(\rho_0, \epsilon)(x) & x\in [0, 1]\\
				\max(\rho_0, \epsilon)(-x+2) & x\in [1, 2],
			\end{cases}
		\end{equation}
		\begin{equation}\label{u_epsilon0}
			m_0^{odd} = \begin{cases}
				-m_0(-x) & x \in [-1, 0]\\
				m_0(x) & x\in [0, 1]\\
				-m_0(-x+2) & x \in [1, 2].
			\end{cases}
		\end{equation}
		Convolve the extensions with a standard mollifier, $\varphi_\epsilon \in C^2_c(0,1)$, and define the regularized initial data $U_{\epsilon 0} := (\rho_{\epsilon 0}, m_{\epsilon 0})$ to \eqref{regularized_problem2} by:
		\begin{equation}\label{U_epsilon0}
			\rho_{\epsilon 0} = (\rho_0^{even}  * \varphi_\epsilon )\cdot\mathbf{1}_{x\in[0,1]}, \;\; m_{\epsilon 0} = (m_0^{odd} * \varphi_\epsilon) \cdot \mathbf{1}_{x\in[0,1]}.
		\end{equation}

		We then have $\rho_{\epsilon0}, m_{\epsilon0} \in H^{2}(0, 1)$, and $\rho_{\epsilon 0} \geq c$, for some constant $c = c(\epsilon) >0$. Note that the sequence $\{\rho_{\epsilon0}\}_\epsilon$ is also uniformly bounded above, since  $$|\rho_{\epsilon 0}(x)| = \left|\int_0^1 \rho_0(y) \varphi_\epsilon(x-y)\,dy\right|\leq \|\rho_0\|_{L^\infty(0,1)}\|\varphi_\epsilon\|_{L^1(0,1)}\leq M_1,$$ where $M_1$ is given in \eqref{init_cond}. Similarly, the sequence $\{m_{\epsilon0}\}$ is also uniformly bounded above by $M_1 M_2$. Regarding the boundary condition, one can also check that $\partial_x \rho_{\epsilon0}(0) = \partial_x\rho_{\epsilon 0}(1) = 0$, and $m_{\epsilon 0}(0) = m_{\epsilon0}(1) = 0$. Therefore, the regularized initial conditions are compatible with the boundary condition in \eqref{regularized_problem2}. Moreover, $\lim_{\epsilon \to 0}\|U_{\epsilon_0} - U_{0}\|_{L^2(0,1)} = 0 $. 
        
        Similarly, we  regularize the noise coefficient $\sigma(x, U)$ in the spatial variable via odd extension and convolution with the standard mollifier $\varphi_\epsilon$:
        \begin{equation}\label{defn_sigmaepsilon}
        \sigma_\epsilon(x, U) = \sigma(x, U) * \varphi_\epsilon,
        \end{equation}
        so that $\sigma_\epsilon \in C^1_c(0,1)$, and $\lim_{\epsilon\to 0}\|\sigma_\epsilon(\cdot, U) - \sigma(\cdot, U)\|_{L^\infty(0,1)} = 0$. We then denote:
        \begin{equation}\label{reg_Phi}
            \Phi_\epsilon(x, U) := \begin{pmatrix}
            0\\ \sigma_\epsilon(x, U)
        \end{pmatrix}.
        \end{equation}

 One can quickly verify that $\sigma_\epsilon$ also satisfies \eqref{noise_assumption}, since \eqref{noise_assumption} are assumptions on the state variables, and convolution in the spatial variable does not change these properties. In particular, throughout the analysis of the approximate system, we will often use the following analogues of the properties from \eqref{noise_assumption}:
 \begin{equation}\label{regularizedLipschitz}
 \sigma_{\epsilon}(x, \rho, 0) = 0, \qquad \sigma(x, 0, m) = 0, \qquad |\nabla_{\rho, m} \sigma_{\epsilon}(x, \rho, m)| \le \sqrt{A_0}.
 \end{equation}
This completes the definition of regularized initial data and noise coefficient.

\smallskip

    We are now in the position to define the concept of a pathwise bounded solution to the regularized system \eqref{regularized_problem1}--\eqref{regularized_problem2}.
     \begin{defn}[Pathwise bounded solution to \eqref{regularized_problem1}--\eqref{regularized_problem2}]
        \label{defn_U_epsilon} Fix $\epsilon>0$, and a stochastic basis $(\Omega, \mathcal{F}, (\mathcal{F}_t)_{t\geq 0}, \mathbb{P}, W)$, where $(\Omega, \mathcal{F}, (\mathcal{F}_t)_{t\geq 0}, \mathbb{P}) $ is a filtered probability space, and $W$ is a real-valued $(\mathcal{F}_t)_{t\geq 0}$-adapted Wiener process. Let $U_{\epsilon0} := (\rho_{\epsilon0}, m_{\epsilon0}) \in H^{2}(0,1)\times H^{2}(0,1)$ be defined by \eqref{U_epsilon0} and let $\Phi_\epsilon$ be defined by \eqref{reg_Phi}. A pathwise bounded solution to \eqref{regularized_problem1}--\eqref{regularized_problem2} is a predictable process $U_\epsilon = (\rho_\epsilon, m_\epsilon)$ such that
		\begin{enumerate}
			\item {$U_\epsilon \in C_{loc}([0, \infty); H^{2}(0, 1))$, $\mathbb{P}$-almost surely.} 
            \item $U_\epsilon$ is $(\mathcal{F}_t)_{t\geq 0}$-adapted.
			
			\item The process $U_\epsilon$ satisfies the following weak formulation $\mathbb{P}$-almost surely, for all test functions $\varphi \in C^2_c((0,1); \mathbb{R}^2)$, and every $t\geq 0$:
			\begin{equation}\label{epsilonweakformulation}
				\begin{split}
				    \langle U_\epsilon(t), \varphi \rangle = \langle U_{\epsilon 0}, \varphi \rangle & + \int_{0}^{t} \langle F(U_\epsilon), \partial_x\varphi \rangle \,ds + \int_{0}^{t} \langle G(U_\epsilon), \varphi \rangle \,ds\\ & + \epsilon \int_{0}^{t}\langle U_\epsilon, \partial_x^2\varphi \rangle \,ds+ \int_{0}^{t}\langle { \Phi_\epsilon(x, U_\epsilon)}, \varphi \rangle\,dW(s),
				\end{split}
			\end{equation}
             where $F,G$ are defined in \eqref{defn_FGPhi}, and $ \Phi_\epsilon$ is defined in \eqref{reg_Phi}.
			
            \item The boundary conditions  $\partial_x\rho_\epsilon (t, 0) = \partial_x \rho_\epsilon(t, 1) = 0$ and $m_\epsilon(t, 0) = m_\epsilon(t, 1) = 0$ are satisfied for all $t \geq 0$,  $\mathbb{P}$-almost surely.
		\end{enumerate}
    \end{defn}
  The main result of this section is the following result on the existence of a pathwise bounded solution to the approximate problem \eqref{regularized_problem1}--\eqref{regularized_problem2} with artificial viscosity.
    
\begin{thm}[Existence of pathwise bounded solution to \eqref{regularized_problem1}--\eqref{regularized_problem2}]\label{thm_existence_pathwise_U_epsilon}
   Fix $\epsilon>0$, and a stochastic basis $(\Omega, \mathcal{F}, (\mathcal{F}_t)_{t\geq 0}, \mathbb{P}, W)$, where $(\Omega, \mathcal{F}, (\mathcal{F}_t)_{t\geq 0}, \mathbb{P}) $ is a filtered probability space, and $W$ is a real-valued $(\mathcal{F}_t)_{t\geq 0}$-adapted Wiener process. Assume that the noise coefficient $\sigma$ satisfies the assumptions stated in \eqref{noise_assumption}. Let the regularized noise coefficient $\sigma_\epsilon(\cdot, U)\in C_c^1([0,1])$ be as defined in \eqref{defn_sigmaepsilon}, and define $U_{\epsilon0} \in H^{2}(0,1)$ via \eqref{U_epsilon0}, for $U$ satisfying the conditions in \eqref{init_cond}. Then, there exists a pathwise bounded solution, $U_\epsilon = (\rho_\epsilon, m_\epsilon)$, as defined in Definition \ref{defn_U_epsilon},  to the problem \eqref{regularized_problem1}--\eqref{regularized_problem2}.

                Furthermore, this solution $U_\epsilon$ satisfies the following:
        \begin{enumerate}
            \item {$\|U_\epsilon\|_{L^\infty((0, \infty)\times (0,1))} \leq C$, for some{ deterministic constant} $C= C(M_1, \gamma)$ independent of $\epsilon$}, $\mathbb{P}$-almost surely.
            \item { For every $T>0$, there exists a random variable $\displaystyle c^{\epsilon}_T(\omega)>0$ depending on $\epsilon, T, \|U_0\|_{L^\infty}, \|u_0\|_{L^\infty}$, and $\displaystyle \int_0^T\int_0^1 \rho_\epsilon (\partial_x u_\epsilon)^2\,dx\,dt$, such that $\rho_\epsilon(t,x) \geq c^{\epsilon}_T(\omega)$, for every $(t,x)\in [0,T]\times[0,1]$, $\mathbb{P}$-almost surely, where $u_0 := m_0/\rho_0$ is the initial fluid velocity.}
            \item For all test functions $\varphi \in C^2_c((0,1))$, { $\psi \in C^1_c([0,\infty))$}, and all entropy-entropy flux pairs $(\eta, H)$ defined in \eqref{entropy_pair_formula}, the following entropy balance equation is satisfied $\mathbb{P}$-almost surely for every { $t \in [0, \infty)$}:
			\begin{equation}\label{epsilonentropy}
				\begin{split}
					&\langle \eta(U_\epsilon(t)), \varphi \rangle\psi(t)  = \langle \eta(U(0)), \varphi \rangle \psi(0)\\
                    &+ \int_0^t \langle \eta(U_\epsilon)(s), \varphi \rangle \psi'(s)\,ds + \int_{0}^{t} \langle H(U_\epsilon), \partial_x\varphi \rangle\psi(s) \,ds + \int_{0}^{t} \langle m_\epsilon \partial_m \eta(U_\epsilon), \varphi \rangle \psi(s) \,ds \\
					&+  \epsilon\int_{0}^{t}\langle \eta(U_\epsilon), \partial_x^2\varphi \rangle \psi(s)\,ds -  {\epsilon\int_{0}^{t}\langle \nabla^2\eta(U_\epsilon) \partial_xU_\epsilon \cdot \partial_x U_\epsilon, \varphi \rangle \psi(s)\,ds}\\
					&+ \int_{0}^{t}\langle \partial_m\eta(U_\epsilon) \sigma(x, U_\epsilon), \varphi \rangle \psi(s)\,dW(s) + \frac{1}{2}\int_{0}^{t} \langle \partial_m^2 \eta(U_\epsilon) { \sigma_\epsilon^2}(x, U_\epsilon), \varphi \rangle \psi(s)\,ds.
				\end{split}
			\end{equation}
        \end{enumerate}
    \end{thm}

    The rest of this section is organized as follows. In Section \ref{sec_splitting_scheme}, we define the splitting scheme used to construct approximate solutions to \eqref{regularized_problem1}-\eqref{regularized_problem2}. This scheme splits the deterministic and stochastic dynamics of the full problem into two subproblems in a way that preserves uniform boundedness properties. 

In Section \ref{sec_tau_existence}, we discuss the existence of solutions to the two subproblems and prove an approximate entropy balance equation satisfied by the approximate solutions at the $\tau$ level.
    
    \subsection{The $\tau$-approximation via a time-splitting scheme} 
    \label{sec_splitting_scheme}
    
    To obtain the pathwise bounded solution that exists on the whole time interval $[0, \infty)$ stated in Theorem \ref{thm_existence_pathwise_U_epsilon}, we will first construct solution on a time interval $[0, T]$ for a fixed but arbitrary $T>0$ via a splitting scheme. For this purpose we fix $N \in \mathbb{Z}_+$, and discretize $[0, T]$ into $N$ subintervals $[t_n, t_{n+1})$ of size $\tau = \Delta t$, where $\tau = \frac{T}{N}$, and $t_n = n\tau$, for $n = 0, 1, \dots, N-1$. 
    The goal is to define the approximate solution $[{U}_\epsilon]_\tau$ to the regularized system (\ref{regularized_problem1}) but defined on the finite rectangle:
    $$Q_T = \{(t,x) | t\in (0,T), x \in (0,1) \}.$$
    To simplify notation, in the remainder of this section we will drop the subscript $\epsilon$ and denote
    $$
    {U}_\tau : = [{U}_\epsilon]_\tau.
    $$
    We will eventually remove the dependence of the solution on $T$ and obtain a pathwise bounded solution, in the sense of Definition~\ref{defn_U_epsilon}, defined on the entire time interval $[0,\infty)$. This extension will be justified by the pathwise uniqueness of the solution.

    On each subinterval $(t_n,t_{n+1})$ we propose to solve a deterministic and a stochastic subproblem in a way that corresponds to an interpolated  Lie-Trotter scheme \cite{Bensoussan_Glowinski_Rascanu_1992}. To state the splitting scheme, for each fixed $\epsilon > 0$ and an \textit{arbitrary but fixed final time} $T>0$, we introduce the following deterministic and stochastic subproblems:

    \medskip
    \noindent \underline{\textbf{Deterministic subproblem}}: Find ${U}(t,x) = ({\rho,}(t,x),{m}(t,x))$ such that for each $t \in (t_n,t_{n+1})$ and $x\in(0,1)$, the following {\emph{deterministic problem}} is satisfied:
	\begin{equation}\label{subproblem_1}
		\begin{cases}
			\partial_t {\rho} + \partial_x {m} = \epsilon \partial_x^2{\rho},
            \\
			\partial_t {m} + \partial_x(\frac{{m}^2}{{\rho}} + p({\rho})) = \epsilon \partial_x^2 {m}-\alpha {m},
            \\
			({\rho}(t_n, x), {m}(t_n, x)) = ({\rho}_{t_n}(x), {m}_{t_n}(x)) & x \in [0, 1],
            \\
            \partial_x  {\rho}(t, 0) = \partial_x {\rho}(t, 1) = 0, & t \in [t_n, t_{n+1}],
            \\
			{m}(t, 0) = {m}(t, 1) = 0, & t \in [t_n, t_{n+1}],
		\end{cases}
	\end{equation}
    where $({\rho}_{t_n}(x), {m}_{t_n}(x))$ is a prescribed initial data -- typically given by the solution at the previous time step. We denote the solution operator associated with \eqref{subproblem_1} by $\mathbf{S}$, so that
    \begin{equation}\label{Ubar}
    {U}(t, x) = \mathbf{S}(t - t_n){U}(t_n, x),
    \quad t \in [t_n,t_{n+1}], x\in(0,1).
    \end{equation}
    Since the coefficients and the boundary data are independent of time, the solution to this problem depends only on the time difference $t - t_n$, and not on the absolute times themselves. In fact, $\mathbf{S}(t)$ forms a (nonlinear) semigroup.
    
    \medskip
    
    \noindent \underline{\textbf{Stochastic subproblem}}: Let $(\Omega, \mathcal{F}, (\mathcal{F}_t)_{0\leq t\leq T}, \mathbb{P}, W)$ be a fixed stochastic basis, where $(\Omega, \mathcal{F},(\mathcal{F}_t)_{0\leq t\leq T}, \mathbb{P})$ is a filtered probability space, and $W$ a real-valued $(\mathcal{F}_t)_{0\leq t\leq T}$-adapted Wiener process. The stochastic subproblem is defined by finding ${U}(t,x) = ({\rho,}(t,x),{m}(t,x))$ such that for each $t \in (t_n,t_{n+1})$ and $x\in(0,1)$, the following holds:
	\begin{equation}\label{subproblem_2}
    \begin{cases}
        {\rho}(t, x) =  {\rho}(t_n, x),
        \\
		\displaystyle {m}(t, x) = {m}(t_n, x) + \int_{t_n}^{t}\sigma_\epsilon(x, {U}(s,x))\,dW(s),
        \\
        ({\rho}(t_n, x), {m}(t_n, x)) = ({\rho}_{t_n}(x), {m}_{t_n}(x)) & x\in [0,1],
        \\
        \partial_x{\rho}(t, 0) = \partial_x{\rho}(t, 1) = 0 & t \in [t_n, t_{n+1}],
        \\
        {m}(t, 0) = {m}(t, 1) = 0  & t \in [t_n, t_{n+1}],
    \end{cases}
	\end{equation}
where the initial condition $({\rho}_{t_n}, {m}_{t_n})$ is a given $\mathcal{F}_{t_n} \times \mathcal{F}_{t_n}$-measurable process.
Notice that the density ${\rho}$ is kept constant in the stochastic subproblem, and only the momentum ${m}$ is updated by the stochastic integral. We denote the solution operator associated with subproblem \eqref{subproblem_2} by $\mathbf{R}$, so that
    \begin{equation}\label{StochasticU}
    \displaystyle {U}(t, x) = \mathbf{R}(t, t_n){U}(t_n, x) := {U}(t_n, x) + \begin{pmatrix}0\\
		\int_{t_n}^{t}\sigma_\epsilon(x, {U})\,dW(s)
	\end{pmatrix}.
    \end{equation}
    \medskip
    \noindent \underline{\textbf{Splitting scheme}}: Having defined the deterministic and stochastic subproblems, we can now define the full splitting scheme. For each given $\tau = T/N$, the scheme, based on the Lie-Trotter formula, is defined as follows:
    \begin{equation} \label{Scheme}
		\begin{split}
			& {U}_\tau(t, x) = \begin{pmatrix}
			    {\rho}_\tau \\ {m}_\tau
			\end{pmatrix} : = \begin{cases}
				U_{\epsilon 0}(x) & t=0\\ 
				\frac{t_{n+1}- t}{\tau} \mathbf{S}(t - t_n){U}_\tau(t_n, x)+ \frac{t - t_n}{\tau}\mathbf{R}(t, t_n)\textbf{S}(t_{n+1} - t_n){U}_\tau(t_n,x) & t \in (t_n, t_{n+1}]
			\end{cases}.
		\end{split}
	\end{equation}
    Notice that the deterministic subproblem takes for the initial data the solution calculated from the previous time step, while the stochastic subproblem takes for the initial data the solution just calculated from the deterministic subproblem, obtained at $t_{n+1}$. 
    At discrete points $t_{n+1}$, this is a discrete Lie-Trotter splitting scheme \cite{Bensoussan_Glowinski_Rascanu_1992}. The interpolation defines a continuous process everywhere on the time interval $[0,T]$.
    
   To simplify calculations later in the paper, we introduce the following notation for the interpolated quantities on the right hand-side of \eqref{Scheme}:

	\begin{equation} \label{U_defn}
		\begin{split}
			& \bar{U}_\tau(t, x)= \begin{pmatrix}
			    \bar{\rho}_\tau (t, x)\\ \bar{m}_\tau(t, x)
			\end{pmatrix} := \textbf{S}(t - t_n){U}_\tau(t_n,x ), \qquad \qquad \ \ \   t \in (t_n, t_{n+1}],\\
			& \tilde{U}_\tau(t, x) = \begin{pmatrix}
			    \tilde{\rho}_\tau(t, x) \\ \tilde{m}_\tau(t, x)
			\end{pmatrix}:=  \mathbf{R}(t, t_n)\textbf{S}(t_{n+1} - t_n){U}_\tau(t_n,x),  t \in (t_n, t_{n+1}].
		\end{split}
	\end{equation}
    The corresponding fluid velocities will be denoted by:
    \begin{equation}\label{defn_small_u}
        {u}_\tau:= \frac{{m}_\tau}{{\rho}_\tau}, \;\; \bar{u}_\tau:= \frac{\bar{m}_\tau}{\bar{\rho}_\tau}, \;\;\tilde{u}_\tau:= \frac{\tilde{m}_\tau}{\tilde{\rho}_\tau}.
    \end{equation}


    \subsubsection{Existence of solutions to the the deterministic and stochastic subproblems}\label{sec_tau_existence}

    In this section, we show the existence of solutions to the deterministic and stochastic subproblems, and prove an entropy equality satisfied by the approximate solution $U_\tau$ to the parabolic regularization. The result is stated in the following proposition.
        
\begin{prop}
\label{prop_existence_split} Fix $\epsilon>0$, $T>0$, and a stochastic basis $(\Omega, \mathcal{F}, (\mathcal{F}_t)_{0\leq t\leq T}, \mathbb{P}, W)$, where $(\Omega, \mathcal{F},(\mathcal{F}_t)_{0\leq t\leq T}, \mathbb{P})$ is a filtered probability space, and $W$ is a real-valued $(\mathcal{F}_t)_{0\leq t\leq T}$-adapted Wiener process. Assume that the initial condition $U_{\epsilon 0}\in H^{2}(0,1)$, $\rho_{\epsilon 0} \geq c$ for some $c >0$, and the noise coefficient $\sigma_\epsilon \in C^1([0,1])$ satisfies \eqref{noise_assumption}.
Then:
\begin{enumerate}

    \item There exists a unique solution $\bar{U}_\tau\in C(t_n,t_{n+1}; H^{2}(0,1))$ to the deterministic subproblem \eqref{subproblem_1} with initial data in $H^{2}(0,1)$.
    %
    \item There exists a unique pathwise solution $\tilde{U}_\tau \in C(t_n,t_{n+1}; H^{ 2}(0,1))$ to the stochastic subproblem \eqref{subproblem_2} with initial data in $H^{2}(0,1)$. 

    %
    \item For every $\tau >0$, the $H^{2}(0,1)$-valued processes ${U}_\tau$ defined by the interpolated Lie-Trotter formula \eqref{Scheme}, and the processes $\bar{U}_\tau, \tilde{U}_\tau$ given by \eqref{U_defn}, are progressively measurable with respect to $(\mathcal{F}_t)_{0\leq t\leq T}$.
\end{enumerate}
\end{prop}
\begin{proof}
We first address the first item of this proposition, namely the existence of solution of the two subproblems \eqref{subproblem_1} and \eqref{subproblem_2}. 

	\textbf{For the existence of solution to the deterministic subproblem} \eqref{subproblem_1}, we refer to \cite{MR1383202}, where the existence of a solution to the isentropic Euler equation (without damping) in the class of functions $L^\infty_tW^{1, \infty}_x \cap C_tL^2_x$, for Lipschitz continuous initial data $U_0=(\rho_0, m_0)$ with $\rho_0 > 0$ on $[0,1]$, is established. Since the system \eqref{subproblem_1} propagates $H^{2}$ regularity of the initial condition (see later calculation in Theorem \ref{regularity_of_Uhat}), we have that $\bar{U}\in C(t_n, t_{n+1};H^{2}(0,1)).$
    For the initial boundary value problem, in particular, with Dirichlet boundary conditions on $m$, we refer to section 4 of \cite{MR1724654} for more details, where they showed global existence of smooth viscous solution to the initial boundary value problem of isentropic Euler equation with artificial viscosity, and since in 1D, $H^{2}(0,1) \subset C^1([0,1])$, the boundary conditions $\partial_x \bar{\rho}(t, 0) = \partial_x \bar{\rho}(t, 1) = 0$ and $\bar{m}(t, 0) = \bar{m}(t, 1) = 0$ are satisfied in the classical sense.
    
    \textbf{For the existence of solution to the stochastic subproblem \eqref{subproblem_2}} , we refer to the proof of Proposition 3.8 in \cite{berthelin_stochastic_2019}, where they used a fixed point method to show that { there exists a solution to \eqref{subproblem_2} in the space $C(t_n, t_{n+1}; H^{2}(\mathbb{T}))$, with initial condition $U(t_n, x)\in H^{2}(\mathbb{T})$.} A similar approach can be adapted for showing existence of solution to \eqref{subproblem_2} in the space $C(t_n, t_{n+1}; H^{2}(0,1))$, with the boundary conditions \eqref{subproblem_2}$_{4,5}$ satisfied for the following reasoning: since the density $\tilde{\rho}$ is kept constant in this subproblem, if $\partial_x \tilde{\rho}(t_n, 0) = \partial_x\tilde{\rho}(t_n, 1) = 0$, then $\partial_x \tilde{\rho}(t, 0) = \partial_x\tilde{\rho}(t, 1) = 0$, for all $t\geq t_n$. Regarding the Dirichlet boundary condition on $\tilde{m}$: since $\sigma_\epsilon(0, U) = \sigma_\epsilon(1, U) = 0 $ by the assumption in \eqref{noise_assumption} and $R(t, t_n)\tilde{m}(t_n, x) = \tilde{m}(t_n, x) + \int_{t_n}^{t}\sigma_\epsilon(x, \tilde{U})\,dW(s) = 0$, for $x = 0, 1$, the Dirichlet boundary condition is preserved for $\tilde{m}$ in the stochastic subproblem. Finally, we comment on the measurability of the solution to \eqref{subproblem_2}.  Since the function $t \mapsto \tilde{U}(t)$ is continuous, and $\tilde{U}$ is $(\mathcal{F}_t)_{t_n\leq t\leq t_{n+1}}$-adapted, it is progressively measurable with respect to the filtration $(\mathcal{F}_t)_{t_n\leq t\leq t_{n+1}}$. 
    


Next, we discuss the last item of this proposition, namely the \textbf{measurability of the approximate solution.} We first observe that the function $\tilde{U}_\tau(t_1)$ defined in \eqref{U_defn} is $\mathcal{F}_{t_1}$- measurable. Therefore, ${U}_\tau(t_1)$ is $\mathcal{F}_{t_1}$-measurable due to its definition \eqref{Scheme}. To show that ${U}_\tau(t)$ is $\mathcal{F}_t$-measurable for every $t\in [t_1, t_2]$, we first recall from \eqref{Scheme} that for $t\in [t_1, t_2]$:
\[{U}_\tau(t) = \frac{t_{n+1}-t}{\tau}\mathbf{S}(t-t_n){U}_\tau(t_n) + \frac{t-t_n}{\tau}\mathbf{R}(t, t_n)\mathbf{S}(\tau){U}_\tau(t_n) = \frac{t_{n+1}-t}{\tau}\bar{U}_\tau(t) + \frac{t-t_n}{\tau}\tilde{U}_\tau(t).\]
The first term $\mathbf{S}(t-t_n){U}_\tau(t_n)$ is $\mathcal{F}_{t_n}$-measurable for every $t\in [t_n, t_{n+1}]$, because ${U}_\tau(t_n) \mapsto \mathbf{S}(t-t_n){U}_\tau(t_n)$ is Lipschitz continuous from $H^{2}(0,1)\to C(t_n, t_{n+1}; H^{2}(0,1))$. The second term $\mathbf{R}(t, t_1)\mathbf{S}(\tau){U}_\tau(t_n)$ is $\mathcal{F}_t$-measurable for every $t\in [t_n ,t_{n+1}]$, because $\mathbf{S}(\tau){U}_\tau(t_n)$ is $\mathcal{F}_{t_n}$-measurable, and $\mathbf{R}(t,t_n)$ operates on $\mathbf{S}(\tau){U}_\tau(t_n)$ by adding $\int_{t_n}^t \Phi_\epsilon(x, \tilde{U}_\tau(s))\,dW(s)$, which is a $(\mathcal{F}_t)_{t_n\leq t\leq t_{n+1}}$-adapted martingale. Therefore, ${U}_\tau(t)$ is $\mathcal{F}_t$-measurable for every $t\in [t_n, t_{n+1}]$. Repeating this process, one obtains that all three processes ${U}_\tau$, $\bar{U}_\tau$, and $\tilde{U}_\tau$ are $(\mathcal{F}_t)_{0\leq t\leq T}$-adapted. Since the processes ${U}_\tau$, $\bar{U}_\tau$, and $\tilde{U}_\tau$ are all $(\mathcal{F}_{t})_{0 \leq t \leq T}$-adapted processes that are left continuous in time almost surely, they are all progressively measurable with respect to $(\mathcal{F}_t)_{0\leq t\leq T}$.

\end{proof}

    \section{Passage to the limit as $\tau \to 0$}\label{tauto0}

    In this section, our primary goal is to take the limit $\tau \to 0$ and prove Theorem \ref{thm_existence_pathwise_U_epsilon}. 
For this purpose, we start by deriving the necessary uniform $L^\infty$ bounds for the approximate solutions, presented in Section \ref{sec_preliminary_estimates}.
Section \ref{sec_timereg_estimate} focuses on time-regularity estimates for $U_\tau$, which are essential for establishing the tightness of probability laws in the appropriate function spaces. 

In Sections \ref{entropydissipation}, \ref{positivitytau}, and \ref{sec_diffestimate}, we establish key properties of the approximate solutions, including a uniform-in-$\tau$ strictly positive lower bound for the density $\rho_\tau$, as well as the convergence of the various approximate forms ($U_\tau, \bar{U}_\tau, \tilde{U}_\tau$) in suitable functional spaces. 
These estimates are crucial for passing to the limit in the entropy balance equation \eqref{epsilonentropy}. 

Subsequently, in Section \ref{skorohodsection_tau}, we apply the Skorohod Representation Theorem to pass to the limit $\tau \to 0$, yielding a martingale solution to the regularized problem \eqref{regularized_problem1}--\eqref{regularized_problem2} on $[0, T]$ that satisfies \eqref{epsilonentropy}. 
Finally, in Section \ref{sec_uniqueness_pathwise_U_epsilon}, we construct a pathwise bounded solution, as per Definition \ref{defn_U_epsilon}, on the interval $[0, \infty)$. 
This is achieved using the Gy\"{o}ngy-Krylov Theorem (see \cite{MR4491500}) and a standard gluing argument, leveraging the pathwise uniqueness property of the solution $U_\epsilon$.

    
	\subsection{Uniform $L^{\infty}(Q_{T})$ estimates on approximate solutions}\label{sec_preliminary_estimates}
    
      In this section, we establish uniform $L^\infty$ bounds for the approximate solution $U_\tau$. 
      To this end, we first show that the solution operator $\mathbf{R}$ associated with the stochastic subproblem~\eqref{subproblem_2} maps any bounded interval $-M \le m \le M$ onto itself, almost surely, where $M>0$ is determined by the compact support of the noise coefficient $\sigma_\epsilon$ with respect to the state variable $m$ (see Lemma 1 from \cite{bauzet_time-splitting_2014}). We state the result in terms of the solution operator $\mathbf{R}(t, s)$ associated with subproblem \eqref{subproblem_2} defined on an arbitrary interval $[s,T]$, with  initial data $U(s,x)$. Operator $\mathbf{R}(t, s)$ gives rise to the continuous process $U(t,x)$ for all $t \in [s,T]$:
      \begin{equation}\label{subproblem_2_cont}
    \displaystyle {U}(t, x) = \mathbf{R}(t, s){U}(s, x) := {U}(s, x) + \begin{pmatrix}0\\
		\int_{s}^{t}\sigma_\epsilon(x, {U})\,dW(s')
	\end{pmatrix}.
    \end{equation}
    Therefore, since only $m$ is modified in the stochastic subproblem, the invariant region $[-M,M]$ is determined only by the state variable $m$ and the support of the noise coefficient $\sigma_\epsilon$. 
      More precisely, we have the following result.

	\begin{lem}\label{R_operator bound}
        Let $\textbf{R}(t,s)$ be the map defined in \eqref{subproblem_2_cont}  where the noise coefficient $\sigma_\epsilon(x, \cdot, \cdot)$ is compactly supported on $[0, M_1]\times[-M_2, M_2]$. Then almost surely, for all $t >s$, the operator $\mathbf{R}(t, s)$ maps $[0, M_1]\times[-M_2, M_2]$ onto itself and is the identity operator outside of this compact set.
	\end{lem}

	\begin{proof}        
        Since the density variable is kept constant in the stochastic subproblem \eqref{subproblem_2_cont}, the first component of the operator $\mathbf{R}(t, s)$ is exactly the identity operator both on the support of $\sigma_\epsilon(x, \cdot, \cdot)$ and outside of this support. 
        
        To prove that the second component of $\mathbf{R}(t, s)$ has the property stated in this lemma,
       we need to use a suitable auxiliary function to help manifest this property. Indeed, we consider a smooth function $f(z) \in C^{\infty}(\mathbb{R})$, vanishing in $[-M_2, M_2]$ and strictly increasing for $|z| \ge M_2$. We then apply It\^{o}'s formula with the map $m \mapsto f(m)$ to \eqref{subproblem_2_cont}$_2$ and obtain:
		\begin{equation}
			f(m(t, x)) = f(m(s, x)) + \int_{s}^{t} \partial_m f(m(s', x))\sigma_\epsilon(x, U)\,dW(s')+ \frac{1}{2}\int_{s}^{t} \partial_m^2 f(m(s', x))\sigma_\epsilon^2(x, U)\,ds'.
		\end{equation}
		
		If $m(s, x) \in [-M_2, M_2]$ then $f(m(s, x)) = 0 = f(m(t, x))$ implies that $m(t, x)\in[-M_2, M_2]$. 
        
        If $m(s, x) \notin [-M, M]$, then   $ \sigma_\epsilon(x, U) = 0$ and $\sigma_\epsilon^2(x, U) = 0$  due to the fact that $\sigma_\epsilon(x, \rho, \cdot)$ has compact support on $[-M_2, M_2]$. Therefore, $f(m(s, x)) = f(m(t, x))$, and by the injectivity of $f$, we have that $m(s, x) = m(t, x)$, which implies that the second component of $\mathbf{R
        }(t, s)$ is the identity operator. We thus conclude the proof.
	\end{proof}
	Using the above lemma, we now prove $L^\infty(Q_T)$ bounds, uniform in $\tau$ and  $\epsilon$ and independent of $T$, on the approximate solution $U_\tau$ and the solutions of the two subproblems $\bar{U}_\tau$, $\tilde{U}_\tau$.
	\begin{prop}[Uniform $L^\infty$ bounds]	\label{L_infty_estimates}
		Let $U_0$ be the initial condition to the hyperbolic problem \eqref{problem}, and $U_{\epsilon 0}$ be the initial condition to the parabolic regularization, defined by \eqref{U_epsilon0}. Let $U_\tau$ denote the approximate solution generated by
the time-splitting scheme~\eqref{Scheme}, obtained by interpolating between the states $\bar U_\tau$ and $\tilde U_\tau$ defined in~\eqref{U_defn}.
        
        Then, there exists a deterministic constant $C:=C(\|U_0\|_{L^\infty}, \gamma)$, independent of $\epsilon$, $\tau$, and $T$, such that:
        
        \begin{itemize}
        \item $\|\bar{U}_\tau\|_{L^\infty(Q_T)}, \|\tilde{U}_\tau\|_{L^\infty(Q_T)}, \|{U}_\tau\|_{L^\infty(Q_T)} \leq C$, $\mathbb{P}$-almost surely;
        \item $\|\bar{u}_\tau\|_{L^\infty(Q_T)}, \|\tilde{u}_\tau\|_{L^\infty(Q_T)}, \|{u}_\tau\|_{L^\infty(Q_T)} \leq C$ , $\mathbb{P}$-almost surely,
        \end{itemize}
        where $\bar{u}_\tau$, $\tilde{u}_\tau$, and ${u}_\tau$ are the corresponding velocities defined by \eqref{defn_small_u}.
        
	\end{prop}
	\begin{proof}
    The proof is based on introducing an invariant region (independent of $\epsilon$, $\tau$ and $T$) for $(\rho,m)$ satisfying the deterministic subproblem \eqref{subproblem_1}, and showing that all the quantities associated with the splitting scheme belong to the invariant region.
    
    More specifically, we introduce the Riemann invariants associated with system \eqref{problem}, which are defined by  
		\begin{equation}\label{Riemann_invariants}
			w = \frac{m}{\rho} + \rho^\theta, \;\; z = \frac{m}{\rho} -\rho^\theta,  \text{ where } \theta
             = \frac{\gamma -1}{2}.
		\end{equation}
		For each constant $C>0$, we define the invariant set $\Sigma_C$ using the Riemann invariants: 
		\[\Sigma_C := \{(\rho, m) \in \mathbb{R}_+ \times \mathbb{R}: -C \leq z < w \leq C\}.\] 
		One can check, using the theory developed by Cheuh-Conley-Smoller in \cite{MR430536} (see Theorem 4.4 \cite{MR430536}), that $\Sigma_C$ defines an invariant region for the deterministic subproblem \eqref{subproblem_1} for all $\epsilon >0 $ in the sense that if the initial condition $U_{\epsilon 0} \in \Sigma_C$, then the solution $U_\epsilon(t)$ to  \eqref{subproblem_1} satisfies $U_\epsilon(t) \in \Sigma_C$ for all $t >0$ and all $\epsilon>0$. 
        
        Since from the definition of $U_{\epsilon 0}$ \eqref{U_epsilon0} we have $0 \leq \rho_{\epsilon 0} \leq M_1$ and $-M_2 \leq u_{\epsilon 0} \leq M_2$ for all $\epsilon >0$, we conclude that $U_{\epsilon 0} \in \Sigma_C$ for some constant $C>0$ independent of $\epsilon$. 
        Therefore, $\textbf{S}(t - t_n)U_{\epsilon 0} \in \Sigma_C$ for all $t \in [t_n,t_{n+1}]$, $n = 0,...,N-1$, 
        and $\epsilon>0$, which then implies that $\|\bar{\rho}_\tau\|_{L^\infty(Q_T)}\leq C, \|\bar{u}_\tau\|_{L^\infty(Q_T)} \leq C$, where $C = C(\|U_0\|_{L^\infty}, \gamma)$ is independent of $\tau>0$, $\epsilon>0$, $T>0$, and $\omega\in \Omega$. 

        {
    Next, by Lemma \ref{R_operator bound}, we know that 
    $\textbf{R}(t,t_n)\textbf{S}(t_{n+1} - t_n)U_{\epsilon 0} \in $ supp $\sigma$ for any  $t \in [t_n,t_{n+1}], n = 0,...,N-1$,
    }
        and $\text{supp }\sigma(x, \cdot, \cdot) \subset \Sigma_C$ by the assumption on the noise \eqref{noise_assumption}.  Therefore we can conclude that, $\tilde{U}_\tau(t, x) \in \Sigma_C$ for all $x\in [0, 1]$ and $t \in [0, T]$, for any fixed $T>0$. In particular, $\|\tilde{\rho}_\tau\|_{L^\infty(Q_T)}\leq C, \|\tilde{u}_\tau\|_{L^\infty(Q_T)} \leq C$. 
        
        Lastly, since ${U}_\tau$ is an affine combination of $\bar{U}_\tau$ and $\tilde{U}_\tau$, we conclude that ${U}_\tau \in \Sigma_C$, $\|{u}_\tau\|_{L^\infty(Q_T)}\leq C$, and $\|{U}_\tau\|_{L^\infty(Q_T)} \leq C$, where $C = C(\|U_0\|_{L^\infty}, \gamma)$.
	\end{proof}

     Based on the uniform $L^\infty$ bounds on the approximate solutions derived above, we now deduce $L^\infty$ estimates, independent of $\tau$, $\epsilon$, and $T$, on the entropy and on the gradient of entropy associated with approximate solutions.

	\begin{lem}[Uniform $L^\infty$ bounds on entropy and its gradient]\label{lem_Linfty_eta_gradient_eta} 
    Let $U_0$ be the initial condition to the hyperbolic problem \eqref{problem}, and let $U_{\epsilon 0}$ be the initial condition to the parabolic regularization, defined by \eqref{U_epsilon0}. Let $U_\tau$ denote the approximate solution generated by
the time-splitting scheme~\eqref{Scheme}, obtained by interpolating between the states $\bar U_\tau$ and $\tilde U_\tau$ defined in~\eqref{U_defn}.
 Let $\eta$ be an entropy function associated with the original problem \eqref{problem}, as defined in \eqref{entropy_pair_formula} via a convex function $g \in C^{2}(\mathbb{R})$. Then, the following uniform estimates hold: 
		\begin{itemize}
			\item $\| \eta(\bar{U}_\tau)\|_{L^\infty(Q_T)},  \| \eta(\tilde{U}_\tau)\|_{L^\infty(Q_T)}, \| \eta({U}_\tau)\|_{L^\infty(Q_T)} \leq C$, $\mathbb{P}$-almost surely
			\item $\|\nabla \eta(\bar{U}_\tau)\|_{L^\infty(Q_T)},  \|\nabla \eta(\tilde{U}_\tau)\|_{L^\infty(Q_T)}, \|\nabla \eta({U}_\tau)\|_{L^\infty(Q_T)} \leq C$, $\mathbb{P}$-almost surely, 
		\end{itemize}
		where $C:=C(\|U_0\|_{L^\infty}, \gamma, g)$ is deterministic and is independent of $\tau$, $\epsilon$, and $T$.
	\end{lem}
	\begin{proof}
    Recall from \eqref{entropy_pair_formula} that for any convex function $g\in C^2(\mathbb{R})$, the following is an entropy function for \eqref{problem}:
		\[\eta(U) = \rho \int_{-1}^{1}g\left(\frac{m}{\rho}+z\rho^\theta\right)(1-z^2)\,dz.\]
		By a direct computation, its gradient $\displaystyle \nabla \eta(U) = (\partial_\rho \eta, \partial_m \eta)$ is given by:
		\begin{equation}\label{gradient_eta_expression}
		    \begin{split}
			\partial_\rho \eta(U) &= \int_{-1}^{1}g\left(\frac{m}{\rho}+z\rho^\theta\right)(1-z^2)\,dz  + \int_{-1}^{1} z\theta \rho^{\theta }g'\left(\frac{m}{\rho}+z\rho^\theta\right)(1-z^2)\,dz,\\
			\partial_m\eta(U) &= \int_{-1}^{1}g'\left(\frac{m}{\rho}+z\rho^\theta\right)(1-z^2)\,dz.
		\end{split}
		\end{equation}
        The uniform boundedness of $\eta$ and $\nabla \eta$ follow directly from
		Lemma \ref{L_infty_estimates}, since the integrands are uniformly bounded in $L^\infty$. More specifically,
        $|\frac{m}{\rho} + z\rho^\theta| \leq C(\|U_0\|_{L^\infty}, \gamma)$ for $z \in [-1, 1]$ for each pair of functions
        $(\bar{\rho}_\tau, \bar{u}_\tau), (\tilde{\rho}_\tau, \tilde{u}_\tau), ({\rho}_\tau, {u}_\tau)$. Since $g$ is continuous, $\eta$, $\partial_\rho \eta$ and $\partial_m \eta$ are all bounded uniformly in $\epsilon$ and $\tau$ by some deterministic constant $C = C(\|U_0\|_{L^\infty}, \gamma)$, independent of $T$. 
	\end{proof}

	\subsection{Time regularity estimates} 
    \label{sec_timereg_estimate}
	
    In this section, we are going to pass to the limit as $\tau \to 0$ in the approximate solutions \eqref{Scheme}, and obtain a martingale solution to the parabolic approximation problem \eqref{regularized_problem1}--\eqref{regularized_problem2}. To do this, we first derive the uniform-in-$\tau$ time regularity of the approximate solution ${U}_\tau$ in the following theorem.
	\begin{thm}[Regularity of ${U}_\tau$] \label{regularity_of_Uhat}
		Suppose $U_{\epsilon0} \in H^{2}(0,1)$ and $\rho_{\epsilon0} \geq c$ for some constant $c > 0$, and the noise coefficient $\sigma_\epsilon$ satisfies \eqref{noise_assumption}. Let ${U}_\tau$ be defined in \eqref{Scheme}, and  $\bar{U}_\tau$, $\tilde{U}_\tau$ be defined in \eqref{U_defn}. Then 
		\begin{equation}
			\mathbb{E}\|{U}_\tau\|^2_{C^\alpha(0, T; H^{1}(0,1))} \leq C, \;\; \text{ where } \alpha \in [0, \frac{1}{4}),
		\end{equation} and
		\begin{equation}
			\mathbb{E}\|{U}_\tau\|^2_{C(0, T; H^{2}(0,1))} \leq C, \;\;
			\mathbb{E}\|\bar{U}_\tau\|^2_{L^\infty(0, T; H^{2}(0,1))} \leq C, \;\;
			\mathbb{E}\|\tilde{U}_\tau\|^2_{L^\infty(0, T; H^{2}(0,1))} \leq C,
		\end{equation}
		where $C = C(\epsilon, T, \|U_{\epsilon0}\|_{H^{2}}, \|U_{\epsilon0}\|_{L^\infty})$.
	\end{thm}

    	\begin{proof} 
        To establish this result, we make some comments about the regularizing properties of the heat equation. Let $K_{\epsilon t}(x, y)$ be the heat kernel associated with the heat equation $\partial_t z - \epsilon \partial_x^2 z = 0$ on $[0,1]$ with homogeneous  Dirichlet boundary condition, which is given by
        \[K_{\epsilon t}(x, y) = \sum_{n=-\infty}^{\infty}\frac{1}{\sqrt{4\pi \epsilon t}}\left(e^{-\frac{(x - y -2n)^2}{4\epsilon t}} - e^{-\frac{(x + y -2n)^2}{4\epsilon t}}\right), \;\; x, y \in [0, 1].\] 
        Let $S_\epsilon(t)$ be the corresponding heat semigroup. 
        A brief calculation gives that for every $y \in [0, 1]$,
        \[\|\partial_t^k \partial_x^j K_{\epsilon t}(\cdot, y)\|_{L^m_x} \leq C(m, k, j)\epsilon^k (\epsilon t)^{-\frac{1}{2}(\frac{1}{m'})-\frac{j}{2}-k}, \]
        where $\frac{1}{m} + \frac{1}{m'} = 1$. Then we have the following regularizing property of $S_\epsilon$, thanks to the Young's inequality for convolution:
		\begin{equation}\label{S_epsilon_youngs}
			\|\partial_t^k\partial_x^jS_\epsilon(t)\|_{L_x^p \to L_x^q} \leq \|\partial_t^k \partial_x^j K_{\epsilon t}(\cdot, y)\|_{L^m_x} \leq C(p, q, k, j)\epsilon^k (\epsilon t)^{-\frac{1}{2}(\frac{1}{p}-\frac{1}{q})-\frac{j}{2}-k}
		\end{equation}
		for $k, j \in \mathbb{N}$, $1 \leq p \leq q \leq +\infty$, and $1 + \frac{1}{q} = \frac{1}{m} + \frac{1}{p}$. Note that we can use the heat semigroup $S_\epsilon(t)$ to express $\bar{U}_\tau(t)$, as follows:
		\begin{equation}
			\bar{U}_\tau(t) = S_\epsilon(t-t_n){U}_\tau(t_n) +\int_{t_n}^{t}\partial_x S_\epsilon(t-s') F(\bar{U}_\tau)\,ds' - \int_{t_n}^{t} S_\epsilon(t-s') G(\bar{U}_\tau)\,ds'.
		\end{equation}
        Using this expression, we can now show the main statements of the theorem.

        \medskip
        
        \noindent \textbf{Step 1: Show that ${U}_{\tau} \in C^{\alpha}(0, T; L^{2}(0, 1))$ for $\alpha \in [0, 1/4)$, almost surely}. To show this, by the Kolmogorov continuity criterion (see Theorem 3.5 and Theorem 5.22 in \cite{da_prato_stochastic_2014}), it suffices to show the following increment estimate for all $1 \le \beta < \infty$:
        \begin{equation}\label{incrementbeta}
        \mathbb{E}\|{U}_{\tau}(t) - {U}_{\tau}(s)\|^{\beta}_{L^{2}(0, 1)} \le C|t - s|^{\beta/4} \qquad \text{ for all } s, t \in [0, T],
        \end{equation}
        for some constant $C := C(\beta, \epsilon, T, \|U_{\epsilon 0}\|_{H^{2}}, \|U_{\epsilon 0}\|_{L^{\infty}})$ that is independent of $\tau$ and $s, t \in [0, T]$. 
        
        First, we establish this inequality in the case where $t_{j} \le s < t\leq t_{j + 1}$. Note that:
        \begin{equation}\label{Uhat_s}
			\begin{split}
				&{U}_\tau(s) = \frac{t_{j+1}-s}{\tau}\mathbf{S}(s-t_j){U}_\tau(t_j) + \frac{s - t_j}{\tau}\textbf{R}(s, t_j)\textbf{S}(\tau){U}_\tau(t_j)\\
				= &\frac{t_{j+1}-s}{\tau}\left[S_\epsilon(s - t_j){U}_\tau(t_j) + \int_{t_j}^{s} \partial_xS_\epsilon(s-s')F(\bar{U}_\tau(s')) - S_\epsilon(s-s')G(\bar{U}_\tau(s'))\,ds'\right] \\
				+ &\frac{s-t_j}{\tau} \left[S_\epsilon(\tau){U}_\tau(t_j) + \int_{t_j}^{t_{j+1}} \partial_x S_\epsilon(t_{j+1}-s')F(\bar{U}_\tau(s')) - S_\epsilon(t_{j+1}-s')G(\bar{U}_\tau)\,ds' + \int_{t_j}^{s} \Phi_\epsilon(x, \tilde{U}_\tau(s'))\,dW(s')\right].
			\end{split}
		\end{equation}
        Therefore, we obtain: \begin{equation}\label{eqn_tj+1-s}
			\begin{split}
				&{U}_\tau(t) - {U}_\tau(s) =  \frac{t_{j+1}-t}{\tau} \left[S_\epsilon(t - t_{j}){U}_\tau(t_j) - S_\epsilon(s-t_j){U}_\tau(t_j)\right] - \frac{t - s}{\tau}S_{\epsilon}(s - t_{j}){U}_{\tau}(t_j) + \frac{t - s}{\tau} S_{\epsilon}(\tau) {U}_{\tau}(t_j) \\
				&+ \frac{t_{j+1}-t}{\tau}\left[\int_{t_j}^{s}\left(\partial_xS_\epsilon(t -s')-\partial_xS_\epsilon(s-s')\right)F(\bar{U}_{\tau}(s'))ds' -\int_{t_j}^{s} \left(S_\epsilon(t-s')-S_\epsilon(s-s')\right)G(\bar{U}_{\tau}(s'))ds'\right]\\
                &+ \frac{t - s}{\tau} \int_{t_{j}}^{s} \partial_{x}S_{\epsilon}(s - s')F(\bar{U}_{\tau}(s')) - S_{\epsilon}(s - s')G(\bar{U}_{\tau}(s')) ds' \\
                &+ \frac{t_{j + 1} - t}{\tau} \int_{s}^{t} \partial_{x}S_{\epsilon}(t - s')F(\bar{U}_{\tau}(s')) - S_{\epsilon}(s - s')G(\bar{U}_{\tau}(s')) ds' \\
                &+ \frac{t - s}{\tau} \int_{t_{j}}^{t_{j + 1}}\partial_xS_\epsilon(t_{j + 1}-s')F(\bar{U}_{\tau}(s')) - S_\epsilon(t_{j + 1}-s')G(\bar{U}_{\tau}(s'))ds'\\
				&+ \frac{t - t_{j}}{\tau} \int_{s}^{t}\Phi_\epsilon(x, \tilde{U}_\tau(s'))\,dW(s') + \frac{t-s}{\tau}\int_{t_j}^{s}\Phi_\epsilon(x, \tilde{U}_\tau(s'))\,dW(s').
			\end{split}
		\end{equation}
		Now, we estimate the $L^2$ norm of each term on the right-hand side using the regularizing property of the heat semigroup \eqref{S_epsilon_youngs}, with the aim of proving \eqref{incrementbeta} for $t_j \le s \le t \le t_{j + 1}$, for arbitrary $\beta > 1$. Below, the symbol "$\lesssim$" means "$\leq C(j, k, p, q)$".

        \medskip
        
		\noindent \textbf{Term 1.} By (3.27) in Proposition 3.5 of \cite{berthelin_stochastic_2019}, we estimate:
        \begin{equation}
			    \mathbb{E}\left\| S_\epsilon(t-t_j){U}_\tau(t_j) - S_\epsilon(s-t_j){U}_\tau(t_j)\right\|^\beta_{L^2(0,1)} \lesssim \epsilon^{-\beta/2} |t - s|^{\beta/2}\|\bar{U}_{\tau}(t_j)\|^\beta_{H^{1}(0,1)}.
		\end{equation}

        \medskip
        
        \noindent \textbf{Term 2.} Next, we obtain that
        \begin{align*}
        \frac{t - s}{\tau} S_{\epsilon}(\tau) {U}_{\tau}(t_j) &- \frac{t - s}{\tau}S_{\epsilon}(s - t_j){U}_{\tau}(t_j) \\
        &= \frac{t - s}{\tau}\Big[\Big(S_{\epsilon}(\tau) - S_{\epsilon}(0)\Big){U}_{\tau}(t_j) - \Big(S_{\epsilon}(s - t_j) - S_{\epsilon}(0)\Big){U}_{\tau}(t_j)\Big] \\
        &= \frac{t - s}{\tau} \left(\int_{0}^{\tau} \partial_{t}S_{\epsilon}(r){U}_{\tau}(t_j) dr + \int_{0}^{s - t_j} \partial_{t}S_{\epsilon}(r) {U}_{\tau}(t_j) dr\right).
        \end{align*}
        Using \eqref{S_epsilon_youngs} for $(j, k, p, q) = (0, 1, \infty, 2)$ and the fact that $0 \le s - t_{j} \le \tau$, we obtain
        \begin{align*}
        \left\|\frac{t - s}{\tau} S_{\epsilon}(\tau) {U}_{\tau}(t_j) - \frac{t - s}{\tau} S_{\epsilon}(s - t_{j}) {U}_{\tau}(t_j)\right\|_{L^{2}(0, 1)}^{\beta} &\le \left(\frac{2(t - s)}{\tau} \int_{0}^{\tau} r^{-3/4} \|{U}_{\tau}(t_j)\|_{L^{\infty}(0, 1)} dr\right)^{\beta} \\
        &\le \left(\frac{2(t - s)}{\tau^{3/4}} \|{U}_{\tau}(t_{j})\|_{L^{\infty}(0, 1)}\right)^{\beta} \lesssim (t - s)^{\beta/4}. 
        \end{align*}

        \medskip
        
        \noindent \textbf{Term 3.} With the choice of $(j, k, p, q)=(1, 1, \infty, 2)$, we use \eqref{S_epsilon_youngs} to estimate: 
        \begin{equation}\label{semgroup_est1}
            \begin{split}
                \mathbb{E}\Bigg\| \int_{t_j}^{s}  \partial_x S_\epsilon(t-s')& F(\bar{U}_\tau(s')) - \partial_x S_\epsilon(s-s') F(\bar{U}_\tau(s')) \,ds'\Bigg\|_{L^2(0,1)}^\beta \\ 
                &\leq  \mathbb{E} \left(\int_{t_j}^s \int_{s-s'}^{t-s'} \left\|\partial_t \partial_x S_\epsilon(t') F(\bar{U}_\tau(s'))\right\|_{L^2(0, 1)}\,dt'\,ds'\right)^\beta \\
                &\lesssim \mathbb{E}\left( \int_{t_j}^s \int_{s-s'}^{t-s'} (\epsilon)(\epsilon t')^{-\frac{5}{4}} \|F(\bar{U}_\tau(s'))\|_{L^\infty(0,1)}\,dt'\,ds' \right)^\beta \\
                & \lesssim \mathbb{E}\left(\epsilon^{-1/4}\int_{t_j}^s \left(-(t-s')^{-1/4} + (s-s')^{-1/4}\right)\,ds' \|F(\bar{U}_\tau)\|_{L^\infty(0, T; L^{\infty}(0, 1))} \right)^\beta\\
                & \lesssim \mathbb{E}\left[\epsilon^{-1/4}\left((t-t_j)^{3/4} - (s-t_j)^{3/4} + (t - s)^{3/4}\right)\|F(\bar{U}_\tau)\|_{L^\infty(0,1)}\right]^\beta \\
                &\lesssim \epsilon^{-\beta/4}(t-s)^{3\beta/4}\mathbb{E}\|F(\bar{U}_\tau)\|^\beta_{L^\infty(0, T; L^{\infty}(0,1))},
            \end{split}
        \end{equation}
        where we recall that $t_j \le s \le t \le t_{j+1}$, and where we use the following inequality for $p = 3/4$:
        \begin{equation}\label{alphabetaineq}
        0 \le \alpha^{p} - \beta^{p} \le (\alpha - \beta)^{p}, \qquad \text{ for } 0 \le \alpha \le \beta \text{ and } 0 < p \le 1.
        \end{equation}
        
        
        Similarly, with the choice of $(j, k, p, q) = (0, 1, +\infty, 2)$ in \eqref{S_epsilon_youngs}, we have by \eqref{alphabetaineq} and calculation similar to the one above that:
        \begin{align}\label{semgroup_est2}
				\mathbb{E}\Bigg\|  \int_{t_j}^{s}  S_\epsilon(t-s') &G(\bar{U}_\tau(s')) -  S_\epsilon(s-s') G(\bar{U}_\tau(s')) \,ds'\Bigg\|^\beta_{L^2(0,1)} \nonumber \\
                &\lesssim \mathbb{E}\left[\epsilon^{1/4}\left((t-t_j)^{5/4} - (s - t_j)^{5/4} + (t-s)^{5/4}\right)\|G\|_{L^\infty(0, T; L^{\infty}(0,1))}\right]^\beta \nonumber \\
                &\lesssim \epsilon (t-s)^{5\beta/4} \mathbb{E}\|G(\bar{U}_\tau)\|_{L^\infty(0, T; L^{\infty}(0,1))}^\beta.
        \end{align}
        
        \medskip
        
        \noindent \textbf{Terms 4-6 (terms involving $F(\bar{U}_{\tau})$).} With the choice of $(j, k, p, q) = (1, 0, +\infty, 2)$ and by the fact that $t_{j} \le s \le t \le t_{j + 1}$, we have that
            \begin{align}\label{semgroup_est3}
             \mathbb{E}\left\|\frac{t - s}{\tau}\int_{t_j}^{s}  \partial_x S_\epsilon(s-s') F(\bar{U}_\tau) \,ds'\right\|^\beta_{L^2(0, 1)} &\lesssim \mathbb{E}\left[\frac{t - s}{\tau}\epsilon^{-1/4} \int_{t_j}^{s} (s-s')^{-1/4} ds' \|F(\bar{U}_\tau)\|_{L^\infty(0, T; L^{\infty}(0,1))}\right]^\beta\nonumber  \\
             &\lesssim \Big[\epsilon^{-1}(t-s) \tau^{-1} (s - t_j)^{3/4} \mathbb{E}\|F(\bar{U}_\tau)\|_{L^\infty(0, T; L^{\infty}(0, 1))}\Big]^\beta \nonumber \\
             &\le \epsilon^{-\beta/4} (t - s)^{3\beta/4} \mathbb{E}\|F(\bar{U}_{\tau})\|_{L^{\infty}(0, T; L^{\infty}(0, 1))}^{\beta}.
        \end{align}
        Similarly, we compute that
        \begin{multline*}
        \mathbb{E} \left\|\frac{t_{j + 1} - t}{\tau} \int_{s}^{t} \partial_{x}S_{\epsilon}(t - s')F(\bar{U}_{\tau}(s')) ds' + \frac{t - s}{\tau} \int_{t_{j}}^{t_{j + 1}} \partial_{x}S_{\epsilon}(t_{j + 1} - s')F(\bar{U}_{\tau}(s')) ds'\right\|_{L^{2}(0, 1)}^{\beta} \\
        \lesssim \epsilon^{-\beta/4} \left[\frac{t_{j + 1} - t}{\tau} (t - s)^{3/4} + \frac{t - s}{\tau}\tau^{3/4}\right]^{\beta} \mathbb{E}\|F(\bar{U}_{\tau})\|_{L^{\infty}(0, 1)}^{\beta} \lesssim \epsilon^{-\beta/4} (t - s)^{3\beta/4} \mathbb{E}\|F(\bar{U}_{\tau})\|_{L^{\infty}(0, 1)}^{\beta}.
        \end{multline*}

        \medskip
        
        \noindent \textbf{Terms 4-6 (terms involving $G(\bar{U}_{\tau}))$.} With the choice of $(j, k, p, q) = (0, 0, +\infty, 2)$, we estimate for $t_j \le s \le t \le t_{j + 1}$ as for the terms involving $F(\bar{U}_{\tau})$:
        \begin{align}\label{semgroup_est4}
             \mathbb{E}\left\|\frac{t - s}{\tau} \int_{t_j}^{s}  S_\epsilon(s-s') G(\bar{U}_\tau(s')) \,ds'\right\|^\beta_{L^2(0, 1)} &\lesssim \mathbb{E}\left[\frac{t - s}{\tau} \epsilon^{1/4} \int_{t_j}^{s} (s - s')^{1/4} ds' \|G(\bar{U}_{\tau})\|_{L^{\infty}(0, T; L^{\infty}(0, 1))}\right]^{\beta} \nonumber \\
             &\lesssim \epsilon^{\beta/4} (t - s) \tau^{-\beta/4} (s - t_{j})^{5\beta/4} \mathbb{E} \|G(\bar{U}_{\tau})\|_{L^{\infty}(0, T; L^{\infty}(0, 1))}^{\beta} \nonumber \\
             &\lesssim \epsilon^{-\beta/4} (t - s)^{\beta}\tau^{\beta/4} \mathbb{E}\|G(\bar{U}_{\tau})\|_{L^{\infty}(0, T; L^{\infty}(0, 1))}^{\beta}.
        \end{align}
        We obtain a similar estimate for the remaining terms:
        \begin{multline*}
        \mathbb{E} \left\|\frac{t_{j + 1} - t}{\tau} \int_{s}^{t} S_{\epsilon}(s - s')G(\bar{U}_{\tau}(s')) ds' + \frac{t - s}{\tau} \int_{t_{j}}^{t_{j + 1}} S_{\epsilon}(t_{j + 1} - s') G(\bar{U}_{\tau}(s')) ds'\right\|_{L^{2}(0, 1)}^{\beta} \\
        \lesssim \epsilon^{\beta/4} (t - s) \tau^{\beta/4} \mathbb{E} \|G(\bar{U}_{\tau})\|^{\beta}_{L^{\infty}(0, T; L^{\infty}(0, 1))}.
        \end{multline*}
        We refer to Proposition 3.5 in \cite{berthelin_stochastic_2019} for a more detailed calculation.

        \medskip
        
		\noindent \textbf{Estimate of the stochastic integrals.} Lastly, we estimate the It\^{o} integrals in \eqref{eqn_tj+1-s}. Since the first component of the integral is $0$, it suffices to estimate the second component only:
		\begin{equation}
			\begin{split}
				\mathbb{E}\left(\left\|\frac{t - t_j}{\tau}\int_{s}^{t}\sigma_\epsilon(x, \tilde{U}_\tau)\,dW(s') \right\|_{L^2}\right)^\beta & =\left(\frac{t - t_j}{\tau}\right)^\beta \mathbb{E}\left(\int_{0}^{1}\left(\int_{s}^{t}\sigma_\epsilon(x, \tilde{U}_\tau)\,dW(s') \right)^2\,dx \right)^{\beta/2}\\
				& = \left(\frac{t - t_j}{\tau}\right)^\beta \int_{0}^{1}\mathbb{E}\left(\int_{s}^{t}\sigma_\epsilon(x, \tilde{U}_\tau)\,dW(s') \right)^\beta\,dx\\
				& \leq \left(\frac{t - t_j}{\tau}\right)^\beta\int_{0}^{1}\mathbb{E}\left(\int_{s}^{t}\sigma_\epsilon^2(x, \tilde{U}_\tau)\,ds' \right)^{\beta/2}\,dx\\
				& \lesssim \left(\frac{t - t_j}{\tau}\right)^\beta(t - s)^{\beta/2} \leq (t-s)^{\beta/2},
			\end{split}
		\end{equation}
		where in the first inequality we used the Burkholder-Davis-Gundy inequality; and the second inequality we used that $\sigma_\epsilon^2 \leq A_0 m^2$ and $\tilde{m}_\tau$ is uniformly bounded in the invariant region; in the rest of the  inequality we used that $t_j \le s \le t \le t_{j + 1}$.
		In a similar fashion, we have that
		\begin{equation}
			\mathbb{E}\left(\left\|\frac{t - s}{\tau} \int_{t_j}^{s}\sigma_\epsilon(x, \tilde{U}_\tau)\,dW(s') \right\|_{L^2(0,1)}\right)^\beta \lesssim \frac{(t - s)^{\beta}}{\tau^{\beta}}(t-s)^{\beta/2} \lesssim (t - s)^{\beta/2}.
		\end{equation}
		
        \medskip
        
        \noindent \textbf{Conclusion of Step 1.} Combining all of the estimates above, we have the increment estimate \eqref{incrementbeta} for $\beta > 1$, whenever $t_j \le s \le t \le t_{j + 1}$, namely whenever $s$ and $t$ are within the same subinterval $[t_j, t_{j + 1}]$. To establish the increment estimate \eqref{incrementbeta} more generally for all $0 \le s \le t \le T$, it suffices by the triangle inequality to verify the estimate \eqref{incrementbeta} whenever $s = t_{j + 1}$ and $t = t_{k}$ for $j + 1 \le k$. To do this, we express the term ${U}_\tau(t_n)$ for $n \ge j$, starting from initial condition ${U}_\tau(t_j)$:
		\begin{equation}\label{Uhat_tn}
			\begin{split}
				{U}_\tau(t_n)  = &S_\epsilon((n-j)\tau){U}_\tau(t_j) + \sum_{k=j}^{n-1}S_\epsilon((n-k-1)\tau)\int_{t_k}^{t_{k+1}} \partial_xS_\epsilon(t_{k+1}-s')F - S_\epsilon(t_{k+1}-s')G\,ds'\\
				& + \sum_{k=j}^{n-1}S_\epsilon((n-k-1)\tau)\int_{t_k}^{t_{k+1}}\sigma_\epsilon(x, \tilde{U}_\tau(s'))\,dW(s').
			\end{split}
		\end{equation}
		Note that we have summation in this expression because each value ${U}_\tau(t_i)$ for $i = j, ..., n-1$, is carried to the future steps as an initial condition by the heat semigroup. Next we use \eqref{Uhat_tn} (for $n = j+1$ and $n = k$) to obtain
		\begin{equation}
			\begin{split}
				&{U}_\tau(t_n) - {U}_\tau(t_{j+1}) 
                =  \left[S_\epsilon((n-j)\tau)-S_\epsilon(\tau)\right]{U}_\tau(t_j)\\
				& + \left[S_\epsilon((n-j-1)\tau) - S_\epsilon(0)\right]\int_{t_j}^{t_{j+1}} \partial_xS_\epsilon(t_{j+1}-s')F(\bar{U}_\tau) - S_\epsilon(t_{j+1}-s')G(\bar{U}_\tau)\,ds'\\
				& + \sum_{k=j+1}^{n-1}S_\epsilon((n-k-1)\tau)\int_{t_k}^{t_{k+1}}\partial_xS_\epsilon(t_{k+1}-s')F(\bar{U}_\tau) - S_\epsilon(t_{k+1}-s')G(\bar{U}_\tau)\,ds'\\
				& + \left[S_\epsilon((n-j-1)\tau) - S_\epsilon(0)\right]\int_{t_j}^{t_{j+1}}\sigma_\epsilon(x, \tilde{U}_\tau)\,dW(s') + \sum_{k=j+1}^{n-1}S_\epsilon((n-k-1)\tau)\int_{t_k}^{t_{k+1}}\sigma_\epsilon(x, \tilde{U}_\tau)\,dW(s') .
			\end{split}
		\end{equation}
		Note that {the $L^2(0,1)$ norm of all the} individual integrals can be estimated in the same way {using \eqref{S_epsilon_youngs}, as is done in \eqref{semgroup_est1}-\eqref{semgroup_est4}.}  To deal with the effect of $S_\epsilon$ applied to these integrals, we use again \eqref{S_epsilon_youngs} with $(p, q, k, j) = (2, 2, 0, 0)$ to obtain a corresponding increment estimate \eqref{incrementbeta} when $0 \le s = t_{j + 1} \le t = t_{k} \le T$.
		
		Therefore, we have proved that 
		\begin{equation}
			\mathbb{E}\|{U}_\tau(t)-{U}_\tau(s)\|_{L^2(0,1)}^\beta \leq C |t-s|^{\beta/4},
		\end{equation}
		where the constant $C = C(\beta, \epsilon, T, \|U_{\epsilon0}\|_{H^{2}}, \|U_{\epsilon0}\|_{L^\infty})$, and is independent of $\tau$.
		Finally, we can conclude through Kolmogorov's continuity criterion (see Theorem 3.5  and Theorem 5.22 in \cite{da_prato_stochastic_2014}) that ${U}_\tau(t)$ is sample continuous as a $L^2(0, 1)$-valued random variable, and its sample path is $\alpha$-H\"older continuous, where $\alpha \in [0, \frac{1}{4})$, i.e.
		\begin{equation}
			\mathbb{E}\|{U}_\tau(t)\|^2_{C^\alpha(0, T; L^2(0,1))} \leq C(\alpha, \epsilon, T, \|U_{\epsilon0}\|_{H^{2}}, \|U_{\epsilon0}\|_{L^\infty}).
		\end{equation}
		\textbf{Step 2: Show that ${U}_\tau \in C(0, T; H^{2}(0,1))$ almost surely.} First, observe that for $t \in [0, t_1]$, we can show that $\partial_x\bar{U}_\tau(t) \in C(0, t_1;H^1(0,1))$, and $\partial_x\tilde{U}_\tau(t) \in C(0, t_1; H^1(0, 1))$. For the proof we refer to Proposition 3.5 and Appendix B of \cite{berthelin_stochastic_2019} for detailed calculations. Using the definition of ${U}_\tau$ in \eqref{U_defn}, we conclude that ${U}_\tau \in C(0, t_1; H^2(0, 1))$. 
        
        For later subintervals, namely for $t \in [t_n, t_{n+1}]$, we observe that $\bar{U}_\tau(t) = \textbf{S}(t-t_n){U}_\tau(t_n)$, $\tilde{U}_\tau(t) = \textbf{R}(t, t_n)\textbf{S}(t_{n+1}-t_n){U}_\tau(t_n)$, and 
		${U}_\tau(t_n) \in H^2(0, 1)$. 
        Hence, we can also show that $\bar{U}_\tau, \tilde{U}_\tau \in C(t_n, t_{n+1}; H^2(0, 1))$. Since $\lim_{t\to t_n^-}{U}_\tau(t) = \lim_{t \to t_n^+} {U}_\tau(t)$ for any $n=1,..., N-1$, we can conclude that ${U}_\tau \in C(0, T; H^2(0, 1))$. We also obtain $\bar{U}_\tau, \tilde{U}_\tau \in L^\infty(0, T; H^2(0, 1))$, since there is a potential jump discontinuity for each of $\bar{U}_{\tau}$ and $\tilde{U}_{\tau}$, at each time $t=t_n$.
        
        \medskip
        
		\noindent \textbf{Step 3: Show that $\partial_x {U}_\tau \in C^\alpha(0, T; L^2(0,1))$ almost surely.} By Step 2, $\partial_x{U}_\tau$ is well-defined. Therefore we can differentiate \eqref{Uhat_s} to obtain
		\begin{equation}
			\begin{split}
				\partial_x {U}_\tau(t) = &\frac{t_{j+1}-t}{\tau}\partial_x\textbf{S}(t-t_j){U}_\tau(t_j) + \frac{t - t_j}{\tau}\partial_x\textbf{R}(t, t_j)\textbf{S}(\tau){U}_\tau(t_j).
			\end{split}
		\end{equation}
		Then, we implement the exact same process as in Step 1 and use Kolmogorov continuity criterion to obtain $\partial_x{U}_\tau \in C^\alpha(0, T; L^2(0, 1))$ almost surely.

        \end{proof}

        \if 1 = 0
	\begin{proof}
		\textbf{Step 1: Show that ${U}_\tau \in C^\alpha(0, T; L^2([0,1]))$.} Let $K_{\epsilon t}(x, y)$ be the heat kernel associated with the heat equation $\partial_t z - \epsilon \partial_x^2 z = 0$ on $[0,1]$ with homogeneous  Dirichlet boundary condition, which is given by
        \[K_{\epsilon t}(x, y) = \sum_{n=-\infty}^{\infty}\frac{1}{\sqrt{4\pi \epsilon t}}\left(e^{-\frac{(x - y -2n)^2}{4\epsilon t}} - e^{-\frac{(x + y -2n)^2}{4\epsilon t}}\right), \;\; x, y \in [0, 1].\] 
        Let $S_\epsilon(t)$ be the corresponding heat semigroup. 
        A brief calculation gives that for every $y \in [0, 1]$,
        \[\|\partial_t^k \partial_x^j K_{\epsilon t}\|_{L^m_x} \leq C(m, k, j)\epsilon^k (\epsilon t)^{-\frac{1}{2}(\frac{1}{m'})-\frac{j}{2}-k}, \]
        where $\frac{1}{m} + \frac{1}{m'} = 1$. Then we have the following regularizing property of $S_\epsilon$, thanks to the Young's inequality for convolution:
		\begin{equation}\label{S_epsilon_youngs}
			\|\partial_t^k\partial_x^jS_\epsilon(t)\|_{L_x^p \to L_x^q} \leq \|\partial_t^k \partial_x^j K_{\epsilon t}\|_{L^m_x} \leq C(p, q, k, j)\epsilon^k (\epsilon t)^{-\frac{1}{2}(\frac{1}{p}-\frac{1}{q})-\frac{j}{2}-k}
		\end{equation}
		for $k, j \in \mathbb{N}$, $1 \leq p \leq q \leq +\infty$, and $1 + \frac{1}{q} = \frac{1}{m} + \frac{1}{p}$. We then use the heat semigroup $S_\epsilon(t)$ to express $\bar{U}_\tau(t)$ as follows:
		\begin{equation}
			\bar{U}_\tau(t) = S_\epsilon(t-t_n){U}_\tau(t_n) +\int_{t_n}^{t}\partial_x S_\epsilon(t-s') F(\bar{U}_\tau)\,ds' - \int_{t_n}^{t} S_\epsilon(t-s') G(\bar{U}_\tau)\,ds'.
		\end{equation}
		
		Take $0 \leq s < t\leq T$, and consider $s\in [t_j, t_{j+1}]$ and $ t\in [t_n, t_{n+1}]$ for $j \le n$. We are going to estimate the fourth moment of the difference between ${U}_\tau(t)$ and ${U}_\tau(s)$. We divide this difference, ${U}_\tau(t) - {U}_\tau(s)$, into three parts and estimate them one by one.
		\begin{equation}\label{three_chunks}
			{U}_\tau(t) - {U}_\tau(s) = \left({U}_\tau(t)-{U}_\tau(t_n) \right)+ \left({U}_\tau(t_n) - {U}_{\tau}(t_{j+1})\right) + \left({U}_{\tau}(t_{j+1}) - {U}_{\tau}(s)\right).
		\end{equation}
		The first and third term can be estimated in the same way, so we will only demonstrate the process on the third term below. First, we express ${U}_\tau(s)$ and ${U}_\tau(t_{j+1})$:
		\begin{equation}\label{Uhat_s}
			\begin{split}
				&{U}_\tau(s) = \frac{t_{j+1}-s}{\tau}\mathbf{S}(s-t_j){U}_\tau(t_j) + \frac{s - t_j}{\tau}\textbf{R}(s, t_j)\textbf{S}(\tau){U}_\tau(t_j)\\
				= &\frac{t_{j+1}-s}{\tau}\left[S_\epsilon(s - t_j){U}_\tau(t_j) + \int_{t_j}^{s} \partial_xS_\epsilon(s-s')F(\bar{U}_\tau(s')) - S_\epsilon(s-s')G(\bar{U}_\tau(s'))\,ds'\right] \\
				+ &\frac{s-t_j}{\tau} \left[S_\epsilon(\tau){U}_\tau(t_j) + \int_{t_j}^{t_{j+1}} \partial_x S_\epsilon(t_{j+1}-s')F(\bar{U}_\tau(s')) - S_\epsilon(t_{j+1}-s')G(\bar{U}_\tau)\,ds' + \int_{t_j}^{s} \Phi(x, \tilde{U}_\tau(s'))\,dW(s')\right],
			\end{split}
		\end{equation}
		\begin{align}\label{Uhat_tj+1}
			{U}_\tau(t_{j+1}) = S_\epsilon(\tau){U}_\tau(t_j) + \int_{t_j}^{t_{j+1}} \partial_xS_\epsilon(t_{j+1} -s')F(\bar{U}_\tau(s')) &- S_\epsilon(t_{j+1}-s')G(\bar{U}_\tau(s'))\,ds' \nonumber\\
            &+ \int_{t_j}^{t_{j+1}} \Phi(x, \tilde{U}_\tau(s'))\,dW(s').
		\end{align}
		Next we subtract \eqref{Uhat_s} from \eqref{Uhat_tj+1} to obtain:
		\begin{equation}\label{eqn_tj+1-s}
			\begin{split}
				&{U}_\tau(t_{j+1}) - {U}_\tau(s) =  \frac{t_{j+1}-s}{\tau} \left[S_\epsilon(\tau){U}_\tau(t_j) - S_\epsilon(s-t_j){U}_\tau(t_j)\right]\\
				&+ \frac{t_{j+1}-s}{\tau}\left[\int_{t_j}^{s}\left(\partial_xS_\epsilon(t_{j+1}-s')-\partial_xS_\epsilon(s-s')\right)F(\bar{U}_{\tau}(s'))ds' -\int_{t_j}^{s} \left(S_\epsilon(t_{j+1}-s')-S_\epsilon(s-s')\right)G(\bar{U}_{\tau}(s'))ds'\right]\\
				 & +\frac{t_{j+1}-s}{\tau}\int_{s}^{t_{j+1}}\partial_xS_\epsilon(t_{j+1}-s')F(\bar{U}_{\tau}(s')) - S_\epsilon(t_{j+1}-s')G(\bar{U}_{\tau}(s'))ds'\\
				&+ \int_{s}^{t_{j+1}}\Phi(x, \tilde{U}_\tau(s'))\,dW(s') + \frac{t_{j+1}-s}{\tau}\int_{t_j}^{s}\Phi(x, \tilde{U}_\tau(s'))\,dW(s').
			\end{split}
		\end{equation}
		Now, we estimate the $L^2$ norm of each term using the regularizing property of the heat semigroup \eqref{S_epsilon_youngs}. Below, the symbol "$\lesssim$" means "$\leq C(j, k, p, q)$".\\
        
		Without loss of generality, we can assume that $t_j=0$ here, since after we proved that ${U}_\tau \in C^\alpha(0, t_1;H^1(0, 1))$ for this base case, we can use induction to show that this proof works for every time $s \in [0, T]$. By (3.27) in Proposition 3.5 of \cite{berthelin_stochastic_2019}, we then have that:
        \begin{equation}
			\begin{split}
			    \mathbb{E}\left\| S_\epsilon(\tau){U}_\tau(t_j) - S_\epsilon(s-t_j){U}_\tau(t_j)\right\|^4_{L^2([0,1])} =& \mathbb{E}\left\| S_\epsilon(\tau)U_{\epsilon0} - S_\epsilon(s-t_j)U_{\epsilon 0}\right\|^4_{L^2([0,1])}\\
                &\lesssim \epsilon^{-2} |t_{j+1}-s|^{2}\|U_{\epsilon0}\|^4_{H^{1}(0,1)}.
			\end{split}
		\end{equation}
		
		We then estimate the other four deterministic integrals in \eqref{eqn_tj+1-s}, using \eqref{S_epsilon_youngs}.  With the choice of \textcolor{orange}{$(k, j, p, q)=(1, 1, \infty, 2)$,
         we estimate the first term in the following way, where in the second inequality we use \eqref{S_epsilon_youngs}: 
        \begin{equation}\label{semgroup_est1}
            \begin{split}
                \mathbb{E}&\left\| \int_{t_j}^{s}  \partial_x S_\epsilon(t_{j+1}-s') F(\bar{U}_\tau) - \partial_x S_\epsilon(s-s') F(\bar{U}_\tau) \,ds'\right\|_{L^2(0,1)}^4 \\ 
                &\leq  \mathbb{E} \left(\int_{t_j}^s \int_{s-s'}^{t_{j+1}-s'} \left\|\partial_t \partial_x S_\epsilon(t') F(\bar{U}_\tau)\right\|_{L^2(0, 1)}\,dt'\,ds'\right)^4 \lesssim \mathbb{E}\left( \int_{t_j}^s \int_{s-s'}^{t_{j+1}-s'} (\epsilon)(\epsilon t')^{-\frac{5}{4}} \|F(\bar{U}_\tau)\|_{L^\infty(0,1)}\,dt'\,ds' \right)^4 \\
                & \lesssim \mathbb{E}\left(\epsilon^{-1/4}\int_{t_j}^s \left(-(t_{j+1}-s')^{-1/4} + (s-s')^{-1/4}\right)\,ds' \|F(\bar{U}_\tau)\|_{L^\infty(0,1)} \right)^4\\
                & \lesssim \mathbb{E}\left[\epsilon^{-1/4}\left((t_{j+1}-s)^{3/4} + (s-t_j)^{3/4} - \tau^{3/4}\right)\|F(\bar{U}_\tau)\|_{L^\infty(0,1)}\right]^4 \lesssim \epsilon^{-1}(t-s)^{3}\mathbb{E}\|F(\bar{U}_\tau)\|^4_{L^\infty(0,1)}.
            \end{split}
        \end{equation}
        }
       With the choice of $(j, k, p, q) = (0, 1, +\infty, 2)$, we have for the second term,
        \begin{equation}\label{semgroup_est2}
			\begin{split}
				&\mathbb{E}\left\|  \int_{t_j}^{s}  S_\epsilon(t_{j+1}-s') G(\bar{U}_\tau) -  S_\epsilon(s-s') G(\bar{U}_\tau) \,ds'\right\|^4_{L^2(0,1)}\\
                & \qquad \lesssim \mathbb{E}\left[\epsilon^{1/4}\left((s-t_j)^{5/4} + (t_{j+1}-s)^{5/4} - \tau^{5/4}\right)\|G\|^4_{L^\infty([0,1])}\right]^4 \lesssim \epsilon (t-s)^5 \mathbb{E}\|G(\bar{U}_\tau\|_{L^\infty(0,1)}^4.
			\end{split}
        \end{equation}
        With the choice of $(j, k, p, q) = (1, 0, +\infty, 2)$, we have for the third term, 
            \begin{equation}\label{semgroup_est3}
             \mathbb{E}\left\|\int_{s}^{t_{j+1}}  \partial_x S_\epsilon(t_{j+1}-s') F(\bar{U}_\tau) \,ds'\right\|^4_{L^2} \lesssim \mathbb{E}\left[\epsilon^{-1/4} (t_{j+1}-s)^{3/4} \|F(\bar{U}_\tau\|_{L^\infty(0,1)}\right]^4  \lesssim \epsilon^{-1}(t-s)^3\mathbb{E}\|F(\bar{U}_\tau\|_{L^\infty(0,1)}^4.
        \end{equation}
        With the choice of $(j, k, p, q) = (0, 0, +\infty, 2)$, we have for the fourth term, 
        \begin{equation}\label{semgroup_est4}
             \mathbb{E}\left\|\int_{s}^{t_{j+1}}  S_\epsilon(t_{j+1}-s') G(\bar{U}_\tau) \,ds'\right\|^4_{L^2} \lesssim \mathbb{E}\left[\epsilon^{1/4} (t_{j+1}-s)^{5/4} \|G(\bar{U}_\tau\|_{L^\infty(0,1)}\right]^4  \lesssim \epsilon(t-s)^5\mathbb{E}\|G(\bar{U}_\tau\|_{L^\infty(0,1)}^4.
        \end{equation}

		We refer to Proposition 3.5 in \cite{berthelin_stochastic_2019} for detailed calculation. \\
		Lastly we estimate the Ito integrals in \eqref{eqn_tj+1-s}. Since the first component of the integral is $0$, it suffices to estimate the second component only:
		\begin{equation}
			\begin{split}
				\mathbb{E}\left(\left\|\frac{t_{j+1}-s}{\tau}\int_{t_j}^{s}\sigma_\epsilon(x, \tilde{U}_\tau)\,dW(s') \right\|_{L^2}\right)^4 & =\left(\frac{t_{j+1}-s}{\tau}\right)^4 \mathbb{E}\left(\int_{0}^{1}\left(\int_{t_j}^{s}\sigma_\epsilon(x, \tilde{U}_\tau)\,dW(s') \right)^2\,dx \right)^2\\
				& = \left(\frac{t_{j+1}-s}{\tau}\right)^4 \int_{0}^{1}\mathbb{E}\left(\int_{t_j}^{s}\sigma_\epsilon(x, \tilde{U}_\tau)\,dW(s') \right)^4\,dx\\
				& \leq \left(\frac{t_{j+1}-s}{\tau}\right)^4\int_{0}^{1}\mathbb{E}\left(\int_{t_j}^{s}\sigma_\epsilon^2(x, \tilde{U}_\tau)\,ds' \right)^2\,dx\\
				& \lesssim \left(\frac{t_{j+1}-s}{\tau}\right)^4(s-t_j)^2 \leq (t_{j+1}-s)^2 \leq (t-s)^2
			\end{split}
		\end{equation}
		where in the first inequality we used the Burkholder-Davis-Gundy inequality; and the second inequality we used that $\sigma_\epsilon^2 \leq A_0 m^2$ and $\tilde{m}_\tau$ is uniformly bounded in the invariant region; in the rest of the  inequality we used that $t_{j+1}-s < \tau$, $s-t_j < \tau$, and $t_{j+1}-s < \tau$.
		In a similar fashion, we have that
		\begin{equation}
			\mathbb{E}\left(\left\|\int_{s}^{t_{j+1}}\sigma_\epsilon(x, \tilde{U}_\tau)\,dW(s') \right\|_{L^2(0,1)}\right)^4 \lesssim (t_{j+1}-s)^2 \leq (t-s)^2.
		\end{equation}
		Therefore we conclude that 
		\begin{equation}
			\mathbb{E}\|{U}_\tau(t_{j+1}) - {U}_\tau(s)\|_{L^2([0,1])}^4 \leq C(t - s)^2,
		\end{equation}
        for some constant $C=C(\epsilon, \gamma, \alpha, A_0, \|U_{0}\|_{L^\infty})$ that is independent of $\tau$.
		Note that if $t_{j+1} = t_n$, then we are done with this Step 1, since the second term in \eqref{three_chunks}, ${U}_\tau(t_n)-{U}_\tau(t_{j+1}) = 0$. Hence, when estimating the second term in \eqref{three_chunks}, the assumption is that $t_{j+1} \neq t_n$, and therefore $\tau < t-s$. Below, we express the term ${U}_\tau(t_n)$ starting from initial condition ${U}_\tau(t_j)$:
		\begin{equation}\label{Uhat_tn}
			\begin{split}
				{U}_\tau(t_n)  = &S_\epsilon((n-j)\tau){U}_\tau(t_j) + \sum_{k=j}^{n-1}S_\epsilon((n-k-1)\tau)\int_{t_k}^{t_{k+1}} \partial_xS_\epsilon(t_{k+1}-s')F - S_\epsilon(t_{k+1}-s')G\,ds'\\
				& + \sum_{k=j}^{n-1}S_\epsilon((n-k-1)\tau)\int_{t_k}^{t_{k+1}}\sigma_\epsilon(x, \tilde{U}_\tau(s'))\,dW(s').
			\end{split}
		\end{equation}
		Note that we have summation in this expression because each value ${U}_\tau(t_k)$ for $k = j, ..., n-1$, is carried to their future steps as initial condition by the heat semigroup. Next we subtract \eqref{Uhat_tj+1} from \eqref{Uhat_tn} to obtain
		\begin{equation}
			\begin{split}
				&{U}_\tau(t_n) - {U}_\tau(t_{j+1}) \\
                = & \left[S_\epsilon((n-j)\tau)-S_\epsilon(\tau)\right]{U}_\tau(t_j)\\
				& + \left[S_\epsilon((n-j-1)\tau) - S_\epsilon(0)\right]\int_{t_j}^{t_{j+1}} \partial_xS_\epsilon(t_{j+1}-s')F(\bar{U}_\tau) - S_\epsilon(t_{j+1}-s')G(\bar{U}_\tau)\,ds'\\
				& + \sum_{k=j+1}^{n-1}S_\epsilon((n-k-1)\tau)\int_{t_k}^{t_{k+1}}\partial_xS_\epsilon(t_{k+1}-s')F(\bar{U}_\tau) - S_\epsilon(t_{k+1}-s')G(\bar{U}_\tau)\,ds'\\
				& + \left[S_\epsilon((n-j-1)\tau) - S_\epsilon(0)\right]\int_{t_j}^{t_{j+1}}\sigma_\epsilon(x, \tilde{U}_\tau)\,dW(s') + \sum_{k=j+1}^{n-1}S_\epsilon((n-k-1)\tau)\int_{t_k}^{t_{k+1}}\sigma_\epsilon(x, \tilde{U}_\tau)\,dW(s') .
			\end{split}
		\end{equation}
		Note that {the $L^2(0,1)$ norm of all the} individual integrals can be estimated in the same way {using \eqref{S_epsilon_youngs}, as is done in \eqref{semgroup_est1}-\eqref{semgroup_est4}.}  To deal with the effect of $S_\epsilon$ applied to these integrals, we use again \eqref{S_epsilon_youngs} with $(p, q, k, j) = (2, 2, 0, 0)$ to conclude that
		\begin{equation}
			\|S(t)\|_{L^2_x \to L^2_x} \leq C,
		\end{equation}
		where $C$ is a constant independent of time $t$. Therefore, all of the terms are still bounded by the same order of $(t-s)$ as when they are not propagated by $S_\epsilon$.
		
		Therefore, we have proved that 
		\begin{equation}
			\mathbb{E}\|{U}_\tau(t)-{U}_\tau(s)\|_{L^2([0,1])}^4 \leq C |t-s|^2,
		\end{equation}
		where the constant $C = C(\epsilon, T, \|U_{\epsilon0}\|_{H^{2}}, \|U_{\epsilon0}\|_{L^\infty})$, and is independent of $\tau$.
		Finally, we can conclude through Kolmogorov's continuity criterion (see Theorem 3.5  and Theorem 5.22 in \cite{da_prato_stochastic_2014}) that ${U}_\tau(t)$ is sample continuous as a $L^2([0, 1])$-valued random variable, and its sample path is $\alpha$-H\"{o}lder continuous, where $\alpha \in [0, \frac{1}{4})$, i.e.
		\begin{equation}
			\mathbb{E}\|{U}_\tau(t)\|^2_{C^\alpha(0, T; L^2([0,1]))} \leq C(\epsilon, T, \|U_{\epsilon0}\|_{H^{2}}, \|U_{\epsilon0}\|_{L^\infty}).
		\end{equation}
		\textbf{Step 2: Prove that ${U}_\tau \in C(0, T; H^{2}(0,1))$}\\
		We show this statement inductively: \\
		\textbf{Base step}: For $t \in [0, t_1]$, we can show that $\partial_x\bar{U}_\tau(t) \in C(0, t_1;H^1(0,1))$, and $\partial_x\tilde{U}_\tau(t) \in C(0, t_1; H^1(0, 1))$. For the proof we refer to Proposition 3.5 and Appendix B of \cite{berthelin_stochastic_2019} for detailed calculations. Using the definition of ${U}_\tau$ in \eqref{U_defn}, we conclude that ${U}_\tau \in C(0, t_1; H^2(0, 1))$. \\
		\textbf{ Inductive step}: For $t \in [t_n, t_{n+1}]$, since $\bar{U}_\tau(t) = S(t-t_n){U}_\tau(t_n)$, and $\tilde{U}_\tau(t) = R(t, t_n)S(t-t_n){U}_\tau(t_n)$, and 
		${U}_\tau(t_n) \in H^2(0, 1)$, we can show that $\bar{U}_\tau, \tilde{U}_\tau \in C(t_n, t_{n+1}; H^2(0, 1))$. Since $\lim_{t\to t_n^-}{U}_\tau(t) = \lim_{t \to t_n^+} {U}_\tau(t)$ for any $n=1,..., \tau\lceil \frac{T}{\tau}\rceil$, we can conclude that ${U}_\tau \in C(0, T; H^2(0, 1))$. And $\bar{U}_\tau, \tilde{U}_\tau \in L^\infty(0, T; H^2(0, 1))$, since there is jump discontinuity at each $t=t_n$.\\ 
		\\
		\textbf{Step 3: Prove that $\partial_x {U}_\tau \in C^\alpha(0, T; L^2([0,1]))$}\\
		By Step 2, $\partial_x{U}_\tau$ is well defined, therefore we can differentiate \eqref{Uhat_s} to obtain
		\begin{equation}
			\begin{split}
				\partial_x {U}_\tau(t) = &\frac{t_{j+1}-t}{\tau}S(t-t_j)\partial_x{U}_\tau(t_j) + \frac{t - t_j}{\tau}R(t, t_j)S(\tau)\partial_x{U}_\tau(t_j).
			\end{split}
		\end{equation}
		Then, we implement the exact same process as in Step 1 and use Kolmogorov continuity criterion to obtain upgraded regularity in time, namely $\partial_x{U}_\tau \in C^\alpha(0, T; L^2(0, 1))$.

        \end{proof}
        \fi

    \subsection{Uniform entropy dissipation estimates}\label{entropydissipation}

    Recall that the approximate solutions satisfy an entropy balance equation \eqref{entropy_eq_Uhat_full} for the linear interpolant ${U}_{\tau}$. In this section, we will use the entropy balance equation for the approximate solutions to derive uniform bounds on the entropy dissipation (uniformly in $\tau$), which will be important for deriving uniform positivity estimates on the density in the next subsection, namely Section \ref{positivitytau}. In particular, we show the following uniform estimate.

    \begin{prop}\label{entropydissipationprop}
        There exists a constant $C=C(\epsilon, U_{\epsilon 0}, \gamma, T)$ that is independent of $\tau$, such that the approximate solution $(\bar{\rho}_{\tau}, \bar{u}_{\tau})$ to the deterministic subproblem, defined in \eqref{U_defn} and \eqref{defn_small_u}, satisfies the following entropy dissipation estimate:
        \begin{equation*}
        \mathbb{E} \int_{0}^{T} \int_{0}^{1} \bar{\rho}_{\tau}(\partial_{x}\bar{u}_{\tau})^{2} dx ds \leq C.
        \end{equation*}
    \end{prop}

    \begin{proof}
    We start by noticing that \eqref{entropy_eq_Ubar}, \eqref{entropy_eq_Utilde}, and \eqref{entropy_eq_Uhat} imply the following entropy balance equation for the approximate solutions $\tilde{U}_\tau$ and $\bar{U}_\tau$  to the two subproblems, holding for every $T >0$ $\mathbb{P}$-almost surely, for all deterministic test functions $\varphi \in C^{2}_{c}((0, 1))$:
        \begin{equation}
            \begin{split}
			&\int_{0}^{1}\eta(\tilde{U}_\tau(T,x))\varphi(x)\,dx \\
            = & \int_0^1\eta(U_{\epsilon 0}(x))\varphi(x)\,dx + \int_{0}^{T}\int_{0}^1 H(\bar{U}_\tau)\partial_x\varphi \,dx\,ds - \alpha\int_{0}^{T}\int_{0}^{1} \bar{m} \partial_m \eta(\bar{U}_\tau)\varphi(x)\,dx\,ds \\
            & - \epsilon\int_{0}^{T}\int_{0}^{1}\langle \nabla^2\eta(\bar{U}_\tau)\partial_x\bar{U}_\tau, \partial_x\bar{U}_\tau\rangle \varphi(x)\,dx\,ds + \epsilon\int_{0}^{T}\int_{0}^{1}\eta(\bar{U}_\tau)\partial_x^2\varphi \,dx\,ds \\
            & + \frac{1}{2}\int_{0}^{T} \int_{0}^{1} \partial_m^2\eta(\tilde{U}_\tau)\sigma_\epsilon^2(x, \tilde{U}_\tau) \varphi(x)\,dx\,ds + \int_{0}^{T}\int_{0}^{1}\partial_m\eta(\tilde{U}_\tau)\sigma_\epsilon(x, \tilde{U}_\tau)\varphi(x)\,dx\,dW(s).\\
            \end{split}
        \end{equation}
    By substituting $\varphi(x) = 1$, $\eta = \eta_E(\rho, m) = \frac{m^2}{2\rho} + p(\rho)$, and taking expectation on both sides of the equation we obtain:
    \begin{equation}\label{eqn_EHessian}
    \begin{split}
        \mathbb{E}\epsilon\int_{0}^{T}\int_{0}^{1}\langle \nabla^2\eta_E(\bar{U}_\tau)\partial_x\bar{U}_\tau, &\partial_x\bar{U}_\tau\rangle \,dx\,ds  = \mathbb{E}\int_0^1\eta_E(U_{\epsilon 0}(x))\,dx -\mathbb{E}\int_{0}^{1}\eta_E(\tilde{U}_\tau(T,x))\,dx\\
        &- \alpha \mathbb{E} \int_{0}^{T}\int_{0}^{1} \frac{\bar{m}_\tau^2}{\bar{\rho}_\tau}\,dx\,ds + \frac{1}{2}\mathbb{E} \int_{0}^{T} \int_{0}^{1} \frac{1}{\tilde{\rho}_\tau}\sigma_\epsilon^2(x, \tilde{U}_\tau) \,dx\,ds.
    \end{split}
    \end{equation}
  We express the Hessian term on the left hand-side as follows:
    \begin{equation}\label{HessianExpression}
    \langle \nabla^2 \eta_E(\bar{U}_\tau)\partial_x \bar{U}_\tau, \partial_x \bar{U}_\tau \rangle = \kappa \gamma (\bar{\rho}_\tau)^{\gamma -2}(\partial_x\bar{\rho}_\tau)^2 + \bar{\rho}_\tau (\partial_x \bar{u}_\tau)^2,\end{equation}
    and notice that the term we want to estimate is the second term on the right-hand side. 
    For this purpose we estimate the terms on the right-hand side of \eqref{eqn_EHessian} as follows.
    
    First, using the Lipschitz continuity assumption on $\sigma_\epsilon$, namely $|\sigma_\epsilon(\cdot,\cdot, m)|^2 \leq A_0 m^2$, together with Lemma \ref{L_infty_estimates}, we obtain that the last term on the right hand-side of \eqref{eqn_EHessian} can be estimated as:
    \begin{equation}\label{rhs1}
        \left|\frac{1}{2}\mathbb{E} \int_{0}^{T} \int_{0}^{1} \frac{1}{\tilde{\rho}_\tau}\sigma_\epsilon^2(x, \tilde{U}_\tau) \,dx\,ds\right| \leq \left|\frac{A_0}{2}\mathbb{E} \int_0^T \int_0^1 \frac{\tilde{m}_\tau^2}{\tilde{\rho}_\tau}\,dx\,ds \right|\leq C(U_{\epsilon 0}, \gamma, T).
    \end{equation}
    Similarly, using Lemma \ref{L_infty_estimates} and Lemma \ref{lem_Linfty_eta_gradient_eta}, we can estimate the remaining terms to obtain:
    \begin{equation}\label{rhs2}
        \left|\mathbb{E}\int_0^1\eta_E(U_{\epsilon 0}(x))\,dx -\mathbb{E}\int_{0}^{1}\eta_E(\tilde{U}_\tau(T,x))\,dx\\
        - \alpha \mathbb{E} \int_{0}^{T}\int_{0}^{1} \frac{\bar{m}_\tau^2}{\bar{\rho}_\tau}\,dx\,ds\right| \leq C(U_{\epsilon 0}, \gamma, T).
    \end{equation}
    Finally, by combining \eqref{HessianExpression}, \eqref{rhs1} and \eqref{rhs2}, and by using the fact that $\kappa \gamma (\bar{\rho}_\tau)^{\gamma -2}(\partial_x\bar{\rho}_\tau)^2$ is nonnegative, we get the desired estimate
    \begin{equation*}
        \mathbb{E}\int_0^T \int_0^1 \bar{\rho}_\tau (\partial_x \bar{u}_\tau)^2 \,dx\,ds \leq C(\epsilon, U_{\epsilon 0}, \gamma, T)
    \end{equation*}
    where $C(\epsilon, U_{\epsilon 0}, \gamma, T)$ is independent of $\tau$.
    \end{proof}
    
    \subsection{Bounds on positivity of approximate density}\label{positivitytau}

    In this section, we discuss the positivity properties of the approximate fluid densities $\{\rho_\tau\}_{\tau > 0}$, which will be important for passing to the limit in the approximate entropy balance equation \eqref{entropy_eq_Uhat_full}.
    More precisely, we prove that there is a uniform in $\tau$ lower bound on approximate densities that depends only on the lower bound of initial density, the entropy dissipation estimated in Proposition~\ref{entropydissipationprop}, the $L^\infty$ estimate of fluid velocity, and the length of the time interval $[0,T]$.

    
    
    To show positivity of the approximate fluid densities, recall from the definition of the approximate solutions defined by our splitting scheme  in \eqref{Scheme} and \eqref{U_defn}, that density is updated only in the deterministic subproblem. In particular, $\bar{\rho}_{\tau}$ is a weak solution to the regularized continuity equation
    \begin{equation}\label{rhobareqn}
    \partial_{t}\bar{\rho}_{\tau} + \partial_{x}(\bar{\rho}_{\tau} \bar{u}_{\tau}) = \epsilon \Delta \bar{\rho}_{\tau} \quad \text{ on } [0, 1],
    \end{equation}
    with Neumann boundary conditions, see \eqref{subproblem_1}, $\bar{\mathbb{P}}$-almost surely. 
    Hence, we can show that $\bar{\rho}_{\tau}$ is strictly positive, $\bar{\mathbb{P}}$-almost surely, by using the positivity properties of the regularized continuity equation. 
    
    For this purpose, we recall the following (deterministic) positivity result for the regularized continuity equation which can be proven with a slight modification of the proof for Theorem A.2 in \cite{berthelin_stochastic_2019}.
    

    \begin{prop}\label{bvpositivity}
        Let $\rho_{0} \in H^{1}(0, 1)$ satisfy $\min_{x \in (0, 1)} \rho_{0}(x) \ge c > 0$ for some positive constant $c$, and suppose that $\rho(x)$ is a weak solution to
        \begin{equation*}
        \partial_{t}\rho + \partial_{x}(\rho u) = \epsilon \Delta \rho, \qquad \text{ on } [0, 1],
        \end{equation*}
        with Neumann boundary conditions, for some given function $u \in L^{\infty}((0, T) \times (0, 1))$. Then, there exists a positive constant $c_{0}$, depending only on $c$, $T$, $\|u\|_{L^{\infty}((0, T) \times (0, 1))}$, and
        \begin{equation*}
        \int_{0}^{T} \int_{0}^{1} \rho (\partial_{x}u)^{2} dx dt,
        \end{equation*}
        such that $\rho(t, x) \ge c_{0} > 0$ for all $(t, x) \in [0, T] \times [0, 1]$.
    \end{prop}
    
    We will use this result and the previously derived uniform bounds on entropy dissipation and fluid velocity, to derive the following positivity result for the approximate densities. This result is similar to Theorem A.2 in \cite{berthelin_stochastic_2019}.
    
	\begin{prop} \label{positive_density}
		Let the initial data $U_{\epsilon 0} = (\rho_{\epsilon 0}, m_{\epsilon 0})$ be defined by \eqref{U_epsilon0} with $\rho_{\epsilon 0} \geq c_0(\epsilon)>0$. Let ${U}_\tau =(\rho_\tau,m_\tau)$ be the approximate solution of the Lie-Trotter splitting scheme defined in \eqref{Scheme}. Then there exists a random variable $c_{min} > 0$ depending only on $c_0, T, \displaystyle \int_{Q_T}\bar{\rho}_\tau |\partial_x\bar{u}_\tau|^2\,dx\,dt$, and $\|\bar{u}_\tau\|_{L^{\infty}((0, T) \times (0, 1))}$, such that $\bar{\rho}_\tau \ge c_{min}$ and ${\rho}_\tau \geq c_{min} $ almost surely.
	\end{prop}

    \begin{proof}
     Since the approximate density is updated only during the deterministic subproblem, it is sufficient to prove that the densities $\{\bar{\rho}_\tau\}_{\tau > 0}$,  obtained by solving the deterministic subproblems \eqref{subproblem_1}, are uniformly bounded away from zero.
     This will follow by
      combining the positivity result on the fluid density in Proposition \ref{bvpositivity} with the uniform $L^{\infty}$ bounds in Proposition \ref{L_infty_estimates}, the uniform bounds on the entropy dissipation in Proposition \ref{entropydissipationprop}, and the fact that the approximate solution $\bar{\rho}_{\tau}$ satisfies the continuity equation \eqref{rhobareqn}. Namely, we obtain:

        \begin{equation*}
        \bar{\rho}_{\tau} \ge c_{min}(\omega) > 0, \qquad \forall \tau > 0
        \end{equation*}
        for a constant $c_{min}(\omega)$ depending only on the lower bound on $\rho_{\epsilon 0}$ (see \eqref{U_epsilon0}), and the quantities $\displaystyle \int_{Q_{T}} \bar{\rho}_{\tau} |\partial_{x}\bar{u}_{\tau}|^{2} dx dt$ and $\|\bar{u}_{\tau}\|_{L^{\infty}((0, T) \times (0, 1))}$ evaluated for the specific outcome $\omega \in \Omega$, which are almost surely finite by the estimates in Proposition \ref{L_infty_estimates} and Proposition \ref{entropydissipationprop}. 
    \end{proof}
    
    \if 1 = 0 \begin{rem}\label{rem_positive_density_rhoepsilon}
        We remark here that due to Corollary 3.15 of \cite{berthelin_stochastic_2019}, which gives uniform in $\tau$ bounds on the quantity $\int_{Q_T}\bar{\rho}^\tau|\partial_x \bar{u}_\tau|^2\,dx\,dt$ and $\int_{Q_T}|\bar{u}^\tau|^m\,dx\,dt$, the random variable $c_{min}$ in Proposition \ref{positive_density} is independent of $\tau$. 
    \end{rem}
    \fi
\subsection{Approximate entropy equality}\label{sec_approx_entropy}
In this section, we prove an approximate entropy equality satisfied by the approximate solution $U_\tau$, $\bar{U}_\tau$, and $\tilde{U}_\tau$ defined by the splitting scheme \eqref{Scheme}-\eqref{U_defn}. We will then, in Section \ref{skorohodsection_tau}, show that when passing $\tau \to 0$ under this approximate entropy equality, we obtain the entropy equality \eqref{epsilonentropy} given in Theorem \ref{thm_existence_pathwise_U_epsilon}.
\begin{prop}
     For every $n \in \{0, 1,...,N-1\}$, and every $t\in (t_n, t_{n+1}]$, the approximate solutions ${U}_\tau(t,x)$ defined by \eqref{Scheme}, and $\bar{U}_\tau(t,x)$, $\tilde{U}_\tau(t, x)$ defined by \eqref{U_defn} satisfy the following entropy balance equation for every $\varphi \in C_c^2((0,1))$:
\begin{equation}\label{entropy_eq_Uhat_full}
		\begin{split}
	   &\int_{0}^{1}\eta({U}_\tau(t))\varphi(x)\,dx =\int_{0}^{1}\eta(U_{\epsilon 0})\varphi(x)\,dx  + \int_{0}^{t}\int_{0}^1 H(\bar{U}_\tau)\partial_x\varphi \,dx\,ds - \alpha\int_{0}^{t}\int_{0}^{1} \bar{m} \partial_m \eta(\bar{U}_\tau)\varphi(x)\,dx\,ds \\
			+& \epsilon\int_{0}^{t}\int_{0}^{1}\eta(\bar{U}_\tau)\partial_x^2\varphi - \langle \nabla^2\eta(\bar{U}_\tau)\partial_x\bar{U}_\tau, \partial_x\bar{U}_\tau \rangle \varphi(x) \,dx\,ds \\
			+& \frac{1}{2}\int_{0}^{t_n} \int_{0}^{1}  \partial_m^2\eta(\tilde{U}_\tau) \sigma_\epsilon^2(x, \tilde{U}_\tau)\varphi(x)\,dx\,ds + \int_{0}^{t_n}\int_{0}^{1}\partial_m\eta(\tilde{U}_\tau)\sigma_\epsilon(x, \tilde{U}_\tau)\varphi(x)\,dx\,dW(s)\\
			+ & \frac{t-t_n}{\tau}\left(\int_{t}^{t_{n+1}}\int_{0}^1 H(\bar{U}_\tau)\partial_x\varphi - \alpha\bar{m}_\tau \partial_m \eta(\bar{U}_\tau)\varphi(x)
			+ \epsilon \eta(\bar{U}_\tau)\partial_x^2\varphi - \epsilon \langle \nabla^2\eta(\bar{U}_\tau)\partial_x\bar{U}_\tau, \partial_x\bar{U}_\tau\rangle \varphi(x) \,dx\,ds \right)\\
			+ &\frac{t-t_n}{\tau}\left(\frac{1}{2}\int_{t_n}^{t} \int_{0}^{1} \partial_m^2\eta(\tilde{U}_\tau) \sigma_\epsilon^2(x, \tilde{U}_\tau)\varphi(x)\,dx\,ds + \int_{t_n}^{t}\int_{0}^{1}\partial_m\eta(\tilde{U}_\tau)\sigma_\epsilon(x, \tilde{U}_\tau)\varphi(x)\,dx\,dW(s)\right)\\
            + & \int_0^1 \mathrm{R}_2(\bar{U}_\tau- \tilde{U}_\tau)\varphi(x)\,dx,
		\end{split}
	\end{equation}
    where $\mathrm{R}_2$ is given by:
    \[\mathrm{R}_2(\bar{U}_\tau-\tilde{U}_\tau) = \frac{1}{2}(\tilde{U}_\tau - \bar{U}_\tau)^T\left((\frac{t-t_n}{\tau})^2\nabla^2\eta(\mathbf{\xi}_1)-\frac{t-t_n}{\tau}\nabla^2\eta(\mathbf{\xi}_2)\right)(\tilde{U}_\tau - \bar{U}_\tau),\]
for some $\xi_1=\xi_1(t, x), \xi_2=\xi_2(t, x) \in \mathbb{R}^2$ on the line segment joining $\bar{U}_\tau(t, x)$ and $\tilde{U}_\tau(t, x)$, for every $(t, x) \in Q_T$.
\end{prop}
\begin{proof}
    For every $t \in (t_n, t_{n+1}]$, $\bar{U}_\tau(t)$ is the solution to the deterministic subproblem \eqref{subproblem_1} and thus satisfies for every entropy-entropy flux pair $(\eta, H)$ defined in \eqref{entropy_pair_formula}, and every $\varphi \in C_c^2((0, 1))$, the following entropy equality:
\begin{equation}\label{entropy_eq_Ubar}
		\begin{split}
			&\int_{0}^{1}\eta(\bar{U}_\tau(t))\varphi(x)\,dx + \epsilon\int_{t_n}^{t}\int_{0}^{1}\langle \nabla^2\eta(\bar{U}_\tau)\partial_x\bar{U}_\tau, \partial_x\bar{U}_\tau\rangle \varphi(x)\,dx\,ds = \int_{0}^{1}\eta({U}_\tau(t_n))\varphi(x)\,dx  \\
			&+ \int_{t_n}^{t}\int_{0}^1 H(\bar{U}_\tau)\partial_x\varphi \,dx\,ds - \alpha\int_{t_n}^{t}\int_{0}^{1} \bar{m} \partial_m \eta(\bar{U}_\tau)\varphi(x)\,dx\,ds + \epsilon\int_{t_n}^{t}\int_{0}^{1}\eta(\bar{U}_\tau)\partial_x^2\varphi \,dx\,ds .
		\end{split}
	\end{equation}
	By applying It\^{o}'s formula with $U \mapsto \eta(U)$ to the equation \eqref{subproblem_2_cont}, we obtain that $\tilde{U}_\tau(t, x)$ satisfies the following entropy equality, for every $\varphi \in C^2_c((0, 1))$, and every $t \in (t_n, t_{n+1}]$:
	\begin{equation}\label{entropy_eq_Utilde}
		\begin{split}
			\int_{0}^{1}\eta(\tilde{U}_\tau(t))\varphi(x)\,dx =& \int_{0}^{1}\eta(\textbf{S}(\tau){U}_\tau(t_n,x))\varphi(x)\,dx+ \frac{1}{2}\int_{t_n}^{t} \int_{0}^{1} \partial_m^2\eta(\tilde{U}_\tau)\sigma_\epsilon^2(x, \tilde{U}_\tau) \varphi(x)\,dx\,ds \\ &+ \int_{t_n}^{t}\int_{0}^{1}\partial_m\eta(\tilde{U}_\tau)\sigma_\epsilon(x, \tilde{U}_\tau)\varphi(x)\,dx\,dW(s),
            \end{split}
        \end{equation}
        where $\textbf{S}$ is the solution operator to the deterministic subproblem defined in \eqref{subproblem_1}. 
        To obtain the expression for $\eta({U}_\tau)$, we calculate the Taylor expansion of $\eta({U}_\tau)$ and $\eta(\tilde{U}_\tau)$ centered around the point $\bar{U}_\tau$ up to the second order:
        {
        \begin{equation}\label{eta_taylorexpansion}
        \begin{split}
            \eta(U_\tau)&=\eta \left(\frac{t_{n+1}- t}{\tau} \bar{U}_\tau+ \frac{t - t_n}{\tau}\tilde{U}_\tau\right)
            = \eta\left(\bar{U}_\tau + \frac{t-t_n}{\tau}(\tilde{U}_\tau-\bar{U}_\tau)\right)\\
            &= \eta(\bar{U}_\tau) + \frac{t-t_n}{\tau}\nabla \eta(\bar{U}_\tau)\cdot (\tilde{U}_\tau - \bar{U}_\tau) + \frac{1}{2}(\frac{t-t_n}{\tau})^2(\tilde{U}_\tau - \bar{U}_\tau)^T\nabla^2\eta(\mathbf{\xi}_1)(\tilde{U}_\tau - \bar{U}_\tau),\\
            \eta(\tilde{U}_\tau)& = \eta(\bar{U}_\tau) + \nabla\eta(\bar{U}_\tau)\cdot (\tilde{U}_\tau - \bar{U}_\tau) + \frac{1}{2}(\tilde{U}_\tau - \bar{U}_\tau)^T\nabla^2\eta(\mathbf{\xi}_2)(\tilde{U}_\tau - \bar{U}_\tau),
        \end{split}
        \end{equation}
         where $\xi_1=\xi_1(t, x), \xi_2=\xi_2(t, x) \in \mathbb{R}^2$ are two points on the line segment joining $\bar{U}_\tau(t, x)$ and $\tilde{U}_\tau(t, x)$, for every $(t, x) \in Q_T$.  We remark here that $\eta$ is twice differentiable in the state space $(0, \infty) \times \mathbb{R}$, because $\bar{\rho}_\tau$ and $\tilde{\rho}_\tau$ have strictly positive lower bounds, $\mathbb{P}$-almost surely by Proposition \ref{positive_density}. Therefore, $\nabla^2 \eta(\xi)$ is always well-defined for $\xi=\xi(t, x)\in \mathbb{R}^2$ on the line segment joining $\bar{U}_\tau(t, x)$ and $\tilde{U}_\tau(t, x)$.
        }
        
        We then multiply \eqref{eta_taylorexpansion}$_2$ by $\frac{t-t_n}{\tau}$, subtract it from \eqref{eta_taylorexpansion}$_1$, and move $-\frac{t-t_n}{\tau}\eta(\tilde{U}_\tau)$ to the right hand side to obtain \eqref{entropy_eq_Uhat}:
	\begin{equation}\label{entropy_eq_Uhat}
		\begin{split}
			\eta({U}_\tau) &= \eta \left(\frac{t_{n+1}- t}{\tau} \bar{U}_\tau+ \frac{t - t_n}{\tau}\tilde{U}_\tau\right)\\
			& = \frac{t_{n+1}- t}{\tau} \eta (\bar{U}_\tau)+ \frac{t - t_n}{\tau}\eta(\tilde{U}_\tau) + \mathrm{R}_2(\bar{U}_\tau-\tilde{U}_\tau),
		\end{split}
	\end{equation}
	 where $\mathrm{R}_2(\bar{U}_\tau-\tilde{U}_\tau)$ is the second order Taylor remainder term given by 
     \[\mathrm{R}_2(\bar{U}_\tau-\tilde{U}_\tau) = \frac{1}{2}(\tilde{U}_\tau - \bar{U}_\tau)^T\left((\frac{t-t_n}{\tau})^2\nabla^2\eta(\mathbf{\xi}_1)-\frac{t-t_n}{\tau}\nabla^2\eta(\mathbf{\xi}_2)\right)(\tilde{U}_\tau - \bar{U}_\tau),\]
      where $\xi_1=\xi_1(t, x), \xi_2=\xi_2(t, x) \in \mathbb{R}^2$ are two points on the line segment joining $\bar{U}_\tau(t, x)$ and $\tilde{U}_\tau(t, x)$, for every $(t, x) \in Q_T$. Note that the squared coefficient $(\frac{t-t_n}{\tau})^2$ comes from the second order Taylor remainder term of the linear interpolant, whereas the term with $\frac{t-t_n}{\tau}$ comes purely from the linear scaling of $\eta(\tilde{U}_\tau)$.  
      
      We then substitute $\eta(\bar{U}_\tau)$ and $\eta(\tilde{U}_\tau)$ in \eqref{entropy_eq_Uhat} with the expressions from \eqref{eta_taylorexpansion}, and recursively express the initial term. After the substitution, we multiply the equation by $\varphi \in C^2_c((0,1))$ and integrate over $[0,1]$, to finally obtain the entropy equality for ${U}_\tau$ starting from $U_{\epsilon 0}$ given by \eqref{entropy_eq_Uhat_full}.
\end{proof}
\subsection{Difference estimates on approximate solutions}\label{sec_diffestimate}

Before passing to the limit, we make one last observation about the approximate solutions. Namely, we recall from \eqref{U_defn} that there are three definitions of approximate solutions: $\bar{U}_{\tau}$ taking into account the deterministic subproblems only, $\tilde{U}_{\tau}$ taking into account the stochastic subproblems only, and the linear interpolant ${U}_{\tau}$. These are all important to consider, since all forms of these solutions appear within the approximate entropy balance equation \eqref{entropy_eq_Uhat_full}. However, we will prove in this section, that all three forms of the approximate solution are actually close together in an appropriate norm, and hence, when we pass to the limit after using the Skorohod representation theorem, they will all converge to the same limit as $\tau \to 0$.
        
        Specifically, we will prove the following lemma, which provides pairwise estimates on the difference between the approximate solutions $\tilde{U}_\tau$, $\bar{U}_\tau$ to the subproblems \eqref{subproblem_1} and \eqref{subproblem_2}, and the approximate solution  ${U}_\tau$, defined in \eqref{Scheme}. 
        These estimates will be useful later when we pass $\tau \to 0$ and identify the limit. 
    \begin{lem}\label{lem_difference_estimates}
        Fix $T>0$. Consider the solutions to the deterministic and stochastic subproblems $\bar{U}_\tau$ and $\tilde{U}_\tau$ defined in \eqref{U_defn}, and the approximate solution ${U}_\tau$ defined by the interpolated Lie-Trotter splitting scheme \eqref{Scheme}. Then 
        \begin{equation}
        \begin{split}
            \lim_{\tau\to 0}\mathbb{E}\sup_{0\leq t\leq T}\left\|\tilde{U}_\tau(t) - \bar{U}_\tau(t)\right\|_{L^2(0,1)}^2 =0.
        \end{split} 
        \end{equation}
        The same result holds for the pairs $(U_\tau, \bar{U}_\tau)$ and $(U_\tau, \tilde{U}_\tau)$.
    \end{lem}
    \begin{proof} \if 1 = 0 {\color{orange} Rewritten: }
    {We start by
 using the definition of $\tilde{U}_\tau$ and $\bar{U}_\tau$ in \eqref{U_defn}, to obtain that for $t=0$,
 \begin{equation}\label{diff_expressiont0}
     \tilde{U}_\tau(0) - \bar{U}_\tau(0) = \int_{0}^{t_1}\partial_x F(\bar{U}_\tau) - G(\bar{U}_\tau) + \epsilon\partial_x^2\bar{U}_\tau\,ds.
\end{equation}}
  Using H\"{o}lder's inequality and the regularity results in Theorem \ref{regularity_of_Uhat}, we have that
  \begin{equation}\label{diff_estimatet0}
    \begin{split}
        \mathbb{E}\|\tilde{U}_\tau(0) - \bar{U}_\tau(0)\|_{L^2(0,1)}^2 & = \mathbb{E}\int_0^1 \left|\int_{0}^{t_1}\partial_x F(\bar{U}_\tau) - G(\bar{U}_\tau) + \epsilon\partial_x^2\bar{U}_\tau\,ds\right|^2\,dx \\
        & \leq \mathbb{E}\|\partial_x F(\bar{U}_\tau) - G(\bar{U}_\tau) + \epsilon\partial_x^2\bar{U}_\tau\|_{L^\infty(0, t_1; L^2(0,1))}^2\tau \leq C(U_0, \gamma, \alpha, \epsilon, T)\tau.
    \end{split}
  \end{equation}
  \fi
 Consider any $t \in [0, T]$. By
 the definition of $\tilde{U}_\tau$ and $\bar{U}_\tau$ in \eqref{U_defn}, we obtain that for $t=0$:
 \begin{equation}\label{diff_expressiont0}
     \tilde{U}_\tau(0) - \bar{U}_\tau(0) = \int_{0}^{t_1}\Big(\partial_x F(\bar{U}_\tau) - G(\bar{U}_\tau) + \epsilon\partial_x^2\bar{U}_\tau\Big)ds.
\end{equation}
and more generally, for $t\in (t_n, t_{n+1}]$:
  \begin{equation}\label{diff_expression}
      \tilde{U}_\tau(t) - \bar{U}_\tau(t) = \int_{t}^{t_{n+1}}\Big(\partial_x F(\bar{U}_\tau) - G(\bar{U}_\tau) + \epsilon\partial_x^2\bar{U}_\tau\Big)ds + \int_{t_n}^{t} \Phi_\epsilon(x, \tilde{U}_\tau)\,dW(s).
  \end{equation}
  By using \eqref{diff_expressiont0} and \eqref{diff_expression} along with the triangle inequality and the inequality $(A + B)^{2} \le 2(A^2 + B^2)$, we obtain the following estimate:
        \begin{equation}\label{diff_Ubar_Utilde}
			\begin{split}
				&\mathbb{E}\sup_{0< t\leq T}\left\|\tilde{U}_\tau(t) - \bar{U}_\tau(t)\right\|_{L^2(0,1)}^2  = \mathbb{E}\max_{n\in\{0,\dots, N-1\}}\sup_{t_n\leq t\leq t_{n+1}}\left\|\tilde{U}_\tau(t) - \bar{U}_\tau(t)\right\|_{L^2(0,1)}^2 \\
                 = &\mathbb{E}\max_{n}\sup_{t_n \leq t\leq t_{n+1}}\int_{0}^{1}\left|\int_{t}^{t_{n+1}}\partial_x F(\bar{U}_\tau) - G(\bar{U}_\tau) + \epsilon\partial_x^2\bar{U}_\tau\,ds + \int_{t_n}^{t} \Phi_\epsilon(x,\tilde{U}_\tau)\,dW(s)\right|^2\,dx \leq  2(I_{1} + I_{2}), 
                \end{split}
                \end{equation}
                where
                \begin{equation*}
                I_{1} := \tau \left(\mathbb{E}\max_{n}\sup_{t_n\le t\leq t_{n+1}}\int_{t}^{t_{n+1}}\int_{0}^{1}\left|\partial_x F(\bar{U}_\tau) - G(\bar{U}_\tau) + \epsilon\partial_x^2\bar{U}_\tau\right|^2 \,dx\,ds\right),
                \end{equation*}
        \begin{equation*}
                I_{2} := \mathbb{E} \max_{n}\sup_{t_n \le t\leq t_{n+1}} \int_{0}^{1}\left| \int_{t_n}^{t}  \Phi_\epsilon(x, \tilde{U}_\tau)\,dW(s)\right|^2\,dx .
		\end{equation*}
        Note that we used H\"{o}lder's inequality and Fubini's theorem to obtain the term $I_{1}$ in the previous inequality \eqref{diff_Ubar_Utilde}.
        
        \medskip
        
        \noindent \textbf{Estimate of $I_{1}$.} We estimate the first integral by using H\"{o}lder's inequality and Theorem 
        \ref{regularity_of_Uhat} to deduce that:
        \begin{equation}\label{det_term}
            \begin{split}
               &\mathbb{E}\max_{n}\sup_{t_n \le t\leq t_{n+1}}\int_{t}^{t_{n+1}}\int_{0}^{1}\left|\partial_x F(\bar{U}_\tau) - G(\bar{U}_\tau) + \epsilon\partial_x^2\bar{U}_\tau\right|^2 \,dx\,ds \\
               \leq & \mathbb{E}\left\|\partial_x F(\bar{U}_\tau) - G(\bar{U}_\tau) + \epsilon\partial_x^2\bar{U}_\tau\right\|_{L^\infty(0,T; L^2(0,1))}^2|t_{n+1}-t| \leq C(U_0, \gamma, \alpha, \epsilon, T) \tau.
            \end{split}
        \end{equation}
        So we conclude that $I_{1} \le C(U_0, \gamma, \alpha, \epsilon, T) \tau^{2}$, and hence $I_{1} \to 0$ as $\tau \to 0$. 
        
        \medskip
        
        \noindent \textbf{Estimate of $I_{2}.$} Next, we estimate the stochastic integral. 
        \if 1 = 0 
        Note that we cannot apply the Burkholder-Davis-Gundy inequality directly because we cannot exchange expectation and maximum. Instead, we will aim to show that 
        $$M_\tau(t):= \int_0^t \Phi_\epsilon(x, \tilde{U}_\tau)\,dW(s)$$ is H\"{o}lder continuous as an $L^2(0,1)$-valued stochastic process in time on $[0, T]$, by using the Kolmogorov continuity theorem, see \cite{da_prato_stochastic_2014}. Let $0\le s<t\le T$. We use the It\^{o} isometry to estimate the following increment:
        \begin{equation}
            \mathbb{E}\left\|M_\tau(t) - M_\tau(s)\right\|_{L^2(0,1)}^2 = \mathbb{E} \int_{0}^{1}\left| \int_{s}^{t}  \Phi_\epsilon(x, \tilde{U}_\tau(s'))\,dW(s')\right|^2\,dx =  \int_{0}^{1} \left[\mathbb{E} \int_{s}^{t} \sigma_\epsilon^2(x, \tilde{U}_\tau(s'))\,ds'\right]\,dx.
        \end{equation}
        By the assumptions on $\sigma_\epsilon$ in \eqref{regularizedLipschitz}, which together imply $|\sigma_\epsilon(x, \tilde{U}_\tau)|\leq \sqrt{A_0}|\tilde{m}_\tau|$, and Proposition \ref{L_infty_estimates}, we have 
        \begin{equation}
             \mathbb{E}\left\|M_\tau(t) - M_\tau(s)\right\|_{L^2(0,1)}^2 \leq C |t-s|,
        \end{equation}
        for some deterministic and uniform-in-$\tau$ constant $C=C(\|U_0\|_{L^\infty}, \gamma, A_0)$. Therefore, by the Kolmogorov continuity theorem, $M_\tau(t)$ is in $C^\beta(0, T; L^2(0, 1))$, with $\beta\in [0, 1/2)$, $\mathbb{P}$-almost surely. 
        
        For concreteness, taking $\beta = 1/4$, we have for example that for all $0 \le s \le t \le T$, there exists a uniform constant $c(\omega)$ potentially depending on the outcome $\omega \in \Omega$ such that
        \begin{equation*}
        \|M_{\tau}(t) - M_{\tau}(s)\|_{L^{2}(0, 1)}^{2} \le c(\omega)|t - s|^{1/2},
        \end{equation*}
        $\mathbb{P}$-almost surely. Hence, our quantity of interest in $I_2$ satisfies:
        \begin{equation}
            \sup_{t_n\leq t\leq t_{n+1}}\int_{0}^{1}\left| \int_{t_n}^{t}  \Phi_\epsilon(x, \tilde{U}_\tau)\,dW(s)\right|^2\,dx = \sup_{t_n \leq t\leq t_{n+1}} \left\|M_\tau(t) - M_\tau(t_n)\right\|_{L^2(0,1)}^2 \leq c(\omega) \tau^{1/2},
        \end{equation}
        for every $n\in\{0,...,N-1\}$ and almost every $\omega \in \Omega$. 
        Therefore
        \begin{equation}{\label{KCT_bound}}
            \max_n \sup_{t_n<t\leq t_{n+1}}\int_{0}^{1}\left| \int_{t_n}^{t}  \Phi_\epsilon(x, \tilde{U}_\tau)\,dW(s)\right|^2\,dx \leq c(\omega)\tau^{1/2},
        \end{equation}
        and thus, we have the following almost sure convergence:
        \begin{equation}\label{aslimit}
            \lim_{\tau \to 0}\max_n \sup_{t_n<t\leq t_{n+1}}\int_{0}^{1}\left| \int_{t_n}^{t}  \Phi_\epsilon(x, \tilde{U}_\tau)\,dW(s)\right|^2\,dx =0, \;\; \mathbb{P}\text{-almost surely. }
        \end{equation}
        \fi 
        Consider $p > 1$. We use the Burkholder-Davis-Gundy inequality, the estimate that $|\sigma_{\epsilon}(x, \tilde{U}_{\tau})| \le \sqrt{A_0} |\tilde{m}_{\tau}|$ from \eqref{regularizedLipschitz}, and the uniform bounds in Proposition \ref{L_infty_estimates} on $\|\tilde{m}_{\tau}\|_{L^{\infty}(Q_{T})}$, to estimate that
        \begin{align}\label{vitalimoment}
        \mathbb{E} \left(\max_{n}\sup_{t_{n} \le t \le t_{n + 1}} \int_{0}^{1} \left|\int_{t_{n}}^{t} \Phi_{\epsilon}(x, \tilde{U}_{\tau}) dW(s)\right|^{2} dx\right)^{p} &= \mathbb{E} \left[\max_{n}\sup_{t_{n} \le t \le t_{n + 1}} \left( \int_{0}^{1} \left|\int_{t_{n}}^{t} \Phi_{\epsilon}(x, \tilde{U}_{\tau}) dW(s)\right|^{2} dx\right)^{p}\right] \nonumber \\
        &\le \sum_{n = 0}^{N - 1} \mathbb{E}\left(\sup_{t_{n} \le t \le t_{n + 1}} \left\|\int_{t_{n}}^{t} \Phi_{\epsilon}(x, \tilde{U}_{\tau}) dW(s)\right\|^{2p}_{L^{2}(0, 1)} \right) \nonumber \\
        &\le \sum_{n = 0}^{N - 1} \mathbb{E} \left(\int_{t_n}^{t_{n+1}} \|\sigma_{\epsilon}(x, \tilde{U}_{\tau}(s))\|^{2}_{L^{2}(0, 1)} ds\right)^{p} \nonumber \\
        &\le (A_0)^p \sum_{n = 0}^{N - 1} \mathbb{E} \left(\int_{t_{n}}^{t_{n + 1}} \|\tilde{m}_{\tau}\|_{L^{2}(0, 1)}^{2} ds\right)^{p} \nonumber \\
        &\le (A_{0})^{p}C \sum_{n = 0}^{N - 1} \tau^{p} \le (A_{0})^{p} C (N\tau^{p}) \nonumber \\
        &= (A_{0})^{p} C T \tau^{p - 1}.
        \end{align}
        Hence, by using the fact that $\mathbb{E}|X| \le \Big[\mathbb{E}(|X|^{p})\Big]^{1/p}$, we have that $I_{2} \to 0$, as $\tau \to 0$.

        \if 1 = 0
        \begin{multline}
            \mathbb{E}\max_n \sup_{t_n<t\leq t_{n+1}}\int_{0}^{1}\left| \int_{t_n}^{t}  \Phi_\epsilon(x, \tilde{U}_\tau)\,dW(s)\right|^2\,dx = \mathbb{E}\max_n\sup_{t_n<t\leq t_{n+1}} \left\|M_\tau(t)-M_\tau(t_n)\right\|_{L^2(0, 1)}^2 \\ \leq \mathbb{E}\sum_{n=0}^{N-1}\sup_{t_n<t\leq t_{n+1}} \left\|M_\tau(t)-M_\tau(t_n)\right\|_{L^2(0, 1)}^2  = \sum_{n=0}^{N-1}\mathbb{E}\sup_{t_n<t\leq t_{n+1}} \left\|M_\tau(t)-M_\tau(t_n)\right\|_{L^2(0, 1)}^2.
        \end{multline}
      Applying \eqref{KCT_L2} to the right hand side of the above equation, we obtain
      \begin{equation}\label{boundmaxsup}
          \mathbb{E}\max_n\sup_{t_n<t\leq t_{n+1}} \left\|M_\tau(t)-M_\tau(t_n)\right\|_{L^2(0, 1)}^2 \leq \sum_{n=0}^{N-1}C|t-t_n| \leq  N\cdot C\tau = C\cdot T = C(\|U_0\|_{L^\infty}, \gamma, A_0, T).
      \end{equation}
      Note that the constant $C(\|U_0\|_{L^\infty}, \gamma, A_0, T)$ is independent of $\tau$. 
     With \eqref{aslimit} and \eqref{boundmaxsup}, we can then apply the dominated convergence theorem and obtain:
     \begin{equation}\label{stoch_term}
         \lim_{\tau\to 0}\mathbb{E}\max_n\sup_{t_n<t\leq t_{n+1}} \left\|M_\tau(t)-M_\tau(t_n)\right\|_{L^2(0, 1)}^2 = 0.
     \end{equation}

     \fi 

     \medskip

     \noindent \textbf{Conclusion.} Combining the estimates of $I_{1}$ and $I_{2}$ in \eqref{diff_Ubar_Utilde}, we obtain that
        \begin{equation}
            \lim_{\tau\to 0}\mathbb{E}\sup_{0\leq t\leq T}\left\|\tilde{U}_\tau(t) - \bar{U}_\tau(t)\right\|_{L^2(0,1)}^2 =0.
        \end{equation}
         To see that the same result holds for the pairs $(U_\tau, \bar{U}_\tau)$ and $(U_\tau, \tilde{U}_\tau)$, we simply observe that
        
        \begin{equation*}
        |{U}_\tau(t) - \bar{U}_\tau(t)| = \frac{t - t_n}{\tau} |\tilde{U}_\tau(t) - \bar{U}_\tau(t)|\leq |\tilde{U}_\tau(t) - \bar{U}_\tau(t)|,
        \end{equation*}
        \begin{equation*}
        |{U}_\tau(t) - \tilde{U}_\tau(t)| = \frac{t_{n+1}-t}{\tau}|\tilde{U}_\tau(t) - \bar{U}_\tau(t)| \leq |\tilde{U}_\tau(t) - \bar{U}_\tau(t)|.
        \end{equation*}
    \end{proof}

    Next we give a corresponding difference estimate for the gradient of the three different approximate solutions, $\bar{U}_\tau$, $\tilde{U}_\tau$, and ${U}_\tau$.
    \begin{lem}\label{lem_diff_estimates_grad}
        Fix $T>0$. Let ${U}_\tau$, $\bar{U}_\tau$, and $\tilde{U}_\tau$ be defined in \eqref{Scheme} and \eqref{U_defn}. Then
        \begin{equation}
        \begin{split}
            \lim_{\tau\to0}\mathbb{E}\sup_{0\leq t\leq T}\|\partial_x \bar{U}_\tau(t) - \partial_x \tilde{U}_\tau(t)\|^2_{L^2(0,1)} =0.
        \end{split}
        \end{equation}
        Moreover, the same result holds for the pairs $(U_\tau, \bar{U}_\tau)$ and $(U_\tau, \tilde{U}_\tau)$.
    \end{lem}
    \begin{proof}
       { We start by recalling the definition of $\bar{U}_\tau(t)$ in \eqref{U_defn} and that $\bar{U}_\tau(0)=U_{\epsilon 0}$. When $t\in (t_n, t_{n+1}]$, $\bar{U}_\tau(t)$ can be expressed using the solution operator $\mathbf{S}$ (see \eqref{Ubar}) and the heat semigroup $S_\epsilon$ (introduced in Theorem \ref{regularity_of_Uhat}) in the following way:
        \begin{equation}\label{Ubar_repeat}
           \bar{U}_\tau(t)=
               \mathbf{S}(t-t_n)U_\tau(t_n) = S_\epsilon(t-t_n)U_\tau(t_n) + \int_{t_n}^t S_\epsilon(t-r)\partial_x F(\bar{U}_\tau(r))\,dr - \int_{t_n}^t S_\epsilon(t-r)G(\bar{U}_\tau(r))\,dr.
       \end{equation}
       By the definition of the regularized initial condition $U_{\epsilon 0}$ (see \eqref{rem_init_cond}), the spatial derivative $\partial_x\bar{U}_\tau(0) = \partial_xU_\tau(0) = \partial_xU_{\epsilon0}$ is well defined. 
       By Theorem \ref{regularity_of_Uhat}, $\bar{U}_\tau \in L^\infty(0, T; H^{2}(0,1))$ and $U_\tau \in C(0, T; H^{2}(0,1))$, $\mathbb{P}$-almost surely. Therefore, we can also then differentiate with respect to $x$ on both sides of \eqref{Ubar_repeat} to obtain by properties of convolution with respect to differentiation:
       }
       \begin{equation}\label{partialx_bar}
           \partial_x\bar{U}_\tau(t)= \begin{cases}
               \partial_x U_\tau(0) & t=0 \\ 
               \begin{split}
                   \partial_x\left(\mathbf{S}(t-t_n)U_\tau(t_n)\right) &= S_\epsilon(t - t_n)\partial_x{U_\tau(t_n)} + \int_{t_n}^t \partial_{x}S_\epsilon(t-r)\partial_x F(\bar{U}_\tau(r))\,dr \\ & - \int_{t_n}^t \partial_x S_\epsilon(t-r) G(\bar{U}_\tau(r))\,dr
               \end{split} & t\in(t_n, t_{n+1}].
           \end{cases} 
       \end{equation}
       Similarly, we recall the definition of $\tilde{U}_\tau$ in \eqref{U_defn}$_2$ and differentiate \eqref{U_defn}$_2$ on both sides with respect to $x$ to obtain: 
       \begin{equation}\label{partialx_tilde} 
       \partial_x \tilde{U}_\tau(t)=\begin{cases} \partial_x(\mathbf{S}(\tau) {U}_\tau(0)) & t=0\\
             \partial_x(\mathbf{S}(\tau) {U}_\tau(t_{n})) + \int_{t_n}^t \partial_x \Phi_\epsilon(x, \tilde{U}_\tau)\,dW(r)& t\in (t_n, t_{n+1}].
       \end{cases}
       \end{equation}
       After replacing $\partial_x(\textbf{S}(\tau)U_\tau(0))$ and $\partial_x(\textbf{S}(\tau)U_\tau(t_n))$ with the expression involving the heat semigroup in \eqref{partialx_bar}, and subtracting \eqref{partialx_bar} from \eqref{partialx_tilde}, we have that 
    for $t=0$:
        \begin{equation}
        \begin{split}
            &\mathbb{E}\|\partial_x\bar{U}_\tau(0) - \partial_x\tilde{U}_\tau(0)\|_{L^2(0,1)}^2 \\
            =& \mathbb{E}\int_0^1\left|[S_\epsilon(\tau)-S_\epsilon(0)]\partial_x U_\tau(0) + \int_0^{t_1} S_\epsilon(t_1-r)\partial_x^2 F(\bar{U}_\tau(r)) - S_\epsilon(t_{1}-r)\partial_xG(\bar{U}_\tau(r))\,dr\right|^2\,dx.
        \end{split}
        \end{equation}
       For $t\in (t_n, t_{n+1}]$, we apply the triangle inequality and obtain
		\begin{equation}
			\begin{split}
				&\mathbb{E}\sup_{0< t\leq T}\left\|\partial_x\bar{U}_\tau(t) - \partial_x\tilde{U}_\tau(t)\right\|^2_{L^2(0,1)} \leq \mathbb{E}\max_n\sup_{t_n< t\leq t_{n+1}}\int_0^1 \left|\left[S_\epsilon(\tau) - S_\epsilon(t-t_n) \right]\partial_x {U}_\tau(t_n)\right|^2\,dx\\
                &+ \mathbb{E}\max_n \int_{0}^{1}\left|\int_{t_n}^{t_{n+1}} \partial_x S_\epsilon(t_{n+1}-r) \partial_x F(\bar{U}_\tau) - \partial_x S_\epsilon(t_{n+1}-r) G(\bar{U}_\tau)\,dr \right|^2\,dx\\
				&+\mathbb{E}\max_n\sup_{t_n< t\leq t_{n+1}}\int_0^1 \left| \int_{t_n}^{t} \partial_x S_\epsilon(t-r) \partial_xF(\bar{U}_\tau) -\partial_x S_\epsilon(t-r)G(\bar{U}_\tau) \,dr\right|^2\,dx\\
                & +\mathbb{E}\max_n\sup_{t_n< t\leq t_{n+1}}\int_0^1\left|\int_{t_n}^{t}\partial_x\Phi_\epsilon(x, \tilde{U}_\tau)\,dW(r)\right|^2\,dx =: I_0 + I_1 + I_2 + I_3.
			\end{split}
		\end{equation}
        We then only demonstrate the estimates for the $t\in(t_n, t_{n+1}]$ case below, since a similar argument can be applied to the more specific case of $t=0$. For this purpose, we recall the estimate on the heat semigroup from \eqref{S_epsilon_youngs}:
        \begin{equation*}
			\|\partial_t^k\partial_x^jS_\epsilon(t)\|_{L_x^p \to L_x^q} \leq \|\partial_t^k \partial_x^j K_{\epsilon t}(\cdot, y)\|_{L^m_x} \leq C(p, q, k, j)\epsilon^k (\epsilon t)^{-\frac{1}{2}(\frac{1}{p}-\frac{1}{q})-\frac{j}{2}-k}
		\end{equation*}
		for $k, j \in \mathbb{N}$, $1 \leq p \leq q \leq +\infty$, and $1 + \frac{1}{q} = \frac{1}{m} + \frac{1}{p}$.
        
        We first write the integrand in $I_0$ as 
        \[\left[S_\epsilon(\tau) - S_\epsilon(t-t_n) \right]\partial_x {U}_\tau(t_n) = \int_{t-t_n}^\tau \partial_rS_\epsilon(r) \partial_x {U}_\tau(t_n)\,dr.\]
         Using \eqref{S_epsilon_youngs} for $(j, k, p, q) = (0, 1, \infty, 2)$, the fact that $\partial_x {U}_\tau(t_n) \in H^{1}(0,1) \hookrightarrow L^\infty(0,1)$ by Theorem \ref{regularity_of_Uhat}, and the fact that $|\tau -(t-t_n)|\leq \tau$, we have that
         \begin{equation}
             I_0 \leq \mathbb{E}\max_n \sup_{t_n< t\leq t_{n+1}}\left(\int_{t-t_n}^\tau r^{-3/4}\|\partial_x \bar{U}_\tau(t_n)\|_{L^\infty(0,1)}\,dr \right)^2\leq C\tau^{1/2},
         \end{equation}
       for some constant $C=C(U_{0}, \epsilon, \gamma, \alpha, T)$. We estimate $I_1$ in a similar manner as we estimate the deterministic integral in Lemma \ref{lem_difference_estimates}. By H\"{o}lder's inequality and \eqref{S_epsilon_youngs} for $(j, k, p, q) = (1, 0, \infty, 2)$, we have that:  
        \begin{equation}
        \begin{split}
            I_1 &\le \tau \cdot \mathbb{E}\max_{n} \int_{t_n}^{t_{n+1}} \int_0^1 \left|\partial_x S_\epsilon(t_{n+1}-r) \partial_x F(\bar{U}_\tau) - \partial_x S_\epsilon(t_{n+1}-r) G(\bar{U}_\tau)\right|^2\,dx\,dr\\
            & \leq \tau \cdot \mathbb{E} \max_n \int_{t_{n}}^{t_{n+1}} \|\partial_x S_\epsilon(t_{n+1}-\cdot) \partial_x F(\bar{U}_\tau) - \partial_x S_\epsilon(t_{n+1}-\cdot) G(\bar{U}_\tau)\|_{L^2(0,1)}^2 dr \\
            &\leq C\tau \cdot \mathbb{E} \max_{n} \int_{t_n}^{t_{n+1}} (t_{n+1} - r)^{-1/2} \Big(\|\partial_{x}F(\bar{U}_{\tau})\|_{L^{\infty}(0, 1)} + \|G(\bar{U}_{\tau})\|_{L^{\infty}(0, 1)}\Big)^2dr \\
            &\le C\tau^{3/2} \cdot \mathbb{E} \Big(1 + \|\partial_{x}\bar{U}_{\tau}\|_{L^{\infty}((0, T) \times (0, 1))}^{2}\Big) \\
            &\le C\tau^{3/2} \cdot \mathbb{E}\Big(1 + \|\bar{U}_{\tau}\|_{L^{\infty}(0, T; H^{2}(0, 1))}^{2}\Big) \le C\tau^{3/2}.
        \end{split}
        \end{equation}
        Here, we used Theorem \ref{regularity_of_Uhat}, which implies that $\mathbb{E}\|\bar{U}_{\tau}\|_{L^{\infty}(0, T; H^{2}(0, 1))}^{2} \le C$, and we used the $L^\infty$ boundedness in Proposition \ref{L_infty_estimates} to conclude that $\|\partial_{x}F(\bar{U}_{\tau})\|_{L^{\infty}(0, 1)} \le C\|\partial_{x}\bar{U}_{\tau}\|_{L^{\infty}(0, 1)}$ and $\|G(\bar{U}_{\tau})\|_{L^{\infty}(0, 1)} \le C$. Similarly, $I_2 \leq C\tau^{3/2}$.
		For the stochastic integral, $I_3$, we estimate it in a similar fashion as we did with the stochastic integral in Lemma \ref{lem_difference_estimates} and conclude that 
		\begin{equation}
        \begin{split}
            \lim_{\tau \to 0} I_3  = \lim_{\tau \to 0} \mathbb{E}\max_n \sup_{t_n< t\leq t_{n+1}}\int_{0}^{1}\left|\int_{t_n}^{t}\partial_x\Phi_\epsilon(x, \tilde{U}_\tau)\,dW(r)\right|^2\,dx =0. 
        \end{split}
		\end{equation}
		Together we have obtained that
        \[\lim_{\tau\to0}\mathbb{E}\sup_{0\leq t\leq T}\int_{0}^{1}|\partial_x\bar{U}_\tau(t) - \partial_x\tilde{U}_\tau(t)|^2\,dx  =0.\]
        To show that the same result applies to the pairs $(\partial_x U_\tau, \partial_x\bar{U}_\tau)$ and $(\partial_x U_\tau, \partial_x \tilde{U}_\tau)$, we simply note that $\partial_x(\bar{U}_\tau(t) - {U}_\tau(t)) = \frac{t-t_n}{\tau}\partial_x(\tilde{U}_\tau(t) - \bar{U}_\tau(t))$, and $\partial_x(\tilde{U}_\tau(t) - {U}_\tau(t)) = \frac{t_{n+1} -t}{\tau}\partial_x(\tilde{U}_\tau(t) - \bar{U}_\tau(t))$. Therefore,
        \[|\partial_x({U}_\tau(t) - \bar{U}_\tau(t))| \leq |\partial_x(\tilde{U}_\tau - \bar{U}_\tau)|, \qquad |\partial_x({U}_\tau(t) - \tilde{U}_\tau(t))|  \leq |\partial_x(\tilde{U}_\tau - \bar{U}_\tau)|.\]
        We then conclude the proof of this lemma.
    \end{proof}
    
\subsection{Skorohod representation theorem}\label{skorohodsection_tau}

We will now combine all of the uniform bounds on $U_{\tau}$ independent of $\tau$, which are defined in Section \ref{sec_preliminary_estimates}--Section \ref{sec_diffestimate}, in order to pass to the limit in the approximate (random-in-time) solutions $U_{\tau}$ as $\tau \to 0$. To do this, we will use a classical \textit{stochastic compactness argument}, which involves examining the laws of the approximate solutions in $\tau$, and then showing that the laws converge weakly (along a subsequence in $\tau$) to a limiting probability measure as $\tau \to 0$. This will be accomplished by showing that the laws of the approximate solutions are uniformly \textit{tight} in $\tau$. Once we have weak convergence of the laws of the approximate solutions, we can then use the classical Skorohod representation theorem to obtain almost sure convergence of the actual random variables defining the approximate solutions, but at the expense of transferring to a different probability space. We can then transfer this almost sure convergence of the approximate solutions back to the \textit{original probability space} by  using a strong pathwise uniqueness property of the approximate system \eqref{regularized_problem1} via a standard Gy\"{o}ngy-Krylov argument, as shown in Section \ref{sec_uniqueness_pathwise_U_epsilon}. We carry out this stochastic compactness procedure below.

The approximate $\tau$-level quantities that we need to keep track of are $(U_{\tau}, \bar\rho_{\tau}^{1/2} \partial_{x}\bar{u}_{\tau}, W)$. We remark that we keep track of $\bar\rho_{\tau}^{1/2} \partial_{x}\bar{u}_{\tau}$ in order to obtain uniform control of the approximate densities away from vacuum. We recall the definition of the approximate fluid density $\bar{\rho}_{\tau}$ from the definition of $\bar{U}_{\tau} := (\bar{\rho}_{\tau}, \bar{m}_{\tau})$ in \eqref{U_defn}, and we recall that $\bar{u}_{\tau} := \bar{m}_{\tau}/\bar{\rho}_{\tau}$ is the approximate fluid velocity, which is well-defined by the lower bound in Proposition \ref{positive_density}. For these approximate quantities $(U_{\tau}, \bar\rho_{\tau}^{1/2}\partial_{x}\bar{u}_{\tau}, W)$, we define the following path spaces: 
\begin{equation*}
\mathcal{X}_{U} := C(0, T; H^1(0,1)), \quad \mathcal{X}_{\rho^{1/2}\partial_{x}u}: = (L^2((0, T) \times (0, 1)), w), \quad \mathcal{X}_W := C(0, T; \mathbb{R}), 
\end{equation*}
and the combined path space:
\begin{equation}\label{phasespace}
\mathcal{X}:= \mathcal{X}_{U} \times \mathcal{X}_{\rho^{1/2}\partial_{x}{u}}\times \mathcal{X}_W.
\end{equation}
Here, $(L^{2}((0, T) \times (0, 1)), w)$ refers to the space $L^{2}((0, T) \times (0, 1))$ with the topology of weak convergence. We denote the joint law of the approximate quantities in the path space $\mathcal{X}$ by $\mu_{\tau}$, namely:
\begin{equation*}
\mu_{\tau} := \text{Law}_{\mathcal{X}}(U_{\tau}, \bar\rho_{\tau}^{1/2} \partial_{x}\bar{u}_{\tau}, W).
\end{equation*}
To show that the probability measures $\mu_{\tau}$ on $\mathcal{X}$ converge weakly along a subsequence in $\tau$, we are going to use the regularity of ${U}_\tau$ (Theorem \ref{regularity_of_Uhat}) and $\bar{\rho}_\tau^{1/2} \partial_x \bar{u}_\tau$ (Proposition \ref{entropydissipationprop}) to deduce uniform tightness of the joint laws of $\{{U}_\tau\}_{\tau}$, $\{\bar{\rho}_\tau^{1/2} \partial_x \bar{u}_\tau\}_{\tau}$, and $W$ in the space $\mathcal{X}$. 

	\begin{thm}[Tightness of law]
		Let $\{{\mu}_\tau\}_\tau$ be the joint laws of the random variables $({U}_\tau, \bar{\rho}_{\tau}^{1/2}\partial_{x}\bar{u}_\tau, W)$. Then $\{\mu_\tau\}_\tau$ is uniformly tight in the path space $\mathcal{X}$.
	\end{thm}
	\begin{proof}
    \textbf{Tightness of laws of $\{{U}_\tau\}_{\tau}$.} 
		Define the following subset of $\mathcal{X}_{U}$:
        \begin{equation*}
        K_M := \{U \in C(0, T; H^1(0,1)) : \|U\|_{C^\alpha((0, T); H^1(0,1))} + \|U\|_{C(0,T; H^2(0,1))}> M\}.
        \end{equation*}
        By Arzela-Ascoli, $C^\alpha(0, T;H^1(0,1)) \cap C(0,T; H^2(0,1))$ compactly embeds in $C(0, T; H^1(0,1))$, and hence, to show the tightness of the laws of $\{U_{\tau}\}_{\tau}$ in $\mathcal{X}_{U}$, we observe that: 
		\begin{equation}
			\begin{split}
				\mathbb{P}\left({U}_\tau \notin K_M\right) &= \mathbb{P}\left(\|{U}_\tau\|_{C^\alpha((0, T); H^1(0,1))} + \|{U}_{\tau}\|_{C(0,T; H^2(0,1))}> M\right) \\
				&\leq \frac{\mathbb{E}\|{U}_\tau\|^2_{C^\alpha((0, T); H^1(0,1))} +\mathbb{E} \|{U}_{\tau}\|^2_{C(0,T; H^2(0,1))}}{M^2}\leq \frac{C}{M^2},
			\end{split}
		\end{equation}
		where we use Markov's inequality, and the estimates in Theorem \ref{regularity_of_Uhat}. Therefore we have shown that for every $\epsilon>0$, we can find a compact set $K \subset\subset C(0, T; H^1(0,1))$ such that $\mu_{{U}_\tau}({U}_\tau \notin K) < \epsilon$, which implies that the laws of $\{U_\tau\}_{\tau}$ form a tight family of measures on $\mathcal{X}_{U} = C(0, T;H^1(0,1))$.

        \medskip
        
        \noindent \textbf{Tightness of laws of $\{\bar{\rho}_{\tau}^{1/2} \partial_{x}\bar{u}_{\tau}\}_{\tau}$.} 
        By Proposition \ref{entropydissipationprop}, 
        \begin{equation*}
        \mathbb{E} \int_{0}^{T} \int_{0}^{1} \bar{\rho}_{\tau} (\partial_{x}\bar{u}_{\tau})^{2} dx dt \le C.
        \end{equation*}
        for a constant $C$ that is independent of $\tau$. This shows the tightness of laws of $\{\bar{\rho}_{\tau}^{1/2}\partial_{x}\bar{u}_{\tau}\}_{\tau}$ in $(L^{2}((0, T) \times (0, 1)), w)$ with the weak topology, since any closed ball in $L^{2}((0, T) \times (0, 1))$ is weakly compact by the Banach-Alaoglu theorem.

        \medskip

        \noindent \textbf{Conclusion of the proof.} Finally, the law of the Brownian motions are trivially tight in $C(0, T; \mathbb{R})$, since the Brownian motion is the same, independently of $\tau$. Therefore, the collection of joint laws, $\{\mu_\tau\}_\tau$ is uniformly tight in the phase space $\mathcal{X}$, defined in \eqref{phasespace}.

	\end{proof}

	From this result on the tightness of laws, we can then apply the Skorohod Representation Theorem (see for example pg.~70 of \cite{MR2893652}) to obtain an almost surely convergent sequence of random variables defined on a new probability space, which is stated in Theorem \ref{pass_tau_0} below. Though this requires transferring to a different probability space, the almost sure convergence that we obtain on the new probability space allows us to pass to the limit in the $\tau$-level entropy equality as $\tau \to 0$. 
    
	\begin{thm}[Passage to the limit]\label{pass_tau_0}
		Suppose the initial condition $U_{\epsilon 0}$ is defined by \eqref{U_epsilon0}, and the noise coefficient $\sigma_\epsilon$ satisfies (\ref{noise_assumption}). Let ${U}_\tau$ be defined as in (\ref{U_defn}), then there exists a sequence of $\mathcal{X}$-valued random variables $({U}^\sharp_\tau, \ell^\sharp_\tau, W^\sharp_\tau)_\tau$, and a triple of $\mathcal{X}$-valued random variables $(U_\epsilon, \ell_\epsilon, W^\sharp)$ on the probability space $(\tilde{\Omega}, \tilde{\mathcal{F}}, \tilde{\mathbb{P}})$, such that
		\begin{itemize}
			\item $({U}^\sharp_\tau, \ell^\sharp_\tau, W^\sharp_\tau) \to (U_\epsilon, \ell_\epsilon, W^\sharp)$ as $\tau \to 0$, in the topology of $\mathcal{X}_{U} \times \mathcal{X}_{\rho^{1/2}\partial_x u} \times \mathcal{X}_{W}$, $\tilde{\mathbb{P}}$-almost surely.
			\item $({U}_\tau, \bar{\rho}_\tau^{1/2} \partial_{x}\bar{u}_{\tau}, W)$ and $({U}^\sharp_\tau, \ell^\sharp_\tau, W^\sharp_\tau)$ have the same law for each $\tau > 0$.
		\end{itemize}
        Moreover, $U_\epsilon \in C(0, T; H^2(0,1))$, $\tilde{\mathbb{P}}$-almost surely.
	\end{thm}
       
       \begin{rem}
        Since every closed ball in $C(0, T; H^2(0,1))$ is a measurable subset of $C(0, T; H^1(0,1))$, we also have that the limit $U_\epsilon \in C(0, T; H^2(0,1))$, $\tilde{\mathbb{P}}$-almost surely.
    \end{rem}
    
	\begin{rem}
	     Let $\tilde{\mathcal{F}}_t'$ be the $\sigma$-algebra generated by the random variables $U_\epsilon(s),  W^\sharp(s)$, for all $s \leq t$. Define
        \[\mathcal{N}:= \{\mathcal{A}\in \tilde{\mathcal{F}} \ | \ \tilde{\mathbb{P}}(\mathcal{A}) = 0\}, \qquad \tilde{\mathcal{F}}_t := \cap_{s \geq t}\sigma(\tilde{\mathcal{F}}_s'\cup \mathcal{N}).\]
        Then, $(\tilde{\mathcal{F}}_t)_{0\leq t\leq T}$ defines a complete right continuous filtration for which $(U_\epsilon, W^\sharp)$ is an adapted process. The construction of a complete right continuous filtration for ${U}_\tau^\sharp$ (for each $\tau$) on the new probability space $(\tilde\Omega, \tilde{\mathcal{F}}, \tilde{\mathbb{P}})$, can be done in a similar fashion. 
	\end{rem}
    
    Next we proceed to identify the limit $U_{\epsilon}$ in Theorem \ref{pass_tau_0}, as a weak entropy solution to the system \eqref{regularized_problem1}--\eqref{regularized_problem2} in the sense of Theorem \ref{thm_existence_pathwise_U_epsilon}, by passing to the limit in the entropy equality \eqref{entropy_eq_Uhat_full}. Note however that the entropy equality \eqref{entropy_eq_Uhat_full} involves all of the approximations ${U}_{\tau}$, $\bar{U}_{\tau}$, $\tilde{U}_{\tau}$ defined in \eqref{U_defn}, whereas we only passed ${U}_{\tau}$ through Skorohod to obtain ${U}^{\sharp}_{\tau}$ (see Theorem \ref{pass_tau_0}). Therefore, on the new probability space $(\tilde\Omega, \tilde{\mathcal{F}}, \tilde{\mathbb{P}})$, we also define the other corresponding approximate solution forms:
    \begin{equation}\label{defn_Ubar_sharp}
		   \begin{split}
		       \bar{U}_\tau^\sharp(t) := \textbf{S}(t-t_n) {U}_\tau^\sharp(t_n), \qquad \tilde{U}_\tau^\sharp(t) := \textbf{R}(t, t_n) \textbf{S}(\tau){U}_\tau^\sharp(t_n), \qquad \text{for } t \in (t_n, t_{n+1}].
		   \end{split} 
	\end{equation}
    Recall that by Theorem \ref{pass_tau_0}, ${U}_\tau^\sharp$ and ${U}_\tau$ coincide in law in $C(0, T; H^{1}(0, 1))$. Therefore, $\bar{U}_\tau^\sharp, \tilde{U}_\tau^\sharp$ coincide in law with $\bar{U}_\tau, \tilde{U}_\tau$ in $L^{\infty}(0, T; H^{1}(0, 1))$. We make the following observation about the new approximate solutions \eqref{defn_Ubar_sharp}, namely that they converge as $\tau \to 0$ to the same limit $U_{\epsilon}$, obtained in Theorem \ref{pass_tau_0}.

    \begin{prop}\label{samelimit}
    As $\tau \to 0$, 
    \begin{equation*}
    \bar{U}^{\sharp}_{\tau} \to U_{\epsilon}, \qquad \tilde{U}^{\sharp}_{\tau} \to U_{\epsilon}, \qquad \text{in } L^{\infty}(0, T; H^{1}(0, 1)), \quad \tilde{\mathbb{P}}\text{-almost surely},
    \end{equation*}
    where the limit $U_{\epsilon}$ is previously defined in Theorem \ref{pass_tau_0}. In addition, we also have the stronger convergence that $\bar{\rho}^{\sharp}_{\tau} \to \rho_{\epsilon}$ in $C(0, T; H^{1}(0, 1))$, $\tilde{\mathbb{P}}$-almost surely.
    
    \end{prop}

    \begin{proof}
        This is a direct consequence of the fact that by Lemma \ref{lem_difference_estimates}, Lemma \ref{lem_diff_estimates_grad}, and equivalence of laws in Theorem \ref{pass_tau_0}, we have that :
        \begin{equation}\label{diff_estimates_sharp}
            \lim_{\tau \to 0}\mathbb{E}\|{U}_\tau^\sharp(t) - \bar{U}_\tau^\sharp(t)\|_{L^{\infty}(0, T; H^1(0, 1))}^2 =0, \qquad \lim_{\tau \to 0}\mathbb{E}\|{U}_\tau^\sharp(t) - \tilde{U}_\tau^\sharp(t)\|_{L^{\infty}(0, T; H^1(0, 1))}^2 =0.
        \end{equation}

        Then, we obtain the desired result by the convergence of ${U}^{\sharp}_{\tau} \to U_{\epsilon}$ in $C(0, T; H^{1}(0, 1))$, $\tilde{\mathbb{P}}$-almost surely, due to Theorem \ref{pass_tau_0}. The stronger convergence of $\bar{\rho}^{\sharp}_{\tau} \to \rho_{\epsilon}$ in $C(0, T; H^{1}(0, 1))$, $\tilde{\mathbb{P}}$-almost surely, follows from the fact that $\bar{\rho}_{\tau}$ on the original probability space is continuous in time, by the definition in \eqref{U_defn} and the fact that the stochastic subproblem does not change $\rho$ (since the stochasticity is only in the momentum equation).
    \end{proof}

    We now use Theorem \ref{pass_tau_0} and Proposition \ref{samelimit} to show the following almost everywhere convergence result, which will be important for the $\tau \to 0$ limit passage of the approximate entropy formulation.

    \begin{prop}\label{partialUaeconv}
        For almost every $(\tilde\omega, t, x) \in \tilde\Omega \times [0, T] \times [0, 1]$:
        \begin{equation*}
        U^{\sharp}_{\tau}(\tilde\omega, t, x) \to U_{\epsilon}(\tilde\omega, t, x), \qquad \bar{U}^{\sharp}_{\tau}(\tilde\omega, t, x) \to U_{\epsilon}(\tilde\omega, t, x), \qquad \tilde{U}^{\sharp}_{\tau}(\tilde\omega, t, x) \to U_{\epsilon}(\tilde\omega, t, x),
        \end{equation*}
        \begin{equation*}
        \partial_{x} U^{\sharp}_{\tau}(\tilde\omega, t, x) \to \partial_{x} U_{\epsilon}(\tilde\omega, t, x), \qquad \partial_{x} \bar{U}^{\sharp}_{\tau}(\tilde\omega, t, x) \to \partial_{x} U_{\epsilon}(\tilde\omega, t, x), \qquad \partial_{x} \tilde{U}^{\sharp}_{\tau}(\tilde\omega, t, x) \to \partial_{x} U_{\epsilon}(\tilde\omega, t, x).
        \end{equation*}
    \end{prop}

    \begin{proof}
        From Theorem \ref{regularity_of_Uhat} and equivalence of laws, we have that
        \begin{equation*}
        \mathbb{E}\|U^{\sharp}_{\tau}(\tilde\omega, t, x)\|_{C^{\alpha}(0, T; H^{1}(0, 1))}^{2} \le C, \qquad \text{ for } \alpha \in [0, 1/4),
        \end{equation*}
        for a constant $C$ that depends on $\alpha$, $\epsilon$, and $T$, but is independent of $\tau$. Therefore, for $i = 0, 1$:
        \begin{equation*}
        \mathbb{E} \|\partial_{x}^{i}U^{\sharp}_{\tau}\|_{L^{1}((0, T) \times (0, 1))}^{2} \le C,
        \end{equation*}
        for a constant $C$ that is independent of $\tau$. By Theorem \ref{pass_tau_0}, we have that $U^{\sharp}_{\tau} \to U_{\epsilon}$, $\tilde{\mathbb{P}}$-almost surely, as $\tau \to 0$ in $L^{\infty}(0, T; H^{1}(0, 1))$, and hence for $i = 0, 1$, we have that $\partial_{x}^{i}U^{\sharp}_{\tau} \to \partial_{x}^{i} U_{\epsilon}$ as $\tau \to 0$, $\tilde{\mathbb{P}}$-almost surely, in $L^{1}((0, T) \times (0, 1))$. So by interpreting the approximate solutions $U^{\sharp}_{\tau}$ as $L^{1}((0, T) \times (0, 1))$ valued random variables, we can apply the Vitali convergence theorem to deduce that $U^{\sharp}_{\tau} \to U_{\epsilon}$ in $L^{1}(\tilde\Omega \times (0, T) \times (0, 1))$ as $\tau \to 0$, which implies the desired almost everywhere convergence. A similar argument completes the verification of the remainder of the convergences.
    \end{proof}

    Next, we show positivity of the limiting density $\rho_{\epsilon}$ in $U_{\epsilon} := (\rho_{\epsilon}, m_{\epsilon})$, by using the convergence of the terms $\ell_{\tau}^\sharp \to \ell_{\epsilon}$ in Theorem \ref{pass_tau_0} corresponding to the quantities $\bar{\rho}_{\tau}^{1/2} \partial_{x}\bar{u}_{\tau}$, in the space $L^{2}((0, T) \times (0, 1))$ with the weak topology, which are needed to apply Proposition \ref{bvpositivity}. 

    \begin{lem}[Positivity of the limiting density]\label{limitingpositive}
        We can identify the terms $\ell_{\tau}^\sharp$ and $\ell_{\epsilon}$ on the new probability space as:
        \begin{equation*}
    \ell_{\tau}^\sharp = (\bar{\rho}^{\sharp}_{\tau})^{1/2} \partial_{x}\bar{u}_{\tau}^\sharp, \qquad \ell_{\epsilon} = \rho_{\epsilon}^{1/2} \partial_{x}u_{\epsilon}.
        \end{equation*}
        Furthermore, for all $\tau > 0$,
        \begin{equation*}
        \min_{(t, x) \in [0, T] \times [0, 1]} \bar{\rho}_{\tau}^{\sharp}(t, x,\tilde\omega) \ge c(\tilde\omega) > 0, \qquad \min_{(t, x) \in [0, T] \times [0, 1]} {\rho}_{\tau}^{\sharp}(t, x,\tilde\omega) \ge c(\tilde\omega) > 0,
        \end{equation*}
        and
        \begin{equation*}
        \min_{(t, x) \in [0, T] \times [0, 1]} \rho_{\epsilon}(t, x,\tilde\omega) \ge c(\tilde\omega) > 0 \qquad 
        \end{equation*}
        $\tilde{\mathbb{P}}$-almost surely, for an almost surely positive random variable $c(\tilde\omega)$ that is independent of $\tau$.
    \end{lem}

    \begin{proof}
        The fact that $\ell_{\tau}^\sharp = (\bar{\rho}_{\tau}^{\sharp})^{1/2} \partial_{x} \bar{u}_{\tau}^\sharp$ follows by the equivalence of laws in Theorem \ref{pass_tau_0}. Since by Theorem \ref{pass_tau_0},
        \begin{equation*}
        (\bar{\rho}^{\sharp}_{\tau})^{1/2} \partial_{x}\bar{u}_{\tau}^\sharp \rightharpoonup \ell_{\epsilon}, \quad \text{weakly in } L^{2}((0, T) \times (0, 1)), \quad \tilde{\mathbb{P}}\text{-almost surely},
        \end{equation*}
        we conclude by the uniform boundedness principle, that there exists a positive random variable $C(\tilde\omega)$ independent of $\tau$, such that 
        \begin{equation}\label{cminbound1}
        \|(\bar{\rho}^{\sharp}_{\tau})^{1/2}\partial_{x}\bar{u}_{\tau}^\sharp\|_{L^{2}((0, T) \times (0, 1))} = \int_{0}^{T} \int_{0}^{1} \bar{\rho}^{\sharp}_{\tau} (\partial_{x} \bar{u}_{\tau}^\sharp)^{2} dx dt \le C(\tilde\omega).
        \end{equation}
        In addition, by the equivalence of laws in Theorem \ref{pass_tau_0} and the uniform $L^{\infty}$ bounds in Proposition \ref{L_infty_estimates}, we have that for $\bar{u}_\tau^\sharp := \frac{\bar{m}_\tau^\sharp}{\bar{\rho}_\tau^\sharp}$, there exists a deterministic constant $C_{\epsilon}$ that depends only on $\epsilon$, such that 
        \begin{equation}\label{cminbound2}
        \|\bar{u}^{\sharp}_{\tau}\|_{L^{\infty}((0, T) \times (0, 1))} \le C_{\epsilon}, \qquad \tilde{\mathbb{P}}\text{-almost surely.}
        \end{equation}
        Since $\bar{\rho}^{\sharp}_{\tau}$ also satisfies the continuity equation \eqref{rhobareqn} with $\bar{u}^{\sharp}_{\tau}$, we can hence apply Proposition \ref{bvpositivity} to conclude that 
        \begin{equation}\label{cminlower}
        \bar{\rho}_{\tau}^{\sharp} \ge c_{min}(\tilde\omega) > 0,
        \end{equation}
        for some almost surely positive random variable $c_{min}(\tilde\omega)$ that crucially is \textit{independent of $\tau$}. The fact that $c_{min}(\tilde\omega)$ is independent of $\tau$ follows from Proposition \ref{bvpositivity} and the fact that the estimates \eqref{cminbound1} and \eqref{cminbound2} are independent of $\tau$. Moreover, since density is kept constant in the stochastic subproblem, $\tilde{\rho}_\tau^\sharp \geq c_{min}(\tilde{\omega}) >0$, and thus, $\rho_{\tau}^{\sharp}$ (which is a linear interpolation of $\bar\rho^{\sharp}_{\tau}$ and $\tilde\rho^{\sharp}_{\tau}$ by \eqref{U_defn} and equivalence of laws) satisfies ${\rho}_\tau^\sharp \geq c_{min}(\tilde{\omega})>0$.

        Hence, because $\bar{\rho}^{\sharp}_{\tau} \to \rho_\epsilon$ in $C(0, T; H^{1}(0, 1))$ and hence in $C([0, T] \times [0, 1])$ also, $\tilde{\mathbb{P}}$-almost surely, we conclude from \eqref{cminlower} that 
        \begin{equation*}
        \rho_{\epsilon} \ge c_{min}(\tilde\omega) > 0,
        \end{equation*}
        also, for the same $c_{min}(\tilde\omega)$ as in \eqref{cminlower}.

        To identify the limit $\ell_{\epsilon}$ as $\rho_{\epsilon} \partial_{x}u_{\epsilon}$, we will prove that the convergences in Theorem \ref{pass_tau_0} imply that 
        \begin{equation}\label{goalweak}
        \ell_{\tau}^\sharp = (\bar{\rho}_{\tau}^\sharp)^{1/2} \partial_{x}\bar{u}_{\tau}^\sharp \rightharpoonup \rho_{\epsilon}^{1/2} \partial_{x}u_{\epsilon}, \quad \text{ weakly in } L^{2}((0, T) \times (0, 1)), \quad \tilde{\mathbb{P}}\text{-almost surely.}
        \end{equation}
        It suffices to show \eqref{goalweak} in order to identify the limit $\ell_{\epsilon}$ in Theorem \ref{pass_tau_0} as $\rho_{\epsilon}^{1/2}\partial_{x}u_{\epsilon}$, and hence conclude the proof, so we will focus on establishing the weak convergence in \eqref{goalweak} for the remainder of the proof. By integration by parts, we will establish \eqref{goalweak} once we show that for all $\varphi \in C_{c}^{\infty}([0, T] \times [0, 1])$:
        \begin{equation}\label{goalweak2}
        \int_{0}^{T} \int_{0}^{1} \frac{\bar{m}_{\tau}^\sharp \partial_{x}\bar{\rho}_{\tau}^\sharp}{(\bar{\rho}_{\tau}^\sharp)^{3/2}} \varphi dx dt \to \int_{0}^{T} \int_{0}^{1} \frac{m_{\epsilon} \partial_{x}\rho_{\epsilon}}{\rho_{\epsilon}^{3/2}} \varphi dx dt,
        \end{equation}
        \begin{equation}\label{goalweak3}
        \int_{0}^{T} \int_{0}^{1} \frac{\partial_x \bar{m}_{\tau}^{\sharp}}{(\bar{\rho}_{\tau}^{\sharp})^{1/2}} \varphi dx dt \to \int_{0}^{T} \int_{0}^{1} \frac{\partial_x m_{\epsilon}}{\rho_{\epsilon}^{1/2}} \varphi dx dt,
        \end{equation}
        as $\tau \to 0$, $\tilde{\mathbb{P}}$-almost surely. We will now establish each of these convergences.
        
        \medskip
        
        \noindent \textbf{Proof of \eqref{goalweak2}.} To see this, we note that uniformly in $\tau$:
        \begin{equation}\label{conv1}
        \left|\int_{0}^{T} \int_{0}^{1} \frac{\bar{m}_{\tau}^\sharp}{(\bar{\rho}^\sharp)^{3/2}} \varphi (\partial_{x}\rho_{\epsilon} - \partial_{x}\bar{\rho}_{\tau}^\sharp) dx dt\right| \le c(\tilde{\omega}) \int_{0}^{T} \int_{0}^{1} |\partial_{x}\rho_{\epsilon} - \partial_{x}\bar{\rho}_{\tau}^\sharp| dx dt \le c(\tilde\omega) T \|\rho_{\epsilon} - \bar{\rho}_{\tau}^\sharp\|_{C(0, T; H^{1}(0, 1))},
        \end{equation}
        by the uniform lower bound in \eqref{cminlower}, and the fact that $\|\bar{m}_{\tau}^\sharp\|_{C([0, T] \times [0, 1])} \le C$ independently of $\tau$  by Proposition \ref{L_infty_estimates} and equivalence of laws. In addition, note that $\displaystyle \left|\frac{\bar{m}_{\tau}^\sharp}{(\bar{\rho}_{\tau}^\sharp)^{3/2}}\right| \le c(\tilde\omega)$ uniformly in $\tau$ and $\displaystyle \frac{\bar{m}_{\tau}^\sharp}{(\bar{\rho}_{\tau}^\sharp)^{3/2}} \to \frac{m_{\epsilon}}{\rho_{\epsilon}^{3/2}}$ by the lower bound \eqref{cminlower} and the convergences in $C([0, T] \times [0, 1])$ in Proposition \ref{samelimit} by Sobolev embedding. Hence, by the dominated convergence theorem,
        \begin{equation}\label{conv2}
        \left|\int_{0}^{T} \int_{0}^{1} \left(\frac{m_{\epsilon}}{\rho_{\epsilon}^{3/2}} - \frac{\bar{m}_{\tau}^\sharp}{(\bar{\rho}_{\tau}^\sharp)^{3/2}}\right) \varphi \partial_{x}\rho_{\epsilon} dx dt\right| \to 0, \qquad \text{ as } \tau \to 0,
        \end{equation}
        $\tilde{\mathbb{P}}$-almost surely. So by combining \eqref{conv1} and \eqref{conv2} and using Proposition \ref{samelimit}, we obtain the convergence \eqref{goalweak2}.

        \medskip

        \noindent \textbf{Proof of \eqref{goalweak3}.} Similarly, by the uniform lower bound in \eqref{cminlower},
        \begin{equation*}
        \left|\int_{0}^{T} \int_{0}^{1} \frac{1}{(\bar{\rho}_{\tau}^{\sharp})^{1/2}} \varphi (\partial_{x} m_{\epsilon} - \partial_{x} \bar{m}_{\tau}^{\sharp})\right| \le C(\tilde \omega) T \|m_{\epsilon} - \bar{m}_{\tau}^{\sharp}\|_{L^{\infty}(0, T; H^{1}(0, 1))},
        \end{equation*}
        and similarly, by dominated convergence and the uniform lower bound in \eqref{cminlower},
        \begin{equation*}
        \left|\int_{0}^{T} \int_{0}^{1} \left(\frac{1}{\rho_{\epsilon}^{1/2}} - \frac{1}{(\bar{\rho}^{\sharp}_{\tau})^{1/2}}\right) \varphi (\partial_{x}m_{\epsilon})\right| \to 0, \qquad \text{ as } \tau \to 0.
        \end{equation*}
        The proof of \eqref{goalweak3} is thus completed by using Proposition \ref{samelimit}.
    \end{proof}

    Finally, before we carry out the limit passage as $\tau \to 0$ in the approximate entropy balance equation \eqref{entropy_eq_Uhat_full}, we note that this approximate entropy balance equation contains various nonlinear functions of the approximate solution $U_{\epsilon}$, defined via the entropy-flux pairs $(\eta, H)$. Thus, it will be useful to derive certain analytic bounds on these nonlinear functionals of $U_{\epsilon}$, as functions of the state variables $(\rho, m)$. These properties can be derived algebraically from the explicit formula for the entropy-flux pairs $(\eta, H)$ given in \eqref{entropy_pair_formula}.
    
    \begin{lem}\label{lem_lipschitz}
        Let $(\eta, H)$ be the entropy-flux pair defined via the relations \eqref{entropy_pair_formula} for functions convex $g \in C^{2}(\mathbb{R})$. Then, for all $(\rho, m) \in \mathbb{R}^{+} \times \mathbb{R}$ satisfying
        \begin{equation}\label{Mbounds}
        0 < \rho \le M, \qquad \left|\frac{m}{\rho}\right| \le M
        \end{equation}
        for some positive constant $M$, there exists a constant $C_{g, M}$ depending only on the choice of the convex function $g \in C^{2}(\mathbb{R})$ and the positive constant $M$, such that
        \begin{equation*}
        |\eta(\rho, m)| \le C_{g, M}\rho, \qquad |H(\rho, m)| \le C_{g, M} \rho, 
        \end{equation*}
        \begin{equation*}
        |\nabla_{\rho, m} \eta(\rho, m)| \le C_{g, M}, \qquad \rho |\nabla^{2}_{\rho, m} \eta(\rho, m)| \le C_{g, M}.
        \end{equation*}
        \if 1 = 0
        \begin{itemize}
            \item The function $H(\rho, m)$ is Lipschitz continuous as a function of $(\rho, m)$, on any compact set contained in $(0, \infty)\times \mathbb{R}$.
            \item The function $\partial_{m}\eta(\rho, m)$ is Lipschitz continuous as a function of $(\rho, m)$, on any compact set contained in $[0, \infty) \times \mathbb{R}$ (containing the vacuum set $\rho = 0$).
            \item The function $\nabla^2\eta(\rho, m)$ is Lipschitz continuous as a function of $(\rho, m)$ on any compact set contained in $(0, \infty)\times \mathbb{R}$. 
        \end{itemize}
        \fi
    \end{lem}
    \begin{proof} 
    This result is proved in Proposition 1.3 in \cite{KuanTawriTrivisa2025}, and it follows immediately by direct computations using the formulas in \eqref{entropy_pair_formula} along with the inequalities in \eqref{Mbounds}. We refer the reader to \cite{KuanTawriTrivisa2025} for the details of the proof.
    
    \if 1 = 0
    We recall the following formulas from \eqref{entropy_pair_formula}, for the entropy-flux pair $(\eta, H)$ given a convex function $g \in C^{2}(\mathbb{R})$:
    \begin{equation*}
    \eta(\rho, m) = \rho \int_{-1}^{1} g\left(\frac{m}{\rho} + z\rho^{\theta}\right) (1 - z^{2}) dz,
    \end{equation*}
    \begin{equation*}
    H(\rho, m) = \rho \int_{-1}^{1} g\left(\frac{m}{\rho} + z\rho^{\theta}\right) \left(\frac{m}{\rho} + z\theta \rho^{\theta}\right) (1 - z^{2}) dz.
    \end{equation*}

    \if 1 = 0
        To show the first statement, we carry out a direct computation of $\nabla_{\rho, m} H$ using the definition of the entropy flux function $H$ given in \eqref{entropy_pair_formula}, and we can then verify that $\nabla_{\rho, m} H$ is continuous on any compact set of $(0, \infty)\times \mathbb{R}$ (excluding the vacuum). By Lemma \ref{limitingpositive}, we have that 
        \begin{equation}\label{compactawayvacuum}
            (\bar{\rho}_\tau^\sharp, \bar{m}_\tau^\sharp) \in [c(\tilde{\omega}), C] \times [-C, C],
        \end{equation} for some real-valued strictly positive random variable $c(\tilde{\omega})$ independent of $\tau$, and a deterministic constant $C$ also independent of $\tau$. Therefore, we can conclude that the function $H(\bar{\rho}_\tau^\sharp, \bar{m}_\tau^\sharp)$ is $C^1$ on any compact set of $(0, \infty)\times\mathbb{R}$, and hence it is locally Lipschitz on $(0, \infty)\times \mathbb{R}$ with a Lipschitz constant $L=L(\tilde{\omega})$ depending on the outcome $\tilde{\omega}\in \tilde{\Omega}$ but uniform in $\tau$.

         To show the second statement, we carry out a similar direct computation of $\nabla_{\rho, m} D^\alpha \eta$, for $|\alpha|=2$ using the definition of the entropy function $\eta$ given in \eqref{entropy_pair_formula}. Then, again by Lemma \ref{limitingpositive} and the fact that the function $g$ in the definition of entropy function $\eta$ \eqref{entropy_pair_formula} belongs to $C^2$, we can conclude that $\nabla_{\rho, m} D^\alpha \eta$ as a function of $(\bar{\rho}_\tau^\sharp, \bar{m}_\tau^\sharp)$ is continuous on any compact set of $(0, \infty)\times \mathbb{R}$, and thus $\nabla^2 \eta$  as a function of $(\bar{\rho}_\tau^\sharp, \bar{m}_\tau^\sharp)$ is locally Lipschitz on $(0, \infty)\times \mathbb{R}$ with a Lipschitz constant $L=L(\tilde{\omega})$ uniform in $\tau>0$.
        Furthermore, the squared regularized noise coefficient $\sigma_\epsilon^2(x, U)$ is Lipschitz in the state variables. More specifically,
        \[|\sigma_\epsilon^2(x, U_1)-\sigma_\epsilon^2(x, U_2)| =
        |\sigma_\epsilon(x, U_1)-\sigma_\epsilon(x, U_2)||\sigma_\epsilon(x, U_1)+\sigma_\epsilon(x, U_2)|.\]
        By the Lipschitz assumption on the regularized noise coefficient $\sigma_\epsilon$ followed from \eqref{noise_assumption}
        \[|\sigma_\epsilon(x, U_1)-\sigma_\epsilon(x, U_2)|\leq \sqrt{A_0}|U_1 - U_2|.\]
        Moreover, by \eqref{sigma_bound} and Proposition \ref{L_infty_estimates}, the following bound hold for some $C=C(\|U_0\|_{L^\infty(0,1)}, \gamma)$
        \[|\sigma_\epsilon(x, U_1)+\sigma_\epsilon(x, U_2)| \leq \|U_1\|_{L^\infty(Q_T)} + \|U_2\|_{L^\infty(Q_T)} \leq C.\]
        Therefore, $|\sigma_\epsilon^2(x, U_1)-\sigma_\epsilon^2(x, U_2)| \leq C |U_1 - U_2|$, for some constant $C=C(\|U_0\|_{L^\infty(0,1)}, \gamma, A_0)$. We can thus conclude that $\partial_m^2 \eta(\cdot) \sigma^2(x, \cdot)$ as a function of $(\tilde{\rho}_\tau^\sharp, \tilde{m}_\tau^\sharp)$ is locally Lipschitz on $(0, \infty)\times\mathbb{R}$ with Lipschitz constant $L=L(\tilde{\omega})$ uniform in $\tau$.

        The third statement is shown by the argument used in proving the second statement. 
        
        To show the last statement, it suffices to show that $\partial_m \eta(\cdot, \cdot)\sigma_\epsilon(x, \cdot, \cdot)$ is $C^1$ in $[0, \infty)\times \mathbb{R}$, \textit{crucially including the vacuum state of the density variable}. We first make the observation that $\bar{u}_\tau^\sharp:=\frac{\bar{m}_\tau^\sharp}{\bar{\rho}_\tau^\sharp}$ is uniformly bounded in $L^\infty(Q_T)$ due to Proposition \ref{L_infty_estimates} by a deterministic constant and the equivalence of laws between $\bar{U}_\tau^\sharp$ and $\bar{U}_\tau$. With this observation, we then compute the expression for $D^\alpha \eta$ with $|\alpha|=2$ and conclude that $|D^\alpha \eta(\bar{\rho}_\tau^\sharp, \bar{m}_\tau^\sharp)|\leq \frac{C}{|\bar{\rho}_\tau^\sharp|}$ for some deterministic constant $C$ independent of $\tau$. Moreover, as a consequence of the Lipschitz assumption on the noise coefficient \eqref{noise_assumption}: $|\nabla_{\rho, m}\sigma_\epsilon(x, \rho, m)| \leq \sqrt{A_0}$, we have that $|\sigma_\epsilon(x, \bar{\rho}_\tau^\sharp, \bar{m}_\tau^\sharp) | \leq \sqrt{A_0}|\bar{m}_\tau^\sharp|$. Therefore
        \begin{equation}\label{product1}
            |D^\alpha\eta(\bar{\rho}_\tau^\sharp, \bar{m}_\tau^\sharp) \sigma_\epsilon(x, \bar{\rho}_\tau^\sharp, \bar{m}_\tau^\sharp)| \leq \frac{C}{|\bar{\rho}_\tau^\sharp|}\sqrt{A_0}|\bar{m}_\tau^\sharp|= C\sqrt{A_0}|\bar{u}_\tau^\sharp| \leq C,  \;\; |\alpha|=2,
        \end{equation}
        where $C$ is a deterministic constant independent of $\tau$. In addition, 
        \begin{equation}\label{product2}
            |\partial_m\eta(\bar{\rho}_\tau^\sharp, \bar{m}_\tau^\sharp)\nabla_{\rho, m}\sigma_\epsilon(x, \bar{\rho}_\tau^\sharp, \bar{m}_\tau^\sharp)|\leq C.
        \end{equation}
        Combining \eqref{product1} and \eqref{product2}, we have that
        \[|\nabla_{\rho, m} (\partial_m \eta(\bar{\rho}_\tau^\sharp, \bar{m}_\tau^\sharp) \sigma_\epsilon(x, \bar{\rho}_\tau^\sharp, \bar{m}_\tau^\sharp)| \leq C.\]
        Therefore, $\partial_m\eta(\cdot, \cdot)\sigma_\epsilon(x, \cdot, \cdot)$ as a function of $(\bar{\rho}_\tau^\sharp, \bar{m}_\tau^\sharp)$ is $C^1$ on $[0, \infty)\times \mathbb{R}$ (including the vacuum), and hence it is Lipschitz with a Lipschitz constant $L$ uniform in $\tau>0$ and $\tilde{\omega}\in \tilde{\Omega}$ on $(0, \infty)\times \mathbb{R}$. 

        Finally, the same results hold for $(\tilde{\rho}_\tau^\sharp, \tilde{m}_\tau^\sharp)$, and $(\rho_\tau^\sharp, m_\tau^\sharp)$ because both $\tilde{\rho}_\tau^\sharp$ and $\rho_\tau^\sharp$ are bounded below by some strictly positive random variable $c(\tilde{\omega})$, due to Lemma \ref{limitingpositive} and the fact that the density is not updated in the stochastic subproblem. We can then use the same argument to conclude the Lipschitz property of the four functions.
        \fi
        \fi
    \end{proof}
    Now, we have all of the necessary components needed to pass to the limit as $\tau \to 0$ in the approximate entropy balance equation \eqref{entropy_eq_Uhat_full}. In particular, we proceed to show that the limit $U_\epsilon$ satisfies the weak formulation of the parabolic system \eqref{regularized_problem1}--\eqref{regularized_problem2} and the desired entropy balance equation at the $\epsilon$ level, stated in Theorem \ref{thm_existence_pathwise_U_epsilon}.
    
	\begin{thm}[Identification of the limit]\label{thm_identify_the_limit_tau}
		Suppose $U_{\epsilon0} \in H^{2}(0,1)$ and $\rho_{\epsilon0} \geq c_\epsilon$ for some (deterministic) constant $c_\epsilon > 0$, and the noise coefficient $\sigma_\epsilon $ satisfies \eqref{noise_assumption}.
		Let $U_\epsilon$ be defined as in Theorem \ref{pass_tau_0}. Then, the following properties hold: 
		\begin{itemize}
			\item $U_\epsilon$ satisfies the following weak formulation for every $\varphi \in C_c^2((0,1))$:
			\begin{equation}\label{epsweakthm}
				\begin{split}
					\int_{0}^{1}U_\epsilon(t)\varphi(x)\,dx =& \int_{0}^{1}U_{\epsilon0}\varphi(x)\,dx  + \int_{0}^{t}\int_{0}^1 F(U_\epsilon)\partial_x\varphi \,dx\,ds - \int_{0}^{t}\int_{0}^{1} \begin{pmatrix}
						0\\ \alpha m_\epsilon 
					\end{pmatrix}\varphi(x)\,dx\,ds \\
					+& \epsilon\int_{0}^{t}\int_{0}^{1}U_\epsilon\partial_x^2\varphi + \int_{0}^{t}\int_{0}^{1}\begin{pmatrix}
						0 \\ \sigma_\epsilon(x, U_\epsilon)
					\end{pmatrix}\varphi(x)\,dx\,dW^\sharp(s), \;\;\; \tilde{\mathbb{P}}\text{-almost surely,}\\
				\end{split}
			\end{equation}
            where the function $F$ is defined in \eqref{defn_FGPhi}.
            \smallskip
			\item More generally, for any entropy-flux pair $(\eta, H)$ generated by a convex function $g \in C^{2}(\mathbb{R})$ via the formulas \eqref{entropy_pair_formula}, $U_\epsilon$ satisfies the following entropy equality for every $\varphi \in C_c^2((0,1))$:
			\begin{equation}
				\begin{split}
					&\int_{0}^{1}\eta(U_\epsilon(t))\varphi(x)\,dx = \int_{0}^{1}\eta(U_{\epsilon0})\varphi(x)\,dx  + \int_{0}^{t}\int_{0}^1 H(U_\epsilon)\partial_x\varphi \,dx\,ds\\& - \int_{0}^{t}\int_{0}^{1} m_\epsilon\partial_m\eta(U_\epsilon)\varphi(x)\,dx\,ds 
					+ \epsilon\int_{0}^{t}\int_{0}^{1}\eta(U_\epsilon)\partial_x^2\varphi \,dx\,ds
					+ \epsilon \int_{0}^{t}\int_{0}^{1}\langle \nabla^2\eta(U_\epsilon)\partial_xU_\epsilon, \partial_xU_\epsilon\rangle\varphi(x)\,dx\,ds\\
					+& 	 \int_{0}^{t}\int_{0}^{1}\partial_m\eta(U_\epsilon)\sigma_\epsilon(x,U_\epsilon)\varphi(x)\,dx\,dW^\sharp(s) + \frac{1}{2}\int_{0}^{t}\int_{0}^{1}\partial_m^2\eta(U_\epsilon)\sigma_\epsilon^2(x, U_\epsilon)\varphi(x)\,dx\,ds,
				\end{split}
			\end{equation}
            $\tilde{\mathbb{P}}$-almost surely.
		\end{itemize} 
	\end{thm}
	\begin{proof}
		%
         
        Recall that ${U}_\tau(t)$ satisfies the entropy balance equation \eqref{entropy_eq_Uhat_full}. Then by the equivalence of laws between ${U}_\tau^\sharp(t)$, and the definition of $\bar{U}_\tau^\sharp$ and $\tilde{U}_\tau^\sharp$ in \eqref{defn_Ubar_sharp}, the same approximate balance equation {\eqref{entropy_eq_Uhat_full}} holds for the random variables on the probability space $(\tilde{\Omega}, \tilde{\mathcal{F}}, \tilde{\mathbb{P}})$ given by Theorem \ref{pass_tau_0}. Below, we rewrite the equation \eqref{entropy_eq_Uhat_full} so that the numerical errors between the linear interpolant ${U}_\tau^\sharp$ and the solutions to the two subproblems $\bar{U}_\tau^\sharp$ and $\tilde{U}_\tau^\sharp$, see the definitions in \eqref{U_defn}, are reflected:
		\begin{equation}\label{approx_entropy_full}
			\begin{split}
				&\int_{0}^{1}\eta({U}_\tau^\sharp(t))\varphi(x)\,dx \\
				=& \int_{0}^{1}\eta(U_{\epsilon 0})\varphi(x)\,dx  + \int_{0}^{t}\int_{0}^1 H({U}_\tau^\sharp)\partial_x\varphi \,dx\,ds - \alpha\int_{0}^{t}\int_{0}^{1} {m}_\tau^\sharp \partial_m \eta({U}_\tau^\sharp)\varphi(x)\,dx\,ds \\
				+& \epsilon\int_{0}^{t}\int_{0}^{1}\eta({U}_\tau^\sharp)\partial_x^2\varphi - \langle \nabla^2\eta({U}_\tau^\sharp)\partial_x{U}_\tau^\sharp, \partial_x{U}_\tau^\sharp\rangle \varphi(x) \,dx\,ds \\
				+& \frac{1}{2}\int_{0}^{t} \int_{0}^{1} \sigma_\epsilon^2(x, \tilde{U}_\tau^\sharp) \partial_m^2\eta({U}_\tau^\sharp)\varphi(x)\,dx\,ds + \int_{0}^{t}\int_{0}^{1}\partial_m\eta({U}_\tau^\sharp)\sigma_\epsilon(x, {U}_\tau^\sharp)\varphi(x)\,dx\,dW^\sharp(s)\\
				+ &  \int_{0}^{t}\int_{0}^1 \left(H(\bar{U}_\tau^\sharp) - H({U}_\tau^\sharp)\right)\partial_x\varphi \,dx\,ds - \alpha\int_{0}^{t}\int_{0}^{1}\left( \bar{m}_\tau^\sharp \partial_m \eta(\bar{U}_\tau^\sharp) - {m}_\tau^\sharp \partial_m \eta({U}_\tau^\sharp)\right)\varphi(x)\,dx\,ds \\
				+& \epsilon\int_{0}^{t}\int_{0}^{1}\left(\eta(\bar{U}_\tau^\sharp) - \eta({U}_\tau^\sharp)\right)\partial_x^2\varphi -\left(\langle \nabla^2 \eta(\bar{U}_\tau^\sharp)\partial_x\bar{U}_\tau^\sharp, \partial_x\bar{U}_\tau^\sharp\rangle - \langle \nabla^2\eta({U}_\tau^\sharp)\partial_x{U}_\tau^\sharp, \partial_x{U}_\tau^\sharp\rangle \right)\varphi(x) \,dx\,ds \\
				+& \frac{1}{2}\int_{0}^{t} \int_{0}^{1} \left(\partial_m^2\eta(\tilde{U}_\tau^\sharp)\sigma_\epsilon^2(x, \tilde{U}_\tau^\sharp) -  \partial_m^2\eta({U}_\tau^\sharp)\sigma_\epsilon^2(x, {U}_\tau^\sharp)\right)\varphi(x)\,dx\,ds\\
                +& \int_{0}^{t}\int_{0}^{1}\left(\partial_m\eta(\tilde{U}_\tau^\sharp)\sigma_\epsilon(x, \tilde{U}_\tau^\sharp) - \partial_m\eta({U}_\tau^\sharp)\sigma_\epsilon(x, {U}_\tau^\sharp)\right)\varphi(x)\,dx\,dW^\sharp(s)\\
				+ & \frac{t-t_n}{\tau}\left[\int_{t}^{t_{n+1}}\int_{0}^1 \left(H(\bar{U}_\tau^\sharp)\partial_x\varphi - \alpha\bar{m}_\tau^\sharp \partial_m \eta(\bar{U}_\tau^\sharp)\varphi(x)
				+ \epsilon \eta(\bar{U}_\tau^\sharp)\partial_x^2\varphi - \epsilon \langle \nabla^2\eta(\bar{U}_\tau^\sharp)\partial_x\bar{U}_\tau^\sharp, \partial_x\bar{U}_\tau^\sharp\rangle \varphi(x)\right) \,dx\,ds \right]\\
				+ &\frac{t - t_{n}}{\tau}\left[\frac{1}{2}\int_{t_n}^{t} \int_{0}^{1} \partial_m^2\eta(\tilde{U}_\tau^\sharp)\sigma_\epsilon^2(x, \tilde{U}_\tau^\sharp) \varphi(x)\,dx\,ds + \int_{t_n}^{t}\int_{0}^{1}\partial_m\eta(\tilde{U}_\tau^\sharp)\sigma_\epsilon(x, \tilde{U}_\tau^\sharp)\varphi(x)\,dx\,dW^\sharp(s)\right]\\
				+& \int_{0}^{1}\mathrm{R}_2(\bar{U}_\tau^\sharp-\tilde{U}_\tau^\sharp)\varphi(x)\,dx\,ds\\
				= : & \int_{0}^{1}\eta(U_{\epsilon 0})\varphi(x)\,dx  + \int_{0}^{t}\int_{0}^1 H({U}_\tau^\sharp)\partial_x\varphi \,dx\,ds - \alpha\int_{0}^{t}\int_{0}^{1} {m}_\tau^\sharp \partial_m \eta({U}_\tau^\sharp)\varphi(x)\,dx\,ds \\
				+& \epsilon\int_{0}^{t}\int_{0}^{1}\eta({U}_\tau^\sharp)\partial_x^2\varphi - \langle \nabla^2\eta({U}_\tau^\sharp)\partial_x{U}_\tau^\sharp, \partial_x{U}_\tau^\sharp\rangle \varphi(x) \,dx\,ds \\
				+& \frac{1}{2}\int_{0}^{t} \int_{0}^{1}  \partial_m^2\eta({U}_\tau^\sharp)\sigma_\epsilon^2(x, {U}_\tau^\sharp)\varphi(x)\,dx\,ds + \int_{0}^{t}\int_{0}^{1}\partial_m\eta({U}_\tau^\sharp)\sigma_\epsilon(x, {U}_\tau^\sharp)\varphi(x)\,dx\,dW^\sharp(s)\\
				&+ \sum_{i=1}^{10}E_i(t),
			\end{split}
		\end{equation}
         where $\mathrm{R}_2(\bar{U}_\tau^\sharp-\tilde{U}_\tau^\sharp)$ is the second order Taylor remainder term given by 
         \[\mathrm{R}_2(\bar{U}_\tau^\sharp-\tilde{U}_\tau^\sharp) =  \frac{1}{2}(\tilde{U}_\tau^\sharp - \bar{U}_\tau^\sharp)^T\left((\frac{t-t_n}{\tau})^2\nabla^2\eta(\mathbf{\xi}_1^\sharp)-\frac{t-t_n}{\tau}\nabla^2\eta(\mathbf{\xi}_2^\sharp)\right)(\tilde{U}_\tau^\sharp - \bar{U}_\tau^\sharp),\]
         where $\xi_1^\sharp=\xi_1^\sharp(t, x), \xi_2^\sharp=\xi_2^\sharp(t, x) \in \mathbb{R}^2$ are two points on the line segment joining $\bar{U}_\tau^\sharp(t, x)$ and $\tilde{U}_\tau^\sharp(t, x)$, for every $(t, x) \in Q_T$ (see calculation in \eqref{eta_taylorexpansion}-\eqref{entropy_eq_Uhat}).
          We remark that the error terms $\{E_i\}_{i=1}^6$ arise due to the numerical error between ${U}_\tau^\sharp$, $\bar{U}_\tau^\sharp$, and $\tilde{U}_\tau^\sharp$ as defined in \eqref{U_defn}, and the error terms $\{E_i\}_{i=7}^{10}$ are terms that are already in the entropy equality \eqref{entropy_eq_Uhat_full}, which, we recall, arise from calculating the Taylor expansion of $\eta({U}_\tau^\sharp)$, see \eqref{entropy_eq_Uhat}. Our goal is to show that as $\tau \to 0$, $E_i\to 0$, $\tilde{\mathbb{P}}$-almost surely, as this will allow us to recover the entropy equality \eqref{epsilonentropy} as we pass $\tau \to 0$. For convenience, we collect all of the uniform bounds in $\tau$ and convergences that we have altogether here. By equivalence of laws, from Proposition \ref{L_infty_estimates} and Lemma \ref{limitingpositive}, we have that for the approximate solutions and the fluid velocity $u$:
          \begin{equation}\label{Utaubound}
          \|U^{\sharp}_{\tau}\|_{L^{\infty}(Q_{T})} \le C, \qquad \|\bar{U}^{\sharp}_{\tau}\|_{L^{\infty}(Q_{T})} \le C, \qquad \|\tilde{U}_{\tau}^{\sharp}\|_{L^{\infty(Q_{T})}} \le C,
          \end{equation}
          \begin{equation}\label{littleutaubound}
          \|u^{\sharp}_{\tau}\|_{L^{\infty}(Q_{T})} \le C, \qquad \|\bar{u}^{\sharp}_{\tau}\|_{L^{\infty}(Q_{T})} \le C, \qquad \|\tilde{u}_{\tau}^{\sharp}\|_{L^{\infty(Q_{T})}} \le C,
          \end{equation}
          for a deterministic constant $C$, $\tilde{\mathbb{P}}$-almost surely, and 
          \begin{equation}\label{rhotaubounds}
          0 < c(\tilde\omega) \le \rho^{\sharp}_{\tau} \le C,
          \end{equation}
          for a positive constant $c(\tilde\omega)$ depending on the outcome, $\tilde{\mathbb{P}}$-almost surely, but independent of $\tau$. Furthermore, from Proposition \ref{pass_tau_0} and Proposition \ref{samelimit}, we have the $\tilde{\mathbb{P}}$-almost sure convergences:
          \begin{equation}\label{tauconvergences}
          U^{\sharp}_{\tau} \to U_{\epsilon}, \qquad \bar{U}^{\sharp}_{\tau} \to U_{\epsilon}, \qquad \tilde{U}^{\sharp}_{\tau} \to U_{\epsilon}, \qquad \text{ in } L^{\infty}(0, T; H^{1}(0, 1)).
          \end{equation}
          Using these bounds and convergences, we now proceed to estimate the error terms, as follows.

\medskip

        \noindent \textbf{Estimates of $E_{1}$, $E_{2}$, and $E_{3}$.} 
        By \eqref{Utaubound}, \eqref{littleutaubound}, and the bound on $H(\rho, m)$ in Lemma \ref{lem_lipschitz}, we have that for some deterministic constant $C$, $|H(\bar{U}^{\sharp}_{\tau}) - H(U^{\sharp}_{\tau})| \le C$, $\tilde{\mathbb{P}}$-almost surely. Hence, since $\bar{U}^{\sharp}_{\tau} - U^{\sharp}_{\tau} \to 0$ in $L^{\infty}(Q_T)$, $\tilde{\mathbb{P}}$-almost surely using \eqref{tauconvergences}, and since $H(\rho, m)$ is continuous in $(\rho, m)$ away from vacuum, we have by dominated convergence that
		\begin{equation*}
				|E_1| := \left|\int_{0}^{t}\int_{0}^1 \left(H(\bar{U}_\tau^\sharp) - H({U}_\tau^\sharp)\right)\partial_x\varphi \,dx\,ds \right| \to 0,
		\end{equation*}
		$\tilde{\mathbb{P}}$-almost surely. We can also show that $E_{2}$ and $E_{3}$ converge to 0, $\tilde{\mathbb{P}}$-almost surely as $\tau \to 0$ by a similar argument using the dominated convergence theorem.

        \medskip
        
		\noindent \textbf{Estimate of $E_{4}$.} Next, we estimate $E_4$:
		\begin{equation*}
				E_4 := \epsilon\int_{0}^{t}\int_{0}^{1}\left( \langle \nabla^2 \eta(\bar{U}_\tau^\sharp)\partial_x\bar{U}_\tau^\sharp, \partial_x\bar{U}_\tau^\sharp \rangle - \langle \nabla^2\eta({U}_\tau^\sharp)\partial_x{U}_\tau^\sharp, \partial_x{U}_\tau^\sharp\rangle \right)\varphi(x) \,dx\,ds.	
		\end{equation*}
    By the estimates in Lemma \ref{lem_lipschitz} and the vacuum estimate \eqref{rhotaubounds}, there exists a random variable $C(\tilde\omega) > 0$ independent of $\tau$, such that for all $\tau > 0$,
  \begin{equation}\label{HessianUuniform}
  \|\nabla^2\eta(\bar{U}^{\sharp}_{\tau})\|_{L^\infty(Q_T)} \leq C(\tilde\omega), \qquad \|\nabla^{2}\eta(U^{\sharp}_{\tau})\|_{L^{\infty}(Q_{T})} \le C(\tilde\omega).
  \end{equation}
  Therefore, by \eqref{HessianUuniform},
   \begin{equation*}
   \left|\left( \langle \nabla^2 \eta(\bar{U}_\tau^\sharp)\partial_x\bar{U}_\tau^\sharp, \partial_x\bar{U}_\tau^\sharp \rangle - \langle \nabla^2\eta({U}_\tau^\sharp)\partial_x{U}_\tau^\sharp, \partial_x{U}_\tau^\sharp\rangle \right)\varphi(x)\right| \le C(\tilde\omega) \|\varphi\|_{L^{\infty}(0, 1)} \Big(|\partial_{x}\bar{U}^{\sharp}_{\tau}|^{2} + |\partial_{x}U^{\sharp}_{\tau}|^{2}\Big).
   \end{equation*}
   Furthermore, by Proposition \ref{partialUaeconv}, we have that $\tilde{\mathbb{P}}$-almost surely:
   \begin{equation*}
   |\partial_{x}\bar{U}^{\sharp}_{\tau}|^{2} + |\partial_{x}U^{\sharp}_{\tau}|^{2} \to 2|\partial_{x}U_{\epsilon}|^{2}, \qquad \text{ for almost every } (t, x) \in [0, T] \times [0, 1].
   \end{equation*}
   In addition, by the convergences in \eqref{tauconvergences}, we have that
   \begin{equation*}
   \int_{0}^{t} \int_{0}^{1} \Big(|\partial_{x} \bar{U}^{\sharp}_{\tau}|^{2} + |\partial_{x}U^{\sharp}_{\tau}|^{2}\Big) dx dt \to 2\int_{0}^{t} \int_{0}^{1} |\partial_{x}U_{\epsilon}|^{2} dx ds.
   \end{equation*}
   So $E_{4} \to 0$ as $\tau \to 0$, $\tilde{\mathbb{P}}$-almost surely, by the generalized Lebesgue dominated convergence theorem (see Theorem 11 in Section 4.4 of \cite{royden}).

        \medskip
        
        \noindent \textbf{Estimate of $E_{5}$.} By the assumption \eqref{regularizedLipschitz} on the regularized noise coefficient $\sigma_{\epsilon}$, we have that $|\sigma_{\epsilon}(x, \rho, m)| \le \sqrt{A_{0}} \rho$. So by \eqref{Utaubound}, the bounds on $\partial^{2}_m\eta(\rho, m)$ in Lemma \ref{lem_lipschitz}, and the vacuum bounds \eqref{rhotaubounds}, there exists a random constant $C(\tilde\omega)$ independent of $\tau$, such that
        \begin{equation*}
        \left|\left(\sigma^2_\epsilon(x, \tilde{U}_\tau^\sharp)\partial_m^2\eta(\tilde{U}_\tau^\sharp) - \sigma^2_\epsilon(x, {U}_\tau^\sharp) \partial_m^2\eta({U}_\tau^\sharp)\right)\varphi(x)\right| \le \|\varphi\|_{L^{\infty}(0, 1)} C(\tilde\omega).
        \end{equation*}
        So by the convergence of $\tilde{U}^{\sharp}_{\tau}$ and $U^{\sharp}_{\tau}$ to $U_{\epsilon}$, $\tilde{\mathbb{P}}$-almost surely in $L^{\infty}(Q_{T})$, combined with the continuity of $\sigma_{\epsilon}(\rho, m)$ and the continuity of $\partial^{2}_{m}\eta(\rho, m)$ away from vacuum, we have that 
		\begin{equation*}
				E_5 = \frac{1}{2}\int_{0}^{t} \int_{0}^{1} \left(\sigma^2_\epsilon(x, \tilde{U}_\tau^\sharp)\partial_m^2\eta(\tilde{U}_\tau^\sharp) - \sigma^2_\epsilon(x, {U}_\tau^\sharp) \partial_m^2\eta({U}_\tau^\sharp)\right)\varphi(x)\,dx\,ds \to 0,
		\end{equation*}
        $\tilde{\mathbb{P}}$-almost surely by the dominated convergence theorem.
        
        \medskip

        \noindent \textbf{Estimate of $E_{6}$.} For $E_6$, we can take expectation of this stochastic integral and use the It\^{o} isometry along with the Cauchy-Schwarz inequality, to obtain
		\begin{equation}\label{E6calculation}
			\begin{split}
				\mathbb{E}|E_6|^2 & = \mathbb{E}\left|\int_{0}^{t}\int_{0}^{1}\left(\partial_m\eta(\tilde{U}_\tau^\sharp)\sigma_\epsilon(x, \tilde{U}_\tau^\sharp) - \partial_m\eta({U}_\tau^\sharp)\sigma_\epsilon(x, {U}_\tau^\sharp)\right)\varphi(x)\,dx\,dW^\sharp(s)\right|^2\\
				& = \mathbb{E}\int_{0}^{t} \int_{0}^{1}\Big|\left(\partial_m\eta(\tilde{U}_\tau^\sharp)\sigma_\epsilon(x, \tilde{U}_\tau^\sharp) - \partial_m\eta({U}_\tau^\sharp)\sigma_\epsilon(x, {U}_\tau^\sharp)\right)\varphi(x)\Big|^{2}\,dx\,ds.
			\end{split}
		\end{equation}
        As in the estimate of $E_{5}$, we can show that $\mathbb{E}|E_{6}|^{2} \to 0$ as $\tau \to 0$ by the dominated convergence theorem, so therefore, $E_{6} \to 0$ as $\tau \to 0$, $\tilde{\mathbb{P}}$-almost surely, along a subsequence in $\tau$.

        \medskip
        
		\noindent \textbf{Estimate of $E_{7}$, $E_8$, $E_9$.} To estimate $E_7$, we use Lemma \ref{L_infty_estimates}, Lemma \ref{lem_Linfty_eta_gradient_eta}, and Lemma \ref{limitingpositive}, to conclude that for every $\omega \in \tilde{\Omega}$, all the terms in the integrand, except the Hessian term, are uniformly bounded in $\tau$. However, by the uniform lower bound in \eqref{rhotaubounds} and the convergence in Proposition \ref{samelimit}, we estimate:
        \begin{equation*}
        \left|\int_{t}^{t_{n+1}} \int_{0}^{1} \epsilon \langle \nabla^{2}\eta(\bar{U}^{\sharp}_{\tau})\partial_x \bar{U}^{\sharp}_{\tau}, \partial_x \bar{U}^{\sharp}_{\tau} \rangle dx ds \right| \le C(\tilde\omega) \epsilon \int_{t}^{t_{n+1}} \|\bar{U}^{\sharp}_{\tau}(s)\|_{H^{1}(0, 1)}^{2} ds \le C(\tilde\omega) \epsilon \tau \|\bar{U}_{\tau}^{\sharp}\|_{L^{\infty}(0, T; H^{1}(0, 1))}^{2} \to 0.
        \end{equation*}
        Therefore, we have that:
		\begin{equation}
			\begin{split}
				|E_7 | &= \left(\frac{t-t_n}{\tau}\right)\left|\int_{t}^{t_{n+1}}\int_{0}^1 H(\bar{U}_\tau^\sharp)\partial_x\varphi - \alpha\bar{m}_\tau^\sharp \partial_m \eta(\bar{U}_\tau^\sharp)\varphi(x)
				+ \epsilon \eta(\bar{U}_\tau^\sharp)\partial_x^2\varphi\right.\\ 
                &\left.\qquad \qquad \qquad \qquad - \epsilon \langle \nabla^2\eta(\bar{U}_\tau^\sharp)\partial_x\bar{U}_\tau^\sharp, \partial_x\bar{U}_\tau^\sharp\rangle \varphi(x) \,dx\,ds \right|  \to 0, \qquad \text{ as } \tau \to 0.
			\end{split}
		\end{equation}
        To estimate $E_8$, we use Lemma \ref{lem_lipschitz} and the bounds in \eqref{Utaubound} and \eqref{littleutaubound} to obtain:
        \begin{equation}
            |E_8| = \frac{t_{n+1}-t}{\tau}\left(\frac{1}{2}\int_{t_n}^{t} \int_{0}^{1} \sigma^2_\epsilon(x, \tilde{U}_\tau^\sharp) \partial_m^2\eta(\tilde{U}_\tau^\sharp)\varphi(x)\,dx\,ds \right) \to 0, \qquad \text{ as } \tau \to 0.
        \end{equation}
        To estimate $E_9$, we use the It\^{o} Isometry as is done in $E_6$, and obtain:
        \begin{equation}
            \mathbb{E}|E_9|^2 = \left(\frac{t_{n+1}-t}{\tau}\right)^2\mathbb{E}\left(\int_{t_n}^{t}\int_{0}^{1}\partial_m\eta(\tilde{U}_\tau^\sharp)\sigma_\epsilon(x, \tilde{U}_\tau^\sharp)\varphi(x)\,dx\,dW^\sharp(s)\right)^2 \to 0, \qquad \text{ as } \tau \to 0.
        \end{equation}
        We conclude that $E_{7}, E_{8}, E_{9} \to 0$ as $\tau \to 0$, $\tilde{\mathbb{P}}$-almost surely, potentially along a subsequence in $\tau$.
        
        \medskip
        \noindent \textbf{Estimate of $E_{10}$.} Finally, we estimate $E_{10}$. 
		\begin{equation}\label{E10}
			\begin{split}
				|E_{10}| &=  \left|\int_{0}^{1}\mathrm{R}_2(\bar{U}_\tau^\sharp-\tilde{U}_\tau^\sharp)\varphi(x)\,dx\right| \\
                &=\mathbb{E}\left| \frac{1}{2}\int_{0}^{1}(\tilde{U}_\tau^\sharp - \bar{U}_\tau^\sharp)^T\left(\left(\frac{t-t_n}{\tau}\right)^2\nabla^2\eta(\mathbf{\xi}_1^\sharp)-\frac{t-t_n}{\tau}\nabla^2\eta(\mathbf{\xi}_2^\sharp)\right)(\tilde{U}_\tau^\sharp - \bar{U}_\tau^\sharp)\varphi(x)\,dx \right|,
			\end{split}
		\end{equation}
        where $\xi_1^\sharp, \xi_2^\sharp \in \mathbb{R}^2$ are two points on the line segment joining $\bar{U}_\tau^\sharp$ and $\tilde{U}_\tau^\sharp$. By \eqref{Utaubound}, for every $(t, x) \in Q_T$,
        \begin{equation*}
        \xi_{1}^\sharp(t, x), \xi_{2}^\sharp(t, x) \in [c(\tilde\omega), C] \times [-C, C],
        \end{equation*}
        independently of $\tau$. By Lemma \ref{lem_lipschitz} combined with \eqref{littleutaubound} and \eqref{rhotaubounds}, $\nabla^2\eta(\xi_1^\sharp)$, $\nabla^2 \eta(\xi_2^\sharp)$ are therefore bounded by some constant $C(\tilde{\omega})$ uniformly in $\tau$.  We can then apply H\"{o}lder's inequality on the integral in \eqref{E10}  to obtain
        \begin{align}
        \begin{split}
            |E_{10}| &\leq C\|\varphi\|_{L^\infty(0,1)}\left\|(\tilde{U}_\tau^\sharp - \bar{U}_\tau^\sharp)^T\left(\left(\frac{t-t_n}{\tau}\right)^2\nabla^2\eta(\mathbf{\xi}_1^\sharp)-\frac{t-t_n}{\tau}\nabla^2\eta(\mathbf{\xi}_2^\sharp)\right)\right\|_{L^\infty(0,1)} \left\|\bar{U}_\tau^\sharp(t) - \tilde{U}_\tau^\sharp(t)\right\|_{L^1(0,1)} \\
            &\le C(\tilde\omega) \|\varphi\|_{L^{\infty}(0, 1)} \|\bar{U}^{\sharp}_{\tau}(t) - \tilde{U}^{\sharp}_{\tau}(t)\|_{L^{1}(0, 1)}.
        \end{split}
        \end{align}
        Lastly, we apply Cauchy-Schwarz inequality on the term $\|\bar{U}_\tau^\sharp(t) - \tilde{U}_\tau^\sharp(t)\|_{L^1(0,1)}$, and use the difference estimates provided by Lemma \ref{lem_difference_estimates} to obtain
        \begin{equation}
            |E_{10}|\leq C(\tilde\omega) \|\bar{U}_\tau^\sharp(t) - \tilde{U}_\tau^\sharp(t)\|_{L^1(0,1)} \leq C(\tilde\omega) \|\bar{U}_\tau^\sharp(t) - \tilde{U}_\tau^\sharp(t)\|_{L^2(0,1)}.
        \end{equation}
        By Lemma \ref{lem_difference_estimates}, $|E_{10}| \to 0$, as $\tau \to 0$, $\tilde{P}$-almost surely, along a subsequence of $\tau$.
		
		\color{black}
		
		\medskip
        
        \noindent \textbf{The limiting entropy equality.} We have now shown that for a fixed $\epsilon>0$, $E_i \to 0$ for $i = 1,...,10$, $\tilde{\mathbb{P}}$-almost surely as $\tau \to 0$, up to a subsequence $\{\tau_n\}$ in $\tau$. We then proceed to show that $U_\epsilon$ satisfies the entropy equality \eqref{epsilonentropy}, by passing $\tau \to 0$ in the terms preceding $\sum_{i=1}^{10}E_i$ following $\tau$-level entropy equality, from \eqref{approx_entropy_full}:
		\begin{equation}\label{tauentropy}
			\begin{split}
				&\int_{0}^{1}\eta({U}_\tau^\sharp(t))\,dx= \int_{0}^{1}\eta(U_{\epsilon 0})\varphi(x)\,dx  + \int_{0}^{t}\int_{0}^1 H({U}_\tau^\sharp)\partial_x\varphi \,dx\,ds - \alpha\int_{0}^{t}\int_{0}^{1} {m}_\tau^\sharp \partial_m \eta({U}_\tau^\sharp)\varphi(x)\,dx\,ds \\
				 & \qquad + \epsilon\int_{0}^{t}\int_{0}^{1}\eta({U}_\tau^\sharp)\partial_x^2\varphi - \langle \nabla^2\eta({U}_\tau^\sharp)\partial_x{U}_\tau^\sharp, \partial_x{U}_\tau^\sharp\rangle \varphi(x) \,dx\,ds \\
				&\qquad + \frac{1}{2}\int_{0}^{t} \int_{0}^{1} \partial_m^2\eta({U}_\tau^\sharp)\sigma^2_\epsilon(x, {U}_\tau^\sharp) \varphi(x)\,dx\,ds + \int_{0}^{t}\int_{0}^{1}\partial_m\eta({U}_\tau^\sharp)\sigma_\epsilon(x, {U}_\tau^\sharp)\varphi(x)\,dx\,dW^\sharp(s) +\sum_{i=1}^{10} {E}_i.
			\end{split}
		\end{equation}
        To pass to the limit as $\tau \to 0$, note that ${U}_\tau^\sharp \to U_\epsilon$ in $C(0, T; H^1(0,1))$, $\tilde{\mathbb{P}}$-almost surely by Proposition \ref{samelimit}. Furthermore, the limit passage as $\tau \to 0$ in the terms involving nonlinear functionals such as $\eta$, $H$, $m\partial_m\eta$, $\sigma_\epsilon\partial_m\eta$, can be handled using the dominated convergence theorem and the convergence \eqref{tauconvergences}. For the Hessian term, we note that by Proposition \ref{limitingpositive}, for every $\tilde{\omega} \in \tilde{\Omega}$, ${\rho}_\tau^\sharp \geq c_\epsilon(\tilde{\omega})>0$, for every $\tau$, and $\rho_\epsilon \geq c_\epsilon(\tilde{\omega}) >0$. We can then carry out the exact same procedure as we did in estimating $E_4$ to conclude that for every $t>0$:
        \begin{equation*}
            \epsilon \int_0^t \int_0^1 \left(\langle \nabla^2\eta({U}_\tau^\sharp)\partial_x{U}_\tau^\sharp, \partial_x{U}_\tau^\sharp\rangle- \langle \nabla^2\eta(U_\epsilon)\partial_xU_\epsilon, \partial_x U_\epsilon\rangle\right) \varphi(x)   \,dx\,ds \to 0,
        \end{equation*}
        as $\tau \to 0$ along a subsequence, $\tilde{\mathbb{P}}$-almost surely. Finally, for the It\^{o} integral, by a computation similar to that for term $E_{6}$ in \eqref{E6calculation}, we can show that
        \begin{equation*}
        \mathbb{E} \int_{0}^{t} \left|\int_{0}^{1} \Big(\partial_{m}\eta(U^\sharp_{\tau}) \sigma_{\epsilon}(x, U^{\sharp}_{\tau}) - \partial_{m}\eta(U_{\epsilon})\sigma_{\epsilon}(x, U_{\epsilon})\Big) \varphi(x) \right|^{2} dt \to 0, \qquad \text{ as } \tau \to 0.
        \end{equation*}
        Combining this with the convergence $W^\sharp_\tau \to W^\sharp$ in $C(0, T; \mathbb{R})$ by Theorem \ref{pass_tau_0}, $\tilde{\mathbb{P}}$-almost surely, we can use Lemma 2.1 in \cite{bensoussan_stochastic_1995} to conclude that 
		\[\int_{0}^{t}\int_{0}^{1}\partial_m\eta({U}_\tau^\sharp)\sigma_\epsilon(x, {U}_\tau^\sharp)\varphi(x)\,dx\,dW^\sharp_\tau(s) \to \int_{0}^{t}\int_{0}^{1}\partial_m\eta(U_\epsilon)\sigma_\epsilon(x, U_\epsilon)\varphi(x)\,dx\,dW^\sharp(s), \;\; \text{ in probability},\]
        and hence, up to a subsequence $\{\tau_n\}$, this convergence also holds, $\tilde{\mathbb{P}}$-almost surely.
	
    Therefore, passing to the limit as $\tau \to 0$ along a subsequence in \eqref{tauentropy}, we conclude that the limiting process  $U_\epsilon$ satisfies the entropy balance equation for every test function $\varphi \in C^2_c(0,1)$, $\tilde{\mathbb{P}}$-almost surely: 
		\begin{equation}\label{entropylimiteps}
			\begin{split}
				\int_{0}^{1}\eta(U_\epsilon(t))\,dx= & \int_{0}^{1}\eta(U_{\epsilon 0})\varphi(x)\,dx  + \int_{0}^{t}\int_{0}^1 H(U_\epsilon)\partial_x\varphi \,dx\,ds - \alpha\int_{0}^{t}\int_{0}^{1} m_\epsilon \partial_m \eta(U_\epsilon)\varphi(x)\,dx\,ds \\
				+& \epsilon\int_{0}^{t}\int_{0}^{1}\eta(U_\epsilon)\partial_x^2\varphi\,dx\,ds - \epsilon\int_{0}^{t}\int_{0}^{1}\langle \nabla^2\eta(U_\epsilon)\partial_xU_\epsilon, \partial_xU_\epsilon\rangle \varphi(x) \,dx\,ds \\
				+& \frac{1}{2}\int_{0}^{t} \int_{0}^{1}  \partial_m^2\eta(U_\epsilon)\sigma^2_\epsilon(x, U_\epsilon)\varphi(x)\,dx\,ds + \int_{0}^{t}\int_{0}^{1}\partial_m\eta(U_\epsilon)\sigma_\epsilon(x, U_\epsilon)\varphi(x)\,dx\,dW^\sharp(s).
			\end{split}
		\end{equation}
		Taking $g(\xi) = 1$, and $g(\xi) = \xi$ respectively in the formula \eqref{entropy_pair_formula} for the entropy function $\eta$ and substituting into the entropy equality \eqref{entropylimiteps}, we see that $U_{\epsilon}$ satisfies the weak formulation \eqref{epsweakthm} of the $\epsilon$ level approximate problem also.
	\end{proof}

	\subsection{Uniqueness of pathwise solution}\label{sec_uniqueness_pathwise_U_epsilon}
    Now that we have constructed a martingale solution to the regularized problem \eqref{regularized_problem1}--\eqref{regularized_problem2}, namely, a solution satisfying the weak formulation \eqref{epsweakthm} potentially on a different probability space $(\tilde\Omega, \tilde{\mathcal{F}}, \tilde{\mathbb{P}})$, we want to show that the solution is actually pathwise (in particular, it can be defined on the original probability space without transferring to a different probability space). To prove that such solution is pathwise (Theorem \ref{thm_existence_pathwise_U_epsilon}), it suffices to show that pathwise solutions to the approximate $\epsilon$-level system \eqref{regularized_problem1}--\eqref{regularized_problem2} are unique, which then proves Theorem \ref{thm_existence_pathwise_U_epsilon}. This is because pathwise uniqueness will allow us to show that the approximate solutions $U_{\tau}$ defined on the original probability space in \eqref{U_defn} converge almost surely along a subsequence to a limiting solution on the same initial probability space, via a standard Gy\"{o}ngy-Krylov argument \cite{MR4491500}. 
    
    The pathwise uniqueness of solutions to \eqref{regularized_problem1}--\eqref{regularized_problem2} comes from the uniqueness of solution of the subproblems, and thus the uniqueness of the sequence of approximate solutions $\{{U}_\tau\}$, and the proof can be completed through a standard doubling of variables technique. We state the theorem below without proof, and refer to Theorem 3.3 of \cite{berthelin_stochastic_2019} for the proof and technical details. 
	
	\begin{thm}(Uniqueness of a pathwise solution to the $\epsilon$-level problem)\label{thm_uniqueness_pathwise}
		Let $U_{\epsilon0} \in H^{2}(0, 1)$ be given such that $\rho_{\epsilon 0 }\geq c_0$ for some constant $c_0 > 0$, and suppose that the noise coefficient $\sigma_\epsilon(\cdot, U) \in C^1(0,1)$ satisfies \eqref{noise_assumption}. Then, there is at most one pathwise solution (in the sense of Definition \ref{defn_U_epsilon}) to the regularized system (\ref{regularized_problem1}) with artificial viscosity.
	\end{thm}
    
    Finally we are ready to prove the main existence result, Theorem \ref{thm_existence_pathwise_U_epsilon}, of (unique) pathwise solutions to the regularized problem \eqref{regularized_problem1}--\eqref{regularized_problem2}.
	\begin{proof}[Proof of Theorem \ref{thm_existence_pathwise_U_epsilon}] 
        { We consider the given filtration $(\mathcal{F}_t)_{t\geq 0}$ fixed in Theorem \ref{thm_existence_pathwise_U_epsilon} and restrict it on the interval $[0, T]$ to obtain $(\mathcal{F}_t)_{0\leq t\leq T}$. Then by Theorem \ref{pass_tau_0} and Theorem \ref{thm_uniqueness_pathwise}, we can invoke a standard Gy\"{o}ngy-Krylov argument to conclude that there exists a unique pathwise solution $U_\epsilon^T$ on the time interval $[0, T]$, defined on the given filtered probability space $(\Omega, \mathcal{F}, (\mathcal{F}_t)_{0\leq t\leq T}, \mathbb{P})$, for each $T > 0$. We then define the solution, $U_\epsilon$, on the whole time interval $[0, \infty)$, in the following way:
        \[U_\epsilon := U_\epsilon^T, \;\; \text{ if } 0\leq t\leq T.\]
        $U_\epsilon$ is well-defined in this way because $U_\epsilon^T$ is pathwise unique; namely, for $T_{1}, T_{2} > 0$, the corresponding solutions $U^{T_{1}}$ and $U^{T_{2}}$ on the time intervals $[0, T_1]$ and $[0, T_2]$ respectively, will agree on the overlap $[0, \min(T_1, T_2)]$. Since $U_\epsilon^T \in C(0, T; H^2(0,1))$ for every $T>0$, $\mathbb{P}$-almost surely, we have that $U_\epsilon \in C_{loc}([0, \infty); H^2(0,1))$, $\mathbb{P}$-almost surely.
        }

        By Theorem \ref{thm_identify_the_limit_tau}, $U_\epsilon$ satisfies the entropy equality \eqref{epsilonentropy} and thus the weak formulation \eqref{epsilonweakformulation}. Next, we verify that the boundary condition holds. 
        By Proposition \ref{prop_existence_split}, we have that $\partial_x {\rho}_\tau^\sharp(t, 0) = \partial_x {\rho}_\tau^\sharp(t, 1) = 0$, and ${m}_\tau^\sharp(t, 0) = {m}_\tau^\sharp(t, 1) = 0$, $\tilde{\mathbb{P}}$-almost surely. Recall that $\partial_x {\rho}_\tau^\sharp(t) \rightharpoonup \partial_x \rho_\epsilon(t)$ in $H^1(0, 1)$,  ${m}_\tau^\sharp \to m_\epsilon$ in $H^1(0,1)$, for every $t\geq 0$, $\tilde{\mathbb{P}}$ -almost surely. Then, by the trace theorem, there exists a bounded linear operator $T_\rho: H^1(0, 1) \to \mathbb{R}^2$ and $T_m: H^1(0, 1) \to \mathbb{R}^2$, such that $T_\rho(\partial_x {\rho}_\tau^\sharp, (t, \cdot)) = (0,0)$, $T_m({m}_\tau^\sharp(t, \cdot)) =  (0,0)$ for every $t\geq 0$, $\tilde{\mathbb{P}}$-almost surely, and we can define $\partial_x \rho_\epsilon(t, i) := \lim_{\tau \to 0} T_\rho(\partial_x {\rho}_\tau^\sharp (t, \cdot))(i) = 0$, $m_\epsilon(t, i) := \lim_{\tau \to 0} T_m({m}_\tau^\sharp (t, \cdot))( i) = 0, i = \{0,1\}$ for all $t \geq 0$, $\omega \in \tilde{\Omega}$.
        
        Then, we prove the first item in Theorem \ref{thm_existence_pathwise_U_epsilon}. By Lemma \ref{L_infty_estimates} and equivalence of law of ${U}_\tau$ and ${U}_\tau^\sharp$ due to Theorem \ref{pass_tau_0}, $\|{U}^\sharp_\tau\|_{L^\infty((0,\infty)\times(0,1))} \leq C$,
        for some $C = C(\|U_0\|_{L^\infty}, \gamma)$, that is independent of both $\tau$ and $\epsilon$, $\tilde{\mathbb{P}}$-almost surely. Then, by the lower semi-continuity of norm, $\|U_\epsilon\|_{L^\infty((0,\infty)\times(0,1))} \leq \liminf_{\tau\to 0}\|{U}^\sharp_\tau\|_{L^\infty((0,\infty)\times(0,1))} \leq C(\|U_0\|_{L^\infty}, \gamma)$, $\tilde{\mathbb{P}}$-almost surely. Lastly, the second item in Theorem \ref{thm_existence_pathwise_U_epsilon} is proven in Proposition \ref{limitingpositive}.
    \end{proof}

    \if 1 = 0
	
	\subsection{Other estimates}
    {\color{blue}{Are these lemmas ever used? I think Proposition 3.4 is unnecessary, and I'm not sure about Lemma 3.3.}}{\color{orange} Lemma 3.3 is needed for passing $\epsilon \to 0$ (even though I think it can be stated just in terms of $U_\epsilon$).}

    \fi
	
	\section{Existence of martingale $L^\infty$ weak entropy solution via $\epsilon \to 0$ limit}\label{sec_existence_martingale_soln}
    
	In this section, our goal is to pass $\epsilon \to 0$ in the approximate solutions $U_{\epsilon}$ in the regularized problem with artificial viscosity (Theorem \ref{thm_existence_pathwise_U_epsilon}), and hence prove the main result, Theorem \ref{main_theorem}, on existence of $L^{\infty}({(0, \infty)\times (0,1)})$ martingale weak solutions to the initial problem \eqref{problem} without artificial viscosity. To do this, we will use a Skorohod representation theorem argument: we will obtain tightness of laws of $\{U_\epsilon\}$ in the space {$C_{w,loc}([0, \infty); L^2(0, 1))\cap C_{loc}([0, \infty); H^{-3}(0, 1))$}, and then obtain a sequence of random variables that converges almost surely in this space using the Skorohod representation theorem. Then, we use the theory of Young measure and the standard theory of compensated compactness and reduction of Young measure to obtain strong convergence as $\epsilon \to 0$ in the nonlinear terms in the entropy inequality \eqref{entropy_ineq}, to conclude that the limiting solution satisfies the weak formulation for the original problem.
 
    \subsection{Compactness of $U_\epsilon$}
    
    To begin with, we prove the following lemma which gives uniform in $\epsilon$ bounds in the space $L_\omega^2C^\alpha_t H^{-2}_x$. Recall that many of the past uniform bounds that we obtained are dependent on $\epsilon$: namely, the spatial regularity we obtain from Theorem \ref{regularity_of_Uhat} depends on $\epsilon$, and the $L^\infty_{t,x}$ bound given by Theorem \ref{thm_existence_pathwise_U_epsilon} is independent of $\epsilon$. However, we can show a spatial regularity bound on the approximate solutions $U_{\epsilon}$ in $H^{-2}_x$ that is independent of $\epsilon$, which is what we accomplish in the following lemma. The proof of the following lemma largely follows from Proposition 3.25 in \cite{berthelin_stochastic_2019}.
    \begin{lem}[Uniform Bound in $L^2_\omega C^\beta_t H^{-2}_x$]\label{holder_space_bound_B}
         Let $U_{\epsilon 0}$ be defined by \eqref{U_epsilon0}, and suppose that the noise satisfies the assumption \eqref{noise_assumption}. Let $U_\epsilon$ be the pathwise bounded solution to \eqref{regularized_problem1}--\eqref{regularized_problem2} given by Theorem \ref{thm_existence_pathwise_U_epsilon}.
		Then, $U_\epsilon(t)$ as a process in time taking values in $H^{-2}(0, 1)$, has a $\beta$-H\"{o}lder continuous modification almost surely, for $\beta\in [0, \frac{1}{4})$. Moreover,  $\mathbb{E}\|U_\epsilon\|^2_{C^{\beta}_{loc}(0, \infty; H^{-2}(0,1))}\leq C$, for some $C=C(\gamma, \alpha, A_0, \|U_{0}\|_{L^\infty})$, that we emphasize is crucially independent of $\epsilon$. 
    \end{lem}
    \begin{proof}
        Consider a deterministic test function $\varphi \in H^{2}((0,1); \mathbb{R}^2)$, and note that $U_\epsilon$ satisfies the following weak formulation, for every $t>0$, $\mathbb{P}$-almost surely:
    \begin{equation}
			\begin{split}
				\langle U_\epsilon(t), \varphi \rangle &=  \langle U_{\epsilon0}, \varphi \rangle + \int_{0}^{t} \langle  F(U_\epsilon(s)), \partial_x\varphi\rangle\,ds + \epsilon\int_0^t \langle U_\epsilon(s), \partial_x^2\varphi \rangle \,ds - \int_{0}^{t} \langle G(U_\epsilon(s)), \varphi\rangle \,ds \\& + \int_{0}^{t}\langle \Phi_\epsilon(x, U_\epsilon(s)), \varphi\rangle\,dW(s).
			\end{split}
		\end{equation}
To show that $U_{\epsilon}(t)$ is H\"{o}lder-continuous almost surely, we use the Kolmogorov continuity criterion, which requires that we take $0\leq s< t\leq T$ for any $T>0$, and estimate the following time increment:
\begin{equation}
			\begin{split}
				\mathbb{E}\left|\langle U_\epsilon(t)-U_\epsilon(s), \varphi \rangle\right|^4 &\leq  \mathbb{E}\left| \int_{s}^{t} \langle  F(U_\epsilon(s')), \partial_x\varphi\rangle\,ds' + \epsilon\int_s^t \langle U_\epsilon(s'), \partial_x^2\varphi \rangle \,ds' - \int_{s}^{t} \langle G(U_\epsilon(s')), \varphi\rangle \,ds'\right|^4 \\& + \mathbb{E}\left|\int_{s}^{t}\langle \Phi_\epsilon(x, U_\epsilon(s')), \varphi\rangle\,dW(s')\right|^4 =: I_1 + I_2.
			\end{split}
		\end{equation}
Since $\|U_\epsilon\|_{L^\infty((0,\infty)\times (0,1))}$ is bounded uniformly in $\epsilon$, $\mathbb{P}$-almost surely, due to Theorem \ref{thm_existence_pathwise_U_epsilon}, we get
\begin{equation}
    I_1 \leq C(\gamma, \alpha, \|U_0\|_{L^\infty}) |t-s|^4.
\end{equation}
To estimate the stochastic integral $I_2$, we use the Burkholder-Davis-Gundy inequality, the Lipschitz continuity of $\sigma_\epsilon$, namely $\sigma_\epsilon(\cdot, \cdot, m)^2 \leq A_0 m^2$ as in \eqref{noise_assumption}, and Lemma \ref{L_infty_estimates} to conclude that:
\begin{equation}
    I_2 \leq \mathbb{E}\left|\int_s^t \sigma_\epsilon(x, U_\epsilon(s'))^2 \,ds'\right|^2 \leq A_0 \mathbb{E}\left|\int_s^t m_\epsilon^2(s')\,ds'\right|^2 \leq C(\gamma, \alpha, A_0, \|U_0\|_{L^\infty}) |t-s|^2.
\end{equation}
Together we have
    \begin{equation}
        \mathbb{E}|\langle U_\epsilon(t) - U_\epsilon(s), \varphi \rangle|^4 \leq C(\gamma, \alpha, A_0, \|U_0\|_{L^\infty}) |t-s|^2. 
    \end{equation}
    Then by the Kolmogorov continuity criterion (see Theorem 3.5 and Theorem 5.22 in \cite{da_prato_stochastic_2014}), we have that $U_\epsilon(t)$ is a $H^{-2}(0,1)$-valued stochastic process with $\beta$-H\"{o}lder continuous sample paths almost surely, for $\beta \in [0, \frac{1}{4})$, and we have the associated estimate:
    \begin{equation}
    \mathbb{E}\|U_\epsilon\|_{C^\beta(0, T; H^{-2}(0,1))} \leq C(\gamma, \alpha, A_0, \|U_0\|_{L^\infty}, T).
    \end{equation}
    Since the calculation holds for every $T>0$, we can also conclude that $U_\epsilon(t)$ is locally (in time) $\beta$-H\"{o}lder continuous on $(0, \infty)$, i.e. $U_\epsilon \in C^\beta_{loc}(0, \infty; H^{-2}(0,1))$ almost surely, for $\beta \in [0, \frac{1}{4})$.
    \end{proof}
    {We then use the above lemma to obtain tightness of laws of the sequence $\{U_\epsilon\}_\epsilon$ in the artificial viscosity parameter $\epsilon$, in the phase space 
    \begin{equation}\label{phaseeps}
    \mathcal{X} := C_{w,loc}([0, \infty); L^2(0,1)) \cap C_{loc}([0, \infty); H^{-3}(0, 1)).
    \end{equation}
    Since $U_{\epsilon}$ takes values in $\mathcal{X}$ almost surely, for each $\epsilon > 0$, we can denote the law of $U_{\epsilon}$ on $\mathcal{X}$ by the probability measure $\mu_{\epsilon}$ defined on $\mathcal{X}$.
    
    \begin{thm}[Tightness of laws of $\{U_\epsilon\}_\epsilon$] \label{thm_tightness_Uepsilon}
    Let $U_{\epsilon0}$ be defined by \eqref{U_epsilon0} and suppose that the noise coefficient $\sigma_\epsilon$ satisfies \eqref{noise_assumption}. Let $U_\epsilon$ be the pathwise bounded solution to \eqref{regularized_problem1}--\eqref{regularized_problem2}, and let $\{\mu_\epsilon\}_\epsilon$ denote the family of laws of the corresponding random variables $\{U_\epsilon\}_\epsilon$ on the phase space $\mathcal{X}$, defined in \eqref{phaseeps}. Then $\{\mu_\epsilon\}_\epsilon$ is uniformly tight in $\mathcal{X}$.
    \end{thm}
    \begin{proof}
    By Lemma \ref{holder_space_bound_B}, for each $T > 0$:
    \begin{equation}\label{Cbetaest}
    \mathbb{E}\|U_\epsilon\|_{C^\beta(0, T; H^{-2}(0, 1))}^2 \leq C(\|U_0\|_{L^\infty}, \gamma, \alpha, A_0, T),
    \end{equation}
    where we emphasize that the constant on the right hand side is independent of $\epsilon$. By Theorem \ref{thm_existence_pathwise_U_epsilon}, we have that $\|U_\epsilon\|_{L^\infty((0, \infty)\times (0,1))} \leq C$ for some $C = C(\|U_0\|_{L^\infty}, \gamma)$ independent of $\epsilon$, $\mathbb{P}$-almost surely. Therefore, 
    \begin{equation}\label{L2est}
    \mathbb{E}\|U_\epsilon\|^2_{L^\infty(0, \infty; L^2(0,1))} \leq C(\|U_0\|_{L^\infty}, \gamma),
    \end{equation}
    where the constant on the right hand side is independent of $\epsilon > 0$. Recall the compact embeddings:
    \begin{equation*}
    L^\infty(0, T; L^2(0, 1)) \cap C^\beta(0, T; H^{-2}(0, 1)) \subset \subset C_{w}(0, T; L^2(0, 1)),
    \end{equation*}
    \begin{equation*}
    C^{\beta}(0, T; H^{-2}(0, 1)) \subset \subset C(0, T; H^{-3}(0, 1)).
    \end{equation*}
    To show the tightness result, consider an arbitrary $\delta > 0$. By \eqref{Cbetaest}, for each positive integer $N$, we can choose a corresponding $c_{\delta, N}$ such that
    \begin{equation*}
    \mathbb{P}(\|U_{\epsilon}\|_{C^{\beta}(0, N; H^{-2}(0, 1))} > c_{\delta, N}) \le \delta \cdot 2^{-N + 1},
    \end{equation*}
    and by \eqref{L2est}, there exists $C_{\delta}$ sufficiently large such that
    \begin{equation*}
    \mathbb{P}(\|U_{\epsilon}\|_{L^{\infty}([0, \infty); L^{2}(0, 1))} > C_{\delta}) \le \delta/2.
    \end{equation*}
    Then, by a diagonalization argument (where we diagonalize with respect to the positive integer parameter $N$), we have that the set:
    \begin{equation*}
    K_{\delta} := \{\|U_{\epsilon}\|_{L^{\infty}([0, \infty); L^{2}(0, 1))} \le C_{\delta}\} \cap \Big(\bigcap_{N \in \mathbb{N}} \{\|U_{\epsilon}\|_{C^{\beta}(0, N; H^{-2}(0, 1))} \le c_{\delta, N}\}\Big)
    \end{equation*}
    is a compact subset of $\mathcal{X} := C_{w, loc}([0, \infty); L^{2}(0, 1)) \cap C_{loc}([0, \infty); H^{-2}(0, 1))$, and furthermore, we can verify that $\mu_{\epsilon}(K_{\delta}^{c}) \le \delta$.
    We thus have that $\{\mu_\epsilon\}_\epsilon$ is uniformly tight in $C_{w,loc}([0, \infty); L^2(0, 1)) \cap C_{loc}([0, \infty); H^{-3}(0, 1))$.
    \end{proof}

    \subsection{Compactness of Young measures}
    By the Skorohod representation theorem (as in Section \ref{skorohodsection_tau}), the tightness result in Theorem \ref{thm_tightness_Uepsilon} would imply almost sure convergence in the space $\mathcal{X} := C_{w, loc}([0, \infty); L^{2}(0, 1)) \cap C_{loc}([0, \infty); H^{-3}(0, 1))$ on a potentially different probability space. However, this function space is too weak to pass to the limit under nonlinear terms. Hence, we need to utilize the standard theory of Young measure and compensated compactness to pass to the limit in nonlinear functions such as $\eta(U_{\epsilon})$ and $H(U_{\epsilon})$ for entropy-flux pairs $(\eta, H)$, which appear in the approximate $\epsilon$-level entropy formulation. 
    
    This compensated compactness theory was first developed by \cite{MR584398}, and it was generalized to the stochastic setting by Feng and Nualart in \cite{feng_stochastic_2008}, and was used in \cite{berthelin_stochastic_2019} to obtain strong convergence to the limiting solution. We thus only provide an outline in the following exposition and refer interested readers to Section 5 of \cite{berthelin_stochastic_2019}.
    }

    To simplify the notation, we denote the time-space domain as $Q$, and the state variable space for density and fluid velocity as $E$, as follows:
    \[Q:= [0, \infty) \times [0,1], \qquad E := (0,\infty) \times \mathbb{R}.\]
    Instead of using the fluid density and momentum $(\rho, m)$ as the state variables, it is standard in this compensated compactness argument to instead use the fluid density and \textit{velocity} $(\rho, u)$ as state variables. This is because the entropy and entropy flux functions $(\eta, H)$ are polynomial-type functions in $\rho$ and $u$, which helps simplify the compensated compactness arguments. If one writes the entropy and entropy flux functions in terms of $\rho$ and $m$ instead, $\rho$ can appear in the denominator, which causes issues potentially when there is vacuum.
    
For each artificial viscosity parameter $\epsilon > 0$, consider the approximate solution $U_{\epsilon} = (\rho_{\epsilon}, m_{\epsilon})$ for the fluid density and momentum, and recall that the approximate fluid velocity is given by $u_{\epsilon} := m_{\epsilon}/\rho_{\epsilon}$. {For every $(x, t)\in Q$, we define a sequence of random slicing measures $\{\nu_\epsilon^{x, t}\}_\epsilon$ corresponding to the approximate solutions in $\epsilon$, in the following way:
    \begin{equation}
        \langle \nu_\epsilon^{x, t}, \varphi\rangle: = \langle \delta_{\rho_\epsilon(x, t), u_\epsilon(x, t)}, \varphi\rangle = \varphi(\rho_\epsilon(x, t), u_\epsilon(x, t)), \;\; \forall \varphi \in C_b(E).
    \end{equation}
    We then define the sequence of random Young measures associated with the approximate solutions $\{U_\epsilon\}_\epsilon$ in $\epsilon$ as 
	\begin{equation}\label{nu_defn}
		\begin{split}
			\nu_\epsilon = \nu_\epsilon^{x, t} \rtimes \lambda , \; \text{ where $\lambda$ is the Lebesgue measure on $Q$.}\\
		\end{split}
	\end{equation} 
We define the space of Young measures $\mathcal{Y}$ as follows:
\begin{equation}\label{defn_spaceYoungmeas}
        \mathcal{Y} := \{\nu \in \mathcal{P}_1(Q \times E); \pi^*\nu = \lambda\},
    \end{equation}
where $\mathcal{P}_1(Q\times E)$ is the set of probability measures on $Q \times E$, $\pi$ is the projection map $Q \times E \to Q$, $\lambda$ is the Lebesgue measure on $Q$, and $\pi^*\nu$ denotes the push forward measure of $\nu$ under $\pi$. The space $\mathcal{Y}$ is endowed with the (vague) weak-star topology, i.e. $\nu_n \to \nu $ in $\mathcal{Y}$ 
if for every $\varphi \in C_b(E)$, and $\psi \in C_c^\infty(Q)$:
\begin{equation*}
\int_{Q}\psi(x, t)\int_{E} \varphi(p)\,d\nu_n^{x, t}(p)\,d\lambda(x, t) \to \int_{Q}\psi(x, t)\int_{E} \varphi(p)\,d\nu^{x, t}(p)\,d\lambda(x, t).
\end{equation*}

Next, we refer to Proposition 4.4 in \cite{berthelin_stochastic_2019} for the following result on the tightness of Young measures, which can be easily extended to infinite time domain by a diagonalization argument using compactly supported test functions on $Q$.
\begin{thm}[Tightness of $\{\nu_\epsilon\}_\epsilon$]
Let $\{\nu_\epsilon\}_\epsilon$ be defined by \eqref{nu_defn}. $\{\nu_\epsilon\}_\epsilon$ is uniformly tight in the space of Young measures, $\mathcal{Y}$, which is defined by \eqref{defn_spaceYoungmeas}.
\end{thm}}
    \subsection{Passage of $\epsilon \to 0$ and reduction of Young measure}
    
     Using this tightness result, we can now pass $\epsilon \to 0$ in the approximate solutions $U_{\epsilon}$ and their corresponding Young measures $\nu_{\epsilon}$, using the Skorohod representation theorem in the following path spaces:
    \[
	\begin{split}
	    &\mathcal{X}_U := C_{w,loc}([0, \infty); L^2(0,1)) \cap C_{loc}([0, \infty); H^{-3}(0,1)) , \\
        &\mathcal{Y} := \{\nu \in \mathcal{P}_1(Q \times \mathbb{R}_+\times \mathbb{R}); \pi^*\nu = \lambda\}, \qquad \text{ see also \eqref{defn_spaceYoungmeas}}\\
        &\mathcal{X}_W := C_{loc}([0, \infty); \mathbb{R}). \\
	\end{split} \] 

	\begin{thm}[Pass $\epsilon \to 0$]\label{pass_epsilon_0} Assume that the initial condition $U_{\epsilon 0}$ satisfies \eqref{init_cond}, and the noise coefficient $\sigma_\epsilon$ satisfies \eqref{noise_assumption}. Let $U_\epsilon$ be the process obtained in Theorem \ref{pass_tau_0} and $\nu_\epsilon$ be the random Young measure defined as in \eqref{nu_defn}.
    Then there exists a probability space $(\bar{\Omega}, \bar{\mathcal{F}}, \bar{\mathbb{P}})$, an enumerable set $\Lambda$, a sequence of $\mathcal{X}_U\times  \mathcal{Y}  \times  \mathcal{X}_W$-valued random variables $(\bar{U}_\epsilon, \bar{\nu}_\epsilon, \bar{W}_\epsilon)_{\epsilon\in \Lambda}$, and $(\bar{U}, \bar{\nu},\bar W)$ such that 
		\begin{itemize}
			\item For any $\epsilon\in\Lambda$, the law of $(\bar{U}_\epsilon, \bar{\nu}_\epsilon, \bar{W}_\epsilon)$ coincides with that of $(U_\epsilon, \nu_\epsilon, W)$.
			\item $(\bar{U}_\epsilon, \bar{\nu}_\epsilon,\bar W_\epsilon)_{\epsilon\in\Lambda}$ converges to $(\bar{U}, \bar{\nu},\bar W)$ $\bar{\mathbb{P}}$-almost surely in the topology of $\mathcal{X}_U \times  \mathcal{Y}\times \mathcal{X}_W$.
		\end{itemize}
        Moreover, $\bar{U} \in C_{loc}([0, \infty); H^{-2}(0,1))$, $\bar{\mathbb{P}}$-almost surely.   
	\end{thm}
	\begin{rem}
	    Let $\bar{\mathcal{F}}_t'$ be the $\sigma$-algebra generated by the random variables $\bar{U}(s), \bar{\nu}_{x,s}\rtimes \lambda, \bar{W}(s)$, for all $s \leq t$. Define
        \[\mathcal{N}:= \{\mathcal{A}\in \bar{\mathcal{F}}|\bar{\mathbb{P}}(\mathcal{A}) = 0\}, \;\; \bar{\mathcal{F}}_t := \cap_{s \geq t}\sigma(\bar{\mathcal{F}}_s'\cup \mathcal{N}).\]
        Then, $(\bar{\mathcal{F}}_t)_{t\geq 0}$ defines a complete right continuous filtration for which $(\bar{U}, \bar{\nu}, \bar{W})$ is an adapted process. The construction of a complete right continuous filtration $(\bar{\mathcal{F}^\epsilon_t})_{t\geq 0}$ for $(\bar{U}_\epsilon, \bar{\nu}_\epsilon, \bar W_\epsilon)$ can be done in a similar fashion.
	\end{rem}
    \begin{rem}
        Since any closed ball in $C^\beta_{loc}([0, \infty); H^{-2}(0,1))$ is also closed in $\mathcal{X}_U$, it follows that $\bar{U} \in C_{loc}([0, \infty); H^{-2}(0,1))$, $\bar{\mathbb{P}}$-almost surely. 
    \end{rem}

    These convergences as $\epsilon \to 0$ given by the Skorohod representation theorem are sufficiently strong to pass to the limit in all of the resulting terms of the approximate entropy formulation in $\epsilon$, and hence obtain a limiting martingale $L^{\infty}$ weak entropy solution in the sense of Definition \ref{martingale_soln_defn}.
    
    \begin{prop}\label{prop_barUepsilon}
        Let $(\bar{U}_\epsilon,\bar W_\epsilon)$ and $ (\bar U,\bar W)$ be the processes defined on the probability space $(\bar{\Omega}, \bar{\mathcal{F}}, \bar{\mathbb{P}})$ constructed in Theorem \ref{pass_epsilon_0}. Then, 
        \begin{enumerate} 
        \item On the new filtered probability space $(\bar{\Omega}, \bar{\mathcal{F}}, (\bar{\mathcal{F}^\epsilon_t})_{t\geq 0}, \bar{\mathbb{P}})$, $\bar W_\epsilon$ is an $(\bar{\mathcal{F}^\epsilon_t})_{t\geq 0}$-Wiener process and, $\bar{U}_\epsilon$ solves \eqref{regularized_problem1}--\eqref{regularized_problem2} in the sense of Definition \ref{defn_U_epsilon}. 
        \item There exists a positive deterministic constant $C=C(U_0, \gamma)$, independent of $\epsilon$, such that $\|\bar{U}_\epsilon\|_{L^\infty((0, \infty)\times(0,1))}\leq C$, $\bar{\mathbb{P}}$-almost surely.
        \item The new filtered probability space $(\bar{\Omega}, \bar{\mathcal{F}}, (\bar{\mathcal{F}_t})_{t\geq 0}, \bar{\mathbb{P}})$, the $(\bar{\mathcal{F}_t})_{t\geq 0}$-Wiener process $\bar W$ and, the process $\bar{U}$ form a martingale $L^\infty$ weak entropy solution to \eqref{regularized_problem1}--\eqref{regularized_problem2} in the sense of Definition \ref{martingale_soln_defn}. 
        \end{enumerate}
    \end{prop}
    \begin{proof}
        The first and the second items in this proposition are consequences of the equivalence of laws of $\bar{U}_\epsilon$ and $U_\epsilon$ given by Theorem \ref{pass_epsilon_0}. To prove the first point, we additionally refer to the argument presented in Theorem 2.9.1 in \cite{MR3791804}.

	To prove the third point in the statement of this proposition, we apply stochastic version of the standard theory of compensated compactness, as developed in \cite{berthelin_stochastic_2019}, and the procedure of reducing Young measures, to show that $\bar{\nu}^{x, t}$ is either $\delta_{\bar{\rho}(x, t), \bar{u}(x, t)}$, the Dirac mass centered at $(\bar{\rho}(x, t), \bar{u}(x, t))$, or is supported only in the vacuum region $\{\bar{\rho}=0\}$. Then we can show that, outside of the vacuum region, $\bar{U}$ satisfies the entropy inequality \eqref{entropy_ineq}, $\bar{\mathbb{P}}$-almost surely. We omit the details and refer the reader to the exposition in  \cite{berthelin_stochastic_2019, feng_stochastic_2008}.
	    \end{proof}
        This concludes the proof of our first main result, Theorem \ref{main_theorem}, on the existence of martingale $L^{\infty}$ weak entropy solutions to \eqref{problem}. 

	\section{Long-time behavior of martingale $L^\infty$ weak entropy solutions}\label{sec_long_time_behavior}
   In the previous section, we proved the existence of a martingale $L^{\infty}$ weak entropy  solution to \eqref{problem}. In what follows, we will study the long-time asymptotics of $U=(\rho, m)$, the solution to \eqref{problem} constructed in Theorem \ref{main_theorem}. More precisely, the goal of this section is to prove Theorem \ref{thm_limit_longt}, which states that $\displaystyle \int_0^1 \left[\left(\rho(t) - \rho_*\right)^2 + m^2(t)\right] \,dx \to 0$ as $t \to \infty$ almost surely, where $\displaystyle \rho_* := \int_0^1 \rho_0(x)\,dx$ is the total initial mass of the system. 

We recall that the existence proof of Theorem \ref{main_theorem} involved a two-layer approximation given by a time-splitting scheme and a viscous regularization. We recall from Theorem \ref{thm_existence_pathwise_U_epsilon} that for any $\epsilon>0$ there exists a pathwise bounded solution $U_\epsilon=(\rho_\epsilon, m_\epsilon)$ to the parabolic approximation \eqref{regularized_problem1}--\eqref{regularized_problem2},  for the regularized initial data {$(\rho_{\epsilon 0}, m_{\epsilon0})$ defined in \eqref{U_epsilon0}. For this initial data, we denote the total initial mass by $\displaystyle\rho_{\epsilon *} := \int_0^1\rho_{\epsilon0}\,dx$.
 Note that the approximate solution $U_\epsilon$ satisfies the entropy {\it balance} \eqref{epsilonentropy}, whereas the limiting process $U$, obtained by passing $\epsilon\to 0$, only satisfies an entropy {\it inequality} \eqref{entropy_ineq}.
This means that, 
several tools essential to our stochastic analysis, such as the It\^{o} formula, which can only be applied to the entropy balance equation, are available exclusively at the $\epsilon-$approximation level. 
Hence to prove Theorem \ref{thm_limit_longt}, we seek to prove an analogous decay estimate first at the level of $\epsilon$ approximation and then recover the desired result by passing $\epsilon\to 0$. In other words, we first show that 
there exist a set ${\Omega}_0 \subset {\Omega}$ with $\mathbb{P}({\Omega}_0)=1$, and a constant $r = r(\alpha, \gamma, A_0, M_1)>0$ such that for every $\omega\in {\Omega}_0$, there exists $C_0 = C_0(\omega, \alpha, \gamma, A_0, \|U_0\|_{L^\infty})>0$ such that
\begin{equation}\label{epsilondecay}
        \int_0^1\left[(\rho_\epsilon(t) - \rho_{\epsilon *})^2 + m_\epsilon(t)^2\right]\,dx\leq C_0 e^{-rt},\quad\text{for every $t>0$}.
    \end{equation}
    This result will be established in Section \ref{sec_a.s._convergence} (see in particular \eqref{ineq_a.s._expdecay_bound}). 

To prove \eqref{epsilondecay}, we first show in Section \ref{sec_decay_in_expectation} (see Theorem \ref{thm_approx_longt}), that the entities $\displaystyle\|\rho_\epsilon(t) - \rho_{\epsilon*}\|_{L^2(0,1)} $ and $ \displaystyle\|m_\epsilon(t)\|_{L^2(0,1)}$ converge to 0 as $t\to \infty$, in \textbf{$L^2({\Omega})$}. We then strengthen these second-moment estimates by deriving decay rates for the corresponding fourth moments in Lemma \ref{lem_2ndmoment_eta*_decay}.  
We emphasize that these estimates are obtained {\it in expectation} and are therefore insufficient on their own to establish \eqref{epsilondecay}.

We overcome this limitation, i.e. establish {\it pathwise} control of $\displaystyle \|\rho_\epsilon(t) - \rho_{\epsilon*}\|_{L^2(0,1)} $ and $\displaystyle \|m_\epsilon(t)\|_{L^2(0,1)}$ {\it for every} $t\in[0,\infty)$, in Section \ref{sec_a.s._convergence}.
The fourth-moment bounds from Lemma \ref{lem_2ndmoment_eta*_decay} furnish, via a recent result of Yuskovych \cite{MR4671722}, a sufficient condition for proving the \textit{almost sure decay-in-time of the stochastic integrals} appearing in the energy of the
$\epsilon$-approximate system. This is shown in Theorem \ref{thm_a.s._decay} which in turn yields the desired estimate \eqref{epsilondecay}. Finally, Theorem \ref{thm_limit_longt} follows by passing $\epsilon\to 0$ in \eqref{epsilondecay}.

	

	\subsection{$L^2_\omega$ Convergence}\label{sec_decay_in_expectation}
  The main result of this section is Theorem \ref{thm_approx_longt}, which establishes an exponential-in-time decay, in expectation, for solutions of the $\epsilon$-regularized system \eqref{regularized_problem1}--\eqref{regularized_problem2}.  While the proof of this theorem follows the entropy-dissipation framework developed in \cite{huang_convergence_2005,huang_asymptotic_2006} for the hyperbolic system \eqref{problem} without stochastic forcing, it is modified in our proof and adapted to the parabolically-regularized stochastic system \eqref{regularized_problem1}--\eqref{regularized_problem2}.
    Before stating the theorem, we recall some important assumptions on the parameters associated with the system \eqref{problem}. {The adiabatic constant of the pressure $\gamma$, the damping coefficient $\alpha$ and, the noise coefficient $\sigma$, and the spatially-regularized noise coefficient $\sigma_\epsilon$(see Remark \ref{rem_init_cond}) satisfy $$\gamma>1,\qquad\alpha>0,\qquad|\nabla_{\rho, m} \sigma| \leq \sqrt{A_0}, \qquad|\nabla_{\rho, m} \sigma_\epsilon| \leq \sqrt{A_0},$$
    for some constant $A_0 > 0$, which shall be chosen appropriately later in this section.} The deterministic initial data $(\rho_0, m_0) $ for the problem \eqref{problem} is assumed to be uniformly bounded by some positive constants $M_1$ and $M_2$ in the following way, ensuring that the initial velocity is bounded as well:
    \begin{equation}\label{init_condn_again}
        0 \leq \rho_0(x) \leq M_1, \qquad  |m_0(x)| \leq M_2 \rho_0(x).
    \end{equation}
 We also recall that, in Remark \ref{rem_init_cond}, we constructed $U_{\epsilon 0} = (\rho_{\epsilon 0}, m_{\epsilon 0}) \in H^{2}(0,1)^2$ and $\sigma_\epsilon$ as regularizations of the given initial condition $U_0 = (\rho_0, m_0)$ and the noise coefficient $\sigma$, respectively. These serve as the initial condition and noise coefficient, respectively, for the parabolic approximate system \eqref{regularized_problem1}--\eqref{regularized_problem2}.
   
	\begin{thm}\label{thm_approx_longt}
		Given a stochastic basis $({\Omega}, {\mathcal{F}}, ({\mathcal{F}}_t)_{t\geq 0}, {\mathbb{P}},{W})$, for any $\epsilon>0$, let $U_\epsilon = (\rho_\epsilon, m_\epsilon)$ be the pathwise bounded solution to \eqref{regularized_problem1}--\eqref{regularized_problem2} in the sense of Definition \ref{defn_U_epsilon}, obtained in Theorem \ref{thm_existence_pathwise_U_epsilon}. 
         Denote by $\displaystyle\rho_{\epsilon *} := \int_{0}^{1}\rho_{\epsilon 0}\,dx$.  {There exist positive, deterministic constants $C_0 = C_0(\|U_0\|_{L^\infty}, \gamma, \alpha)$, $\tilde{C} = \tilde{C}(\gamma, \alpha, M_1)$}, $C=C(\alpha, \gamma, A_0, M_1)$ such that if $A_0 < \tilde{C}\min{(\alpha, 1)} $ then,
		\begin{equation}\label{ineq_expdecay_mrho_l2}
			\mathbb{E}\int_{0}^{1}(\rho_\epsilon(t) - \rho_{\epsilon *})^2 + m_\epsilon^2(t) \,dx \leq C_0 e^{-Ct},
		\end{equation}
       for every $t>0$.
	\end{thm}
	\begin{proof}
    The proof of this theorem is divided into 5 steps.\\
    In Step 1, we rewrite the regularized system \eqref{regularized_problem1}--\eqref{regularized_problem2} in terms of exponentially-scaled density and momentum variables and the anti-derivative of the density. \\
    In Step 2, we find an appropriate integral equation (and then an important inequality) satisfied by these new random variables. \\
    In Step 3, we select a suitable entropy-entropy flux pair for the approximate system \eqref{regularized_problem1}--\eqref{regularized_problem2}. \\
    In Step 4, we combine the inequality from Step 2 with the entropy equality from Step 3 which yields an integral inequality (see \eqref{combined_ineq}) for a quantity that controls our object of interest, $\displaystyle(\rho_\epsilon(t) - \rho_{\epsilon*})^2 + m_\epsilon^2(t)$ (see Lemma \ref{lem_eta*_dominates}). \\
    Finally, in Step 5, we apply expectation and then Gronwall's lemma to the aforementioned integral inequality to arrive at the desired result.

    \noindent\textbf{Step 1: Transforming the variables}\\
		We will first rewrite the $\epsilon-$approximation system in {\eqref{regularized_problem1}--\eqref{regularized_problem2}} in terms of the following exponentially-scaled variables: 
		\begin{equation}\label{defn_wyz}
			\begin{split}
				 w_\epsilon(x, t) &= e^{M   t}(\rho_\epsilon(x, t) - \rho_{\epsilon*}),\\
			y_\epsilon(x, t)  &= -\int_{0}^{x}e^{M   t}(\rho_\epsilon(\xi, t) - \rho_{\epsilon*})\,d\xi = -\int_{0}^{x}w_\epsilon(\xi, t)\,d\xi,\\
			z_\epsilon(x, t) & = e^{M   t}m_\epsilon ,
			\end{split}
		\end{equation}
		where $\displaystyle\rho_{\epsilon *}:= \int_0^1 \rho_{\epsilon 0}\,dx$, the constant $M >0$ will be chosen large enough and independent of $\epsilon$ later in \eqref{ineq_chooseMepsilon}. 
        The convergence rate is independent of this constant $M $.\\
		Note the following boundary conditions hold {pointwise} for the transformed variables:
		\begin{equation}\label{bc_wyzepsilon}
			\partial_x w_\epsilon(0, t) = \partial_x w_\epsilon(1, t) = 0, \;\; y_\epsilon(0, t) = y_\epsilon(1, t) = 0, \;\; z_\epsilon(0, t) = z_\epsilon(1, t) = 0.
		\end{equation}
		We also calculate the derivatives of $y_\epsilon$ as follows: 
		\begin{equation}\label{derivatives_yepsilon}
			\begin{split}
				{\partial_x y_\epsilon} &= -w_\epsilon,\\
				\partial_t y_\epsilon &= -\int_{0}^{x}\partial_t[e^{M   t}(\rho_\epsilon - \rho_*)]\,d\xi  = -\int_{0}^{x}M   e^{M   t}(\rho_\epsilon - \rho_*) + e^{M   t}\partial_t(\rho_\epsilon - \rho_*)\,d\xi\\
				& = M   y_\epsilon - \int_{0}^{x}e^{M   t}[-\partial_x m_\epsilon  + \epsilon \partial_x^2\rho_\epsilon]\\
				& = M   y_\epsilon +z_\epsilon - \epsilon\partial_x\rho_\epsilon = M    y_\epsilon+z_\epsilon - \epsilon\partial_x w_\epsilon.
			\end{split}
		\end{equation}
        {We apply the Ito's formula to $m_\epsilon  \mapsto e^{M   t}m_\epsilon  =: z_\epsilon$ and obtain:
        \[dz_\epsilon = M  e^{M  t}m_\epsilon \,dt + e^{M  t}\,dm_\epsilon = M  z_\epsilon\,dt + e^{M  t}\,dm_\epsilon.\]
        We then use the expression of $dm_\epsilon$ in equation \eqref{regularized_problem1}--\eqref{regularized_problem2} and rewrite the original equation \eqref{regularized_problem1}--\eqref{regularized_problem2} in terms of the transformed variables $(w_\epsilon, z_\epsilon)$ as: 
		\begin{equation}\label{transformed_eqn}
			\begin{cases}
			{	\partial_t w_\epsilon + \partial_x z_\epsilon  = M  w_\epsilon + \epsilon \partial_x^2 w_\epsilon}\\
				dz_\epsilon + \partial_x(e^{M  t}(\frac{m_\epsilon^2}{\rho_\epsilon} + p(\rho_\epsilon) - p(\rho_*)))\,dt + (\alpha - M  )z_\epsilon\,dt = \epsilon \partial_x^2 z_\epsilon + e^{M  t}\sigma_\epsilon(x, U_\epsilon)\,dW(t).
			\end{cases}
		\end{equation}

	\noindent	\textbf{Step 2: Testing \eqref{transformed_eqn}$_2$ with $y_\epsilon$}
      
		First, we multiply the second equation in \eqref{transformed_eqn} with $y_\epsilon$, and use the product rule and \eqref{derivatives_yepsilon} to obtain the following equation:
		\[y_\epsilon\,dz_\epsilon = d(y_\epsilon z_\epsilon) - z_\epsilon \partial_t y_\epsilon\,dt = d(y_\epsilon z_\epsilon) - z_\epsilon(M  y_\epsilon + z_\epsilon - \epsilon\partial_x w_\epsilon)\,dt.\]
We then integrate both sides in space over $[0,1]$ and integrate by parts when necessary. The whole equation then reads:
		\begin{multline}
			d\int_{0}^{1} y_\epsilon z_\epsilon\,dx - \int_{0}^{1}M  z_\epsilon y_\epsilon\,dx\,dt - \int_{0}^{1}z_\epsilon^2\,dx\,dt + \epsilon\int_{0}^{1}z_\epsilon \partial_xw_\epsilon\,dx\,dt - \int_{0}^{1}e^{M  t}\frac{m_\epsilon^2}{\rho_\epsilon}\partial_x y_\epsilon \,dx\,dt\\
            - \int_{0}^{1}e^{M  t}[p(\rho_\epsilon)-p(\rho_*)]\partial_x y_\epsilon\,dx\,dt  + \int_{0}^{1}(\alpha - M  )y_\epsilon z_\epsilon\,dx\,dt \\ = \int_{0}^{1}y_\epsilon e^{M  t}\sigma_\epsilon(x, U_\epsilon)\,dx\,dW(t) + \epsilon\int_{0}^{1}y_\epsilon\partial_x^2 z_\epsilon\,dx\,dt.
		\end{multline}
		We then carry out the following computations in the order listed below:
		\begin{enumerate}
			\item Substitute in $\partial_x y_\epsilon = -{e^{M  t}}(\rho_\epsilon - \rho_{\epsilon *})$ and realize that $-e^{2M  t}\frac{m_\epsilon^2}{\rho_\epsilon}(\rho_\epsilon - \rho_{\epsilon*})= - z_\epsilon^2 + \frac{z_\epsilon^2 \rho_{\epsilon *}}{\rho_\epsilon}$ .
			\item Integrate by parts the last term:
			\[\begin{split}
			    \int_{0}^{1}y_\epsilon \partial_x^2 z_\epsilon\,dx  = y_\epsilon\partial_x z_\epsilon \Big|_0^1 - \int_{0}^{1}\partial_x y_\epsilon \partial_x z_\epsilon\,dx 
			= \int_{0}^{1}w_\epsilon \partial_x z_\epsilon\,dx 
            &=  w_\epsilon z_\epsilon \Big|_0^1 - \int_{0}^{1}z_\epsilon \partial_x w_\epsilon\,dx\\
			&= - \int_{0}^{1}z_\epsilon \partial_x w_\epsilon \,dx.
			\end{split}
			\]
			\item Move $-2M  y_\epsilon z_\epsilon$  to the right hand side of the equation.
			\item Rewrite the damping term as follows,
			\[\begin{split}
				\int_{0}^{1}\alpha y_\epsilon z_\epsilon &= \int_{0}^{1}\alpha y_\epsilon(\partial_t y_\epsilon - M  y_\epsilon + \epsilon\partial_x w_\epsilon)\,dx \\
                & = \frac{\alpha}{2}\frac{d}{dt}\int_{0}^{1}y_\epsilon^2\,dx - \alpha M  \int_{0}^{1} y_\epsilon^2\,dx + \epsilon \alpha \int_{0}^{1} y_\epsilon \partial_x w_\epsilon\,dx\\
				& = \frac{\alpha}{2}\frac{d}{dt}\int_{0}^{1}y_\epsilon^2\,dx - \alpha M  \int_{0}^{1} y_\epsilon^2\,dx + \epsilon \alpha \left(y_\epsilon w_\epsilon\Big|_0^1 - \int_{0}^{1} \partial_xy_\epsilon w_\epsilon\,dx\right)\\
				& = \frac{\alpha}{2}\frac{d}{dt}\int_{0}^{1}y_\epsilon^2\,dx - \alpha M  \int_{0}^{1} y_\epsilon^2\,dx + \epsilon \alpha \int_{0}^{1}w_\epsilon^2\,dx.
			\end{split}\]
		\end{enumerate}
		After carrying out the steps above, we integrate over $[0, t]$, for an arbitrary $t>0$, to obtain
		\begin{equation}\label{y_eqn}
            \begin{split}
                &\int_{0}^{1}y_\epsilon z_\epsilon(t) + \frac{\alpha}{2}y_\epsilon^2(t)\,dx + \epsilon \int_{0}^{t}\int_{0}^{1}\alpha w_\epsilon^2 + 2z_\epsilon \partial_x w_\epsilon\,dx\,ds - \int_{0}^{t}\int_{0}^{1}\frac{z_\epsilon^2}{\rho_\epsilon} \rho_{\epsilon *}\,dx\,ds \\
    			&+ \int_{0}^{t}\int_{0}^{1}e^{2M s}[p(\rho_\epsilon) - p(\rho_{\epsilon *})](\rho_\epsilon - \rho_{\epsilon *})\,dx\,ds \\
                =& {\int_{0}^{1}y_\epsilon z_\epsilon(0) + \frac{\alpha}{2}y_\epsilon^2(0)\,dx} + \int_{0}^{t}\int_{0}^{1}2M  y_\epsilon z_\epsilon\,dx\,ds \\ & + \int_{0}^{t}\int_{0}^{1}\alpha M   y_\epsilon^2\,dx\,ds +  \int_0^T\int_{0}^{1}y_\epsilon e^{M  s}\sigma_\epsilon(x, U_\epsilon)\,dx\,dW(s).
            \end{split}
		\end{equation}
    
		Next, we will find lower bounds for the left-hand side term $\int_{0}^{1}z_\epsilon\partial_x w_\epsilon\,dx$.
		We first multiply the first equation in (\ref{transformed_eqn}) by $w_\epsilon$ and integrate over $[0, 1]$. Then, by using the boundary conditions for $z_\epsilon$ and $\partial_x w_\epsilon$ given in \eqref{bc_wyzepsilon}, we integrate-by-parts to obtain
        \begin{equation}
            \begin{split}
                \frac{1}{2}\frac{d}{dt}\int_{0}^{1}w_\epsilon^2 - \int_{0}^{1}z_\epsilon \partial_x w_\epsilon\,dx = \int_{0}^{1} M  w_\epsilon^2 \,dx - \epsilon\int_{0}^{1}(\partial_x w_\epsilon)^2\,dx.
            \end{split}
        \end{equation}
        By applying Poincare's inequality to bound $-\epsilon\int_0^1(\partial_x w_\epsilon)^2\,dx$ we arrive at,
          \begin{equation}
            \begin{split}
                \frac{1}{2}\frac{d}{dt}\int_{0}^{1}w_\epsilon^2 - \int_{0}^{1}z_\epsilon \partial_x w_\epsilon\,dx 
				&\leq \int_{0}^{1} M  w_\epsilon^2 \,dx - \epsilon\int_{0}^{1}w_\epsilon^2\,dx
				= \int_{0}^{1} (M  -\epsilon)w_\epsilon^2\,dx.
            \end{split}
        \end{equation}
{ Next, we integrate the inequality above in time over $[0, t]$  and multiply by $2\epsilon$ to get the following lower bound:
		\begin{equation}
			2\epsilon \int_{0}^{t}\int_{0}^{1}z_\epsilon \partial_x w_\epsilon\,dx\,ds \geq \epsilon \int_{0}^{1}w_\epsilon^2(t) - w_\epsilon^2(0) \,dx - \epsilon\int_{0}^{t}\int_{0}^{1}(2M  -2\epsilon)w_\epsilon^2\,dxds.
		\end{equation}
		By using this lower bound for the left-hand side term of \eqref{y_eqn} we arrive at the following inequality:
		\begin{multline}\label{y_ineq}
			\int_{0}^{1}y_\epsilon z_\epsilon(t) + \frac{\alpha}{2}y_\epsilon^2(t) +\epsilon w_\epsilon^2(t)\,dx + \epsilon \int_{0}^{t}\int_{0}^{1}(2\epsilon+\alpha) w_\epsilon^2 \,dx\,ds - \int_{0}^{t}\int_{0}^{1}\frac{z_\epsilon^2}{\rho_\epsilon}\rho_{\epsilon *}\,dx\,ds \\
			+ \int_{0}^{t}\int_{0}^{1}e^{2M  s}[p(\rho_\epsilon) - p(\rho_{\epsilon *})](\rho_\epsilon - \rho_{\epsilon *})\,dx\,ds \leq { \int_{0}^{1}y_\epsilon z_\epsilon(0) + \frac{\alpha}{2}y_\epsilon^2(0) +\epsilon w_\epsilon^2(0)\,dx} \\
            + 2M  \int_{0}^{t}\int_{0}^{1}(y_\epsilon z_\epsilon + \frac{\alpha}{2}y_\epsilon^2 + \epsilon w_\epsilon^2) \,dx\,ds +  \int_{0}^{t}\int_{0}^{1}y_\epsilon e^{M  s}\sigma_\epsilon(x, U_\epsilon)\,dx\,dW(s).
		\end{multline}
		We now proceed to Step 3 where we will consider an entropy equality, which in Step 4, will be combined with the inequality \eqref{y_ineq}. 
        
		\noindent \textbf{Step 3: Entropy {equation}}\\
		Recall that $U_\epsilon$ 
        satisfies the entropy equality \eqref{epsilonentropy}.
		We substitute $\varphi(x) = 1$, $\psi(t) = e^{2M  t}\mathbbm{1}_{[0, t]}$ and  $\eta = \eta_*$ in \eqref{epsilonentropy}, where $\eta_*$ is defined as
		\begin{equation}\label{defn_eta*}
			\eta_*(U) = \frac{m^2}{2\rho}+p(\rho) - p(\rho_*) - p'(\rho_*)(\rho - \rho_*),\quad U=(\rho,m).
		\end{equation}
     This gives us,
		\begin{equation}\label{entropy_equality_star}
			\begin{split}
				&\int_{0}^{1}e^{2M  t}\eta_*(U_\epsilon(t))\,dx + { \epsilon \int_{0}^{t} e^{2M  s}\int_{0}^{1}\langle\nabla^2\eta_*(U_\epsilon){\partial_xU_\epsilon, \partial_xU_\epsilon\rangle}\,dx\,ds }= \int_{0}^{1}\eta_*(U_\epsilon(0))\,dx \\
                &+ 2M   \int_{0}^{t}\int_{0}^{1}e^{2M  s}\eta_*(U_\epsilon)\,dx\,ds 
				 - \alpha \int_{0}^{t}\int_{0}^{1} e^{2M  s}\frac{m_\epsilon^2}{\rho_\epsilon} \,dx\,ds \\
                &+ \frac{1}{2}\int_{0}^{t}\int_{0}^{1}e^{2M  s}\frac{1}{\rho_\epsilon}{\sigma_\epsilon^2(x, U_\epsilon)}\,dx\,ds + \int_{0}^{t}\int_{0}^{1}e^{2M  s}\frac{m_\epsilon}{\rho_\epsilon}{\sigma_\epsilon(x, U_\epsilon)}\,dx\,dW(s).
			\end{split}
		\end{equation}

        \begin{rem}\label{rem_eta*}
            We note here that a brief calculation gives us $\nabla^2\eta_* = \nabla^2\eta_E$, where $\eta_E$ is defined in \eqref{defn_etaE}. Due to the fact that $\eta_*$ is $C^2$ in $(0, \infty)\times \mathbb{R}$ and that $\rho_\epsilon$ has a lower bound given in {Theorem \ref{thm_existence_pathwise_U_epsilon}}, we know that $\nabla^2\eta_*(U_\epsilon)$ is well-defined. Hence, $\eta_*$ is a valid entropy function.
        \end{rem}
    
	\noindent	\textbf{Step 4: Combining the inequality from Step 2 with the entropy equation from Step 3}\\
       We begin by defining the following entities: 
        \begin{equation}\label{defn_K}
            K_\epsilon = \begin{cases}
                \frac{1}{\alpha} \max\{\Lambda + 2\rho_{\epsilon*}, 2\Lambda\} & \text{ if }0<\alpha<1\\
                \max\{\Lambda + 2\rho_{\epsilon*}, 2\Lambda\} & \text{ if }\alpha \geq 1
            \end{cases}
        \end{equation}
		where 
		\begin{equation}\label{defn_Lambda}
			\Lambda= \sup_{\epsilon>0}\|\rho_\epsilon\|_{{ L^\infty((0,\infty)\times(0,1))}},
		\end{equation}
    and $\rho_{\epsilon*}$ and $\rho_*$ are defined as follows,
        \begin{equation}\label{defn_rhoepsilon*}
            \rho_{\epsilon*}:=\int_0^1\rho_{\epsilon 0}(x)\,dx, \;\;\quad \rho_{*}:=\int_0^1\rho_{0}(x)\,dx.
        \end{equation}
      
        \begin{rem}\label{rem_Lambda}
           We emphasize that, due to statement (1) in Theorem \ref{thm_existence_pathwise_U_epsilon}, $\Lambda$ exists, is independent of $\omega \in \Omega$ and depends only on $\gamma$ and $M_1$ (see \eqref{init_condn_again}).
           
           Recall from Remark \ref{rem_init_cond} that $\rho_{\epsilon 0}$ is bounded independently of $\omega$ and $\epsilon$  Recall from Remark \ref{rem_init_cond} that $|\rho_{\epsilon 0}(x)| \leq M_1$, which means that the total mass $\rho_{\epsilon*}$ is also bounded above by $M_1$.  Therefore, for some deterministic  constant $C=C(M_1,\alpha,\gamma)$, we have 
           \begin{align}\label{supK}
               \sup_{\epsilon>0}|K_\epsilon| \leq C.
           \end{align}
        \end{rem}
        
We now multiply the entropy equality (\ref{entropy_equality_star}) with $K_\epsilon$ and add to (\ref{y_ineq}). 
		The combined inequality reads:
		\begin{equation*}
			\begin{split}
				& \int_{0}^{1}\left(K_\epsilon e^{2M  t}\eta_*(U_\epsilon(t)) + y_\epsilon z_\epsilon(t) + \frac{\alpha}{2}y_\epsilon^2(t) + \epsilon w_\epsilon^2(t)\right)\,dx + \epsilon \int_{0}^{t}\int_{0}^{1}(2\epsilon +\alpha) w_\epsilon^2 \,dx\,ds \\
                &+ { \epsilon  K_\epsilon\int_{0}^{t} e^{2M  s}\int_{0}^{1}\langle\nabla^2\eta_*(U_\epsilon){\partial_xU_\epsilon, \partial_xU_\epsilon\rangle}\,dx\,ds } +  \int_{0}^{t}\int_{0}^{1}\frac{( \alpha K_\epsilon-\rho_{\epsilon *}) z_\epsilon^2}{\rho_\epsilon} \,dx\,ds 
				\\
                &+ \int_{0}^{t}\int_{0}^{1}e^{2M  s}[p(\rho_\epsilon) - p(\rho_{\epsilon *})](\rho_\epsilon - \rho_{\epsilon *})\,dx\,ds\\
				 \leq & \int_{0}^{1}\left(K_\epsilon\eta_*(U_\epsilon(0)) + y_\epsilon z_\epsilon(0) + \frac{\alpha}{2}y_\epsilon^2(0) + \epsilon w_\epsilon^2(0)\right)\,dx \\
                 &+ 2M  \int_{0}^{t}\int_{0}^{1}\left({ K_\epsilon}e^{2M  s}\eta_*(U_\epsilon) + y_\epsilon z_\epsilon + \frac{\alpha}{2}y_\epsilon^2 +  \epsilon w_\epsilon^2\right) \,dx\,ds\\ 
                & +  \int_{0}^{t}\int_{0}^{1}\left(e^{M  s} y_\epsilon + K_\epsilon e^{2M  s}\frac{m_\epsilon}{\rho_\epsilon}\right)\sigma_\epsilon(x, U_\epsilon)\,dx\,dW(s) 
				+ \frac{1}{2}\int_{0}^{t}\int_{0}^{1}K_\epsilon e^{2M  s}\frac{\sigma_\epsilon^2(x, U_\epsilon)}{\rho_\epsilon}\,dx\,ds. 
			\end{split}
		\end{equation*}
	Next, {by using the fact that the Hessian of a convex function is positive semi-definite,} we note that:
		\[\epsilon\int_{0}^{1}w_\epsilon^2(t)\,dx +  \epsilon \int_{0}^{t}\int_{0}^{1}(2\epsilon +\alpha) w_\epsilon^2 \,dx\,ds + { \epsilon  K_\epsilon\int_{0}^{t} e^{2M  s}\int_{0}^{1}\langle\nabla^2\eta_*(U_\epsilon){\partial_xU_\epsilon, \partial_xU_\epsilon \rangle}\,dx\,ds }\geq 0.\]
        Hence, for any $t\geq0$, we obtain,
        \begin{equation}\label{combined_ineq}
			\begin{split}
				& \int_{0}^{1}\left(K_\epsilon e^{2M  t}\eta_*(U_\epsilon(t)) + y_\epsilon z_\epsilon(t) + \frac{\alpha}{2}y_\epsilon^2(t)\right)\,dx  \\
                &+  \int_{0}^{t}\int_{0}^{1}\frac{( \alpha K_\epsilon-\rho_{\epsilon *}) z_\epsilon^2}{\rho_\epsilon} \,dx\,ds 
			+\int_{0}^{t}\int_{0}^{1}e^{2M  s}[p(\rho_\epsilon) - p(\rho_{\epsilon *})](\rho_\epsilon - \rho_{\epsilon *})\,dx\,ds\\
				 \leq & \int_{0}^{1}\left(K_\epsilon\eta_*(U_\epsilon(0)) + y_\epsilon z_\epsilon(0) + \frac{\alpha}{2}y_\epsilon^2(0) + \epsilon w_\epsilon^2(0)\right)\,dx \\
                 &+ 2M  \int_{0}^{t}\int_{0}^{1}\left(K_\epsilon e^{2M  s}\eta_*(U_\epsilon) + y_\epsilon z_\epsilon + \frac{\alpha}{2}y_\epsilon^2 +  \epsilon w_\epsilon^2\right) \,dx\,ds\\ 
                & +  \int_{0}^{t}\int_{0}^{1}(e^{M  s} y_\epsilon + K_\epsilon e^{2M  s}\frac{m_\epsilon}{\rho_\epsilon})\sigma_\epsilon(x, U_\epsilon)\,dx\,dW(s) 
				+ \frac{1}{2}\int_{0}^{t}\int_{0}^{1}K_\epsilon e^{2M  s}\frac{\sigma_\epsilon^2(x, U_\epsilon)}{\rho_\epsilon}\,dx\,ds. 
			\end{split}
		\end{equation}

    Note, for this choice of $K_\epsilon$, that the quantity $ \int_{0}^{1}K_\epsilon e^{2M  t}\eta_*(U_\epsilon(t)) + y_\epsilon z_\epsilon(t) + \frac{\alpha}{2}y_\epsilon^2(t) \,dx$, appearing on the left side of the above inequality, carries important information since, as proven below, it dominates the quantity of interest $\int_{0}^{1} ((\rho_\epsilon - \rho_{\epsilon *})^2 + m_\epsilon^2) \,dx $.
		\begin{lem}\label{lem_eta*_dominates}
			Let $\eta_*$ be defined as \eqref{defn_eta*} 
             and $(y_\epsilon, z_\epsilon)$ be as defined in \eqref{defn_wyz}. Then, there exists a strictly positive constant $C = C(\gamma, M_1)$, such that for every $x \in [0,1]$, $t \geq 0$
			\[\begin{cases}
			    (\rho_\epsilon(t) - \rho_{\epsilon*})^2 + m_\epsilon^2  \leq { C e^{-2M  t}\left( K_\epsilon\eta_*(U_\epsilon)e^{2M  t} + y_\epsilon(t)z_\epsilon(t)+\frac{\alpha}{2}y_\epsilon^2(t) \right)} & 1<\gamma \leq 2\\
                (\rho_\epsilon(t) - \rho_{\epsilon*})^\gamma + m_\epsilon^2  \leq { C e^{-2M  t}\left( K_\epsilon\eta_*(U_\epsilon)e^{2M  t} + y_\epsilon(t)z_\epsilon(t)+\frac{\alpha}{2}y_\epsilon^2(t) \right)} & \gamma > 2.
			\end{cases} \] 
			In addition, we have
			\[K_\epsilon \eta_*(U_\epsilon(t)) \leq {2e^{-2M  t}\left( K_\epsilon e^{2M  t} \eta_*(U_\epsilon(t)) + y_\epsilon(t)z_\epsilon(t)+\frac{\alpha}{2}y_\epsilon^2(t) \right)}.\]  
		\end{lem}
		\begin{proof}
             We recall the definition of $\eta_*$ \eqref{defn_eta*}. We introduce the notation $\tilde{y}_\epsilon := e^{-2M  t}y_\epsilon$, and observe that
            \begin{equation}
                K_\epsilon\eta_*(U_\epsilon) + \tilde{y}_\epsilon m_\epsilon + \frac{\alpha}{2}\tilde{y}_\epsilon^2 = K_\epsilon\frac{m_\epsilon^2}{2\rho_\epsilon} + K_\epsilon[p(\rho_\epsilon) - p(\rho_{\epsilon *}) - p'(\rho_{\epsilon*})(\rho_\epsilon - \rho_{\epsilon*})] + \tilde{y}_\epsilon m_\epsilon + \frac{\alpha}{2}\tilde{y}_\epsilon^2.
            \end{equation}
            Since, by definition, $K_\epsilon \geq \frac{1}{\alpha}2\Lambda $ and $\Lambda = \sup_{\epsilon > 0}\|\rho_\epsilon\|_{L^\infty((0,\infty)\times(0,1))}$, we have that $K_\epsilon \frac{m_\epsilon^2}{4\rho_\epsilon} \geq \frac{1}{\alpha}2\Lambda \frac{m_\epsilon^2}{4\rho._\epsilon} \geq \frac{1}{2\alpha}m_\epsilon^2$. We can then bound the left hand side of the above equation from below in the following way:
            \begin{equation}\label{ineq_lowerbound_Keta*}
                K_\epsilon\eta_*(U_\epsilon) + \tilde{y}_\epsilon m_\epsilon + \frac{\alpha}{2}\tilde{y}_\epsilon^2 \geq K_\epsilon\frac{m_\epsilon^2}{4\rho_\epsilon} + K_\epsilon[p(\rho_\epsilon) - p(\rho_{\epsilon*}) - p'(\rho_{\epsilon *})(\rho_\epsilon - \rho_{\epsilon *})] + \frac{1}{2\alpha}m_\epsilon^2 + \tilde{y}_\epsilon m_\epsilon + \frac{\alpha}{2}\tilde{y}_\epsilon^2.
            \end{equation}
            Next, we use Lemma \ref{lem4.1}, to conclude that 
            \[K_\epsilon [p(\rho_\epsilon) - p(\rho_{\epsilon*}) - p'(\rho_{\epsilon*})(\rho_\epsilon - \rho_{\epsilon*})]  \geq 
            \begin{cases}
                C |\rho_\epsilon - \rho_{\epsilon*}|^2  & 1<\gamma \leq 2\\
                C |\rho_\epsilon - \rho_{\epsilon*}|^\gamma & \gamma > 2
            \end{cases}.\]
            {for some constant $C = C(\gamma, M_1)$, which does not depend on $\epsilon$ and $\omega$ since the $L^\infty$ bounds of $\rho_\epsilon$ depend only on the initial condition $\rho_0$ and are independent of $\epsilon$ as proved in Theorem \ref{thm_existence_pathwise_U_epsilon}.
          
            Finally note that, for $\alpha \in (0, 1)$, we have $K_\epsilon \frac{m_\epsilon^2}{4\rho_\epsilon} \geq \frac{1}{2\alpha}m_\epsilon^2 \geq m_\epsilon^2$. Moreover, $\frac{1}{2\alpha}m_\epsilon^2 + \tilde{y}_\epsilon m_\epsilon + \frac{\alpha}{2}\tilde{y}_\epsilon^2 = (\frac{1}{\sqrt{2\alpha}}m_\epsilon + \frac{\sqrt{\alpha}}{\sqrt{2}}\tilde{y}_\epsilon)^2 \geq 0$. Therefore, we can deduce from \eqref{ineq_lowerbound_Keta*} that 
            \begin{equation}
                K_\epsilon \eta_*(U_\epsilon) + \tilde{y}_\epsilon m_\epsilon + \frac{\alpha}{2}\tilde{y}_\epsilon^2 \geq \begin{cases}
                    m_\epsilon^2 + C|\rho_\epsilon - \rho_{\epsilon*}|^2   & 1 <\gamma \leq 2\\
                    m_\epsilon^2 + C|\rho_\epsilon - \rho_{\epsilon*}|^\gamma  & \gamma > 2
                \end{cases} .
            \end{equation}
            Therefore, the first statement is proven. 
            }
            \\
            {To prove the second statement, we start with \eqref{ineq_lowerbound_Keta*} and use the definition of $\eta_*$ to combine the first and second term on the right hand side and bound it from below by $\frac{K_\epsilon}{2}\eta_*(U_\epsilon)$, and obtain:
            \begin{equation}
                \begin{split}
                    K_\epsilon \eta_*(U_\epsilon) + \tilde{y}_\epsilon m_\epsilon+\frac{\alpha}{2}\tilde{y}_\epsilon^2 & \geq  \frac{K_\epsilon}{2}\eta_*(U_\epsilon) + (\frac{1}{\sqrt{2\alpha}}m_\epsilon + \frac{\sqrt{\alpha}}{\sqrt{2}}\tilde{y}_\epsilon)^2  \geq \frac{K_\epsilon}{2}\eta_*(U_\epsilon), \\
                    K_\epsilon\eta_*(U_\epsilon) &\leq 2(K_\epsilon \eta_*(U_\epsilon) + \tilde{y}_\epsilon m_\epsilon+\frac{\alpha}{2}\tilde{y}_\epsilon^2) .
                \end{split} 
            \end{equation}
            } 
		\end{proof}
  \noindent {\bf Step 5: Bounds and Gronwall Lemma:} We proceed with our proof of Theorem \ref{thm_approx_longt}. Note, due to Lemma \ref{lem_eta*_dominates}, that proving exponential decay of $\displaystyle\mathbb{E}\int_0^1 (m_\epsilon^2(t) + (\rho_\epsilon(t) - \rho_{\epsilon*})^2)\,dx$ now amounts to showing the exponential decay of    
  \begin{align}\label{gronwallentity}
      \mathbb{E}\int_{0}^{1}\left(K_\epsilon e^{2M  t}\eta_*(U_\epsilon(t)) + y_\epsilon(t) z_\epsilon(t)+\frac{\alpha}{2}y_\epsilon^2(t)\right)dx.
  \end{align}
     This will be obtained in this step by an application of the Gronwall lemma.
    
    Therefore, we will next derive an appropriate
     integral inequality for the term \eqref{gronwallentity}, and thus set up the stage for the application of Gronwall's lemma. For that purpose, {we will derive lower bound on the left-hand side of {\eqref{combined_ineq}} and upper bound on the right hand side of {\eqref{combined_ineq}} in terms of ${\int_0^t\int_0^1 \left(K_\epsilon e^{2M s}\eta_*(U_\epsilon(s)) + y_\epsilon(s) z_\epsilon(s) + \frac{\alpha}{2}y_\epsilon^2(s)\right) \,dx\,ds}$. \\
  We start by deriving the following lower bound for the last two terms on the left hand side of \eqref{combined_ineq}:} For some constant $\tilde C$ we shall show that,
     \begin{equation}\label{ineq_lowerbound_2terms}
        \begin{split}
				&\tilde{C} \int_{0}^{t}\int_{0}^{1}\left(\min(\alpha, 1) K_\epsilon e^{2M  s} \eta_*(U_\epsilon(s)) + y_\epsilon (s) z_\epsilon(s) + \frac{\alpha}{2}y_\epsilon^2(s) \right)\,dx\,ds \\
				\leq& \int_{0}^{t}\int_{0}^{1} \frac{\alpha K_\epsilon -\rho_{\epsilon *}}{\rho_\epsilon} (s)z_\epsilon^2(s)\,dx\,ds +  \int_{0}^{t}\int_{0}^{1}e^{2M  s}[p(\rho_\epsilon(s))-p(\rho_{\epsilon *})](\rho_\epsilon(s)-\rho_{\epsilon *})\,dx\,ds.
        \end{split}
        \end{equation}

{
We analyze the term on the left side of the above inequality. Observe that by using the definition of $\eta_*$ and applying Young's inequality to the term $y_\epsilon z_\epsilon$, we have
    \begin{equation}
        \begin{split}\label{case1_ineq1}
            &\int_{0}^{t}\int_{0}^{1}\left(\min(\alpha, 1) K_\epsilon e^{2M  s} \eta_*(U_\epsilon) + y_\epsilon z_\epsilon + \frac{\alpha}{2}y_\epsilon^2 \right)\,dx\,ds \\
            \leq & \int_{0}^{t}\int_{0}^{1}{ \Big(}\min(\alpha, 1) K_\epsilon e^{2M  s}  \frac{m_\epsilon^2{ }}{2\rho_\epsilon} +  \frac{1}{2}z_\epsilon^2\\
            &
            + \min(\alpha, 1) K_\epsilon e^{2M  s} [p(\rho_\epsilon)-p(\rho_{\epsilon*})-p'(\rho_{\epsilon*})(\rho_\epsilon - \rho_{\epsilon *})] 
            + \frac{\alpha+1}{2}y_\epsilon^2 { \Big)}\,dx\,ds.
        \end{split}
    \end{equation}}
Since $z_\epsilon(t)=e^{M  t}m_\epsilon(t)$, we combine the first two terms on the right hand side of \eqref{case1_ineq1} and write
     \begin{equation}
			\begin{split}
				{\int_{0}^{t}\int_{0}^{1}\left(\min(\alpha, 1) K_\epsilon e^{2M  s}  \frac{m_\epsilon^2{ (s)}}{2\rho_\epsilon(s)} + \frac{1}{2}z_\epsilon^2(s)\right)\,dx\,ds} &= \int_{0}^{t}\int_{0}^{1} \frac{\min(\alpha, 1) K_\epsilon +\rho_\epsilon}{2\rho_\epsilon(s)} z_\epsilon^2(s) \,dx\,ds .\\
                \end{split}
		\end{equation}
      {Due to the definition of $K_\epsilon$ in \eqref{defn_K}, we know that $\rho_\epsilon \leq \Lambda \leq \alpha K_\epsilon - 2\rho_{\epsilon*} $, when $0<\alpha < 1$; and $\rho_\epsilon \leq \Lambda \leq K_\epsilon - 2\rho_{\epsilon*} $ when $\alpha \geq 1$. In either case, we have that
      $$\min(\alpha, 1) K_\epsilon + \rho_\epsilon \leq 2\alpha K_\epsilon - 2\rho_{\epsilon*}.$$
      We thus obtain the following inequality for the first two terms on the right hand side of \eqref{case1_ineq1}:
        \begin{equation}\label{case1_ineq2}
			\begin{split}
                \int_{0}^{t}\int_{0}^{1}\left(\min(\alpha, 1) K_\epsilon e^{2M  s}  \frac{m_\epsilon^2{ (s)}}{2\rho_\epsilon(s)} + \frac{1}{2}z_\epsilon^2(s)\right)\,dx\,ds 
                &\leq \int_{0}^{t}\int_{0}^{1} \frac{2\alpha K_\epsilon - 2\rho_{\epsilon*}}{2\rho_\epsilon(s)} z_\epsilon^2(s)\,dx\,ds\\
                &= \int_{0}^{t}\int_{0}^{1} \frac{\alpha K_\epsilon -\rho_{\epsilon*}}{\rho_\epsilon(s)} z_\epsilon^2(s)\,dx\,ds .
			\end{split}
		\end{equation}}

{We next treat the remaining two terms appearing on the right hand side of \eqref{case1_ineq1}. First we use Lemma 5.2 in \cite{pan_initial_2008}, which is also stated in Lemma \ref{lem5.2}, to conclude, for some constant $C = C(\gamma, M_1)$ independent of $\epsilon$ and $\omega \in \Omega$, that,}
		\begin{equation}\label{case1_ineq3}
			\begin{split}
				\int_{0}^{t}\int_{0}^{1} &\min{(\alpha, 1)} K_\epsilon e^{2M  s} [p(\rho_\epsilon)-p(\rho_{\epsilon *})-p'(\rho_{\epsilon*})(\rho_\epsilon - \rho_{\epsilon *})] \,dx\,ds \\
                &\leq C\min{(\alpha, 1)} K_\epsilon \int_{0}^{t}\int_{0}^{1}e^{2M  s} [p(\rho_\epsilon) - p(\rho_{\epsilon *})](\rho_\epsilon - \rho_{\epsilon*})\,dx\,ds\\
               &\leq C_1\int_{0}^{t}\int_{0}^{1}e^{2M  s} [p(\rho_\epsilon) - p(\rho_{\epsilon *})](\rho_\epsilon - \rho_{\epsilon*})\,dx\,ds.
			\end{split}
		\end{equation}
        Here, $C_1>0$ is a deterministic constant  depending on $\gamma,M_1,\alpha$.
        To observe this, recall from \eqref{supK} in Remark {\ref{rem_Lambda}}, that $K_\epsilon$, defined in {\eqref{defn_K}}, is bounded independently of $\epsilon$ and $\omega$.
        
	{ Finally, we observe that} by Poincare's inequality and Lemma 4.1 in \cite{huang_asymptotic_2006} (also stated in Lemma \ref{lem4.1}), we have
		\begin{equation}\label{case1_ineq4}
				\begin{split}
				    \int_{0}^{t}\int_{0}^{1}\frac{\alpha + 1}{2}y_\epsilon^2\,dx\,ds &\leq C\int_{0}^{t}\int_{0}^{1}\frac{\alpha + 1}{2}{ (\partial_x y_\epsilon)^2}\,dx\,ds\\
                    &\leq C_2\int_{0}^{t}\int_{0}^{1} e^{2M  s}[p(\rho_\epsilon)-p(\rho_{\epsilon *})](\rho_\epsilon-\rho_{\epsilon*})\,dx\,ds,
				\end{split}
		\end{equation}
		  where the constant $C_2 $, depending only on $\gamma, \alpha$ and $M_1$, is independent of $\epsilon$ and $\omega \in \Omega$.\\
		Summing up {\eqref{case1_ineq2}, \eqref{case1_ineq3}, and \eqref{case1_ineq4},} we observe that for a deterministic constant $C_3$ depending only on $\gamma,\alpha, M_1$, we have,
		\begin{equation}
			\begin{split}
				&\int_{0}^{t}\int_{0}^{1}\left({\min(\alpha, 1)} K_\epsilon e^{2M  s} \eta_*(U_\epsilon) + y_\epsilon z_\epsilon + \frac{\alpha}{2}y_\epsilon^2 \right)\,dx\,ds \\
				\leq& \int_{0}^{t}\int_{0}^{1} \frac{\alpha K_\epsilon -\rho_{\epsilon *}}{\rho_\epsilon} z_\epsilon^2\,dx\,ds + C_3 \int_{0}^{t}\int_{0}^{1}e^{2M  s}[p(\rho_\epsilon)-p(\rho_{\epsilon *})](\rho_\epsilon-\rho_{\epsilon *})\,dx\,ds,
			\end{split}
		\end{equation}
Hence, for the constant $\tilde{C}>0$ depending on the given data $\alpha, \gamma, M_1$, defined as
          \begin{equation}\label{defn_Ctilde}
            0<  \tilde{C} := \frac{1}{\max(1, C_3)} \leq 1,
          \end{equation}
        we obtain
    \begin{equation*}
			\begin{split}
				&\tilde{C} \int_{0}^{t}\int_{0}^{1}\left(\min(\alpha, 1) K_\epsilon e^{2M  s} \eta_*(U_\epsilon(s)) + y_\epsilon z_\epsilon(s) + \frac{\alpha}{2}y_\epsilon^2(s) \right)\,dx\,ds \\
				\leq& \int_{0}^{t}\int_{0}^{1} \frac{\alpha K_\epsilon -\rho_{\epsilon *}}{\rho_\epsilon} z_\epsilon^2\,dx\,ds +  \int_{0}^{t}\int_{0}^{1}e^{2M  s}[p(\rho_\epsilon)-p(\rho_{\epsilon *})](\rho_\epsilon-\rho_{\epsilon *})\,dx\,ds.
			\end{split}
		\end{equation*}
        { This concludes the proof of {\eqref{ineq_lowerbound_2terms}}.}

{ We then obtain the following inequality from \eqref{combined_ineq}, by using the lower bound we derived in \eqref{ineq_lowerbound_2terms} for the two terms on the left hand side of \eqref{combined_ineq}:
    \begin{equation}
        \begin{split}\label{ineqleft}
            \int_0^1 &\left(K_\epsilon e^{2M  t}\eta_*(U_\epsilon(t)) + y_\epsilon z_\epsilon(t) + \frac{\alpha}{2}y_\epsilon^2(t)\right)\,dx \leq \int_0^1 \left(K_\epsilon \eta_*(U_\epsilon(0)) + y_\epsilon z_\epsilon(0) + \frac{\alpha}{2}y_\epsilon^2(0)+ {\epsilon w_\epsilon^2(0)}\right)\,dx\\
            &+  (2M  -\min(\alpha,1) \tilde{C})\int_{0}^{t}\int_{0}^{1} \left(K_\epsilon e^{2M  s} \eta_*(U_\epsilon(s)) +  y_\epsilon z_\epsilon(s) + \frac{\alpha}{2}y_\epsilon^2(s)\right) \,dx\,ds \\ 
            &+ \frac{1}{2}\int_0^t\int_0^1e^{2M  s} K_\epsilon\frac{ \sigma_\epsilon^2(x, U_\epsilon)}{\rho_\epsilon}\,dx\,ds + \int_{0}^{t}\int_{0}^{1}  \left(\frac{K_\epsilon e^{2M  s}m_\epsilon}{\rho_\epsilon}+y_\epsilon e^{M  s}\right)\sigma_\epsilon(x, U_\epsilon)\,dx\,dW(s).
        \end{split}
    \end{equation}
Next, we recall the assumption of the Lipschitz continuity \eqref{noise_assumption} of the noise coefficient $\sigma_\epsilon$ and the definition of $\eta_*$ given in \eqref{defn_eta*}.  Together, they give us $\frac{\sigma_\epsilon^2}{\rho} \leq A_0 \frac{m^2}{\rho} \leq 2A_0\eta_*(U)$.
Moreover,  recall that one of the consequences of Lemma \ref{lem_eta*_dominates} is the bound {$A_0K_\epsilon e^{2M t} \eta_*(U_\epsilon) \leq 2A_0(K_\epsilon e^{2M t} \eta_*(U_\epsilon) + y_\epsilon z_\epsilon + \frac{\alpha}{2}y_\epsilon^2)$}. Hence, we have
$$K_\epsilon\frac{\sigma_\epsilon^2}{\rho} \leq A_0 K_\epsilon\frac{m^2}{\rho} \leq 2A_0K_\epsilon\eta_*(U)\leq 2A_0(K_\epsilon e^{2M t} \eta_*(U_\epsilon) + y_\epsilon z_\epsilon + \frac{\alpha}{2}y_\epsilon^2),$$
which is used to bound the Ito correction term in \eqref{ineqleft} from above as follows,
}
 \begin{equation}\label{ineq_before_expectation}
			\begin{split}
				\int_{0}^{1} &\left(K_\epsilon e^{2M  t}\eta_*(U_\epsilon(t)) + y_\epsilon z_\epsilon(t)+\frac{\alpha}{2}y_\epsilon^2(t)\right) \,dx \\  
				&\leq  \int_{0}^{1}\left(K_\epsilon \eta_*(U_\epsilon(0)) + y_\epsilon z_\epsilon(0)+\frac{\alpha}{2}y_\epsilon^2(0) + {\epsilon w_\epsilon^2(0)}\right)\,dx \\
                &+    (2M  -\min(\alpha, 1) \tilde{C} + 2A_0)\int_{0}^{t}\int_{0}^{1} \left(K_\epsilon e^{2M  s} \eta_*(U_\epsilon) +  y_\epsilon z_\epsilon + \frac{\alpha}{2}y_\epsilon^2 \right)\,dx\,ds \\
				&+ \int_{0}^{t}\int_{0}^{1}  \left(\frac{K_\epsilon e^{2M  s}m_\epsilon}{\rho_\epsilon}+y_\epsilon e^{M  s}\right)\sigma_\epsilon(x,U_\epsilon)\,dx\,dW(s).
			\end{split}
		\end{equation}
      Due to the fact that the Ito integral is a martingale, it vanishes after we take expectation on both sides, and we thus obtain, 
		\begin{equation}\label{gronwall_epsilon}
			\begin{split}
				&\mathbb{E}\int_{0}^{1}\left(K_\epsilon e^{2M  t}\eta_*(U_\epsilon(t)) + y_\epsilon z_\epsilon(t)+\frac{\alpha}{2}y_\epsilon^2(t)\right)\,dx\\
				&\qquad \leq  \mathbb{E}\int_{0}^{1}\left(K_\epsilon \eta_*(U_\epsilon(0)) + y_\epsilon z_\epsilon(0)+\frac{\alpha}{2}y_\epsilon^2(0)+ {\epsilon w_\epsilon^2(0)}\right)\,dx \\
				&\qquad+ (2M  -\min(\alpha,1) \tilde{C} + 2A_0) \mathbb{E}\int_{0}^{t} \int_{0}^{1}\left(K_\epsilon e^{2M  s}\eta_*(U_\epsilon) + y_\epsilon z_\epsilon+\frac{\alpha}{2}y_\epsilon^2\right)\,dx\,ds.
			\end{split}
		\end{equation}
       Hereon, we shall consider
        \begin{equation}\label{A0}
            0<A_0 < \frac{\tilde{C}}2\min(\alpha,1),
        \end{equation}
        where the constant $\tilde{C}>0$, which depends only on {$\alpha, \gamma, M_1$}, is defined in {\eqref{defn_Ctilde}}.
       We will now also pick
        \begin{align}\label{ineq_chooseMepsilon}
        M> \frac12(\min(\alpha,1)\tilde{C}  - 2A_0), 
        \end{align}
Hence, we are now in position to apply Gronwall's inequality to \eqref{gronwall_epsilon}. This gives us,
		\begin{align}\label{gronwall}
			\mathbb{E}\int_{0}^{1}&\left(K_\epsilon e^{2M  t}\eta_*(U_\epsilon(t)) + y_\epsilon z_\epsilon(t)+\frac{\alpha}{2}y_\epsilon^2(t)\right)\,dx \\
            \notag&{\leq \left(\mathbb{E}\int_{0}^{1}K_\epsilon \eta_*(U_\epsilon(0)) + y_\epsilon z_\epsilon(0)+\frac{\alpha}{2}y_\epsilon^2(0)+{\epsilon w_\epsilon^2(0)}\,dx\right) } e^{ (2M  -\min(\alpha,1)\tilde{C} + 2A_0)t}.
        \end{align}
    We next divide both sides of the inequality above by $e^{2M  t}$ which gives us the following exponential decay,
            \begin{equation}\label{gronwall_divide_e-2Mt}
				\begin{split}
				    e^{-2M  t}\mathbb{E}\int_{0}^{1}&\left(K_\epsilon e^{2M  t}\eta_*(U_\epsilon(t)) + y_\epsilon z_\epsilon(t)+\frac{\alpha}{2}y_\epsilon^2(t) \right)\,dx \\
                &{\leq \left(\mathbb{E}\int_{0}^{1}K_\epsilon \eta_*(U_\epsilon(0)) + y_\epsilon z_\epsilon(0)+\frac{\alpha}{2}y_\epsilon^2(0)+{\epsilon w_\epsilon^2(0)}\,dx\right) }  e^{ (-\min(\alpha,1) \tilde{C} + 2A_0)t}.
				\end{split}
		\end{equation}
        We will next relate the left hand side of the inequality above to the quantities of interest via  Lemma \ref{lem_eta*_dominates}. For any $\gamma \geq 2$, we use H\"{o}lder's inequality to write
        \begin{equation}
            \mathbb{E}\int_0^1 (\rho_\epsilon(t) - \rho_{\epsilon *})^2\,dx \leq  \left(\mathbb{E}\int_0^1 |\rho_\epsilon(t) - \rho_{\epsilon *}|^\gamma\,dx\right)^\frac{2}{\gamma}\leq C(\gamma, \|U_0\|_{L^\infty}) e^{\frac{2}{\gamma}(-\min(\alpha,1) \tilde{C}+2A_0)t}
        \end{equation}
        Hence, for any $\gamma > 1$, we finally conclude that
		\begin{equation}
               \begin{split}
                   \mathbb{E} \int_{0}^{1}&\left[m_\epsilon^2(t) + (\rho_\epsilon(t)-\rho_{\epsilon *})^2 \right]\,dx \\ 
                   &\leq C \left(\mathbb{E}\int_{0}^{1}\left(K_\epsilon \eta_*(U_\epsilon(0)) + y_\epsilon z_\epsilon(0)+\frac{\alpha}{2}y_\epsilon^2(0)+ \epsilon w_\epsilon^2(0)\right)\,dx\right) e^{(1+\frac{2}{\gamma})(-\min(\alpha,1) \tilde{C} + 2A_0)t}, 
               \end{split}
		\end{equation}
        where $C=C(\gamma, M_1)$ is a constant independent of $\epsilon$ and $\omega \in \Omega$. {Since the sequence $\rho_{\epsilon 0}$ approximating $\rho_0$ is uniformly bounded {due to Remark \ref{rem_init_cond}}, we can bound $\mathbb{E}\int_{0}^{1}K_\epsilon \eta_*(U_\epsilon(0)) + y_\epsilon z_\epsilon(0)+\frac{\alpha}{2}y_\epsilon^2(0) +\epsilon w_\epsilon^2(0)\,dx$ by a constant uniform in $\epsilon$.  We thus conclude the proof of Theorem \ref{thm_approx_longt}}}}. 
		\end{proof}
        }
{A consequence of Theorem \ref{thm_approx_longt} and the intermediate step \eqref{gronwall} is the following corollary.}
    \begin{cor}\label{cor_eta*_1stmoment_decay}
		Let $U_\epsilon = (\rho_\epsilon, m_\epsilon)$ be the pathwise bounded solution to \eqref{regularized_problem1}--\eqref{regularized_problem2} in the sense of Definition \ref{defn_U_epsilon}. Let $\eta_*$ be defined as in \eqref{defn_eta*} and $\bar{\alpha}:=\min(\alpha, 1)$. Then, there exists a constant $C_0 = C_0(\|U_0\|_{L^\infty}, \gamma, \alpha)$ such that
		\begin{equation}
			\mathbb{E}\int_{0}^{1}\eta_*(U_\epsilon(t))\,dx \leq C_0 e^{(-\tilde{C}\bar{\alpha}+2A_0)t},
		\end{equation}
		where the deterministic constant $\tilde{C}$, depending only on $\alpha,\gamma$ and $M_1$ is defined by \eqref{defn_Ctilde}.
	\end{cor}

    \begin{proof}
    By \eqref{gronwall_divide_e-2Mt}, we have, for some constant   for $C=C(\|U_0\|_{L^\infty}, \gamma, \alpha)$, that
        \begin{equation}
        \begin{split}
             \mathbb{E}\int_0^1K_\epsilon \eta_*(U_\epsilon(t)) \,dx & \leq Ce^{(-\tilde{C} \bar{\alpha}+2A_0)t} - e^{-2Mt}\mathbb{E}\int_0^1 y_\epsilon z_\epsilon (t) \,dx - e^{-2Mt}\mathbb{E}\int_0^1 \frac{\alpha}{2}y_\epsilon^2(t) \,dx \\
             & \leq Ce^{(-\tilde{C} \bar{\alpha}+2A_0)t} + e^{-2Mt}\mathbb{E}\int_0^1 |y_\epsilon z_\epsilon|\,dx + e^{-2Mt}\int_0^1 \frac{\alpha}{2}y_\epsilon^2(t)\,dx.
        \end{split}
        \end{equation}
    Then by using the definitions of $y_\epsilon$ and $z_\epsilon$ given in \eqref{defn_wyz},  Young's inequality, and the bounds obtained in Theorem \ref{thm_approx_longt} we obtain,
    \begin{equation}
        \begin{split}
            \mathbb{E}\int_0^1K_\epsilon \eta_*(U_\epsilon(t)) \,dx & \leq Ce^{(-\tilde{C} \bar{\alpha}+2A_0)t} + e^{-2Mt} \frac{\alpha + 1}{2}\int_0^1 y_\epsilon^2(t)\,dx + \frac{1}{2}e^{-2Mt}\mathbb{E}\int_0^1 z_\epsilon^2(t) \,dx\\
            &\leq Ce^{(-\tilde{C} \bar{\alpha}+2A_0)t} + \frac{\alpha + 1}{2}\mathbb{E}\int_0^1 (\rho_\epsilon - \rho_{\epsilon*})^2\,dx + \frac{1}{2}\mathbb{E}\int_0^1 m_\epsilon^2(t)\,dx\\
            & \leq Ce^{-(\tilde{C} \bar{\alpha}-2A_0)t}.
        \end{split}
    \end{equation}
    Thanks to the definition \eqref{defn_K} of $K_\epsilon$, we have that $K_\epsilon \geq 2\Lambda$, and we thus obtain
    \begin{equation}
        \mathbb{E}\int_0^1 \eta_*(U_\epsilon(t)) \,dx \leq \frac{C}{K_\epsilon}e^{-(\tilde{C} \bar{\alpha}-2A_0)t} \leq \frac{C}{2\Lambda}e^{-(\tilde{C} \bar{\alpha}-2A_0)t}.
    \end{equation}
    Since, as explained in Remark \ref{rem_Lambda}, $\Lambda$ depends only on $M_1$ and $\gamma$, we conclude the proof of  Corollary \ref{cor_eta*_1stmoment_decay}.
    \end{proof}
    Next, we will upgrade these bounds by bootstrapping the first moment estimates for $\eta_*$ to obtain second moment estimates for $\eta_*$. These bounds will be critical in establishing decay-in-time estimates for the stochastic integral in the following section.
\begin{lem}\label{lem_2ndmoment_eta*_decay}
		Let $U_\epsilon = (\rho_\epsilon, m_\epsilon)$ be the pathwise {bounded} solution to \eqref{regularized_problem1}--\eqref{regularized_problem2} {in the sense of Definition \ref{defn_U_epsilon}}. Let $\eta_*$ be defined as in \eqref{defn_eta*}, then {there exist positive constants $C_0 = C_0(\|U_{0}\|_{L^\infty}, \gamma, \alpha)$, $\tilde{C} = \tilde{C}(\alpha, \gamma, M_1)$ and $C=C(\alpha, \gamma, A_0, M_1)$, such that if $A_0 < \tilde{C}\min(\alpha, 1)$, then for every $t\geq 0$} we obtain,
		\begin{equation}\label{ineq_2ndmom_eta*_decay}
			\mathbb{E} \left(\int_{0}^{1}\eta_*(U_\epsilon(t))\,dx\right)^2 \leq C_0 e^{-Ct}.
		\end{equation}
	Moreover, {for every $t\geq 0$} we have
		\begin{equation}\label{ineq_2ndmom_rhou^2_decay}
			\begin{split}
				\mathbb{E} \left(\int_{0}^{1}\frac{m_\epsilon^2(t)}{\rho_\epsilon(t)} \,dx\right)^2 \leq C_0 e^{-Ct}, \;\;\;\;
				\mathbb{E} \left(\int_{0}^{1}(\rho_\epsilon(t)-\rho_{\epsilon *})^2\,dx\right)^2 \leq C_0 e^{-Ct},
			\end{split}
		\end{equation}
        where $\rho_{\epsilon *} = \int_0^1 \rho_{\epsilon 0}(x)\,dx$.
	\end{lem}
	
	\begin{proof}
		{For every $t \geq 0$},  $U_\epsilon(t)$ satisfies the following {entropy equality {which is }obtained by taking $\varphi(x) = 1$, and $\psi(t) = e^{ Nt}\mathbbm{1}_{[0,t]}$ for any $N>0$, in \eqref{epsilonentropy}}:
		\begin{equation}\label{entropy_equality_tau}
			\begin{split}
				&\int_{0}^{1}e^{ Nt}\eta_*(U_\epsilon(t))\,dx + \epsilon \int_{0}^{t} e^{ Ns}\int_{0}^{1}\langle \nabla^2\eta_*\partial_x U_\epsilon, \partial_x U_\epsilon\rangle \,dx\,ds\\
                = &\int_{0}^{1}\eta_*(U_\epsilon(0))\,dx +  N \int_{0}^{t}\int_{0}^{1}e^{ Ns}\eta_*(U_\epsilon(s))\,dx\,ds  - \alpha \int_{0}^{t}\int_{0}^{1} e^{ Ns}\frac{(m_\epsilon)^2}{\rho_\epsilon} \,dx\,ds\\
                + & \frac{1}{2}\int_{0}^{t}\int_{0}^{1}e^{ Ns}\frac{1}{\rho_\epsilon}\sigma_\epsilon^2(x, U_\epsilon)\,dx\,ds + \int_{0}^{t}\int_{0}^{1}e^{ Ns}\frac{m_\epsilon}{\rho_\epsilon}\sigma_\epsilon(x, U_\epsilon)\,dx\,dW(s).
			\end{split}
		\end{equation}
	The positive constant $N$ will be chosen later and will depend only $\alpha, A_0$ and the constant $\tilde{C}$ which is defined in \eqref{defn_Ctilde}.\\ 
		We then apply Ito's formula with  $z \mapsto \frac{1}{2}z^2$, to the equation \eqref{entropy_equality_tau} and take expectation on both sides so that the Ito integral vanishes. This gives us,
		\begin{equation}
			\begin{split}\label{secondmomentito}
				&\frac{1}{2}\mathbb{E}\left(\int_{0}^{1}e^{ Nt}\eta_*(U_\epsilon(t))\,dx\right)^2 + \epsilon \mathbb{E}\int_{0}^{t} e^{2 Ns}\int_{0}^{1} \langle \nabla^2\eta_*(U_\epsilon)\partial_x U_\epsilon, \partial_x U_\epsilon\rangle\,dx \int_{0}^{1} \eta_*(U_\epsilon)\,dx\,ds \\
				= & \frac{1}{2}\left(\int_{0}^{1}\eta_*(U_\epsilon(0))\,dx\right)^2 +  N \mathbb{E}\int_{0}^{t}e^{2 Ns}\left(\int_{0}^{1} \eta_*(U_\epsilon(s))\,dx\right)^2\,ds \\
                - &\alpha \mathbb{E}\int_{0}^{t} e^{2 Ns} \left(\int_{0}^{1}\frac{(m_\epsilon)^2}{\rho_\epsilon}\,dx\right) \left(\int_{0}^{1}\eta_*(U_\epsilon)\,dx \right)\,ds\\ 
				&+ \frac{1}{2}\mathbb{E}\int_{0}^{t} e^{2 Ns}\left(\int_{0}^{1} \frac{1}{\rho_\epsilon}\sigma_\epsilon^2(x, U_\epsilon)\,dx \right)\left(\int_{0}^{1}\eta_*(U_\epsilon)\,dx \right)\,ds
                + \frac{1}{2} \mathbb{E} \int_{0}^{t} e^{2 Ns} \left(\int_{0}^{1}{\frac{m_\epsilon}{\rho_\epsilon}}\sigma_\epsilon(x, U_\epsilon)\,dx\right)^2\,ds.
			\end{split}
		\end{equation}
        Note that, in order to apply Ito's formula in this step, having an entropy {\it balance} is crucial.
		Next, we find estimates for the last three terms in the equation above one-by-one. For the first term, we use the fact that $\frac{m^2}{\rho} = 2\eta_* - {2[p(\rho) - p(\rho_*)-p'(\rho_*)(\rho - \rho_*)]} \geq2\eta_* - C_\Lambda $, for some constant $C_\Lambda$ depending on $\Lambda$ (defined in \eqref{defn_Lambda}). This gives us, 
		\begin{equation}
			\begin{split}
				&- \alpha \mathbb{E}\int_{0}^{t} e^{2 Ns} \left(\int_{0}^{1}\frac{(m_\epsilon)^2}{\rho_\epsilon}\,dx\right) \left(\int_{0}^{1}\eta_*(U_\epsilon)\,dx \right)\,ds\\
				\leq & -2\alpha \mathbb{E}\int_{0}^{t} e^{2 Ns} \left(\int_{0}^{1}\eta_*(U_\epsilon)\,dx \right)^2\,ds + \alpha C_\Lambda \mathbb{E}\int_{0}^{t} e^{2 Ns} \left(\int_{0}^{1}\eta_*(U_\epsilon)\,dx \right)\,ds. \\
			\end{split}
		\end{equation}
		{Next, we recall that, thanks to Lemma \ref{lem4.1}, we have $p(\rho) - p(\rho_*) - p'(\rho_*)(\rho - \rho_*) \geq 0$, which further gives us that $ \frac{m^2}{\rho} = 2\eta_*(U) - 2[p(\rho) - p(\rho_*) - p'(\rho_*)(\rho - \rho_*)]\leq 2\eta_*(U)$. Combining this with our assumption $|\sigma_\epsilon(\cdot, U_\epsilon)| \leq \sqrt{A_0}|m_\epsilon|$ given in \eqref{noise_assumption} we obtain the following estimates for the last two terms} of \eqref{secondmomentito},
		\begin{equation}
			\begin{split}
            	& \frac{1}{2}\mathbb{E}\int_{0}^{t} e^{2 Ns}\left(\int_{0}^{1} \frac{1}{\rho_\epsilon}\sigma_\epsilon^2(x, U_\epsilon)\,dx \right)\left(\int_{0}^{1}\eta_*(U_\epsilon)\,dx \right)\,ds \leq A_0 \mathbb{E} \int_{0}^{t} e^{2 Ns} \left(\int_{0}^{1}\eta_*(U_\epsilon(s))\,dx\right)^2\,ds, \\
				& \frac{1}{2}\mathbb{E} \int_{0}^{t} e^{2 Ns} \left(\int_{0}^{1}\frac{m_\epsilon}{\rho_\epsilon} \sigma_\epsilon(x, U_\epsilon)\,dx\right)^2\,ds \leq 2A_0 \mathbb{E} \int_{0}^{t} e^{2 Ns} \left(\int_{0}^{1}\eta_*(U_\epsilon(s))\,dx\right)^2\,ds .
			\end{split}
		\end{equation}
		Combining the above estimates and the fact that the dissipative term $\displaystyle \epsilon \mathbb{E}\int_{0}^{t} \left[e^{2 Ns}\left(\int_{0}^{1} \langle \nabla^2\eta_* \partial_x U_\epsilon, \partial_x U_\epsilon\rangle\,dx\right)\left( \int_{0}^{1} \eta_*(U_\epsilon)\,dx\right)\right]ds$ is nonnegative {due to the convexity of $\eta_*$ and Lemma \ref{lem4.1}, we arrive at,
		\begin{equation}
			 \begin{split}
				\mathbb{E}&e^{2 Nt}\left(\int_{0}^{1}\eta_*(U_\epsilon(t))\,dx\right)^2  
				\leq  \left(\int_{0}^{1}\eta_*(U_\epsilon(0))\,dx\right)^2\\& + 2( N -2\alpha + 3A_0)\mathbb{E}\int_{0}^{t}e^{2 Ns}\left(\int_{0}^{1} \eta_*(U_\epsilon(s))\,dx\right)^2\,ds 
				 +  2\alpha C_\Lambda \mathbb{E}\int_{0}^{t} e^{2 Ns} \left(\int_{0}^{1}\eta_*(U_\epsilon(s))\,dx \right)\,ds.
                \end{split}
            \end{equation}
Next, we use Corollary \ref{cor_eta*_1stmoment_decay} to bound the last term on the right hand side and obtain for some constant $C=C(\alpha, \gamma, \|U_0\|_{L^\infty})$ that,
            \begin{equation}
                \begin{split}\label{secondmomentgronwall}
				&\mathbb{E}e^{2 Nt}\left(\int_{0}^{1}\eta_*(U_\epsilon(t))\,dx\right)^2  \\
                  \leq & \left(\int_{0}^{1}\eta_*(U_\epsilon(0))\,dx\right)^2 + 2( N -2\alpha + 3A_0)\mathbb{E}\int_{0}^{t}e^{2Ns}\left(\int_{0}^{1} \eta_*(U_\epsilon(s))\,dx\right)^2\,ds \\
				& \qquad+ 2\alpha C_\Lambda C(\gamma, \alpha, \|U_0\|_{L^\infty}) \int_{0}^{t} e^{(2 N-\bar{\alpha} \tilde{C}+ 2A_0)s}\,ds	\\
				   \leq & \left(\int_{0}^{1}\eta_*(U_\epsilon(0))\,dx\right)^2 + 2( N -2\alpha + 3A_0)\mathbb{E}\int_{0}^{t}e^{2 Ns}\left(\int_{0}^{1} \eta_*(U_\epsilon(s))\,dx\right)^2\,ds \\
				& \qquad + C(e^{(2 N-\bar{\alpha} \tilde{C}+ 2A_0)t}-1) .
			\end{split}
		\end{equation}
        Here recall that $\bar{\alpha}:= \min(\alpha, 1)$, and that the constant $\tilde{C}$ is defined in \eqref{defn_Ctilde}.\\
		Next, we aim to apply Gronwall's lemma to the inequality above to show exponential decay of $\displaystyle\mathbb{E}\left(\int_0^1 \eta_*(U_\epsilon(t))\,dx\right)^2$. Recall that we assume that $A_0$ satisfies \eqref{A0}.
        We define a new constant,
        \begin{equation}\label{defn_C7epsilon}
		    C_7:= \bar{\alpha} \tilde{C}- 2A_0,
		\end{equation}
        {and note that $C_7> 0$.} 
Now, we pick $N$ that satisfies $2 \left(N -( 2\alpha - 3A_0 )\right) = \frac{C_7}{2}$. For this choice, we can also verify that,
		\begin{align*}
			2 N - C_7= 2(2\alpha - 3A_0)-\frac{C_7}{2} = 2(2\alpha-3A_0)-\frac{\alpha}{2}\tilde{C}+ A_0  >0,\\
			\text{ since for $A_0$, chosen in \eqref{A0}, we have } { A_0 < \frac{\tilde{C}}{2}\min(1,\alpha)<\frac{8- \tilde{C}}{10}\alpha}.
		\end{align*}
      Next, we multiply both sides of the inequality \eqref{secondmomentgronwall} by $e^{-(2 N-C_7)t}$. Since, by choice, $N$ satisfies ${2 N-C_7>0}$ we obtain,
		\begin{equation}
			\begin{split}
				\mathbb{E}e^{C_7t}&\left(\int_{0}^{1}\eta_*( U_\epsilon(t))\,dx\right)^2  \leq {e^{-(2N-C_7)t}}\left(\int_{0}^{1}\eta_*( U_\epsilon(0))\,dx\right)^2\\
                &+ 2( N -(2\alpha - 3A_0))\mathbb{E}\int_{0}^{t}e^{ C_7s}\left(\int_{0}^{1} \eta_*( U_\epsilon(s))\,dx\right)^2\,ds 
				 + C(1-e^{-( 2 N-C_7)t}). 
			\end{split}
		\end{equation}
Finally, we apply Gronwall's lemma to obtain,
		\begin{equation}
        \mathbb{E}e^{C_7t}\left(\int_{0}^{1}\eta_*( U_\epsilon(t))\,dx\right)^2  \leq C e^{(2 N -2(2\alpha - 3A_0))t} = C e^{\frac{C_7}{2}t},
                \end{equation}
         and conclude that for some  constant $C = C(\|U_0\|_{L^\infty}, \gamma, \alpha)>0$ we have
           \begin{equation}
			\mathbb{E} \left(\int_{0}^{1}\eta_*( U_\epsilon(t))\,dx\right)^2  \leq C e^{-\frac{C_7}{2}t},
		\end{equation}
		Therefore, \eqref{ineq_2ndmom_eta*_decay} is proved.\\
		To prove \eqref{ineq_2ndmom_rhou^2_decay}, we use the definition of $\eta_*$ and Lemma \ref{lem4.1}, which give us,
		\begin{equation}\label{eta*_upperbounds_rhousquare}
			\eta_*(U_\epsilon) = 
			\frac{1}{2}\frac{m_\epsilon^2}{\rho_\epsilon} + p(\rho_\epsilon)-p(\rho_*) - p'(\rho_*)(\rho_\epsilon - \rho_{\epsilon*}) \geq \begin{cases}
			    \frac{1}{2}\frac{m_\epsilon^2}{\rho_\epsilon} + C |\rho_\epsilon - \rho_{\epsilon*}|^2 & 1<\gamma \leq 2\\
                \frac{1}{2}\frac{m_\epsilon^2}{\rho_\epsilon} + C |\rho_\epsilon - \rho_{\epsilon*}|^\gamma & \gamma > 2
			\end{cases},
		\end{equation}
     for some constant $C=C(\gamma, M_1)$ independent of $\epsilon$.\\
	Since $\displaystyle\frac{m_\epsilon^2}{\rho_\epsilon} \geq 0$, $(\rho_\epsilon - \rho_{\epsilon *})^2 \geq 0$, and $|\rho_\epsilon - \rho_{\epsilon*}|^\gamma \geq 0$ we conclude, for some positive constant $C = C(\|U_0\|_{L^\infty}, \gamma, \alpha)$, that 
    \[\begin{split}
      \mathbb{E}\left(\int_0^1 \frac{m_\epsilon^2}{\rho_\epsilon}(t) \,dx\right)^2 \leq C e^{-\frac{C_7}{2}t}, \;\;
        \begin{cases}
            \mathbb{E}\left(\int_0^1|\rho_\epsilon - \rho_{\epsilon*}|^2\,dx \right)^2 \leq C e^{-\frac{C_7}{2}t} & 1< \gamma \leq 2\\
            \mathbb{E}\left(\int_0^1 |\rho_\epsilon - \rho_{\epsilon *}|^\gamma \,dx \right)^2 \leq C e^{-\frac{C_7}{2}t} & \gamma > 2.
        \end{cases}
    \end{split}\]
    For the case of $\gamma > 2$, we can use H\"{o}lder's inequality to obtain
\begin{align}\label{secondmoment}
    \mathbb{E}\left(\int_0^1 |\rho_\epsilon - \rho_*|^2 \,dx \right)^2 \leq \left(\mathbb{E}\left(\int_0^1 |\rho_\epsilon - \rho_{\epsilon*}|^\gamma \,dx \right)^2\right)^\frac{2}{\gamma} \leq C  e^{-\frac{C_7}{\gamma}t},
\end{align}
where $C_7$ is defined in \eqref{defn_C7epsilon}.
   We have thus finished the proof of \eqref{ineq_2ndmom_rhou^2_decay}.
	}\end{proof}

	\subsection{Pathwise convergence}\label{sec_a.s._convergence}
In the previous section, Theorem \ref{thm_approx_longt} establishes exponential decay in time of the $L^2(0,1)$ norms only in expectation. As noted in the Introduction, while this result guarantees the existence of a sequence of times $(t_n)_{n\in\mathbb{N}}$ along which the desired decay stated in Theorem \ref{thm_limit_longt} holds almost surely, it does not by itself yield our main conclusion which quantifies the long-time behavior of our martingale solutions. Recall that, in the proof of Theorem \ref{thm_approx_longt}, taking expectation on both sides of the inequality \eqref{ineq_before_expectation}, eliminated the stochastic integral. 
    Hence, in order to prove the stronger desired result, namely pathwise convergence as $t\to\infty$ as stated in {Theorem \ref{thm_limit_longt}}, {we will need to perform more refined estimates on the Ito integral term $\displaystyle\int_0^t\int_0^1 \left(\frac{K_\epsilon e^{2Ms}m_\epsilon}{\rho_\epsilon}+y_\epsilon e^{Ms}\right)\sigma_\epsilon(x,U_\epsilon)\,dx\,dW(s)$ appearing in \eqref{ineq_before_expectation}.}\\
	For ease of readability, we introduce the following notation for the term on the left-hand side of \eqref{ineq_before_expectation} and for the It\^{o} integral term on the right-hand side of \eqref{ineq_before_expectation}: 
	\begin{equation}\label{defn_QS}
		\begin{split}
			&Q_\epsilon(t): = \int_{0}^{1} \left(K_\epsilon e^{2Mt}\eta_*(U_\epsilon(t)) + y_\epsilon (t)z_\epsilon (t) + \frac{1}{2}y_\epsilon^2(t)\right)\,dx,\\
			& \mathcal{S}_\epsilon(t): =  \int_0^t\int_0^1 \left(\frac{K_\epsilon e^{2Ms}m_\epsilon(s)}{\rho_\epsilon(s)}+ e^{Ms} y_\epsilon (s) \right)\sigma_\epsilon(x,U_\epsilon(s))\,dx\,dW(s).
		\end{split}
	\end{equation}
	Using this notation along with the new definition $C_5 := 2M-\min(\alpha, 1)\tilde{C}+ 2A_0$, which is strictly positive due to the choice of $M$ in \eqref{ineq_chooseMepsilon}, the inequality \eqref{ineq_before_expectation} can be succinctly written as follows,
	\begin{equation}\label{ineq_concise_before_gronwall}
		Q_\epsilon(t) \leq Q_\epsilon (0) + C_5 \int_{0}^{t} Q_\epsilon(s)\,ds+ \mathcal{S}_\epsilon(t).
	\end{equation}
Note that, since $\mathcal{S}_\epsilon(t)$ is continuous in time for every $\omega\in \Omega$, it is integrable over any interval $I \subset [0, T]$. Hence, this time we apply Gronwall's lemma {to \eqref{ineq_concise_before_gronwall}} \textbf{without} taking expectation on both sides. This results in the following inequality for any $t\geq 0$,
\begin{align*}
    Q_\epsilon(t) &\leq Q_\epsilon(0) + \mathcal{S}_\epsilon(t) + \int_{0}^{t}C_5\left(Q_\epsilon(0) + \mathcal{S}_\epsilon(s)\right)e^{\int_{s}^{t} C_5 \,dr}\,ds\\
			& = Q_\epsilon(0) + \mathcal{S}_\epsilon(t) +C_5 Q_\epsilon(0) \int_{0}^{t}e^{C_5(t-s)}\,ds + C_5\int_{0}^{t}\mathcal{S}_\epsilon(s)e^{C_5(t-s)}\,ds\\
			& = Q_\epsilon(0) + \mathcal{S}_\epsilon(t) + Q_\epsilon(0)(e^{C_5 t} - 1) + C_5 e^{C_5 t}\int_{0}^{t}\mathcal{S}_\epsilon(s)e^{-C_5 s}\,ds.
\end{align*}
In other words, we have,
	\begin{equation}\label{general_gronwall}
		\begin{split}
			e^{-2Mt}Q_\epsilon(t) & \leq  e^{-(\min(\alpha, 1) \tilde{C}- 2A_0)t}Q_\epsilon(0) + e^{-2Mt}\mathcal{S}_\epsilon(t) +C_5 e^{-(\min(\alpha, 1) \tilde{C}- 2A_0)t}\int_{0}^{t}\mathcal{S}_\epsilon(s)e^{-C_5 s}\,ds.
		\end{split}
	\end{equation}
    Recall, thanks to Lemma \ref{lem_eta*_dominates}, that the term on the left-hand side of the inequality \eqref{general_gronwall} dominates our entities of interest $\|\rho_{\epsilon}(t)-\rho_{\epsilon*}\|^2_{L^2(0,1)}$ and $\|m_\epsilon(t)\|^2_{L^2(0,1)}$.
The main theorem of this section is the following result, Theorem \ref{thm_a.s._decay}, which gives us the decay of all the terms on the right side of \eqref{general_gronwall}.  Thanks to Theorem \ref{thm_a.s._decay}, we will then be able to apply the limit $t \to \infty$ in \eqref{general_gronwall} and, thereby finish the proof of Theorem \ref{thm_limit_longt}, under the assumption \eqref{A0} that the Lipschitz constant of the noise coefficient is sufficiently small i.e. $\displaystyle -(\min(\alpha, 1) \tilde{C}- 2A_0) < 0$ .  
    \begin{thm}\label{thm_a.s._decay}
        Let $\mathcal{S}_\epsilon(t)$ be defined as in \eqref{defn_QS}. Then there exists $\Omega_0\subset \Omega$ with ${\mathbb{P}}(\Omega_0)=1$ such that for every $\omega \in \Omega_0$, there exists a positive constant  $C_0=C_0(\omega, \alpha, A_0, \gamma, \|U_0\|_{L^\infty})$ and a deterministic constant $r = r(\alpha, A_0, \gamma, M_1)$ independent of $\epsilon$, such that
\begin{equation}\label{ineq_a.s._term1_decay}
            e^{-2Mt}  |\mathcal{S}_\epsilon(t)| \leq C_0 e^{-r t},
        \end{equation}
        and
        \begin{equation}\label{ineq_a.s._term2_decay}
            e^{-(\min(\alpha, 1) \tilde{C} - 2A_0)t}\left|\int_{0}^{t}\mathcal{S}_\epsilon(s)e^{-C_5 s}\,ds\right| \leq C_0 e^{-r t},
        \end{equation}
        for every $t>0$. 
    \end{thm}
    
	
	The proof of this theorem relies on Lemma 2.1 in \cite{MR4671722}, which we restate below as Lemma \ref{lem_yuskovych}. This lemma provides a sufficient condition on the stochastic integrand ensuring that the associated stochastic integral is $o(t^\nu)$ as $t\to\infty$, for some $\nu>0$. We will verify that the integrand of the stochastic integral $\mathcal{S}_\epsilon$ satisfies these conditions in the subsequent result, Proposition \ref{prop_expbound_Ito}. The exponential factors accompanying $\mathcal{S}_\epsilon$ in \eqref{ineq_a.s._term1_decay} and \eqref{ineq_a.s._term2_decay} then offset the $o(t^\nu)$ growth yielding the desired conclusion.
    Consequently, we postpone the proof of Theorem \ref{thm_a.s._decay} until after establishing Proposition \ref{prop_expbound_Ito}.
    
	\begin{lem}[Lemma 2.1 in \cite{MR4671722}]\label{lem_yuskovych}
		Let $(\Omega, \mathcal{F}, (\mathcal{F}_t)_{t \geq 0}, \mathbb{P})$ be filtered probability space, B be a one dimensional $(\mathcal{F}_t)_{t \geq 0}$-adapted Wiener process. Let $b=b(t, \omega)$ be an $(\mathcal{F}_t)_{t \geq 0}$-adapted stochastic process and $C > 0$, $\beta \geq 0$ be constants such that
		\begin{align}\label{bound_stoch}
		    \mathbb{E}b^2(t) \leq C(1 + t^{2\beta}), \;\; t \geq 0.
		\end{align}
		Then for any $\nu > \beta +\frac{1}{2}$, almost surely
		\[\frac{1}{t^\nu}\int_{0}^{t} b(s)\,dB(s) \to 0,\;\; t\to \infty.\]
	\end{lem}

    
{Our goal now is to establish that the condition \eqref{bound_stoch} is satisfied by the Ito integral $\mathcal{S}_\epsilon$. This is achieved in the following result, i.e. Proposition \ref{prop_expbound_Ito}. 
    }

	\begin{prop}\label{prop_expbound_Ito}
		Let $U_\epsilon = (\rho_\epsilon, m_\epsilon)$ be the pathwise bounded solution to \eqref{regularized_problem1}--\eqref{regularized_problem2} {in the sense of Definition \ref{defn_U_epsilon}.}
		Then there exist a constant $r \in (0, 2M)$ depending on $\alpha, \gamma, A_0, M_1$, and a positive constant {$C=C(\alpha, \gamma, A_0, \|U_0\|_{L^\infty})$}, {such that for every $t \geq 0$},
		\begin{equation}\label{prop_expbound_Ito_eq}
			\mathbb{E}\left( \int_{0}^{1} \left(\frac{K_\epsilon e^{2Mt}m_\epsilon{(t)}}{\rho_\epsilon(t)}+y_\epsilon{(t)}e^{Mt}\right)\sigma_\epsilon(x,U_\epsilon(t))\,dx\right)^2 \leq C e^{2(2M-r)t}.
		\end{equation}
	\end{prop}

	\begin{proof}[Proof of Proposition \ref{prop_expbound_Ito}]
 To establish this result, we will use the decay estimates for the second moment of the entropy $\eta_*$ derived in Lemma \ref{lem_2ndmoment_eta*_decay}. 
By using Lemma \ref{lem_2ndmoment_eta*_decay}, we will find appropriate bounds, as claimed in Proposition \ref{prop_expbound_Ito}, for the integrand of the stochastic integral $\mathcal{S}_\epsilon$ defined in \eqref{defn_QS}.

First we bound the left-hand side term of \eqref{prop_expbound_Ito_eq} by Young's inequality as follows,
\begin{equation}
\begin{split}\label{prop_expbound_Ito_eq1}
    &\mathbb{E}\left( \int_{0}^{1} \left(\frac{K_\epsilon e^{2Mt}m_\epsilon{(t)}}{\rho_\epsilon(t)}+y_\epsilon{(t)}e^{Mt}\right)\sigma_\epsilon(x,U_\epsilon(t))\,dx\right)^2 \\
    &\qquad\leq C\mathbb{E} \left(\int_{0}^{1} \frac{K_\epsilon e^{2Mt}m_\epsilon{(t)}}{\rho_\epsilon(t)}\sigma_\epsilon(x,U_\epsilon(t))dx\right)^2+C\mathbb{E}\left(\int_{0}^{1} y_\epsilon{(t)}e^{Mt}\sigma_\epsilon(x,U_\epsilon(t))dx\right)^2.
    \end{split}
\end{equation}
{Consider the term $\displaystyle
\mathbb{E}\left(\int_{0}^{1}\frac{K_\epsilon e^{2Mt}m_\epsilon}{\rho_\epsilon}\sigma_\epsilon(x, U_\epsilon)\,dx\right)^2$, appearing on the right hand side of \eqref{prop_expbound_Ito_eq1}. We use the growth assumption on the noise coefficient $\sigma_\epsilon$ given in \eqref{noise_assumption} i.e. $|\sigma_\epsilon|^2 \leq A_0 |m_\epsilon|^2$, and the exponential decay bound on $\displaystyle\mathbb{E}\left(\int_0^1 \frac{m_\epsilon^2(t)}{\rho_\epsilon(t)}\,dx\right)^2$ obtained in Lemma \ref{lem_2ndmoment_eta*_decay}}, to arrive at the following upper bound, 
		\begin{equation}
			\begin{split}
			    \mathbb{E}\left(\int_{0}^{1}\frac{K_\epsilon e^{2Mt}m_\epsilon}{\rho_\epsilon}\sigma_\epsilon(x, U_\epsilon)\,dx\right)^2\leq \mathbb{E}\left(\int_{0}^{1}\frac{\sqrt{A_0}K_\epsilon e^{2Mt}m_\epsilon^2}{\rho_\epsilon}\,dx\right)^2  &\leq A_0K_\epsilon^2e^{2(2Mt)}Ce^{-(C_7/2)t}\\ &= A_0K_\epsilon^2Ce^{2(2M - \frac{C_7}{4})t},
			\end{split}
		\end{equation}
{where $C_7= \min(\alpha, 1) \tilde{C}- 2A_0$ is defined in \eqref{defn_C7epsilon}}, and the deterministic constant $C$ depends only on $\|U_0\|_{L^\infty}, \gamma,$ and $ \alpha$. 

Next, we consider the second term that appears in the right hand side of \eqref{prop_expbound_Ito_eq1}. For simplicity, we notate $\displaystyle\tilde{y}_\epsilon(t, x) = e^{-Mt}y_\epsilon(t, x) = \int_{0}^{x}\left(\rho_\epsilon(t, \xi) - \rho_{\epsilon *}\right) \,d\xi$ and observe that
		\begin{equation}\label{ineq_ysigma}
			\begin{split}
				\mathbb{E}\left(\int_{0}^{1}e^{Mt} y_\epsilon(t) \sigma_\epsilon(x, U_\epsilon)\,dx\right)^2 &\leq A_0 e^{4Mt}\mathbb{E}\left(\int_{0}^{1}\tilde{y}_\epsilon(t) |m_\epsilon(t)| \,dx \right)^2\\
				& \leq A_0 e^{4Mt}\mathbb{E} \left(\left\|\tilde{y}_\epsilon(t)\rho_\epsilon^{1/2}(t)\right\|^2_{L^2(0,1)} \left\|\sqrt{\frac{m_\epsilon(t)^2}{\rho_\epsilon(t)}}\right\|^2_{L^2(0,1)}\right)\\
				& \leq \frac{1}{2}A_0 e^{4Mt}\mathbb{E} \left(\left\|\tilde{y}_\epsilon(t)\rho_\epsilon^{1/2}(t)\right\|_{L^2(0,1)}^4 + \left\|\sqrt{\frac{m_\epsilon(t)^2}{\rho_\epsilon(t)}}\right\|_{L^2(0,1)}^4\right).
			\end{split}
		\end{equation}
        By using the definition of $\tilde{y}_\epsilon$, we can bound $\left\|\tilde{y}_\epsilon(t)\rho^{1/2}_\epsilon(t)\right\|_{L^2(0,1)}^2$ as follows
		\begin{equation*}
			\left\|\tilde{y}_\epsilon(t) \rho_\epsilon^{1/2}(t)\right\|^2_{L^2(0,1)} = \int_{0}^{1} \tilde{y}_\epsilon^2 \rho_\epsilon\,dx \leq \left(\int_{0}^{1} \rho_\epsilon(x)dx\right)\,\left( \int_{0}^{1} (\rho_\epsilon(\xi) - \rho_{\epsilon*})^2\,d\xi\right)  = \rho_{\epsilon *} \int_{0}^{1}(\rho_\epsilon(x) - \rho_{\epsilon*})^2\,dx. 
		\end{equation*}
		Combining the above estimates with {\eqref{secondmoment}}, we find,
		\begin{equation*}
			\begin{split}
			    \mathbb{E}\left\|\tilde{y}_\epsilon(t) \rho_\epsilon^{1/2}(t)\right\|_{L^2(0,1)}^4 \leq \rho_{\epsilon *}^2  \mathbb{E}\left(\int_0^1 (\rho_\epsilon - \rho_{\epsilon *})^2\,dx \right)^2
                &\leq {C}\rho_{\epsilon *}^2  e^{-\frac{2}{\max{(2, \gamma)}}\frac{C_7}{2}t} \\
                &\leq {C}e^{-\frac{2}{\max{(2, \gamma)}}\frac{C_7}{2}t},
			\end{split}
		\end{equation*}
		for some constant $C= C(\|U_0\|_{L^\infty(0,1)}, \gamma, \alpha)$. Applying Lemma \ref{lem_2ndmoment_eta*_decay} to the term $\left\|\sqrt{\frac{m_\epsilon(t)^2}{\rho_\epsilon(t)}}\right\|_{L^2(0,1)}^4$ in \eqref{ineq_ysigma}, we obtain, for some constant $C= C(\|U_0\|_{L^\infty}, \gamma, \alpha, A_0)$, that
		\begin{equation*}
			\mathbb{E}\left(\int_{0}^{1}e^{Mt}y_\epsilon (t)\sigma_\epsilon(x, U_\epsilon(t))\,dx\right)^2 \leq C e^{4Mt} e^{-\frac{2}{\max(2,\gamma)}\frac{C_7}{2}t} = C e^{2(2M - \frac{2}{\max(2,\gamma)}\frac{C_7}{4})t}.
		\end{equation*}
		   We again recall the definition of $C_7>0$ given in \eqref{defn_C7epsilon} (see also \eqref{A0})and define,
        \begin{align}\label{def:r}
            r :=  \frac{ C_7}{2\max(2, \gamma)}.
        \end{align}
        Note that, the parameter $M$ is chosen in \eqref{ineq_chooseMepsilon} such that $2M-C_7= 2M-\min(\alpha,1) \tilde{C}+2A_0 >0$. Therefore, $2M - r = 2M-\frac{C_7}{2\max(2,\gamma)}>2M-\frac{C_7}{4} >0 $.
        Moreover, for this choice of $r$,
        which only depends on $\alpha, \gamma, A_0, {M_1}$,
        we have
		\begin{equation}\label{ineq_prop_5.1}
			\mathbb{E}\left(\int_{0}^{1} \left(\frac{K_\epsilon e^{2Mt}m_\epsilon(t)}{\rho_\epsilon(t)}+ e^{Mt} y_\epsilon(t)\right) \sigma_\epsilon(x, U_\epsilon(t))\,dx\right)^2 \leq C e^{2(2M-r)t},\quad \text{for every $t>0$}
		\end{equation}
      where $C>0$ depends only on $\alpha, \gamma, A_0$ and $ \|U_0\|_{L^\infty(0,1)}$. 
      We thus conclude the proof of Proposition \ref{prop_expbound_Ito}.
        
	\end{proof}
	
	Next, we present the proof of Theorem \ref{thm_a.s._decay} and thus establish almost sure time-asymptotic decay of the stochastic terms appearing in \eqref{general_gronwall}.
	\begin{proof}[Proof of Theorem \ref{thm_a.s._decay}] 
We recall \eqref{ineq_prop_5.1} the consequence of the previous theorem, 
\begin{equation}\label{ineq_prop_5.1_rearranged}
			\mathbb{E}\left(e^{-(2M-r)t}\int_{0}^{1} \left( \frac{K_\epsilon e^{2Mt}m_\epsilon(t)}{\rho_\epsilon(t)}+ e^{Mt} y_\epsilon(t)\right) \sigma_\epsilon(x, U_\epsilon(t))\,dx \right)^2 \leq C,\quad\text{for every $t\geq0$},
		\end{equation}
        where the constants $r = r(\alpha, \gamma, A_0, M_1) =\frac{C_7}{2\max(\gamma, 2)}$ and $C=C(\alpha, \gamma, \|U_0\|_{L^\infty(0,1)}, A_0)$ are independent of $\epsilon$.
Hence, thanks to Lemma \ref{lem_yuskovych}, we know that there exists ${\Omega}_0\subset \Omega$ with ${\mathbb{P}}({\Omega}_0)=1$ such that for every $\omega \in {\Omega}_0$ and for any $\nu > \frac{1}{2}$ we have,
\begin{equation}\label{ineq_apply_Yuskovych}
			\frac{1}{t^\nu}\left|\int_{0}^{t}e^{-(2M-r)s}\int_{0}^{1} \left(\frac{K_\epsilon e^{2Ms}m_\epsilon(s)}{\rho_\epsilon(s)}+ e^{Ms}y_\epsilon(s)\right) \sigma_\epsilon(x, U_\epsilon(s))\,dx\,dW(s)\right| \to 0, t \to \infty,
		\end{equation}
Since $2M - r>0$, we have  $e^{-(2M-r)t}<e^{-(2M-r)s}$ for any $0<s<t$ and thus,
		\begin{equation}\label{ineq_apply_Yuskovych2}
			\begin{split}
				e^{-2Mt}|&\mathcal{S}_\epsilon(t)| = t^\nu e^{-rt} \left(\frac{1}{t^\nu} e^{-(2M-r)t} |\mathcal{S}_\epsilon(t)|\right)\\
		&\leq t^\nu e^{-rt} \left(\frac{1}{t^\nu}\left|\int_{0}^{t} e^{-(2M-r)s} \int_{0}^{1}  \left(\frac{K_\epsilon e^{2Ms}m_\epsilon(s)}{\rho_\epsilon(s)}+ e^{Ms}y_\epsilon(s)\right) \sigma_\epsilon(x, U_\epsilon(s))\,dx\,dW(s)\right| \right).
			\end{split}
		\end{equation}
		Note that \eqref{ineq_apply_Yuskovych} also implies that there for any $\omega\in \Omega_0$ exists a constant $C=C(\omega, \nu)$ such that,
		\begin{equation}
			\frac{1}{t^\nu}\left|\int_{0}^{t}e^{-(2M-r)s}\int_{0}^{1} \left(\frac{K_\epsilon e^{2Ms}m_\epsilon(s)}{\rho_\epsilon(s)}+e^{Ms} y_\epsilon(s)\right)\sigma_\epsilon(x, U_\epsilon(s))\,dx\,dW(s)\right| \leq C(\omega, \nu).
		\end{equation}
		Therefore, \eqref{ineq_apply_Yuskovych2} implies, for any $\omega\in \Omega_0$, that
		\begin{equation}\label{ineq_Comega_bound}
			\begin{split}
			    e^{-2Mt}|\mathcal{S}_\epsilon(t)| \leq C(\omega)t^\nu e^{-rt} 
            \leq C(\omega, \nu, r)e^{-\frac{r}{2}t}, \;\; \text{for every $t>0$.}\\
			\end{split}
		\end{equation}
{The last inequality above holds because $t^\nu e^{-\frac{r}{2}t}$ is uniformly bounded by some constant  $C(\nu, r)$, which means that $C(\nu, r)e^{-\frac{r}{2}t} - t^\nu e^{-rt} \geq 0$ .} This finishes the proof of \eqref{ineq_a.s._term1_decay}.

		We next use \eqref{ineq_Comega_bound} to prove our second claim \eqref{ineq_a.s._term2_decay}. In the following calculation we recall $\bar{\alpha} = \min(\alpha, 1)$. We observe for the left-hand side of \eqref{ineq_a.s._term2_decay} that for any $\omega\in \Omega_0$ the following holds,
        \begin{equation}
            \begin{split}
                &\left|e^{-(\bar{\alpha} \tilde{C} - 2A_0)t}\int_{0}^{t}\mathcal{S}_\epsilon(s)e^{-(2M-\bar{\alpha} \tilde{C} + 2A_0)s}\,ds\right|
                = \left|e^{-C_7t}\int_{0}^{t}\mathcal{S}_\epsilon(s)e^{-(2M-C_7)s}\,ds\right|\\
                \leq & C(\omega, \nu, r)e^{-C_7t} \int_0^t e^{-(\frac{r}{2}-C_7)s}\,ds  = C(\omega, \nu, r) \frac{1}{C_7-\frac{r}2}e^{-C_7t}\left(e^{-(\frac{r}{2}-C_7)t}-1\right) \\
                \leq& C(\omega, \nu, r) e^{-\frac{r}{2}t},
            \end{split}
        \end{equation}
        where in the last inequality we use the definition \eqref{def:r} which ensures that $r<\frac{C_7}{4}$ and thus $\displaystyle\frac{1}{C_7-\frac{r}2}>0$. 
        
		
        This finishes the proof of Theorem \ref{thm_a.s._decay}
	\end{proof}
	Recall that our intermediate goal is to establish \eqref{epsilondecay}, from which our main result Theorem \ref{thm_limit_longt} follows upon passing to the limit $\epsilon\to0$. We are now in a position to carry this out. In Theorem \ref{thm_a.s._decay}, we 
    derived quantitative decay estimates for right-hand side terms in \eqref{general_gronwall}. By Lemma \ref{lem_eta*_dominates}, these bounds on $e^{-2Mt}\mathcal{Q}_\epsilon(t)$, translate into corresponding estimates for the quantity of interest $\|\rho_\epsilon(t)-\rho_{\epsilon*}\|_{L^2(0,1)}^2 + \|m_\epsilon(t)\|^2$. This establishes \eqref{epsilondecay} and passing $\epsilon\to0$ then completes the proof of Theorem \ref{thm_limit_longt}.
	
	\begin{proof}[Proof of Theorem \ref{thm_limit_longt}] 
    We apply Theorem \ref{thm_a.s._decay} to the second and third term on the right hand side of \eqref{general_gronwall}. This gives us the existence of a constant $r = r(\alpha, \gamma, A_0, M_1)$ and ${\Omega}_0\subset \Omega$ with ${\mathbb{P}}({\Omega}_0)=1$ such that for every $\omega \in {\Omega}_0$ there exists a constant $C=C(\omega, \alpha, \gamma, A_0, \|U_0\|_{L^\infty})$ such that, 
    \begin{equation}
        \begin{split}
            e^{-2Mt}Q_\epsilon(t) &\leq e^{-(\min(\alpha, 1)\tilde{C}-2A_0)t}Q_\epsilon(0) + C_0e^{-rt}+C_0e^{-rt} \leq C e^{-rt},\qquad\forall t\geq 0.
        \end{split}
    \end{equation}
    In the last inequality above we additionally used the fact that the $L^\infty(0,1)$ norm of the initial condition, $U_{\epsilon 0}$ is uniformly bounded in $\epsilon$ (see Remark \ref{rem_init_cond}). 
    Next, we apply Lemma \ref{lem_eta*_dominates} to obtain a lower bound for the left hand side of the inequality above,
    \begin{equation}\label{ineq_a.s._expdecay_bound}
        \int_0^1\left[(\rho_\epsilon(t) - \rho_{\epsilon *})^2 + m_\epsilon(t)^2\right]\,dx \leq e^{-2Mt}Q_\epsilon(t) \leq C_0 e^{-rt},
    \end{equation}
    for every $t>0$ and for every $\omega \in {\Omega}_0 \subset {\Omega}$ such that $\mathbb{P}({\Omega}_0)=1$ where ${\Omega}_0$ does not depend on $t>0$. Here, the constants $C_0 = C_0(\omega, \alpha, \gamma, A_0, \|U_0\|_{L^\infty})$, and $r = r(\alpha, \gamma, A_0, M_1)$ are positive and are independent of $\epsilon$.\\
    Now we will pass $\epsilon\to 0$ in \eqref{ineq_a.s._expdecay_bound} to obtain the desired result.
    { For that purpose, recall from Theorem \ref{pass_epsilon_0} that there exist a probability space $(\bar{\Omega}, \bar{\mathcal{F}}, \bar{\mathbb{P}})$ and $\mathcal{X}_U$-valued random variables $\{\bar{U}_\epsilon=(\bar\rho_\epsilon,\bar m_\epsilon)\}$ and $\bar{U}$ defined on this new probability space, such that $\bar{U}_\epsilon \to \bar{U}$, $\bar{\mathbb{P}}$-almost surely in $\mathcal{X}_U$.     Recall from Proposition \ref{prop_barUepsilon} that $\bar U$ is a martingale $L^\infty$ weak entropy solution of \eqref{problem} in the sense of Definition \ref{martingale_soln_defn}.
    
    Moreover, due to the coincidence of laws of $\bar{U}_\epsilon$ and $U_\epsilon$, we know that there exists a set $\bar{\Omega}_0 \subset \bar{\Omega}$ with $\bar {\mathbb{P}}(\bar{\Omega}_0)=1$ such that
    \begin{equation}\label{expnew}
        \int_0^1\left[(\bar{\rho}_\epsilon(t) - {\rho}_{\epsilon *})^2 + \bar{m}_\epsilon(t)^2\right]\,dx \leq C_0 e^{-rt},\quad \text{ for every } \omega \in \bar{\Omega}_0, \,\, t \geq 0.
    \end{equation}
    
   Now note that the almost sure convergence $\bar U_\epsilon \to \bar U$ in $C_{w,loc}(0, \infty; L^2(0, 1))$ implies for every $\varphi \in L^2(0, 1)$, { $\zeta \in L^2(0, 1)$}, and every $t\in[0, T]$ and every $\omega\in\bar{\Omega}_0$, that
    \begin{equation}\label{weak_L2_limit}
        \begin{split}
        \int_0^1 (\bar{\rho}(t) - \rho_*)\varphi \,dx &= \lim_{\epsilon\to 0}\int_0^1 (\bar{\rho}_\epsilon(t) - \rho_{\epsilon *}) \varphi\,dx,\\
        \int_0^1 \bar{m}(t)\zeta\,dx &= \lim_{\epsilon \to 0} \int_0^1 \bar{m}_\epsilon(t){\zeta} \,dx
            .
        \end{split}  
    \end{equation}
    For a fixed $t>0$ we know that $\bar{\rho}(t) \in L^2(0,1)$. Hence, we take $\varphi=(\bar{\rho}(t) - \rho_*)$ in \eqref{weak_L2_limit} to obtain,
     \begin{equation}
        \begin{split}
        \int_0^1 (\bar{\rho}(t) - \rho_*)^2 \,dx &=    \int_0^1 \left((\bar{\rho}(t) - \rho_*)(\bar{\rho}(t) - \rho_*) \right)\,dx = \lim_{\epsilon\to 0}\int_0^1 \left((\bar{\rho}_\epsilon(t) - \rho_{\epsilon *}) (\bar{\rho}(t) - \rho_*) \right)\,dx \\
            &\leq \lim_{\epsilon\to 0} \|\bar{\rho}_\epsilon(t) - \rho_{\epsilon *}\|_{L^2(0,1)}  \|(\bar{\rho}(t) - \rho_*)\|_{L^2(0,1)}, \;\; \bar{\mathbb{P}} \text{-a.s}.
        \end{split}  
    \end{equation}
    Hence, on the set $\bar{\Omega}_0\subset \bar{\Omega}$ with $\bar{\mathbb{P}}(\bar{\Omega}_0)=1$, we conclude for any $t>0$ that,
     \begin{equation} \label{ineq_everyt_rho_limit}
        \begin{split}
      \|\bar{\rho}(t) - \rho_*\|_{L^2(0,1)}&\leq \lim_{\epsilon\to 0} \|\bar{\rho}_\epsilon(t) - \rho_{\epsilon *}\|_{L^2(0,1)},
        \end{split}  
    \end{equation}
 and following similar calculations,
    \begin{equation}\label{ineq_everyt_m_limit}
        \|\bar{m}(t)\|_{L^2(0,1)} \leq \lim_{\epsilon \to 0}\|\bar{m}_\epsilon(t)\|_{L^2(0,1)}, \;\; \bar{\mathbb{P}} \text{-a.s.}
    \end{equation}
  Finally we apply the exponential decay bound in \eqref{expnew}, to the right hand side of \eqref{ineq_everyt_rho_limit} and \eqref{ineq_everyt_m_limit} and obtain {for every $t>0$} and every $\omega\in\bar{\Omega}_0$ with $\bar{\mathbb{P}}(\bar{\Omega}_0)=1$ that
    \begin{equation}
        \|(\bar{\rho}(t) - \rho_*)\|_{L^2(0,1)} + \|\bar{m}(t)\|_{L^2(0,1)}\leq C_0 e^{-\frac{r}{2}t},
    \end{equation}
    where $C_0$ depends on $\omega, \alpha, \gamma, A_0, \|U_0\|_{L^\infty}$, and $r = r(\alpha, \gamma, A_0, M_1)$ is defined in \eqref{def:r}. We have thus finished the proof of our main result Theorem \ref{thm_limit_longt}.
	}
    \end{proof}
  {
    \subsection{Porous Medium Equation}\label{PorousMediumSection}
    In this section, we emphasize the connection between the stationary states of the stochastic isentropic Euler equation obtained in Section \ref{sec_a.s._convergence}, and the long time dynamics of the porous medium equation. We consider the following initial boundary value problem of the porous medium equation with Neumann boundary condition on the pressure:
    \begin{equation}\label{PME}
        \begin{cases}
            \partial_t \rho - \partial_x^2 p(\rho) = 0, & \\
            \rho(x, 0) = \rho_0(x), & 0 \leq x \leq 1,\\
            \partial_x p(\rho(0,t)) = \partial_x p(\rho(1, t)) = 0, & t \geq 0,
        \end{cases}
    \end{equation}
    where $\rho$ represents the density, and $p(\rho)$, representing the pressure, is given by the same power law as in \eqref{problem}, i.e. $p(\rho) = \kappa \rho^\gamma$, where $\gamma > 1$ and $\kappa = \frac{(\gamma-1)^2}{4\gamma}$.
    In \cite{pan_initial_2008}, it is proved that the global {smooth} solution of \eqref{PME}, $\rho_{PME}$ approaches the initial mass $\int_0^1 \rho_0 \,dx$, and that the gradient of the pressure $-\partial_xp(\rho_{PME})$ decays to $0$ in $L^2$ exponentially fast. More precisely, they have the following result:
    \begin{thm}[Theorem 5.3 of \cite{pan_initial_2008}]
        Assume that for some positive constant $M_1$ the initial condition satisfies $\rho_0(x) \in [0, M_1]$ for every $x\in[0,1]$. Let $\rho_* := \int_0^1 \rho_0\,dx$. Then there exist positive constants $C_0$ and $\delta_1$ depending on $\rho_0$ and $\gamma$, such that the global {smooth} solution $\rho_{PME} \in C^1(0, \infty; C^2([0,1]))$ of \eqref{PME} and $m_{PME}:= - \partial_x p(\rho_{PME})$, satisfy
        \begin{equation}\label{ineq_thm5.3}
            \int_0^1 \left[(\rho_{PME}(t) - \rho_*)^2 + m_{PME}^2(t)\right] \,dx \leq C_0 e^{-\delta_1 t}, \;\; \text{for every } t>0.
        \end{equation}
    \end{thm}
    Now, we supplement the stochastic isentropic Euler system \eqref{problem} and the porous medium equation \eqref{PME} with the same bounded deterministic initial density $\rho_0$ satisfying $\rho(x) \in [0, M_1]$ for all $x\in[0,1]$. Let $(\rho_{IE}, m_{IE})$ denote the martingale $L^\infty$ weak entropy solution to the stochastic equations \eqref{problem} in the sense of Definition \ref{martingale_soln_defn}. Let $\rho_{PME}$ denote the global smooth solution to \eqref{PME}, and denote the corresponding momentum by $m_{PME}:= -\partial_xp(\rho_{PME})$. Then, by a simple application of the triangle inequality, we conclude that
        there exist positive constants $\tilde{C} = \tilde{C}(U_0, \gamma, \alpha)$ and $r = r(\alpha, \gamma, A_0, M_1)$, and a measurable set $\bar{\Omega}_0 \subset \bar{\Omega}$ with $\bar{\mathbb{P}}(\bar{\Omega}_0) = 1$ such that if $A_0 \leq \tilde{C}\min{(\alpha, 1)}$, then for every $t>0$, and every $\omega\in \bar\Omega_0$, we have
        \begin{equation}
            \int_0^1 \left[(\rho_{IE}(t) - \rho_{PME}(t))^2 + (m_{IE}(t) - m_{PME}(t))^2\right] \,dx \leq C_0 e^{-rt},
        \end{equation}
        for some positive constant $C_0 = C_0(\omega, \alpha, \gamma, A_0, \|U_0\|_{L^\infty})$.

        In other words we have finished the proof of Theorem \ref{ie-pme}.
    }

	\appendix
    \if 1 = 0
    \section{Lower bound on ${\rho}_\tau^\sharp$ on the new probability space}\label{appendix_lower_bound}
    \fi
    \if 1 = 0
    Let ${U}_\tau = ({\rho}_\tau, {m}_\tau)$ be defined by \eqref{U_defn}, and recall that ${\rho}^{\sharp}_{\tau}$ is defined as the random variables on the new probability space, defined via the Skorohod Representation theorem in Theorem \ref{pass_tau_0}. Furthermore, on the new probability space, we have that for some limiting random variable $\rho_{\epsilon}$, we have that
    \begin{equation*}
    {\rho}^{\sharp}_{\tau} \to \rho_{\epsilon} \qquad \text{ in } C(0, T; H^{1}(\mathbb{T})), \quad \tilde{\mathbb{P}}\text{-almost surely}.
    \end{equation*}
    The goal of this appendix is to show that the limiting fluid density $\rho_{\epsilon}$ is strictly positive. Namely, we will show the following result:
    
    \begin{prop}\label{rhoepspos}
        The limiting fluid density $\rho_{\epsilon}$ satisfies $\rho_{\epsilon} > 0$, $\tilde{\mathbb{P}}$-almost surely.
    \end{prop}

    To do this, we first derive a uniform estimate on positivity of the approximate fluid density $\bar{\rho}^{\sharp}_{\tau}$ on the new probability space, defined in \eqref{defn_Ubar_sharp} and corresponding to $\bar{\rho}_{\tau}$ on the original probability space, defined in \eqref{U_defn}. Recall from Proposition \ref{positive_density} that $\bar{\rho}_{\tau}$ and hence $\bar{\rho}^{\sharp}_{\tau}$ are positive almost surely. However, we will use the continuity equation with artificial viscosity to quantify this positivity more concretely, in the following lemma.

    \begin{lem}\label{logbounds}
        There exists a constant $C$ that is \textit{independent of $\tau$} (depending only on $\epsilon$ and $T > 0$), such that
        \begin{equation*}
        \|\partial_{x}(\log(\bar{\rho}^{\sharp}_{\tau}))\|_{L^{2}([0, T] \times [0, 1])} \le C(\epsilon, T), \qquad \bar{\mathbb{P}}\text{-almost surely,}
        \end{equation*}
        and furthermore,
        \begin{equation*}
        \|\log(\bar{\rho}^{\sharp}_{\tau})\|_{L^{2}(0, T; L^{\infty}(0, 1)} \le C(\epsilon, T), \qquad \bar{\mathbb{P}}\text{-almost surely.}
        \end{equation*}
        Hence, $\bar{\rho}^{\sharp}_{\tau} > 0$, for almost every $(\bar\omega, t, x) \in \bar\Omega \times [0, T] \times [0, 1]$.
    \end{lem}

    \begin{proof}
        By Proposition \ref{positive_density} and the equivalence of laws in Proposition \ref{pass_tau_0}, we have that $\bar{\rho}^{\sharp}_{\tau} > 0$, $\bar{\mathbb{P}}$-almost surely. Because the stochasticity is only in the momentum equation (not the continuity equation), note that $\bar{\rho}^{\sharp}_{\tau}$ is actually a continuous function on $[0, T]$, by the way it is constructed in \eqref{U_defn}. Furthermore, because the solution $\bar{\rho}^{\sharp}_{\tau}$ only captures the deterministic continuity equation dynamics, see the definition in \eqref{U_defn}, it is a weak solution to the continuity equation on $[0, T]$:
    \begin{equation}\label{continuity_eqn_rhosharp}
    \partial_{t}\bar{\rho}^{\sharp}_{\tau} + \partial_{x}(\bar{\rho}^{\sharp}_{\tau} \bar{u}^{\sharp}_{\tau}) = \epsilon \partial_{x}^{2}\bar{\rho}^{\sharp}_{\tau}.
    \end{equation}

    So by the fact that $\bar{\rho}^{\sharp}_{\tau} > 0$, $\bar{\mathbb{P}}$-almost surely, we can test the continuity equation by $1/\bar{\rho}^{\sharp}_{\tau}$ to obtain after integrating by parts:

    \begin{equation}\label{logenergy}
    \int_{0}^{1} \log(\bar{\rho}^{\sharp}_{\tau}(T)) dx - \int_{0}^{1} \log(\rho_{0\epsilon}) dx - \int_{0}^{T} \int_{0}^{1} \frac{\bar{u}^{\sharp}_{\tau} (\partial_{x}\bar{\rho}^{\sharp}_{\tau})}{\bar{\rho}^{\sharp}_{\tau}} dx dt = \epsilon \int_{0}^{T} \int_{0}^{1} \frac{(\partial_{x}\bar{\rho}^{\sharp}_{\tau})^{2}}{(\bar{\rho}^{\sharp}_{\tau})^{2}} dx dt.
    \end{equation}
    Recalling the uniform $L^{\infty}$ bounds in Proposition \ref{L_infty_estimates}, we have that
    \begin{equation}\label{maxbounds}
    0 < \bar{\rho}^{\sharp}_{\tau}(t) \le C \qquad \text{ and } \quad 0 < \bar{u}^{\sharp}_{\tau}(t) \le C,
    \end{equation}
    for all $t \in [0, T]$, $\bar{\mathbb{P}}$-almost surely, where the constant $C$ is independent of $\tau$ and $\epsilon$. Combining this with the lower bound on the initial data $\rho_{0\epsilon} \ge c_{\epsilon} > 0$, see \eqref{rho_epsilon0}, we have that $\bar{\mathbb{P}}$-almost surely:
    \begin{equation*}
    \int_{0}^{1} \log(\bar{\rho}^{\sharp}_{\tau}(T)) dx - \int_{0}^{1} \log(\rho_{0 \epsilon}) dx \le C_{\epsilon},
    \end{equation*}
    for a deterministic constant $C_{\epsilon}$ depending only on $\epsilon$. Using this bound and Cauchy's inequality in \eqref{logenergy}, we obtain:
    \begin{equation*}
    \epsilon \int_{0}^{T} \int_{0}^{1} \frac{(\partial_{x} \bar{\rho}^{\sharp}_{\tau})^{2}}{(\bar{\rho}^{\sharp}_{\tau})^{2}} dx dt \le \frac{\epsilon}{2} \int_{0}^{T} \int_{0}^{1} \frac{(\partial_{x}\bar{\rho}^{\sharp}_{\tau})^{2}}{(\bar{\rho}^{\sharp}_{\tau})^{2}} dx dt + C_{\epsilon} \int_{0}^{T} \int_{0}^{1} |\bar{u}^{\sharp}_{\tau}|^{2} dx dt + C_{\epsilon}.
    \end{equation*}
    Using \eqref{maxbounds}, we obtain that $\bar{\mathbb{P}}$-almost surely,
    \begin{equation*}
    \|\partial_{x}(\log(\bar{\rho}^{\sharp}_{\tau}))\|_{L^{2}([0, T] \times [0, 1])}^{2} = \int_{0}^{T} \int_{0}^{1} \frac{(\partial_{x}\bar{\rho}^{\sharp}_{\tau})^{2}}{(\bar{\rho}^{\sharp}_{\tau})^{2}} dx dt \le C(\epsilon, T).
    \end{equation*}

    To show the remaining bound on $\|\log(\bar{\rho}^{\sharp}_{\tau})\|_{L^{\infty}([0, T] \times [0, 1])}$, we recall that the initial data $\displaystyle \int_{0}^{1} \rho_{0\epsilon}(x) dx = {\color{orange}m}_{\epsilon} > 0$, for all $\epsilon > 0$ sufficiently small, given that the initial data $\rho_0(x)$ has $\displaystyle \int_{0}^{1} \rho_0(x) > 0$. So by conservation of mass:
    \begin{equation*}
    \int_{0}^{1} \bar{\rho}^{\sharp}_{\tau}(t, \cdot) dx = {\color{orange}m}_{\epsilon} > 0, \qquad \text{ for all } t \in [0, T].
    \end{equation*}
    Hence, for each $t \in [0, T]$, there exists some $x_0 \in [0, 1]$ such that $\bar{\rho}^{\sharp}_{\tau}(t, x_0) = m_{\epsilon} > 0$. Then, for all $x \in [0, 1]$, we have that $\bar{\mathbb{P}}$-almost surely:
    \begin{align*}
    \int_{0}^{T}\|\log(\bar{\rho}^{\sharp}_{\tau}(t, \cdot))\|^{2}_{L^{\infty}(0, 1)} dt &= \int_{0}^{T}\left(\sup_{x \in [0, 1]} \left|\int_{x_0}^{x} \partial_{x}(\log(\bar{\rho}^{\sharp}_{\tau})) dx\right|\right)^{2} dt \\
    &\le \int_{0}^{T} \int_{0}^{1} |\partial_{x}(\log(\bar{\rho}^{\sharp}_{\tau}))|^{2} dx dt \le C(\epsilon, T).
    \end{align*}
    
    \end{proof}

    ------
    
    To prove this proposition, we will show that the limiting $\rho_{\epsilon}$ satisfies the following viscous continuity equation, which has positivity properties that will give us the desired result in Proposition \ref{rhoepspos}. For convenience, in the proof of this result, it will be easier to work with the approximate fluid densities instead of the form $\bar{\rho}^{\sharp}_{\tau}$, defined in \eqref{defn_Ubar_sharp} on the new probability space, and defined in \eqref{U_defn} on the original probability space. Namely, we first show the following lemma.

    \begin{lem}\label{limitcontinuity}
        For a limiting $u_{\epsilon} \in L^{\infty}([0, T] \times \mathbb{T})$, the limiting fluid density $\rho_{\epsilon}$ is a weak solution to the following viscous continuity equation:
        \begin{equation}\label{epslimitcontinuity}
        \partial_{t}\rho_{\epsilon} + \partial_{x}(\rho_{\epsilon}u_{\epsilon}) = \epsilon \partial_{x}^{2}\rho_{\epsilon}.
        \end{equation}
    \end{lem}
    
    \begin{proof}[Proof of Lemma \ref{limitcontinuity}]
    The strategy will be to pass to the limit as $\tau \to 0$ in the approximate continuity equation satisfied by $\bar{\rho}^{\sharp}_{\tau}$ to obtain the limiting continuity equation \eqref{epslimitcontinuity} satisfied by $\rho_{\epsilon}$. Because the stochasticity is only in the momentum equation (not the continuity equation), note that $\bar{\rho}^{\sharp}_{\tau}$ is actually a continuous function on $[0, T]$. Furthermore, because the solution $\bar{\rho}^{\sharp}_{\tau}$ only captures the deterministic continuity equation dynamics, see the definition in \eqref{U_defn}, it is a weak solution to the continuity equation:
    \begin{equation}\label{continuity_eqn_rhosharp}
    \partial_{t}\bar{\rho}^{\sharp}_{\tau} + \partial_{x}(\bar{\rho}^{\sharp}_{\tau} \bar{u}^{\sharp}_{\tau}) = \epsilon \partial_{x}^{2}\bar{\rho}^{\sharp}_{\tau},
    \end{equation}
    namely for all deterministic $\psi \in C_{c}^{\infty}([0, T) \times (0, 1))$, we have that 
    \begin{equation}\label{weakformulationrhotau}
    \int_{0}^{T} \int_{0}^{1} \bar{\rho}^{\sharp}_{\tau} \partial_{t}\psi dx dt + \int_{0}^{T} \int_{0}^{1} \bar{\rho}^{\sharp}_{\tau} \bar{u}^{\sharp}_{\tau} \partial_{x}\psi dx dt = \epsilon \int_{0}^{T} \int_{0}^{1} \partial_{x}\bar{\rho}^{\sharp}_{\tau} \partial_{x} \psi dx dt + \int_{0}^{1} \rho_{0}(x) \psi(0, x) dx,
    \end{equation}
    $\bar{\mathbb{P}}$-almost surely. Note that this is why it is more convenient to work with $\bar{\rho}^{\sharp}_{\tau}$ instead of the linear interpolant ${\rho}^{\sharp}_{\tau}$ (which satisfies a continuity equation with additional extra error terms arising due to the linear interpolation). Then, by Theorem \ref{pass_tau_0} and Proposition \ref{samelimit}, we have that
    \begin{equation*}
    \bar{\rho}^{\sharp}_{\tau} \to \rho_{\epsilon} \quad \text{ in } C(0, T; H^{1}(0, 1)), \quad \tilde{\mathbb{P}}\text{-almost surely},
    \end{equation*}
    \begin{equation*}
    \bar{u}^{\sharp}_{\tau} \rightharpoonup u_{\epsilon} \quad \text{ weakly-star in } L^{\infty}((0, T) \times (0, 1)), \quad \tilde{\mathbb{P}}\text{-almost surely}.
    \end{equation*}
    This allows us to obtain the desired result by passing to the limit as $\tau \to 0$ in \eqref{weakformulationrhotau}, for each fixed deterministic test function $\psi \in C_{c}^{\infty}([0, T) \times (0, 1))$.
    
    \if 1 = 0
    To obtain the approximate continuity equation for ${\rho}^{\sharp}_{\tau}$, note that ${U}^{\sharp}_\tau := ({\rho}^{\sharp}_{\tau}, {m}^{\sharp}_{\tau})$ satisfies the approximate entropy equality \eqref{entropy_eq_Uhat_full}. By setting $g(\xi) = 1$ in the definition of the entropy $\eta$ in \eqref{entropy_pair_formula}, we obtain from \eqref{entropy_eq_Uhat_full} that, for all test functions $\varphi \in C_c^2((0, 1))$, and every $t \in [t_n, t_{n+1}]$,  ${\rho}_\tau^\sharp$ satisfies the following weak formulation $\bar{\mathbb{P}}$-almost surely:
    \begin{equation}
        \begin{split}
            \int_0^1 {\rho}^{\sharp}_\tau(t, x) \varphi(x) \,dx &= \int_0^1 {\rho}^{\sharp}_\tau(0, x) \varphi(x) \,dx + \int_0^t\int_0^1 {m}^{\sharp}_\tau \partial_x \varphi\,dx\,ds + \epsilon \int_0^t \int_0^1 {\rho}^{\sharp}_\tau\partial_x^2 \varphi \,dx \,ds\\
            &+ \int_0^t \int_0^1 (\bar{m}^{\sharp}_\tau(s) - {m}^{\sharp}_\tau(s))\partial_x \varphi \,dx\,ds
            + \frac{t - t_n}{\tau} \int_t^{t_{n+1}}\int_0^1 \bar{m}^{\sharp}_\tau(s) \partial_x \varphi \,dx\,ds.
        \end{split}
    \end{equation}
    Denote $\displaystyle \mathcal{E}_1^\tau := \int_0^t \int_0^1 (\bar{m}_\tau(s) - {m}_\tau(s))\partial_x \varphi \,dx\,ds$, and $\displaystyle \mathcal{E}_2^\tau := \frac{t - t_n}{\tau} \int_t^{t_{n+1}}\int_0^1 \bar{m}_\tau(s) \partial_x \varphi \,dx\,ds$. We are going to show that $\mathcal{E}_1^\tau, \mathcal{E}_2^\tau \to 0$, as $\tau \to 0$, $\bar{\mathbb{P}}$-almost surely.
    
    We first estimate $\mathcal{E}_1^\tau$. Due to Theorem \ref{regularity_of_Uhat}, $\bar{U}_\tau \in L^2(\Omega; L^\infty(0, T; H^2(0,1)))$. Due to Lemma \ref{L_infty_estimates} and  Lemma \ref{lem_Linfty_eta_gradient_eta}, we have that $|\bar{U}_\tau(t, x)|, |\tilde{U}_\tau(t, x)| \leq C(\|U_0\|_{L^\infty}, \gamma)$, and $|\nabla (\frac{\bar{m}^2_\tau}{\bar{\rho}_\tau}(t, x)+ p(\bar{\rho}_\tau)(t, x)| \leq C(\|U_0\|_{L^\infty}, \gamma)$, $\mathbb{P}$-almost surely. We then use the It\^{o} isometry on the stochastic integral to obtain:
    \begin{equation}
        \begin{split}
            &\mathbb{E}\left|\int_0^1(\bar{m}_\tau^\sharp(s) - {m}_\tau^\sharp(s))\partial_x \varphi \,dx \right|^2\\
            = & \left(\frac{s - \tau \lfloor \frac{s}{\tau}\rfloor }{\tau}\right)^2\mathbb{E}\left|\int_0^1\int_s^{\tau \lceil \frac{s}{\tau}\rceil}\left[\partial_x(\frac{(\bar{m}_\tau^\sharp)^2}{\bar{\rho}_\tau^\sharp} + p(\bar{\rho}_\tau^\sharp)) - \alpha \bar{m}_\tau^\sharp + \epsilon \partial_x^2 \bar{m}_\tau^\sharp \right]\,ds'\partial_x\varphi \,dx + \int_0^1\int_{\tau \lfloor \frac{s}{\tau}\rfloor}^s \sigma_\epsilon(x, \tilde{U}_\tau^\sharp)\,dW(s')\partial_x \varphi \,dx\right|^2
        \end{split}
    \end{equation}
    {\color{orange}To estimate the first integral, we note that by the definition of $\bar{U}_\tau^\sharp$ in \eqref{defn_Ubar_sharp}, and equivalence of laws between $\bar{U}_\tau^\sharp$ and $\bar{U}_\tau$. We then have $\mathbb{E}\|\bar{U}_\tau^\sharp\|^2_{L^\infty(0, T; H^2(0,1))} \leq C$, for some constant $C$ independent of $\tau$. To estimate the It\^{o} integral, we use the It\^{o} isometry, and the Lipschitz continuity of $\sigma_\epsilon^2(x, \tilde{U}_\tau^\sharp) \leq A_0 |\tilde{m}_\tau^\sharp|^2$, and Lemma \ref{L_infty_estimates}. Altogether, we obtain:}
    \begin{equation}
        \begin{split}
             \mathbb{E}\left|\int_0^1(\bar{m}_\tau^\sharp(s) - {m}_\tau^\sharp(s))\partial_x \varphi \,dx \right|^2 & \leq \left(\frac{s - \tau \lfloor \frac{s}{\tau}\rfloor }{\tau}\right)^2 (\tau \lceil \frac{s}{\tau}\rceil -s )^2 C^2 + \left(\frac{s - \tau \lfloor \frac{s}{\tau}\rfloor }{\tau}\right)^2 (s - \tau \lfloor \frac{s}{\tau}\rfloor)C^2 \\
            & \leq C (\tau^2 + \tau).
        \end{split}
    \end{equation}
    Therefore, for every $s \in [\tau \lfloor \frac{s}{\tau}\rfloor, \tau \lceil \frac{s}{\tau}\rceil] $, $\mathbb{E}|\int_0^1 (\bar{m}_\tau^\sharp(s) - {m}_\tau^\sharp(s))\partial_x\varphi\,dx|^2 \to 0$ as $\tau \to 0$, which implies that $\mathbb{E}|\mathcal{E}_1^\tau|^2 \to 0$, as $\tau \to 0$. Therefore, up to a subsequence $\{\tau_n\}$, $\mathcal{E}_1^{\tau_n} \to 0$, $\tilde{\mathbb{P}}$-almost surely. Next, we estimate $\mathcal{E}_2^\tau$. 
    Due to Lemma \ref{L_infty_estimates} and equivalence of laws between $\bar{U}_\tau^\sharp$ and $\bar{U}_\tau$, $|\bar{m}_\tau^\sharp(t, x)| \leq C(\|U_0\|_{L^\infty}, \gamma)$, $\tilde{\mathbb{P}}$-almost surely. Therefore we have for the second term:
    \begin{equation}
        \begin{split}
        \frac{t - t_n}{\tau} \int_t^{t_{n+1}}\int_0^1 \bar{m}_\tau^\sharp \partial_x \varphi \,dx\,ds \leq \frac{t - t_n}{\tau} |t_{n+1}-t|C \leq C\tau,\qquad \tilde{\mathbb{P}} \text{ -almost surely.}
    \end{split}
    \end{equation}
    This implies that $\mathcal{E}_2^\tau \to 0$, as $\tau \to 0$, $\tilde{\mathbb{P}}$-almost surely.\\
    
    By Theorem \ref{pass_tau_0}, ${U}_\tau^\sharp \to U_\epsilon$ in $C(0, T; H^1(0, 1))$, $\tilde{\mathbb{P}}$-almost surely. Moreover, since $\mathcal{E}_1^{\tau} + \mathcal{E}_2^{\tau} \to 0$, $\tilde{\mathbb{P}}$-almost surely, we can then pass $\tau \to 0 $ in \eqref{continuity_eqn_rhosharp} to obtain that the limiting random variable, $\rho_\epsilon$, satisfies the continuity equation weakly $\tilde{\mathbb{P}}$-almost surely:
    \begin{equation}
        \int_0^1 \rho_\epsilon(t, x) \varphi(x) \,dx = \int_0^1 \rho_{\epsilon 0}(x) \varphi(x) \,dx + \int_0^t\int_0^1 m_\epsilon(s) \partial_x \varphi\,dx\,ds + \epsilon \int_0^t \int_0^1 \rho_\epsilon \partial_x^2 \varphi \,dx \,ds.
    \end{equation}
    \fi
    \end{proof}

    \fi

    \if 1 = 0
    Namely, by combining Lemma \ref{limitcontinuity} and Proposition \ref{bvpositivity}, we see that to prove the positivity result in Proposition \ref{rhoepspos}, the key step is control of the integral $\displaystyle \int_{0}^{T} \int_{0}^{1} \rho(\partial_{x}u)^{2} dx dt$, which we obtain in the following lemma:

    \begin{lem}\label{rhodxu}
        We have the convergence:
        \begin{equation*}
        \Big(\bar{\rho}^{\sharp}_{\tau}\Big)^{1/2} \partial_{x}\bar{u}^{\sharp}_{\tau} \rightharpoonup \rho_{\epsilon}^{1/2} \partial_{x} u_{\epsilon} \quad \text{ weakly in } L^{2}([0, T] \times [0, 1]), \quad \bar{\mathbb{P}}\text{-almost surely.}
        \end{equation*}
        Hence, the limiting fluid density $\rho_{\epsilon}$ and fluid velocity $u_{\epsilon}$ satisfy the bound:
        \begin{equation*}
        \mathbb{E} \int_{0}^{T} \int_{0}^{1} \rho_{\epsilon} (\partial_{x}u_{\epsilon})^{2} dx dt \le C(\epsilon, T) < \infty.
        \end{equation*}
    \end{lem}
    
    \begin{proof}
    Before showing the weak in $L^2(Q_T)$ convergence, we first aim to obtain a $\bar{\mathbb{P}}$-almost sure bound (uniformly in $\tau$) for $\|\sqrt{\bar{\rho}_\tau^\sharp}\|_{H^1}$. We test \eqref{continuity_eqn_rhosharp} with $\log(\bar{\rho}_\tau^\sharp)$, and integrate by parts on $[0,T]\times [0, 1]$ to obtain:
    \begin{align}
        -\int_0^1 (\log(\bar{\rho}_\tau^\sharp(T)) -1)\bar{\rho}_\tau^\sharp(T) \,dx + \int_0^1 (\log(\bar{\rho}^\sharp_\tau(0))-1)\bar{\rho}_\tau^\sharp(0)\,dx &+ \int_0^T\int_0^1 (\partial_x \bar{\rho}_\tau^\sharp)\bar{u}_\tau^\sharp \,dx\,dt \nonumber \\
        &= \epsilon  \int_0^T \int_0^1 \frac{(\partial_x \bar{\rho}_\tau^\sharp)^2}{\bar{\rho}_\tau^\sharp}\,dx\,dt.
    \end{align}
    Then, note that $\bar{\rho}^{\sharp}_{\tau}$ is uniformly bounded above independently of $\tau$, namely $0 < \bar{\rho}^{\sharp}_{\tau} \le M$ independently of $\tau$, due to Theorem \ref{pass_tau_0}. Furthermore, there exists a constant $C$ such that $|z\log(z) - z| \le C_{M}$ for all $0 < z \le M$. Thus, we conclude that: 
    \begin{equation*}
        \left|-\int_0^1 (\log(\bar{\rho}_\tau^\sharp(T)) -1)\bar{\rho}_\tau^\sharp(T) \,dx\right| \leq C, \qquad \tilde{\mathbb{P}} \text{-almost surely.}
    \end{equation*}
    Moreover, due to Theorem \ref{regularity_of_Uhat}, the equivalence of laws between $\bar{\rho}_\tau$ and $\bar{\rho}_\tau^\sharp$ given by Proposition \ref{samelimit}, and Theorem \ref{pass_tau_0}, we have that:
    \begin{equation*}
        \mathbb{E} \int_0^T\int_0^1 (\partial_x \bar{\rho}_\tau^\sharp)\bar{u}_\tau^\sharp \,dx\,dt \leq C_{T} \|\bar{u}^{\sharp}_{\tau}\|_{L^{\infty}(\bar\Omega \times [0, T] \times [0, 1]} \mathbb{E} \|\bar{\rho}_\tau^\sharp\|_{C(0, T; H^1(0,1))} \leq C_{T}.
    \end{equation*}
    Therefore, after realizing that $\frac{\partial_x \bar{\rho}_\tau^\sharp}{\sqrt{\bar{\rho}_\tau^\sharp}} = 2\partial_x\sqrt{\bar{\rho}_\tau^\sharp}$, we can conclude that:
    \begin{equation}
        \epsilon \mathbb{E} \int_0^T \int_0^1 \frac{(\partial_x \bar{\rho}_\tau^\sharp)^2}{\bar{\rho}_\tau^\sharp}\,dx\,dt = 4\epsilon \mathbb{E} \|\partial_x \sqrt{\bar{\rho}_\tau^\sharp}\|^2_{L^2(Q_T)}\leq C,
    \end{equation}
    for some constant $C=C(U_{\epsilon0}, \gamma, T)$. Due to the continuity of $\bar{\rho}_\tau^\sharp$ and Proposition \ref{samelimit}, we have that
    \[\partial_x \sqrt{\bar{\rho}_\tau^\sharp} = \frac{\partial_x \bar{\rho}_\tau^\sharp}{2(\bar{\rho}_\tau^\sharp)^{1/2}} \to  \frac{\partial_x \rho_\epsilon}{2(\rho_\epsilon)^{1/2}} = \partial_x \sqrt{\rho_\epsilon}, \;\; \text{in } C(0, T; H^1(0, 1)), \qquad \tilde{\mathbb{P}} \text{ -almost surely}.\] 
    Then, we can conclude  through the Dominated Convergence Theorem that 
    \begin{equation}\label{sqrtrho_convH1}
        \sqrt{\bar{\rho}_\tau^\sharp} \to \sqrt{\rho}_\epsilon, \;\; \text{ in }H^1(Q_T), \qquad \tilde{\mathbb{P}} \text{ -almost surely.}
    \end{equation}
     We now proceed to show that $(\bar{\rho}_\tau^\sharp)^{1/2}\partial_x \bar{u}_\tau^\sharp \rightharpoonup (\rho_\epsilon)^{1/2}\partial_x u_\epsilon$, weakly in $L^2(Q_T)$. First we make the observation that
     since $\bar{U}_\tau^\sharp$ solves the deterministic subproblem \eqref{subproblem_1}, $\bar{U}_\tau^\sharp$ satisfies the following entropy balance equation $\tilde{\mathbb{P}}$-almost surely, where we take the spatial test function $\varphi = 1$:
    \[ \int_0^1 \eta(\bar{U}_\tau^\sharp(T))\,dx = \int_0^1 \eta(U_{\epsilon 0}) - \int_0^T\int_0^1 \alpha \bar{m}_\tau^\sharp \partial_m \eta(\bar{U}_\tau^\sharp)\,dx\,dt - \epsilon\int_0^T\int_0^1\langle \nabla^2\eta(\bar{U}_\tau^\sharp)\partial_x\bar{U}_\tau^\sharp, \partial_x\bar{U}_\tau^\sharp\rangle\,dx\,dt. \]
    By Lemma \ref{L_infty_estimates} and Lemma \ref{lem_Linfty_eta_gradient_eta}, and the convexity of $\eta$, we can conclude that 
    \begin{equation}\label{hessianbound}
        \epsilon\int_0^T\int_0^1\langle \nabla^2\eta(\bar{U}_\tau^\sharp)\partial_x\bar{U}_\tau^\sharp, \partial_x\bar{U}_\tau^\sharp\rangle\,dx\,dt \leq C,
    \end{equation}
    for some constant $C=C(U_0, \gamma, \alpha, \epsilon, T)$, $\tilde{\mathbb{P}}$-almost surely. We then pick $\eta = \eta_E = \frac{1}{2}\rho u^2 + \frac{\kappa}{\gamma - 1}\rho^\gamma$. With a brief calculation,  we obtain:
    \[\nabla^2\eta_E(\bar{U}_\tau^\sharp)\langle \partial_x \bar{U}_\tau^\sharp, \partial_x \bar{U}_\tau^\sharp\rangle = \kappa \gamma (\bar{\rho}_\tau^\sharp)^{\gamma -2}(\partial_x\bar{\rho}_\tau^\sharp)^2 + \bar{\rho}_\tau^\sharp (\partial_x\bar{u}_\tau^\sharp)^2.\]
    In light of \eqref{hessianbound}, we have that
    \[\int_0^T\int_0^1 \bar{\rho}_\tau^\sharp(\partial_x \bar{u}_\tau^\sharp)^2\,dx\,dt = \|\sqrt{\bar{\rho}_\tau^\sharp}\partial_x \bar{u}_\tau^\sharp\|_{L^2(0,1)}^2 \leq C, \]
    for some constant $C= C(U_{0}, \gamma, \alpha,\epsilon, T)$ independent of $\tau$, $\tilde{\mathbb{P}}$-almost surely. Therefore, $\{\sqrt{\bar{\rho}_\tau^\sharp}\partial_x \bar{u}_\tau^\sharp\}_\tau$ is weakly compact in $L^2(Q_T)$, $\tilde{\mathbb{P}}$-almost surely, namely, there exists $f \in L^2(Q_T)$, such that for every $\varphi \in L^2(Q_T)$, we have
    \[\int_0^T\int_0^1 \sqrt{\bar{\rho}_\tau^\sharp}\partial_x \bar{u}_\tau^\sharp \varphi \,dx\,dt \to \int_0^T\int_0^1 f\varphi \,dx\,dt, \]
    $\tilde{\mathbb{P}}$-almost surely. Then, we take $\varphi = 1$, and integrate by parts on the left hand side we obtain: $\int_0^T\int_0^1 \sqrt{\bar{\rho}_\tau^\sharp}\partial_x \bar{u}_\tau^\sharp  \,dx\,dt = -\int_0^T\int_0^1 \partial_x \sqrt{\bar{\rho}_\tau^\sharp} \cdot \bar{u}_\tau^\sharp \,dx\,dt$. By the Triangle's inequality, we have that:
  \begin{equation}
      \int_0^T\int_0^1 \left|\partial_x \sqrt{\bar{\rho}_\tau^\sharp} \cdot \bar{u}_\tau^\sharp - \partial_x\sqrt{\rho_\epsilon} \cdot u_\epsilon\right|\,dx\,dt \leq \int_0^T\int_0^1 \left|\left(\partial_x \sqrt{\bar{\rho}_\tau^\sharp} -\partial_x\sqrt{\rho_\epsilon}\right) \bar{u}_\tau^\sharp\right|\,dx\,dt + \int_0^T\int_0^1 \left|\partial_x\sqrt{\rho_\epsilon} \left(\bar{u}_\tau^\sharp - u_\epsilon\right)\right|\,dx\,dt.
  \end{equation}
  Since $\partial_x\sqrt{\bar{\rho}_\tau^\sharp} \to \partial_x \sqrt{\rho_\epsilon}$ in $L^2(Q_T)$, $\tilde{\mathbb{P}}$-almost surely by \eqref{sqrtrho_convH1}, and $\bar{u}_\tau^\sharp \rightharpoonup u_\epsilon$ weakly* in $L^\infty(Q_T)$, $\tilde{\mathbb{P}}$-almost surely, due to Theorem \ref{pass_tau_0}, we have that
  \[\int_0^T\int_0^1 \left|\left(\partial_x \sqrt{\bar{\rho}_\tau^\sharp} -\partial_x\sqrt{\rho_\epsilon}\right) \bar{u}_\tau^\sharp\right|\,dx\,dt \to 0, \qquad \int_0^T\int_0^1 \left|\partial_x\sqrt{\rho_\epsilon} \left(\bar{u}_\tau^\sharp - u_\epsilon\right)\right|\,dx\,dt \to 0,\]
  $\tilde{\mathbb{P}}$-almost surely. Therefore,
  \[\int_0^T\int_0^1\partial_x \sqrt{\bar{\rho}_\tau^\sharp}\cdot \bar{u}_\tau^\sharp\,dx\,dt \to \int_0^T\int_0^1 \partial_x \sqrt{\rho_\epsilon}\cdot u_\epsilon \,dx\,dt, \qquad \tilde{\mathbb{P}} \text{ -almost surely.}\]
  After integration by parts, we have that 
   \[\int_0^T\int_0^1\sqrt{\bar{\rho}_\tau^\sharp}\cdot \partial_x \bar{u}_\tau^\sharp\,dx\,dt \to \int_0^T\int_0^1  \sqrt{\rho_\epsilon}\cdot \partial_x u_\epsilon \,dx\,dt, \qquad \tilde{\mathbb{P}} \text{ -almost surely.}\]
   By the uniqueness of limit, we can then identify the weak $L^2$ limit $f$, to be $\sqrt{\rho_\epsilon}\partial_x u_\epsilon$, and consequently, $\mathbb{E}\|\sqrt{\rho_\epsilon}\partial_x u_\epsilon\|_{L^2(Q_T)}^2 \leq C(U_0, \gamma, \alpha, \epsilon, T)$. We thus conclude the proof of this lemma.
    \end{proof}
    Idea, test continuity equation with $\log(\bar{\rho}^{\sharp}_{\tau}$ to get a uniform almost sure bound on $\|\sqrt{\rho}\|_{H^{1}}$. Then, can show strong convergence in $L^{2}$ by dominated convergence theorem and by uniqueness the limit must be $\sqrt{\rho_{\epsilon}}$. (rho bar is continuous). In addition, by weak convergence, $\|\sqrt{\rho_{\epsilon}}\|_{H^{1}}$ must also be finite. So $\rho_{\epsilon} > 0$ almost surely and almost everywhere. Then, get that $\mathbb{E} \int_{Q_{T}} \rho (\partial_{x} u)^{2}$ is finite by using $\sqrt{\rho}$ converges in $H^{1}$ and $u$ converges weakly star in $L^{\infty}$ (integrate by parts).
    
    Using a similar procedure as in Theorem A.2 in \cite{berthelin_stochastic_2019}, one can show that for every $\omega \in \tilde{\Omega}$, there exists a constant $c_\epsilon(\omega)>0$ such that $\rho_\epsilon \geq c_\epsilon(\omega)>0$. Since ${\rho}_\tau \to \rho_\epsilon$, and ${m}_\tau \to m_\epsilon$ in $C(0, T; H^1(0,1))$, $\tilde{\mathbb{P}}$-almost surely, we have that ${\rho}_\tau \geq c_\epsilon(\omega)$ for every $\tau$, $\tilde{\mathbb{P}}$-almost surely. We can then establish the locally Lipschitz property of $\nabla^2 \eta$: for almost every $\omega \in \tilde{\Omega}$, there exists a positive constant $L_\eta(\omega)$, such that whenever $U_1, U_2 \in [c(\omega), C(\omega)] \times [-C(\omega)\times C(\omega)]$, then
    \[|\nabla^2\eta(U_1) - \nabla^2\eta(U_2) | \leq L_\eta(\omega) |U_1 - U_2|\]
    
    \fi

    \section{Appendix}\label{AppendixSection}
    \subsection{Boundary condition}\label{appendix_bc}
    In this section, we rigorously justify how the homogeneous Dirichlet boundary condition is imposed on the momentum $m$ in an appropriate ``distributional sense". More precisely, we show the following proposition.
    \begin{prop}
           Let $U$ be an martingale $L^\infty$ weak entropy solution to \eqref{problem_matrix}. For each point on the boundary $x\in \{0, 1\}$, there exists a unique weak boundary datum $\mathcal{T}_x \in L^\infty_{loc}(0, T; \mathbb{R}^2)$, such that for every $\varphi \in C_c^\infty([0,T)\times \mathbb{R}; \mathbb{R}^2)$,
       \[\begin{split}
           &- \int_0^T\int_0^1 \rho \cdot \partial_t \varphi \,dx\,dt - \int_0^T\int_0^1 m \cdot \partial_x \varphi \,dx\,dt \\
           & \qquad = \int_0^1 \rho_0(x)\varphi(0, x)\,dx + \int_0^T \mathcal{T}_0(t) \varphi(t, 0)\,dt - \int_0^T \mathcal{T}_1(t) \varphi(t, 1)\,dt ,
       \end{split}\]
       $\bar{\mathbb{P}}$-almost surely. Moreover, for every $\chi \in C^1([0, T])$
       \[\int_0^T \mathcal{T}_x(t)\chi(t) \,dt = 0, \;\; x \in \{0,1\}, \;\; \bar{\mathbb{P}} \text{-almost surely.}\]
    \end{prop}
         \begin{proof} The proof largely follows the idea in the deterministic setting (See Section 4 \cite{heidrich_global_1994} and Section 4 \cite{pan_initial_2008}).\\
         Let $U$ be an $L^\infty$ martingale solution to \eqref{problem}. {We define the distribution $v$, which acts on smooth and compactly supported functions $\varphi \in C^\infty_c([0,T)\times \mathbb{R}; \mathbb{R})$
         \begin{equation}\label{vvarphi1}
             v(\varphi) := -\int_0^T\int_0^1\rho \partial_t \varphi \,dx\,dt - \int_0^T\int_0^1 m \partial_x \varphi \,dx\,dt.
         \end{equation} 
         }
        Note that $v(\varphi)$ can also be represented as the following for every $\varphi \in C^\infty_c([0,T)\times \mathbb{R}; \mathbb{R})$:
        \begin{equation}\label{vvarphi2}
            v(\varphi) := \int_0^1\rho_0(x) \varphi(0,x) \,dx + \int_0^T m(t, 0) \varphi(t,0) \,dt - \int_0^T m(t, 1) \varphi(t,1) \,dt.
        \end{equation}
         Then, using $v$, we define the following distributions for test functions $\chi \in C_c^\infty([0, T); \mathbb{R}), \psi \in C_c^\infty(\mathbb{R};\mathbb{R})$: 
         \[\begin{split}
            &\mathcal{T}_i(\psi) = v(\zeta_i \otimes \psi) - \psi(0)v(\zeta_i \otimes \zeta_0) - \psi(1)v(\zeta_i \otimes \zeta_1),\\
             &\mathcal{T}_0(\chi) = -v(\chi \otimes \zeta_0),\\
             &\mathcal{T}_1(\chi) = -v(\chi \otimes \zeta_1),
         \end{split}\]
         where $f_1 \otimes f_2 := f_1(t)f_2(x)$, and the functions $\zeta_i, \zeta_0, \zeta_1 \in C_c^\infty(\mathbb{R}; \mathbb{R})$ satisfy the following conditions:
         \[\begin{split}
              &\zeta_i(0) = 1, \zeta_i(T) = 0, \;\; \\
             &\zeta_0(0) = 1, \zeta_0(1) = 0, \;\; \zeta_1(0) = 0, \zeta_1(1) = 1.
         \end{split}\]
          Since the function $(\chi - \chi(0)\zeta_i)\otimes (\psi - \psi(0)\zeta_0 - \psi(1)\zeta_1)(t, x) = 0$, when $t=0$ or $x = 0$ or $x=1$, we obtain through \eqref{vvarphi2} that $v((\chi - \chi(0)\zeta_i)\otimes (\psi - \psi(0)\zeta_0 - \psi(1)\zeta_1))=0$. Then, we use the linearity of the distribution $v$ to deduce the following identity:
         \[
            v(\chi \otimes \psi) = \chi(0)\mathcal{T}_i(\psi)+\psi(0)\mathcal{T}_0(\chi) - \psi(1)\mathcal{T}_1(\chi).
         \]
         By Theorem 5.1 of \cite{heidrich_global_1994}, $\mathcal{T}_i, \mathcal{T}_0, \mathcal{T}_1 $ are the unique distributions such that the above identity holds.
         
         From here, we show that $\mathcal{T}_i, \mathcal{T}_0, \mathcal{T}_1$ are function-valued; in particular, $\mathcal{T}_i \in L^\infty_{loc}(0,1)$, and $\mathcal{T}_0, \mathcal{T}_1 \in L^\infty_{loc}(0,T).$ 
         
         To do this, we follow a similar procedure as in Section 5 of \cite{heidrich_global_1994}, and aim to obtain that for all test functions $\chi \in C_c^\infty((0, T); \mathbb{R})$, $\psi \in C_c^\infty((0, 1); \mathbb{R})$:
         \[\begin{split}
             &|\mathcal{T}_i(\psi)|\leq \|\rho\|_{L^\infty}\|\psi\|_{L^1(0, 1)} \leq C, \\
             & |\mathcal{T}_0(\chi)| \leq \|m\|_{L^\infty}\|\chi\|_{L^1(0, T)} \leq C, \;\; |\mathcal{T}_1(\chi)| \leq \|m\|_{L^\infty} \|\chi\|_{L^1(0, T)} \leq C,
         \end{split}\]
         for all $k\in \mathbb{Z}_+$. We will demonstrate this for $|\mathcal{T}_0(\chi)|$. We pick a sequence $\psi_n \in C_c^1(\mathbb{R};\mathbb{R})$, such that for, $\psi_n$ increases on $(-\infty, 0]$ and decreases on $[0, \infty)$, satisfying $\psi_n = 1$ on $[-\frac{1}{n}, \frac{1}{n}]$, and $\psi_n = 0$, outside of $(-\frac{2}{n}, \frac{2}{n})$. We pick $\chi \in C_c^1((t_1, t_2))$, then \eqref{vvarphi1} gives:
         \begin{equation}
             \mathcal{T}_0(\chi) = - v(\chi \otimes \psi_n) = -\int_0^T\int_0^1 \rho(x, t) \chi'(t)\psi_n(x)\,dx\,dt - \int_0^T\int_0^1 m(x, t)\chi(t)\psi_n'(x)\,dx\,dt.
         \end{equation}
         By using H\"{o}lder's inequality, and passing $n\to \infty$ using $\|\psi_n\|_{L^1} \leq \frac{4}{n}$, and $\|\psi_n'\|_{L^1} = 2$, we have that 
         \[|\mathcal{T}_0(\chi)| \leq \|\rho\|_{L^\infty(Q_T)}\|\chi'\|_{L^1}\|\psi_n\|_{L^1} + \|m\|_{L^\infty(Q_T)}\|\chi\|_{L^1}\|\psi_n'\|_{L^1} \lesssim \|m\|_{L^\infty(Q_T)}\|\chi\|_{L^1}. \]
         Therefore $\mathcal{T}_0(\chi)$ is bounded and since $C_c^\infty$ is dense in $L^1_{loc}$, and $L^\infty_{loc}$ is the dual of $L^1_{loc}$, we can conclude that $\mathcal{T}_0 \in L^\infty_{loc}(0, T)$. Following a similar procedure, one can show that $\mathcal{T}_i \in L^\infty_{loc}(0,1)$, $\mathcal{T}_1 \in L^\infty_{loc}(0, T)$.
         
         We then use \eqref{vvarphi2} and express the right hand side using the solution to the parabolically regularized problem, and we have the following identity holds for all $\varphi \in C_c^2([0,T)\times \mathbb{R})$:
         \[\begin{split}
             &v(\varphi) = \int_0^1 \varphi(0, x)\mathcal{T}_i(x)\,dx + \int_0^T \varphi(t, 0)\mathcal{T}_0(t)\,dt - \int_0^T \varphi(t, 1)\mathcal{T}_1(t)\,dt \\
             =& \lim_{\epsilon\to 0}\int_0^1 \rho_{\epsilon 0}(x)\varphi(0, x)\,dx - \lim_{\epsilon\to 0}\int_0^1 \rho_\epsilon(T, x)\varphi(T, x)\,dx + \lim_{\epsilon\to 0}\int_0^T m_{\epsilon}(0, t)\varphi(t, 0)\,dx - \lim_{\epsilon\to 0}\int_0^T m_\epsilon(t, 1)\varphi(, x)\,dx.
         \end{split} \]
         But note that for the approximate problem, the boundary condition is satisfied in a classical sense, namely $m_\epsilon(t, 0) = m_\epsilon(t, 1) = 0$ since $U_\epsilon \in C_{loc}(0, T; H^2(0,1))$ so the last two terms vanishes. And 
         then we take, $\varphi(x, t) = \zeta(x)\chi(t)\in C_c^1(\mathbb{R}^2)$, where $\zeta, \chi \in C_c^1(\mathbb{R})$ satisfy $\zeta(0) = 1$, $\zeta(1)=0$, $\chi(0)=\chi(T)=0$, then we have that $\varphi(0, x)= \varphi(T, x) = \varphi(t, 1)=0 $. Therefore, the above equation becomes:
         \[v(\varphi) = \int_0^T\chi(t)\mathcal{T}_0(t)\,dt = 0.\]
        With a similar type of argument, but taking $\zeta(0)= 0$, and $\zeta(1) = 1$, one obtain that
        \[v(\varphi) = \int_0^T\chi(t)\mathcal{T}_1(t)\,dt = 0.\]
        We thus conclude the proof of this proposition.
    \end{proof} 

    \subsection{Lemma on the pressure law}\label{Appendix2}
        In this section, we cite the following two lemma which play important roles in establishing connection between the pressure law and the difference between the density $\rho$ and the total initial mass $\rho_*$ in Section \ref{sec_long_time_behavior}.
		\begin{lem}[Lemma 5.2 \cite{pan_initial_2008}]\label{lem5.2}
		Let $0 \leq a, b \leq M < \infty$, $\gamma \geq 1$, {and $p(a) :=a^\gamma$}. There is a positive constant $C$, depending on $M$ and $\gamma$, such that 
		\[p(a)-p(b) - p'(b)(a - b) \leq C[p(a)-p(b)](a - b).\]
	   \end{lem}
    	\begin{lem}[Lemma 4.1 \cite{huang_asymptotic_2006}]\label{lem4.1}
    		Let $0 \leq a, b \leq M < \infty$, $\gamma \geq 1$, {and $p(a) := a^\gamma$}. Then, there are positive constants $c$ and $C$ depending only on $M$ and $\gamma$, such that 
    		\begin{enumerate}
    			\item $|a-b|^{\gamma+1} \leq (a-b)(p(a)-p(b))$,
    			\item $c |a-b|^2 \leq [p(a)-p(b)-p'(b)(a-b)] \leq C|a-b|, \text{ if } 1 < \gamma \leq 2$,
    			\item $c |a-b|^\gamma \leq [p(a)-p(b)-p'(b)(a-b)] \leq C|a-b|, \text{ if }  \gamma > 2.$
    		\end{enumerate}
    	\end{lem}

\bigskip
\noindent
{\bf{Acknowledgments.}} The authors would like to acknowledge the support by the Mathematics Department at UC Berkeley provided for Rongyi Dai, and the following support by the National Science Foundation: (1) Award DMS-2303177 -- the Mathematical Sciences Postdoctoral Research Fellowship --  awarded to  Jeffrey Kuan, (2) Award DMS-2553666 (transferred from DMS-2407197) awarded to Krutika Tawri, (3) Awards DMS-2408928, DMS-2247000 awarded to Sun\v{c}ica \v{C}ani\'{c}, and (4) Awards DMS-2231533, DMS-2008568 awarded to Konstantina Trivisa. The work of Sun\v{c}ica \v{C}ani\'c was additionally supported in part by the U.S. Department of Energy, Office of Science, Office of Advanced Scientific Computing Research's Applied Mathematics Competitive Portfolios program under Contract No. AC02-05CH11231.

	\bibliography{SCEuler_ref.bib}
	\bibliographystyle{siam}

\end{document}